\documentclass[11pt,a4paper]{amsart}

\usepackage[utf8]{inputenc}
\usepackage{amsmath,amssymb,amsthm}
\usepackage{color}
\usepackage[colorlinks, citecolor=green!70!black]{hyperref}
\usepackage{enumitem}
\usepackage[margin=3cm]{geometry}

\usepackage{amsfonts}
\usepackage{graphicx,psfrag,subfigure}
\usepackage{epstopdf}
\usepackage{esvect}
\usepackage[spanish, english]{babel}
\usepackage{import}
\usepackage{pdfpages}
\usepackage{xcolor}
\usepackage{tikz}
\usetikzlibrary{decorations.markings}

\newcommand{\R}{\mathbb{R}}
\newcommand{\Q}{\mathbb{Q}}
\newcommand{\Sp}{\mathbb{S}}
\newcommand{\Z}{\mathbb{Z}}
\newcommand{\T}{\mathbb{T}}
\newcommand{\C}{\mathbb{C}}
\newcommand{\N}{\mathbb{N}}
\newcommand{\dist}{\mathrm{dist}}
\newcommand{\dom}{\operatorname{dom}}
\renewcommand{\H}{\mathbb{H}}
\newcommand{\D}{\H^2}
\newcommand{\F}{\mathcal{F}}
\newcommand{\tx}{{\widetilde{x}}}
\newcommand{\ty}{\widetilde{y}}
\newcommand{\tf}{\widetilde{f}}
\newcommand{\tS}{\widetilde{S}}
\newcommand{\tgamma}{\widetilde{\gamma}}
\newcommand{\tdist}{\widetilde{\dist}}
\newcommand{\dgamma}{\dot{\gamma}}

\newcommand{\Homeo}{\operatorname{Homeo}}
\newcommand{\wt}{\widetilde}
\newcommand{\Mo}{\mathrm{M}}
\newcommand{\dd}{\,\mathrm{d}}
\newcommand{\M}{\mathcal{M}_{\vartheta>0}^{\textnormal{erg}}(f)}
\DeclareMathOperator{\diam}{diam}
\newcommand{\Id}{\operatorname{Id}}

\newcommand{\pr}{\operatorname{pr}}
\newcommand{\I}{I}
\newcommand{\Hy}{\mathbb{H}}
\newcommand{\cl}{\mathcal{N}}
\newcommand{\varep}{\varepsilon}
\newcommand{\card}{\operatorname{Card}}
\newcommand{\sing}{\operatorname{Sing}}
\newcommand{\len}{\operatorname{len}}
\newcommand{\1}{\mathbf{1}}

\newcommand{\inte}{\operatorname{int}}
\newcommand{\conv}{\operatorname{conv}}
\newcommand{\wh}{\widehat}
\newcommand{\dotLambda}{\dot{\Lambda}}
\newcommand{\dotgamma}{\dot{\gamma}}

\theoremstyle{remark}
\newtheorem{remark}{Remark}[section]

\theoremstyle{definition}
\newtheorem{definition}[remark]{Definition}

\theoremstyle{plain}
\newtheorem{theorem}[remark]{Theorem}
\newtheorem{lemma}[remark]{Lemma}
\newtheorem{proposition}[remark]{Proposition}
\newtheorem*{proposition*}{Proposition}
\newtheorem{coro}[remark]{Corollary}
\newtheorem{claim}[remark]{Claim}

\newtheorem{teorema}{Theorem}
\newtheorem{prop}[teorema]{Proposition}
\newtheorem{corol}[teorema]{Corollary}

\numberwithin{equation}{section}
\counterwithout{figure}{section}

\title{Geodesic tracking and the shape of ergodic rotation sets}
\author{Alejo García-Sassi}
\address{A. García-Sassi: Instituto de Matem\'atica y Estad\'istica, Facultad de Ingenier\'ia, Universidad
de la Rep\'ublica, Julio Herrera y Reissig 565, 11300  Montevideo, Uruguay}
\author{Pierre-Antoine Guihéneuf}
\address{P.-A. Guihéneuf: Institut de Math\'ematiques de Jussieu-Paris Rive Gauche, IMJ-PRG, Sorbonne Universit\'e, Universit\'e Paris-Diderot, CNRS, F-75005, Paris, France}
\author{Pablo Lessa}
\address{P. Lessa: CMAT, Facultad de Ciencias, Igu\'a 4225, 11400 Montevideo, Uruguay}

\begin{document}

\begin{abstract}
We prove a structure theorem for ergodic homological rotation sets of homeomorphisms isotopic to the identity on a closed orientable hyperbolic surface: this set is made of a finite number of pieces that are either one-dimensional or almost convex. The latter ones give birth to horseshoes; in the case of a zero-entropy homeomorphism, we show that there exists a geodesic lamination containing the directions in which generic orbits with respect to ergodic invariant probabilities turn around the surface under iterations of the homeomorphism.
The proof is based on the idea of \emph{geodesic tracking} of orbits that are typical for some invariant measure by geodesics on the surface, that allows to get links between the dynamics of such points and the one of the geodesic flow on some invariant subset of the unit tangent bundle of the surface.
\end{abstract}

\maketitle

\setcounter{tocdepth}{2}
\tableofcontents

\section{Introduction}

Consider a closed orientable hyperbolic surface \(S\) of genus $g$ and denote by $\tS$ its universal cover. This space is conformly equivalent to the hyperbolic disk $\Hy^2$. 
Denote $\dist$ a distance on \(S\) and $\tdist$ its lift to $\tS$ so that $(\tS,\tdist)$ is isometric to $\Hy^2$. 
As a hyperbolic space, the lift $\tS$ has a well defined boundary at infinity $\partial\tS$, which is homeomorphic to the circle.
We denote $\Homeo_0(S)$ the space of homeomorphisms of $S$ that are isotopic to the identity. Let $f \in \textnormal{Homeo}_0(S)$.
The group of covering transformations \(\Gamma: \tS \to \tS\) acts freely and discretely on \(\widetilde{S}\) by isometries which preserve orientation, and there exists a unique \emph{preferred lift} \(\widetilde{f}:\tS \to \tS\) of \(f\) commuting with the covering transformations:
\begin{equation}\label{eq:preferedliftequation}
 \tf \circ \gamma = \gamma \circ \tf\ \text{ for all }\gamma \in \Gamma.
\end{equation}

\subsection{Homological rotation sets}

Let us start by defining the notion of homological rotation set, due to Schwarzman \cite{MR88720}.

Consider the Abelianization \(\Gamma' = \Gamma/[\Gamma,\Gamma]\) of the group of covering transformations \(\Gamma\) as a \(\Z\)-module and identify the first real homology of \(S\) with
\[H_1(S,\R) = \R \otimes_{\Z} \Gamma',\]
(that is formal linear combinations with real coefficients of elements of \(\Gamma'\) where sums with integer coefficients are identified with the actual \(\Z\)-linear combination given by the Abelian group structure).
We recall that as $S$ is a closed surface of genus $g$, the set $H_1(S,\R)$ is a real vector space of dimension $2g$.

We equip the homology $H_1(S,\R)$ with a norm $\|\cdot\|$ and the intersection form $\wedge$ coming from the classical cup product on cohomology via Poincaré duality. This intersection form coincides with the classical algebraic intersection number in restriction to elements of $H_1(S,\Z)$ (\emph{e.g.}\ \cite{lellouch}).
Given \(a \in \Gamma\), we denote \([a] \in H_1(S,\R)\) its \emph{homology class}.

Fix a bounded and measurable fundamental domain \(D \subset \tS\) for the action of \(\Gamma\) on \(\tS\) and let \(x \mapsto \tx\) denote the lift of \(x \in S\) to \(D\).
For each \( y \in S\) let \(a_{y} \in \Gamma\) be the unique element such that \(a_y^{-1}\tilde f(\tilde y) \in D\).

The following definition is allowed by Birkhoff Ergodic Theorem. 

\begin{definition}
Given an ergodic \(f\)-invariant probability measure \(\mu\), the \emph{homological rotation vector} of \(\mu\) is
\begin{equation}\label{eq:homologyequation}
\rho_{H_1}(\mu) = \int_{S}[a_y]\dd\mu(y) = \lim\limits_{n \to +\infty}\frac{1}{n}\sum_{i=0}^{n-1}[a_{f^i(x)}],
\end{equation}
for \(\mu\)-almost every \(x \in S\). If $x\in S$ satisfies \eqref{eq:homologyequation}, we will denote $\rho_{H_1}(x) = \rho_{H_1}(\mu)$.
\end{definition}

\begin{remark}
This definition is independent of the choice of the fundamental domain $D$. To see this, note that the deck transformation $(a_xa_{f(x)}\cdots a_{f^{n-1}(x)})^{-1}$ sends $\tilde f^n(\tilde x)$ to $D$, and that two fundamental domains are at bounded Hausdorff distance, and hence the sums on the right of \eqref{eq:homologyequation} associated to two different fundamental domains only differ by a constant independent of $n$ and $x$ (see also Remark~\ref{Rem:indpFond}).
\end{remark}

\begin{definition}[Homological rotation sets]\label{def:ErgHomRotationSet}
Let $f \in \Homeo_0(S)$. 
The \emph{(homological) rotation set} $\rho_{H_1}(f)$ of $f$ is the set of vectors $r\in H_1(S,\R)$ such that there exists $(x_k)_k\in S^\N$ and $(n_k)_k\in\N^\N$ with $\lim_{k\to+\infty} n_k = +\infty$ and such that 
\[\lim\limits_{k \to +\infty}\frac{1}{n_k}\sum_{i=0}^{n_k-1}[a_{f^{i}(x_k)}] = r.\]
The \emph{ergodic (homological) rotation set} $\rho^{\textnormal{erg}}_{H_1}(f)$ of $f$ is
\[\rho^{\textnormal{erg}}_{H_1}(f) = \big\{\rho_{H_1}(\mu) \mid \mu \in \mathcal{M}^{\textnormal{erg}}(f)\big\},\] 
where $\mathcal{M}^{\textnormal{erg}}(f)$ is the set of $f$-invariant ergodic probability Borel measures.
\end{definition}

It is trivial that the inclusion $\rho^{\textnormal{erg}}_{H_1}(f)\subset \rho_{H_1}(f)$ always holds. We have moreover that \cite[Corollary 1.2]{pollicott} $\rho_{H_1}(f) \subset \conv\rho^{\textnormal{erg}}_{H_1}(f)$. In the case of the torus, there is a more precise description of the relations between these two sets.

Indeed, the same definition of rotation sets can be applied in the case of the two torus $\T^2$, with the difference that the obtained sets do depend on the chosen lift $\tilde f$ of $f$; however the sets associated to two different lifts only differ by an integer translation. 

In \cite{zbMATH04084609}, Misiurewicz and Ziemian show that for $f\in\Homeo_0(\T^2)$, the set $\rho_{H_1}(\tilde f)$ is convex. It allows to prove directly that $\rho_{H_1}(\tilde f) =\conv(\rho^{\textnormal{erg}}_{H_1}(\tilde f))$.  
In \cite{zbMATH00009916} the same authors prove that any point of the interior of $\rho_{H_1}(\tilde f)$ is realised as the rotation vector of some ergodic measure, in other words $\inte(\rho_{H_1}(\tilde f)) = \inte(\rho^{\textnormal{erg}}_{H_1}(\tilde f))$. 
Under the hypothesis that $\rho_{H_1}(\tilde f)$ has nonempty interior, Llibre and Mackay prove that the homeomorphism $f$ has positive topological entropy \cite{llibremackay} and Franks shows that any rational point of the interior of the rotation set is realised as the rotation vector of a periodic orbit \cite{MR0958891}.

If, however, the set $\rho_{H_1}(\tilde f)$ is a segment, it can happen that the set $\rho^{\textnormal{erg}}_{H_1}(\tilde f)$ is totally disconnected (think about a flow on the torus with a lot of essential Reeb components).

For higher genus surfaces, the difference between the sets $\rho_{H_1}(f)$ and $\rho^{\textnormal{erg}}_{H_1}(f)$ can be much bigger: for example the equality $\rho_{H_1}(f) = \conv(\rho^{\textnormal{erg}}_{H_1}(f))$ is no longer true (see Figure~\ref{fig:CrissCrossBitorus}). While much progress has been done in the last 35 years in the topic of the dynamical features of $\rho_{H_1}(\tilde f)$ in the torus case, until recently the attempts to understand the higher genus case led to rather partial results. 
Most of these works suppose some property about the existence of a system of periodic points (or orbits) having certain kind of rotational behaviour, hypothesis that frequently implies  that the rotation set is sufficiently big \cite{pollicott, MR1334719, zbMATH00914982, matsumoto, zbMATH05634807, MR4190050}. 
The only work getting rid of this kind of hypothesis is \cite{lellouch}, where Lellouch gets the existence of horseshoes under the hypothesis of nontrivial intersection number of the rotation vectors of two different ergodic measures. Part of the proof of our main theorem will improve his result, involving in particular the use of some of his techniques.
Note also the work of Alonso, Brum and Passeggi that gets a structure theorem for the rotation set $\rho_{H_1}(f)$ in the $C^0$ generic case, by studying the case of Axiom A diffeomorphisms \cite{MR4578317}.

In the present work, we tackle the case of $\rho^{\textnormal{erg}}_{H_1}(f)$ in the hope it could be a gateway for the understanding of $\rho_{H_1}(f)$. Still, the set $\rho^{\textnormal{erg}}_{H_1}(f)$ seems to have its own interest as --- as we will see --- it carries important properties of the dynamics that cannot be seen directly on the classical rotation set $\rho_{H_1}(f)$ (rotational horseshoes, rotational periodic points, pseudo-Anosov components relative to a finite set\dots).

\begin{teorema}[Shape of ergodic rotation sets]\label{thm:ShapeRotationSet}
Let $f \in \Homeo_0(S)$, where $S$ has genus $g$. Then, its ergodic rotation set $\rho_{H_1}^{\textnormal{erg}}(f)$ can be written as
\[\rho_{H_1}^{\textnormal{erg}}(f) = \rho^1 \cup \rho^+,\]
where
\begin{enumerate}
\item The set $\rho^1$ is included in the union of at most $3g-3$ lines.
\item The set $\rho^+$ is the union of at most $2g-2$ sets $\rho^+_i$, such that, for every $i$:
\begin{itemize}
\item The set ${\rho^{+}_i}$ spans a linear subspace $V_i$ which has a basis formed by elements of $H_1(S,\Z)$;
\item The set $\overline{\rho^{+}_i}$ is a convex set containing $0$ (where \(\overline{\phantom{x}}\) denotes the topological closure);
\item We have $\operatorname{int}_{V_i}(\overline{\rho_i^+})= \operatorname{int}_{V_i}(\rho^+_i)$ (in other words, $\rho_i^+$ is convex up to the fact that elements of $\partial_{V_i}(\rho_i^+)\setminus\operatorname{extrem}(\rho_i^+)$ can be in the complement of $\rho_i^+$);
\item Every element of $\operatorname{int}_{V_i}(\rho^{+}_i) \cap H_1(S,\Q)$ is the rotation vector of some $f$-periodic orbit (because $V_i$ has a rational basis, such elements are dense in $\operatorname{int}_{V_i}(\rho^{+}_i)$). 
\end{itemize} 
\item For every pair of vectors $v, w \in \rho_{H_1}^{\textnormal{erg}}(f)$ such that $v \wedge w \neq 0$, there exists $i$ such that $v, w \in \rho^+_i$. In particular, $\textnormal{span}(\rho^1)$ is a totally isotropic subspace.
\end{enumerate}
Moreover, if $\rho^{\textnormal{erg}}_{H_1}(f)$ is not contained in a finite union of lines, then $f$ has a rotational horseshoe (see Definition~\ref{def:rotationalhorseshoe}) and hence $h_{top}(f)>0$.
\end{teorema}

Note that there is to our knowledge a single known obstruction for a compact and convex subset of the plane to be the rotation set of a torus homeomorphism, due to Le Calvez and Tal \cite{lecalveztalforcing}. Hence, obtaining restrictions on the possible shapes of the sets $\overline{\rho_i^+}$ seems to be a very difficult task.

Note also that there are examples of torus homeomorphisms with nonempty interior rotation set whose ergodic rotation sets $\rho^{\textnormal{erg}}_{H_1}$ are not convex (more precisely, elements of $\partial(\rho^{\textnormal{erg}}_{H_1})\setminus\operatorname{extrem}(\rho^{\textnormal{erg}}_{H_1})$ do not belong to $\rho^{\textnormal{erg}}_{H_1}$), see \cite[Section 3]{zbMATH00009916}. This construction can be easily adapted to higher genus surfaces to show that one cannot hope having the full convexity of the set $\rho_i^+$.

The example of \cite[Figure 1]{guiheneuf2023hyperbolic} shows that there are homeomorphisms of a closed surface $S$ of genus 2 with rotation set $\rho_{H_1}(f)$ having nonempty interior but $\rho_{H_1}^\textrm{erg}(f)$ included in the union of two planes of $H_1(S,\R)$. Hence, the sets $\inte(\rho_{H_1}(f))$ and $\inte(\rho_{H_1}^\textrm{erg}(f))$ are in general different. This suggests that both invariants $\rho_{H_1}(f)$ and $\rho_{H_1}^\textrm{erg}(f)$ have their own interest, as Theorem~\ref{thm:ShapeRotationSet} shows in particular that $\rho_{H_1}^\textrm{erg}(f)$ encodes the rotational periodic points and rotational horseshoes of $f$, while these cannot be deduced from $\rho_{H_1}(f)$ as can be seen in \cite[Figure 1]{guiheneuf2023hyperbolic} (but this latter set contains information about wandering dynamics). 

The ergodic rotation set also encodes the action of the homeomorphism on the fine curve graph \cite{guiheneuf2023hyperbolic}, one can hope that (for example) some criteria of positiveness stable commutator length or existence of free subgroups generated by two elements in $\Homeo_0(S)$ can be phrased in terms of ergodic rotation sets (the question for the torus case is tackled in the recent work \cite{bowden2024boundary}).
\medskip

We will obtain Theorem~\ref{thm:ShapeRotationSet} as a consequence of homotopical (and not only homological) properties of typical trajectories of ergodic measures: one of the morals of our proof strategy is that one has to understand the homotopical properties of the orbits (in the sense of \cite{pa}) to harvest consequences in the homological world.

Let us describe these properties that will lead to a more precise picture of the mechanisms giving birth to the decomposition of the ergodic rotation set. It will allow us to state a more precise version of Theorem~\ref{thm:ShapeRotationSet} in Theorem~\ref{thm:DecompRotSetIntro}.

\subsection{Geodesic tracking}

The key idea of our work is that there exists a tracking of orbits that are typical with respect to some ergodic measure of the homeomorphism, by the orbits of the geodesic flow of the surface. In some sense, the behaviour of typical points for some invariant measure is encoded by a sub-dynamics of the geodesic flow, and one task is to understand to what extent this parallel can be made rigorous. Let us first introduce some definitions. The same kind of ideas is already formulated by Boyland in \cite{boyland}.

\begin{definition}[Rotation distance]
For \(x \in S\) and \(n \in \N\), we define the \emph{rotation distance} \(L_n:S \to [0,+\infty)\) of $x$ as
\begin{equation}\label{eq:rotationdistanceequation}
	L_n(x) = \tdist (\tx, \tf ^n (\tx)),
\end{equation}	
where $\tx$ is any lift of $x$ to $\tS$ (this last quantity does not depend on the chosen lift).
\end{definition}

These functions are well defined and continuous by \eqref{eq:preferedliftequation}.

\begin{definition}[Rotation speed]\label{DefRotSpeed}
The \emph{rotation speed} $\vartheta(x)$ of a point \(x \in S\) is the common value of the following limits if they exist and coincide
\begin{equation}\label{eq:defrhoequation}
\vartheta(x) = \lim\limits_{n \to +\infty}\frac{1}{n}L_n(x) = \lim\limits_{n \to +\infty}\frac{1}{n}L_n(f^{-n}(x)).
\end{equation}	
\end{definition}

\begin{lemma}[Ergodic measures have rotation speed]\label{LemErgoRotSpeed}
Let \(\mu\) be an \(f\)-invariant ergodic probability on \(S\). Then there exists a constant \(\vartheta_\mu\), that we call the \emph{rotation speed} of \(\mu\), such that
\[\vartheta(x) = \vartheta_\mu,\]
for \(\mu\)-almost every point \(x \in S\).	
\end{lemma}

\begin{proof}
It comes from Kingman's subadditive ergodic theorem applied to the maps $L_n(x) = \tdist(\tilde x,\tilde f^n(\tilde x))$.
\end{proof}
 
We will use the term \emph{geodesic} to mean a local isometry \(\gamma: \R \to S\), and we will denote by \(\dgamma:\R \to \textnormal{T}^1S\) the lift of a geodesic to the unit tangent bundle \(\textnormal{T}^1S\) of \(S\). The space of all geodesics (up to reparametrization) is identified with \(\textnormal{T}^1S\) via \(\gamma \mapsto \dgamma(0)\) and in particular is compact and metrizable.

\begin{definition}[Tracking geodesic]\label{DefTrackGeod}
We say \(x \in S\) admits a \emph{tracking geodesic} $\gamma$ if \(\vartheta(x) > 0\) and if for each lift \(\tx\) of \(x\), there exists a lift $\tgamma$ of $\gamma$ such that:
\begin{equation}\label{eq:trackingequation}
\lim\limits_{n \to +\infty}\frac{1}{n}\tdist\big(\tf^n(\tx), \tgamma(n \vartheta(x)\big) = \lim\limits_{n \to +\infty}\frac{1}{n}\tdist\big(\tf^{-n}(\tx), \tgamma(-n \vartheta(x)\big) = 0.
\end{equation}	
\end{definition}

Tracking geodesics (when they exist) are unique up to reparametrization by translation, in particular the set \(\dgamma(\R) \subset \textnormal{T}^1S\) does not depend on the choice of tracking geodesic \(\gamma\) for \(x\).

\begin{definition}[Normalized tracking geodesic]\label{DefNormTrackGeod}
 Given $x \in S$, we say a tracking geodesic $\gamma_x$ is \emph{normalized} if there exist lifts \(\tx, \tgamma _{\tx}\) satisfying Equation~\eqref{eq:trackingequation} and such that
 \begin{equation}\label{eq:normalizedtrackinggeodesic}
 	\tdist (\tx, \tgamma _{\tx}(0)) = \tdist (\tx, \tgamma _{\tx} (\mathbb{R}))
 \end{equation} 
\end{definition}

Note that each point $x$ admitting a tracking geodesic admits a unique normalized tracking geodesic $\gamma_x$.

We denote by \(\M\) the set of ergodic \(f\)-invariant probability measures with positive rotation speed (defined by Lemma~\ref{LemErgoRotSpeed}). The following result implies in particular that if \(\M\) is non-empty then there exist points admitting tracking geodesics.

\begin{teorema}[Tracking geodesics]\label{thm:trackinggeodesictheorem}
 The set \(S_T\) of points in \(S\) that admit a tracking geodesic (Definition~\ref{DefTrackGeod}) is a Borel set, and has full measure for all \(\mu \in \M\).
 Furthermore, there exists a Borel mapping \(x \mapsto \gamma_x\) assigning to each \(x \in S_T\) its normalized tracking geodesic \(\gamma_x:\R \to S\) (see Definition~\ref{DefNormTrackGeod}).
\end{teorema}

Theorem~\ref{thm:trackinggeodesictheorem} was essentially proved for \(\tS\) with non-positive sectional curvature satisfying the visibility property in \cite[Theorem 2 and Corollary 21]{lessa}. However, for constant negative curvature the proof is much simpler so we will give it here for the sake of completeness. It involves mainly geometry of the hyperbolic plane and ergodic theory.

A roughly equivalent result for flows on bundles over a hyperbolic base manifold was proved in \cite[Lemmas 2.2 and 2.3]{boyland}. This work can be adapted to the case of $C^1$ diffeomorphisms on a closed hyperbolic surface by considering its suspension flow. It is unclear to us (and to Philip Boyland himself) whether this proof strategy can be directly adapted to the $C^0$ case.

\subsection{Equidistribution of tracking geodesics and minimal laminations}

The next step is to understand what happens collectively for tracking geodesics associated to some $f$-ergodic measure $\mu$ with positive rotation speed. In particular, we get that tracking geodesics of a $\mu$-typical point equidistribute to a measure $\nu_\mu$ on $\mathrm{T}^1S$; moreover the measure $\nu_\mu$ is invariant and ergodic for the geodesic flow (Theorem~\ref{thm:equidistributiontheoremintro}). 
The support of this measure $\nu_\mu$ gives a set $\dotLambda_\mu \subset \textnormal{T}^1S $ naturally associated to $\mu$, which is invariant and transitive under the geodesic flow.

\begin{definition}
	We say a geodesic \(\gamma\) \emph{equidistributes to a probability measure \(\nu\) on \(\textnormal{T}^1S\)} if for every continuous function \(\varphi:\textnormal{T}^1S \to \R\) we have
	\[\lim_{T \to +\infty}\frac{1}{T}\int_{0}^{T}\varphi(\dot{\gamma}(t))\dd t = \lim_{T \to +\infty}\frac{1}{T}\int_{-T}^{0}\varphi(\dot{\gamma}(t))\dd t = \int_{\textnormal{T}^1S}\varphi(v)\dd \nu(v). \]
\end{definition}

We recall that the support of a probability measure is the smallest closed set with full measure.  

\begin{teorema}[Equidistribution of tracking geodesics]\label{thm:equidistributiontheoremintro}
  For each \(\mu \in \M\) there exists a probability measure \(\nu_\mu\) on \(\textnormal{T}^1S\) such that \({\gamma}_x\) equidistributes to \(\nu_\mu\) for \(\mu\)-a.e.\ \(x \in S\). 
  Furthermore, the measure \(\nu_\mu\) is invariant and ergodic with respect to the geodesic flow on \(\textnormal{T}^1S\).  
Its support is equal to
\[\dotLambda_{\mu}:= \textnormal{supp}(\nu_{\mu}) = \overline{\dotgamma_x(\R)} \]
for \(\mu\)-a.e.\ \(x \in S\)
\end{teorema}

We will denote by $\Lambda_{\mu}$ the projection of $\dotLambda_{\mu}$ on $S$.

We say a geodesic is \emph{simple} if it has no transverse self-intersections.  Such a geodesic is either injective, or periodic (closed). A particularly interesting consequence of Theorem~\ref{thm:equidistributiontheoremintro} is obtained when tracking geodesics are simple.
Recall that a \emph{geodesic lamination} is a closed set \(\Lambda \subset S\) of the form
\[\Lambda = \bigcup\limits_{x \in \Lambda}\beta_x(\R),\]
where each \(\beta_x\) is a simple geodesic with \(\beta_x(0) = x\), and for all \(x,y \in \Lambda\) one has that \(\beta_x(\R)\) and \(\beta_y(\R)\) are either disjoint or coincide.
A geodesic lamination is said to be \emph{minimal} if \(\Lambda = \overline{\beta_x(\R)}\) for all \(x \in \Lambda\).

\begin{teorema}[Simple tracking geodesics]\label{thm:ifsimpletheoremintro}
Suppose \(\mu \in \M\) is such that \(\gamma_x\) is simple for \(\mu\)-a.e.\ \(x \in S\).  Then there exists a minimal geodesic lamination \(\Lambda_\mu \subset S\) such that, for \(\mu\)-a.e.\ \(x \in S\),
\[\overline{\gamma_x(\R)} = \Lambda_\mu.\]
\end{teorema}

It is therefore natural to look for a criterion ensuring that the tracking geodesics are simple. This is performed in the next step of the proof.

\subsection{Intersections of tracking geodesics}

We then prove that the intersection of some tracking geodesics implies the existence of new orbits with some type of rotational behaviour.

Let us recall the definition of homotopic rotation set in the sense of  \cite{pa}. In the following we will denote the geodesics of $\tilde S\simeq\D$ by the couple $(\alpha,\omega)$ of their endpoints in $\partial\tS\simeq\Sp^1$.

\begin{definition}\label{Def:RotVect}
A triple \((\alpha,\omega,v) \in \partial \tS \times \partial \tS \times (0,+\infty)\) is a \emph{rotation vector in the sense of \cite{pa}} if there exists a sequence \(\tx_k\) in \(\tS\) and an increasing sequence \(n_k\) of integers tending to \(+\infty\) such that
\[\lim_{k \to +\infty}\tx_k = \alpha, \lim\limits_{k \to +\infty}\tf^{n_k}(\tx_k) = \omega,\]
and
\[v = \lim\limits_{k \to +\infty}\frac{1}{n_k}\tdist\big(\pr_{(\alpha,\omega)}(\tx_k),\,\pr_{(\alpha,\omega)}(\tf^{n_k}(\tx_k))\big),\]
where $\pr_{(\alpha,\omega)}$ denotes the orthogonal projection on the geodesic ${(\alpha,\omega)}$. 
\end{definition}

\begin{definition}\label{def:dynatrans}
	Let $\mu, \mu' \in \M$. We will say that $\mu$ and $\mu'$ are \emph{dynamically transverse} if there exist $v \in \dotLambda_{\mu}, \ v' \in \dotLambda_{\mu'}$ such that $\pi_S(v) = \pi_S(v')$ and $v \neq v'$. 
\end{definition}

Note that if $v,v' \in \textnormal{T}^1 S$ are as in last definition, then any pair of geodesics $\gamma, \gamma'$ such that $\gamma(t) = v, \ \gamma'(t') = v'$ for some $t, t' \in \R$, intersect transversally.

In the following theorem, if $x$ is a point having a tracking geodesic $\gamma_x$, and if $\tgamma_\tx$ is a lift of this geodesic to $\tS$, we denote $\alpha(\tilde x),\omega(\tilde x)\in\partial\tS$ the $\alpha$ and $\omega$-limits of $\tgamma_\tx$. 

\begin{teorema}[Intersections of tracking geodesics and rotation sets]\label{thm:intertrackingimpliesentropy}
Suppose there exist \(\mu_1,\mu_2\in\M\) which are dynamically transverse.
Then:
\begin{itemize}
\item The homeomorphism $f$ has a topological rotational horseshoe (see Definition~\ref{def:rotationalhorseshoe}), and in particular has positive topological entropy;
\item Both $\big(\alpha(\tilde x_1),\omega(\tilde x_2),\max(\vartheta_{\mu_1}, \vartheta_{\mu_2})\big)$ and $\big(\alpha(\tilde x_2),\omega(\tilde x_1),\max(\vartheta_{\mu_1}, \vartheta_{\mu_2})\big)$ are rotation vectors in the sense of \cite{pa};
\item For any $r\in H_1(S,\R)$ in the triangle spanned by $0, \rho_{H_1}(\mu_1), \rho_{H_1}(\mu_2)$ and any $\varepsilon>0$, there exists a periodic point $z\in S$ such that $\|\rho_{H_1}(z) - r\|\le \varepsilon$.
\end{itemize}
\end{teorema}

We will need an improved version of this statement for the proof of Theorem~\ref{thm:DecompRotSetIntro}, we postpone this statement to Section~\ref{SubsecInter} (Theorem~\ref{TheoInterGeod}).

Theorem~\ref{thm:intertrackingimpliesentropy} improves \cite[Proposition C]{pa} in the case of points that are typical for some measure, and generalizes \cite[Theorem F]{pa} (with weaker conclusions) in the case where the geodesics are not closed.
Note that Theorem~\ref{thm:intertrackingimpliesentropy} strictly improves Lellouch's theorem about the existence of horseshoes and new rotation vectors of $\rho_{H_1}^\textrm{erg}$ in the case of rotation vectors of ergodic measures having nonzero intersection in homology \cite[Théorème A, Théorème C]{lellouch}: indeed, Proposition~\ref{prop:InterHomoImpliesInterGeod} ensures that if two ergodic measures have rotation vectors with nontrivial intersection in homology, then some tracking geodesics for typical points for these measures do intersect.

The proof of Theorem~\ref{thm:intertrackingimpliesentropy} is the most technical part of this article; it involves the powerful \emph{forcing theory} of Le Calvez and Tal \cite{lecalveztalforcing, lct2}, which is based on the equivariant Brouwer theory of Le Calvez \cite{lecalvezfoliations}.
%
%

\subsection{Rotation sets and tracking geodesics}\label{SecIntroRot}

We are now ready to state an improved version of Theorem~\ref{thm:ShapeRotationSet}. The pieces of the rotation set's decomposition are obtained as equivalence classes for the following relation involving intersections of tracking geodesics. 

Let us define the equivalence relation\footnote{This is inspired from the equivalence relation of \cite{guiheneuf2023hyperbolic}.} $\sim$ on $\M$ by: $\mu_1\sim \mu_2$ if one of the following is true (see right after Theorem~\ref{thm:equidistributiontheoremintro} for the definition of the set $\Lambda_\mu$):
\begin{itemize}
\item $\Lambda_{\mu_1} = \Lambda_{\mu_2}$;
\item There exist $\tau_1,\dots,\tau_m\in\M$ such that $\tau_1=\mu_1$, $\tau_m=\mu_2$ and for all $1\le i<m$, the measures $\tau_i$ and $\tau_{i+1}$ are dynamically transverse.
\end{itemize}  
The fact that this is an equivalence relation is straightforward. We then denote $\{\cl_i\}_{i\in I}=\M/\sim$ the equivalence classes of $\sim$. We call $I^1$ the set of classes with the property that any two measures $\mu_1$ and $\mu_2$ of $\cl_i$ satisfy $\Lambda_{\mu_1} = \Lambda_{\mu_2}$; we will see that by Theorem~\ref{TheoInterGeod} this implies that the geodesics corresponding to vectors in \(\dotLambda_{\mu_1}\) are simple.
Let $I^+$ denote the other classes, which are such that for any $\mu\in \cl_i$ with \(i \in I^+\), there exists $\mu'\in\M$ such that $\mu$ and $\mu'$ are dynamically transverse.

This decomposition into classes leads to the main theorem of our article, which improves Theorem~\ref{thm:ShapeRotationSet}:

\begin{teorema}\label{thm:DecompRotSetIntro}
Let $f\in\Homeo_0(S)$. Write
\[\M = \bigsqcup_{i\in I} \cl_i = \bigsqcup_{i\in I^1} \cl_i \sqcup \bigsqcup_{i\in I^+} \cl_i\]
the decomposition of the set of ergodic measures with positive rotation speed into equivalence classes for the relation $\sim$, and denote for $i\in I$,
\[\rho_i = \big\{\rho_{H_1}(\mu)\mid \mu\in\cl_i\big\},\qquad 
V_i = \operatorname{Span}(\rho_i),\qquad
\Lambda_i = \bigcup_{\mu\in\cl_i} \Lambda_\mu.\]
then:
\begin{enumerate}
\item For every $i \in I^1$, we have

\begin{itemize}
	\item $\Lambda_i$ is a minimal geodesic lamination of \(S\). 
	\item $\rho_{I^1} = \bigcup \limits_{i \in I^1}\rho_i$ is included in a union of at most $3g-3$ lines of $H_1(S,\R)$.
\end{itemize}

\item If $I^+ \neq \varnothing$, then $f$ has a topological horseshoe (and in particular, positive topological entropy), and for every $i \in I^+$:

\begin{itemize}
\item The linear subspace $V_i$ has a basis formed by elements of $H_1(S,\Z)$;
\item The set $\overline{\rho_i}$ is a convex set containing $0$;
\item We have $\operatorname{int}_{V_i}(\overline{\rho_i})= \operatorname{int}_{V_i}(\rho_i)$ (in other words, $\rho_i$ is convex up to the fact that elements of $\partial_{V_i}(\rho_i)\setminus\operatorname{extrem}(\rho_i)$ can be in the complement of $\rho_i$);
\item Every element of $\operatorname{int}_{V_i}(\rho_i) \cap H_1(S,\Q)$ is the rotation vector of some $f$-periodic orbit (because $V_i$ has a rational basis, such elements are dense in $\operatorname{int}_{V_i}(\rho_i)$). 
\end{itemize} 

\item For $i,j\in I$, $i\neq j$, for $v_i\in V_i$ and $v_j\in V_j$, we have $v_i\wedge v_j = 0$. Moreover, if $i,j\in I^1$, then $\Lambda_i\cap\Lambda_j = \varnothing$.
\item $\card(I)\le 5g-5$. More precisely, $\card I^1\le 3g-3$ and $\card I^+\le 2g-2$.
\end{enumerate}
\end{teorema}

Note that for $i\in I^+$, we have some kind of ``quantitative estimate'' about the periods of the periodic points realizing the vectors of $\operatorname{int}_{V_i}(\rho_i)\cap H_1(S,\Q)$: for any finite set $F\subset \inte_{V_i}(\rho_i)$, denoting $R = \conv(\{0\}, F)$, there exists a constant $M>0$ such that for any $p/q\in R\cap H_1(S,\Q)$, with $p\in H_1(S,\Z)$ and $q\in\N^*$, there exists a $f$-periodic point with its period dividing $Mq$ and with rotation vector $p/q$.

\subsection{Zero entropy homeomorphisms}

A direct consequence of Theorem~\ref{thm:DecompRotSetIntro} is that if the topological entropy of \(f\) is zero then there are no pieces of dimension $\ge 2$ in the rotation set. The study of zero entropy surface homeomorphisms has a long story and was one of the motivations of the present work.

\subsubsection{Handel's preprint}

Let \(h_{\textnormal{top}}(f)\) denote the topological entropy of \(f\). In an unpublished preprint, among many interesting results, the idea was put forth (see \cite[Theorem 2.5]{handel}) that if \(h_{\textnormal{top}}(f) = 0\), then all \(f\)-orbits should track a geodesic lamination \(\Lambda \subset S\) (at least during the time when they are far from the set of fixed points of \(f\)).

Many of the results and techniques of \cite{handel} are by now well known.
For example, a fixed point result used to exclude \emph{oriented cycles} from the dynamics has now received several proofs and improvements in \cite{handel1}, 
\cite{fixed2}, \cite{juliana2}, \cite{juliana1}, and \cite{fixed1}.   Some of the Homotopy Brouwer Theory has been developed for example in \cite{leroux}.  Some results for area preserving zero entropy homeomorphisms of the two-sphere appear in \cite{frankshandel}.  The new technique of \emph{forcing} developed by Le Calvez and Tal \cite{lecalveztalforcing} has recently been used to obtain many results, including further results on the two-sphere, and a proof of \cite[Theorem 9.1]{handel} for annulus homeomorphisms (see \cite{lct2}). 

However, as far as the authors are aware, there is still no published, or even widely accepted, proof of \cite[Theorem 2.5]{handel}.

\subsubsection{Flows}

Let us consider first the case where \(f\) is the time-one map of a flow on \(S\).  In this case it was shown in \cite{youngflows} that \(h_{\textnormal{top}}(f) = 0\).

In this context, the solution to the so-called Anosov-Weil problem for flows (see \cite[Theorems 5.1.2 and 5.5.1]{nikolaev}, and also \cite{anosovweil1}, \cite{anosovweil2}) implies that, if \(f\) is the time-one map of a \(C^\infty\) flow on \(S\) and has a finite number of fixed points, then each \(\tf\)-orbit is either bounded, at bounded distance from a geodesic ray, or at bounded distance from a geodesic.  Because different flow lines do not intersect, the projection to \(S\) of all geodesics occuring in the last case will belong to some geodesic lamination \(\Lambda \subset S\).  Hence, in this case, a complete picture of tracking by a geodesic lamination is available.

We note that some results of this kind pertaining to flows predate Handel's preprint, such as the classification up to orbit equivalence of certain transitive flows \cite{transitiveflows} (see also \cite{aransonminimal}).

\subsubsection{Non-contractible periodic points}

Without assuming that \(f\) is the time-one map of a flow, a complete picture of tracking by geodesic laminations is attainable for periodic orbits.

The technique involved consists of puncturing \(S\) on some finite \(f\)-invariant set and using the Nielsen-Thurston classification of homeomorphisms up to isotopy on the punctured surface \cite{thurston}.  

This idea has been previously used to ensure the existence of positive entropy, given a certain configuration of periodic points (for example in \cite{llibremackay}, \cite{pollicott}, and \cite{matsumoto}).  But the particular result below seems to have no published (or even well known) proof, so we include one here.

We say an \(f\)-periodic point \(x \in S\) with minimal period \(n\) is \emph{non-contractible} if for some lift \(\tx\) one has \(\tf^n(\tx) \neq \tx\).  In this case, there is a loxodromic element in \(\Gamma\) taking \(\tx\) to \(\tf^n(\tx)\) and its axis projects to a normalized tracking \(\gamma_x\) for \(x\) which is closed.

\begin{theorem}[Tracking of non-contractible periodic orbits by a geodesic lamination]\label{periodicpointtheorem}
Let \(f \in \Homeo_0(S)\) be such that \(h_{\textnormal{top}}(f) = 0\). Then there exists a finite union of pairwise disjoint simple closed geodesics \(\Lambda = \bigsqcup_{i} \alpha_i(\R)\) such that \(\gamma_x(\R) \subset \Lambda\) for all non-contractible periodic points \(x \in S\).
\end{theorem}

\begin{proof}
 First suppose that there exists a non-contractible periodic point \(x\) with minimal period \(n\) such that \(\gamma_x\) is not simple.
 
 Let \(g = f^n\) and let \(t \mapsto g_t\) be an isotopy from the identity to \(g\).  Since \(\Homeo_0(S)\) is contractible (see \cite{hamstrom}), the free homotopy class of the closed curve \(t \mapsto g_t(x)\) given by this isotopy is independent of the given isotopy.  Let us call this class \([\alpha]\).
 
 The class \([\alpha]\) is some power of the free homotopy class of \(\gamma_x\) restricted to its minimal period.  In particular, there is no simple curve in \([\alpha]\). 
 
 From \cite[Theorem 2]{kra} this implies that the Nielsen-Thurston class of \(g\) on \(S \setminus \lbrace x\rbrace\) contains a pseudo-Anosov component.  Hence, from \cite[Theorem 7.7]{boyland2}, \(g\) has positive entropy, which implies that \(f\) does as well.

 Now suppose that \(x \neq y\) are non-contractible periodic points whose tracking geodesics are simple but intersect transversally. Similarly to the previous case we consider \(g = f^n\) for \(n\) the minimal common period of \(x\) and \(y\).
 
 Again because geodesics minimize the geometric intersection number in their free homotopy class, the classes \([\alpha]\) and \([\beta]\) of the isotopy curves \(t \mapsto g_t(x)\) and \(t \mapsto g_t(y)\) have positive intersection number.
 
 From \cite[Main Theorem]{imayoshi} this implies that \(g\) has a pseudo-Anosov component, and hence \(f\) has positive entropy as before.
 
 From the two cases above we obtain that if \(h_{\textnormal{top}}(f) = 0\) then 
 \[\Lambda = \bigcup \gamma_x(\R),\]
 is a finite union of pairwise disjoint simple closed geodesics, as required.
\end{proof}

\subsubsection{Forcing}

The more recent techniques of foliations by Brouwer lines \cite{lecalvezfoliations}, and forcing \cite{lecalveztalforcing,lct2,guiheneufforcing}, have allowed in \cite{pa} extensions of the previous results to non-periodic orbits (and the present work is another example of this).
In particular \cite[Theorem E]{pa} roughly states that if there is a trajectory following a closed geodesic in \(S\) at positive speed, then either the geodesic has no transverse self-intersections or \(h_{\textnormal{top}}(f) > 0\).
\bigskip

\subsubsection{Non-hyperbolic surfaces}

Note that the case of other closed surfaces (\emph{i.e.}\ the sphere $\Sp^2$ and the torus $\T^2$) are way better understood.

For the sphere, Franks and Handel \cite{frankshandel} (for smooth area-preserving diffeomorphisms) and then Le Calvez and Tal \cite{lecalveztalforcing,lct2} (for general homeomorphisms) obtained a classification of the non wandering set of homeomorphisms of $\Sp^2$ having zero entropy (or more generally without topological horseshoe): roughly speaking, this non wandering set is covered by a family of invariant annuli on which all points turn in the same direction.

There is, to our knowledge, no tentative of systematic classification of zero entropy torus homeomorphisms. From Llibre and Mackay \cite{llibremackay}, such homeomorphisms must have their rotation set included in some line of the plane. Hence, for any homeomorphism $f\in\Homeo_0(\T^2)$, there is a line $D$ of the plane such that any lifting orbit $(\tilde f^n(\tilde x))_n$ in the universal cover $\R^2$ of $\T^2$ stays at sublinear distance from $D$.

In some cases, it is possible to prove that this distance stays finite (a property called ``bounded deviation'' or ``bounded displacement''), proving that the dynamics is really much like the one of a flow. It is proven true if the rotation set is a non-degenerate line segment and if the homeomorphism has at least one periodic point \cite{zbMATH06304088,zbMATH06914177, zbMATH07488214}, or if the homeomorphism preserves the area and has its fixed point set nonempty and contained in a topological disk \cite{liu2022noncontractible}.

\subsubsection{A consequence of Theorem~\ref{thm:DecompRotSetIntro}}

A particularly interesting consequence of Theorem~\ref{thm:DecompRotSetIntro} is a weakened form of tracking by a geodesic lamination (compared to Handel's conjecture), which applies only to generic points with respect to certain \(f\)-invariant ergodic measures:

\begin{corol}[Ergodic measures are tracked by geodesic laminations]\label{thm:maintheorem}
Let $f\in \Homeo_0(S)$.
  If $f$ has no topological horseshoe\footnote{In particular, if \(h_{\textnormal{top}}(f) = 0\).} then for each measure \(\mu \in \M\) (meaning that $\mu$ is $f$-invariant, ergodic, and with positive rotation speed, see Lemma~\ref{LemErgoRotSpeed} and \eqref{eq:defrhoequation}) there exists a minimal geodesic lamination \(\Lambda_\mu \subset S\) such that
  \[\overline{\gamma_x(\R)} = \Lambda_\mu,\]
  for \(\mu\)-almost every point \(x \in S\), where \(\gamma_x\) is given by Theorem~\ref{thm:trackinggeodesictheorem}.

  Furthermore, for each pair \(\mu,\mu' \in \M\) one has that either \(\Lambda_\mu\) and \(\Lambda_{\mu'}\) are equal or they are disjoint.
  
  The set 
  \begin{equation}\label{Eq:Lambda}
\Lambda = \bigcup\limits_{\mu \in \M}\Lambda_\mu,
\end{equation}
is a geodesic lamination, which has at most $3g-3$ minimal sublaminations.
\end{corol}

Since each non-contractible periodic orbit supports a measure in \(\M\), the corollary above generalizes Theorem \ref{periodicpointtheorem}.

\subsubsection{Examples with non-trivial laminations}

The last section of the present article is devoted to some examples of dynamics in closed hyperbolic surfaces.
We observe that, even though we have $\dot{\gamma}_x(\R) = \dot{\gamma}_{f(x)}(\R)$ for all \(x \in S_T\), this does not imply that \(\dot{\gamma_x}(\R)\) is \(\mu\)-a.e.\ constant for each \(\mu \in \M\).  This is due to the fact that space of orbits of the geodesic flow on \(\textnormal{T}^1S\) is non-Hausdorff. In fact, non-trivial laminations may occur in Theorem~\ref{thm:trackinggeodesictheorem}.

\begin{prop}[Ergodic measures with many tracking geodesics]\label{thm:nonconstanttrackingtheorem}
  For every closed orientable hyperbolic surface \(S\) there exists a homeomorphism \(f \in \textnormal{Homeo}_0(S)\), and \(\mu \in \M\) such that the map
  \[x \mapsto \dot{\gamma}_x(\R),\]
  is not \(\mu\)-a.e.\ constant, where \(x \mapsto \gamma_x\) is given by Theorem~\ref{thm:trackinggeodesictheorem}.
  More precisely, for any measurable set $E\subset S$ such that $\mu(E)>0$, the set $\{\dot\gamma_x(\R)\mid x\in E\}$ is uncountable.
\end{prop}

Two constructions are indicated in \cite[example S and example \(\Lambda\)]{handel}.  The first, credited to Stepanov, is to take an irrational flow on the two torus, slow it down to add two fixed points, blow these points up to circles and paste.   The second, is to directly consider a non-trivial minimal orientable lamination \(\Lambda\) and let \(f\) be the time-one map of the flow of a vector field which is non-zero and tangent along \(\Lambda\).   We provide details of these constructions for the sake of completeness.

\subsection{Some open questions}

Here are some open questions raised by the techniques developed in the present paper.

\begin{enumerate}[label=\Alph*)]
\item Is it true that if $i,j\in I$ with $i\neq j$, then $\Lambda_i\cap\Lambda_j$ is reduced to a finite set of closed geodesics? We  only prove this result in restriction to the set $I^1$.
\item Are the laminations $\Lambda_i$, for $i\in I^1$, orientable? At least, the tracking geodesics of these laminations get an orientation from the orbits that produce them.
\item Could one get applications of this work to the ``classical'' rotation set $\rho_{H_1}$? Is this latter set a finite union of convex sets? Is it convex when it has nonempty interior ? In both those directions, the present work about typical points may be used as a fulcrum to tackle the case of other orbits: in the case of zero entropy homeomorphisms, is the set $\Lambda$ of \eqref{Eq:Lambda} sufficient to describe the asymptotic behaviour of all non-wandering orbits?
\item Is it possible to get results of bounded displacement with respect to the rotation set $\rho_{H_1}$, under weaker hypothesis than \cite{MR4190050, lellouch}? 
\item Consider the flip on \(\textnormal{T}^1S\) sending each unit tangent vector to its opposite.  We notice that since geodesic currents are in bijection with flip invariant positive measures on \(\textnormal{T}^1S\) which are invariant under the geodesic flow (see for example \cite[Theorem 3.4]{erlandsson}), symmetrizing the measures in \(\mathcal{R}_f\) by the flip yields a natural family of geodesic currents associated to \(f\).  Given a closed geodesic \(C\), the intersection number of these currents with \(C\) encodes the rate of intersection of isotopy curves \(f\) with \(C\). Hence, they also seem like a natural object for further study. Also, it could be possible that considering the suspension $\phi_f$ of the homeomorphism $f$, equipped with the suspension $\dd t\otimes \mu$ of the $f$-ergodic measure $\mu$ yields as in \cite{boyland} to a semi-conjugation between the measured flows $(\phi_f,\dd t\otimes \mu)$ and the geodesic flow on $\mathrm{T}^1 S$ equipped with $\nu_\mu$.
\item There is still quite a large gap between our Corollary~\ref{thm:maintheorem} and \cite[Theorem 2.5]{handel}. Of course, one would like to be able to describe not only the orbits that are typical for measures but all of them, or at least the ones having some form of recurrence (\emph{e.g.}\ recurrent or non wandering). 
One could also ask whether our conclusions can be strengthened to get bounded deviations from the geodesics in the the zero entropy case. Note that at the end of Section~7.2 of \cite{pa} there is an example (based on a construction of Koropecki and Tal) of an ``almost annular'' homeomorphism that has 
unbounded displacements in all directions not intersecting the direction of some embedded invariant annulus $A$, but nontrivial rotation set in restriction to the annulus $A$.
\end{enumerate}

\subsection{Organization of the paper}

We devote Section \ref{sec:tracking} to getting to the proof of Theorem~\ref{thm:trackinggeodesictheorem}.

In Section \ref{sec:equidistributionsection} we prove that for \(\mu \in \M\), almost all tracking geodesics equidistribute to the same measure, and thus we obtain Theorem \ref{thm:equidistributiontheoremintro}. We also show that under the assumption that tracking geodesics are simple, this yields a minimal lamination containing almost every tracking geodesic, which proves Theorem \ref{thm:ifsimpletheoremintro}.

In Section \ref{sec:rotationsection} we link geodesic tracking with a notion of homotopical rotation. We also show that ergodic measures with non-zero homological rotation belong to $\M$.


In Section \ref{sec:forcing2} we show that if \(f\) does not contain a topological horshoe then generic tracking geodesics for measures \(\mu,\mu' \in \M\) cannot intersect transversally.

These results are combined to prove Theorem \ref{thm:DecompRotSet}, which yields a partition of $\M$ into rotationally disjoint equivalence classes. This allows to obtain Theorem \ref{thm:ShapeRotationSet} as a byproduct. By restricting these results to the zero-entropy context, we recover Corollary~\ref{thm:maintheorem}. All of this is done in Section \ref{sec:maintheoremsection}. 

The examples of Proposition~\ref{thm:nonconstanttrackingtheorem} are constructed in Section \ref{sec:examplessection}.


\section*{Acknowledgments}

We would like to thank Alejandro Bellati, Christian Bonatti, Philip Boyland, Viveka Erlandsson, Alejandro Kocsard, Carlos Matheus, Rafael Potrie and Jing Tao for helpful conversations.

We are extremely thankful to Emmanuel Militon who allowed us to adapt his idea of the equivalence relation $\sim$ among ergodic measures to the case of intersection of tracking geodesics.

We warmly thank Patrice Le Calvez for his numerous insights regarding the proof of Proposition~\ref{LemLocalTransverse}.

Thanks to MathAmSud, IFUMI, IMERL, PEPS, CSIC and ANII for aiding the production of this article.

\section{Tracking geodesics --- proof of Theorem~\ref{thm:trackinggeodesictheorem}\label{sec:tracking}}

The goal of this section is to prove Theorem~\ref{thm:trackinggeodesictheorem}. First, we prove that the set $S_T$ of points admitting a tracking geodesic (Theorem~\ref{DefTrackGeod}) is equal to the set $S'_T$ of points with positive rotation speed (Theorem~\ref{DefRotSpeed}) whose limit points in $\partial\D$ are different (Lemma~\ref{lem:trackingcharacterizationlemma}). Then, we prove in Lemma~\ref{LemTypNeq} that points which are typical for some ergodic measure with positive rotation speed, belong to $S'_T$.

\subsection{Characterization of geodesic tracking}

We first show that the set \(S_T\) of points admitting a tracking geodesic (Theorem~\ref{DefTrackGeod}),
and the mapping \(x \mapsto \gamma_x\), are both Borel measurable.

For this purpose, endow \(X = \R^\N\) with the topology of pointwise convergence and \(X \times X\) with the product topology.  

We observe that the subset \(\mathcal{L} \subset X\) consisting of Cauchy sequences, and the subset \(\mathcal{E} \subset X \times X\) of pairs of sequences \(\left((x_n),(y_n)\right)\), such that \(\lim\limits_{n \to +\infty}|x_n - y_n| = 0\), are Borel.

It is immediate that the following mappings into \(X\) are continuous (recall that $L_n$ is defined in \eqref{eq:rotationdistanceequation} by $L_n(x) = \tdist (\tx, \tf ^n (\tx))$)
\[\varphi_1(x) = \left(\frac{1}{n}L_n(x)\right)_{n \in \N},
\qquad 
\varphi_2(x) = \left(\frac{1}{n}L_n(f^{-n}(x))\right)_{n \in \N},\]
and therefore the mapping \(\varphi_3 = (\varphi_1,\varphi_2)\) into \(X \times X\) is continuous as well.

It follows that the set 
\[S_\vartheta = \lbrace x \in S: \varphi_1(x) \in \mathcal{L}, \varphi_2(x) \in \mathcal{L}, \varphi_3(x) \in \mathcal{E}\rbrace,\]
of points $x$ having a well defined rotation speed \(\vartheta(x)\) (defined in \eqref{eq:defrhoequation}) is a Borel subset of \(S\).

Since for each \(\epsilon > 0\), the set \(X_{\epsilon}\) of sequences in \(X\) eventually greater than \(\epsilon\) is Borel, we obtain that the set
\[S_{\vartheta > 0} = S_{\vartheta} \cap \bigcup\limits_{n \in \N}\lbrace x \in S: \varphi_1(x) \in X_{\frac{1}{n}}\rbrace\]
of points with positive rotation speed is Borel as well.

We fix from now on a measurable section \(x \mapsto \tx\) of the covering projection \(\pi : \tilde S\to S\), \emph{i.e.}\ a Borel mapping such that \(\pi(\tx) = x\) for all \(x \in S\). This exists by general measurable selection theorems, and also may be constructed from a polygonal fundamental domain for the group \(\Gamma\) of cover transformations.

We identify \(\tS\) with the unit disc \(\D = \lbrace z \in \C: |z| < 1\rbrace\) and \(\tdist\) with the Poincaré metric on \(\D\).

\begin{lemma}\label{lem:twopointslemma}
  For each \(x \in S_{\vartheta > 0}\) and any lift $\wt x$ of $x$, there exist \(\alpha(\wt x),\omega(\wt x) \in \partial \D = \lbrace z \in \C: |z| = 1\rbrace\) such that
  \(\alpha(\wt x) = \lim_{n \to +\infty}\tf^{-n}(\tx)\text{ and }\omega(\wt x) = \lim_{n \to +\infty}\tf^n(\tx).\)
  
  Furthermore, if \(\beta_{-},\beta_{+}:[0,+\infty) \to \R\) are geodesic rays with limit points \(\alpha(\tx)\) and \(\omega(\tx)\) respectively, then
  \[0 = \lim\limits_{n \to +\infty}\frac{1}{n}\tdist\big(\tf^n(\tx),\beta_+(\R)\big) = \lim\limits_{n \to +\infty}\frac{1}{n}\tdist\big(\tf^{-n}(\tx),\beta_{-}(\R)\big).\]
\end{lemma}

\begin{proof}
  We will prove the existence of \(\omega(\wt x)\) and tracking property for the ray \(\beta_+\).  Applying this to \(f^{-1}\) yields the result for \(\alpha(\wt x)\) and \(\beta_-\) as well. 

  For this purpose, fix the origin $O$ as $O=\beta_+(0)$, which allows to Euclidean polar coordinates \(\tf^n(\tx) = r_n\exp(i \theta_n)\).  

  Picking \(0 < a < \vartheta(x)\) (defined in \eqref{eq:defrhoequation}) we have
  \[an < \tdist(O,\tf^n(\tx)),\]
  for all \(n\) large enough. In particular, this implies that \(\lim_{n \to +\infty}r_n = 1\) since the Poincaré metric is bi-lipschitz with respect to the Euclidean metric on compact subsets of \(\D\).

  To control \(\theta_n\) we observe that since \(L_1:S \to \R, x\mapsto \tdist(\tx,\tilde f(\tx))\) is continuous, there exists \(b > 0\) such that,   for all \(n\in\Z\),
  \[\tdist(\tf^n(\tx), \tf^{n+1}(\tx)) \le b.\]

  We now consider the hyperbolic triangle with vertices \(O,\tf^n(\tx)\) and \(\tf^{n+1}(\tx)\), and notice that the angle at \(O\) (for the Poincaré metric) is \(|\theta_n - \theta_{n+1}|\) because the Poincaré metric is conformal with respect to the Euclidean one. By the hyperbolic law of sines we have
\begin{equation}\label{sinhequation}\sin(|\theta_n - \theta_{n+1}|) \le \sinh(b)\sinh(an)^{-1},\end{equation}
  for all \(n\) large enough, and in particular \(\left(\theta_n\right)_{n \in \N}\) is a Cauchy sequence, hence converges.
  
  We may suppose that \(\beta_+\) is  the geodesic ray from \(O\) to \(\omega(\tx)\).   We fix \(\epsilon > 0\), $\epsilon<\vartheta$, denote by \(d_n = \tdist(\tf^n(\tx),\beta_+(\R))\), and set \(\theta = \lim_{n \to +\infty}\theta_n\) and \(\vartheta = \vartheta(x)\).
  
  From \eqref{sinhequation} applied to $a=\vartheta-\epsilon/2$ and the definition of \(\vartheta\) we obtain
\[|\theta - \theta_n| \le \exp(-(\vartheta-\epsilon)n),\]
and
\[\tdist(O,\tf^n(\tx)) \le (\vartheta+\epsilon)n,\]
for all \(n\) large enough.
  
  We apply the hyperbolic law of sines to the triangle \(O, \tf^n(\tx),p_n\) where \(p_n\) is the closest point to \(\tf^n(\tx)\) in \(\beta_+(\R)\).  This yields:
\[\frac{\sinh(d_n)}{\sin(|\theta-\theta_n|)} = \frac{\sinh(\tdist(O,\tf^n(\tx))}{\sin(\pi/2)} \le \sinh((\vartheta+\epsilon)n),\]
for all \(n\) large enough.
  
It follows that
\[\sinh d_n \le e^{-(\vartheta-\epsilon)n}e^{-(\vartheta+\epsilon)n} = e^{2\epsilon n},\]
and hence \(d_n \le 4\epsilon n\) for all \(n\) large enough, which since \(\epsilon > 0\) was arbitrary yields
\[\lim\limits_{n \to +\infty}\frac{1}{n}d_n = 0,\]
as required.
\end{proof}

We now consider the set \(S_T'\) defined by
\[S_T' = \lbrace x \in S_{\vartheta > 0} \mid  \alpha(\wt x) \neq \omega(\wt x)\rbrace.\]

Since by Lemma~\ref{lem:twopointslemma} one has that \(\alpha(\wt x)\) and \(\omega(\wt x)\) are well defined on \(S_{\vartheta >0}\), it follows that they are Borel measurable and therefore \(S_T'\) is a Borel set.

The claim that the set \(S_T\) of points in \(S\) admitting a tracking geodesic is Borel measurable now follows from the following lemma.

\begin{lemma}\label{lem:trackingcharacterizationlemma}
  One has \(S_T = S_T'\).
\end{lemma}

\begin{proof}
$\mathbf{ S_T\subset S'_T}$: 
  Suppose first that \(x \in S_T\).  By definition, \(\vartheta(x) > 0\), so it follows by Lemma~\ref{lem:twopointslemma} that \(\omega(\wt x)\) and \(\alpha(\wt x)\) exist.  Now let \(\tgamma\) be the geodesic satisfying the tracking Equation~\eqref{eq:trackingequation}; we will show that \(\alpha(\wt x) = \lim_{t \to +\infty}\tgamma(-t)\) and \(\omega(\wt x) = \lim_{t \to +\infty}\tgamma(t)\) which will imply that \(\alpha(\wt x) \neq \omega(\wt x)\). 

  Again it suffices to establish \(\omega(\wt x) = \lim_{t \to +\infty}\tgamma(t)\) since the other claim follows applying this to \(f^{-1}\).

  From the definition \eqref{eq:trackingequation} of tracking geodesics, for each \(\epsilon > 0\) one has 
  \[\tdist\big(\tf^n(\tx),\tgamma(n\vartheta(x))\big) \le \epsilon n\]
  and
  \[(\vartheta(x) - \epsilon) n < \tdist\big(\tgamma(n\vartheta(x)),\tgamma(0)\big),\]
  for all \(n\) large enough.

  Applying the hyperbolic law of sines to the hyperbolic triangle with vertices \(\tgamma(0), \tf^n(\tx)\), and \(\tgamma(n\vartheta(x))\), denoting by \(\theta_n\) the angle at \(\tgamma(0)\) we obtain
  \[\sin(\theta_n) \le \sinh(\epsilon n)\sinh((\vartheta(x) - \epsilon)n)^{-1},\]
  for all \(n\) large enough.  Taking \(\epsilon < \vartheta(x)/2\), this implies that \(\lim_{n \to +\infty}\theta_n = 0\), and hence, \(\omega(\wt x) = \lim_{t \to +\infty}\tgamma(t)\) as claimed. 
  
   From this we obtain that \(\alpha(\wt x) \neq \omega(\wt x)\); we have proved that \(S_T \subset S_T'\).
\bigskip

$\mathbf{ S'_T\subset S_T}$:  We now suppose that \(x \in S_T'\); let \(\tgamma\) denote the unique geodesic with \(\lim_{t \to -\infty}\tgamma(t) = \alpha(\wt x)\) and \(\lim_{t \to +\infty}\tgamma(t) = \omega(\wt x)\). 

Let \(t_n \in \R\) and \(r_n \ge 0\) be defined by
  \begin{equation}\label{eq:distancetogeodesicequation}
    r_n = \tdist\big(\tf^n(\tx),\tgamma(t_n)\big) = \tdist\big(\tf^n(\tx),\tgamma(\R)\big),
  \end{equation}
  and observe that since \(\omega(\wt x) = \lim\limits_{t \to +\infty}\tgamma(t)\) we have
  \begin{equation}\label{eq:tndivergenceequation}
    \lim\limits_{n \to +\infty}t_n = +\infty.
  \end{equation}
  From the triangle inequality it follows that
  \begin{equation}\label{eq:tnrnlowerequation}
    \vartheta = \vartheta(x) = \lim\limits_{n \to +\infty}\frac{1}{n}\tdist\big(\tx,\tf^n(\tx)\big) \le \liminf\limits_{n \to +\infty}\frac{t_n + r_n}{n}.
  \end{equation}

  Suppose that there exists \(\epsilon > 0\) such that
  \begin{equation}\label{eq:tnupperboundequation}
    \frac{t_n}{n} \le (1-\epsilon)\vartheta,
  \end{equation}
  for infinitely many \(n\geq 0\).
  Fix \(b = \max L_1\), and let \(R > 0\) be such that 
  \begin{equation}\label{eq:choiceofRequation}
    \cosh(R)^{-1}b \le (1-\epsilon)\vartheta,
  \end{equation}
  and \(N\) be large enough so that (see equation~\eqref{eq:tnrnlowerequation})
  \begin{equation}\label{eq:largenequation}
\frac{\epsilon}{2}\vartheta n > R + b
\qquad\text{and}\qquad 
    \big(1 -\frac{\epsilon}{2}\big)\vartheta \le \frac{t_n + r_n}{n}.
  \end{equation}
  for all \(n \ge N\).
  
  If \(n \ge N\) satisfies the inequality \eqref{eq:tnupperboundequation} then by the second inequality of \eqref{eq:largenequation} we have
  \begin{equation}\label{eq:rnislargeequation}
r_n \ge \frac{\epsilon}{2}\vartheta n,
\end{equation}
  but, from the first inequality of \eqref{eq:largenequation} this implies that \( r_n> R+b\) and hence that $r_{n+1}>R$.
  The Fermi coordinates \((t,r)\) given by projection and distance to a geodesic are \(\dd r^2 + \cosh(r)^2 \dd t^2\) and hence we obtain for the \(n\) under consideration
  \[|t_{n+1}-t_{n}| \le b\cosh(\min(r_{n},r_{n+1}))^{-1} \le b\cosh(R)^{-1}\le (1-\epsilon)\vartheta,\]
  by inequality~\eqref{eq:choiceofRequation}.

  It follows that, using \eqref{eq:tnupperboundequation} once again,
\[\frac{t_{n+1}}{n+1} = \frac{t_{n+1}-t_n+t_n}{n+1}\le \frac{(1-\epsilon)\vartheta+n(1-\epsilon)\vartheta}{n+1} = (1-\epsilon)\vartheta\]
and by induction we have \eqref{eq:tnupperboundequation} for all \(n\) large enough.

  But this implies that \eqref{eq:rnislargeequation} holds for all \(n\) large enough and therefore
  \[|t_{n+1} - t_n| \le b\cosh(n\epsilon \vartheta /2)^{-1},\]
  for all \(n\) large enough.

However, this contradicts Equation~\eqref{eq:tndivergenceequation}, so we have shown that
\begin{equation}\label{eq:conclpart1lem}
\vartheta \le \liminf\limits_{n \to +\infty}\frac{t_n}{n}.
\end{equation}

Consider the hyperbolic triangle with vertices \(\tgamma(0),\tgamma(t_n), \tf^n(\tx)\); its sides have length \(t_n,r_n\) and \(d_n = \tdist(\tgamma(0),\tf^n(\tx))\).  The angle between the sides of length \(t_n\) and \(r_n\) is a right angle so one has \(d_n \ge t_n\).   However, by hypothesis
\[\lim\limits_{n \to +\infty} \frac{1}{n}d_n = \vartheta(x),\]
so it follows that
\[\limsup\limits_{n \to +\infty}\frac{t_n}{n} \le \vartheta(x),\]
which, combined with \eqref{eq:conclpart1lem}, implies that
\begin{equation}\label{eq:limtnn}
\lim\limits_{n \to +\infty}\frac{t_n}{n} = \vartheta(x).
\end{equation}

Hence, by applying the hyperbolic law of cosines and Equation~\eqref{eq:limtnn} for the last line, we obtain
\begin{align*}
  \vartheta(x) &= \lim\limits_{n \to +\infty}\frac{1}{n}d_n 
       \\ &=\lim\limits_{n \to +\infty}\frac{1}{n}\log \cosh(d_n)
       \\ &=\lim\limits_{n \to +\infty}\frac{1}{n}\log \big(\cosh(t_n)\cosh(r_n)\big)
       \\ &=\lim\limits_{n \to +\infty}\left(\frac{1}{n}t_n + \frac{1}{n}r_n\right),
       \\ &= \vartheta(x) + \lim\limits_{n \to +\infty}\frac{1}{n}r_n,
\end{align*}
from where we deduce that \(\lim_{n\to +\infty}\frac{1}{n}r_n = 0\).

This implies by the triangle inequality that the tracking Equation~\eqref{eq:trackingequation} is satisfied for the positive iterates of \(\tf\).  The same argument shows that Equation~\eqref{eq:trackingequation} is satisfied for negative iterates of \(\tx\) as well.   To conclude we obtain that since \(\tf\) commutes with \(\Gamma\) (Equation~\eqref{eq:preferedliftequation}) the projection of \(\tgamma\) to \(S\) is a tracking geodesic for \(x\).
\end{proof}

We observe that the mapping \(\varphi\) assigning to a triplet \((\alpha,\omega,r)\) with \(\alpha,\omega \in \partial \D\) distinct and \(r \in \D\), the geodesic \(\tgamma\) joining \(\alpha\) to \(\omega\) -- with \(\tgamma(0)\) being the closest point to \(r\) --, is continuous. 
In view of Lemma~\ref{lem:trackingcharacterizationlemma} the measurability of \(x \mapsto \gamma_x\) follows since it is a composition of this mapping at \((\alpha(\wt x),\omega(\wt x),\tx)\) with \(\pi\).

\subsection{Ergodic measures and tracking geodesics}

If \(\mu \in \M\) then by definition \(\mu(S_{\vartheta > 0}) = 1\) and by Lemma~\ref{lem:twopointslemma}, the limit points \(\alpha(\wt x)\) and \(\omega(\wt x)\) are well defined at \(\mu\)-almost every point.

The fact that \(S_T\) has full measure for \(\mu\), and hence the statement of Theorem~\ref{thm:trackinggeodesictheorem}, follows from  Lemma~\ref{lem:trackingcharacterizationlemma} and the following result.

\begin{lemma}\label{LemTypNeq}
  For each \(\mu \in \M\) one has \(\alpha(\wt x) \neq \omega(\wt x)\) for \(\mu\)-a.e.\ \(x \in S\).
\end{lemma}

\begin{proof}
Let \(\beta_+,\beta_-\) be the geodesic rays from a basepoint \(O\) to \(\alpha(\tx)\) and \(\beta(\tx)\) as in Lemma~\ref{lem:twopointslemma}.

We consider the Busemann functions \(B_+\) and \(B_-\) associated to these geodesic half rays, that is
\[B_{\pm}(y) = \lim\limits_{t \to + \infty}\left(t- \tdist(y,\beta_{\pm}(t))\right).\]

We will use that \(B_{\pm}\) is \(1\)-Lipschitz and \(B_{\pm}(\beta_{\pm}(t)) = t\).

Let \(\varphi_{\pm}:S \to \R\) be defined by
\[\varphi_{\pm}(x) = B_{\pm}(\tf(\tx)) - B_{\pm}(\tx),\]
where \(\tx\) is any lift of \(x\).

Because \(B_{\pm}\) is \(1\)-Lipschitz, \(\varphi_{\pm}\) is bounded above by \(L_1\) and therefore integrable on $S$.  By Birkhoff ergodic theorem we have
\begin{equation}\label{eq:birkhoffequation}
\int_{S} \varphi_{\pm}(x) \dd \mu(x) = \lim\limits_{n \to +\infty}\frac{1}{n}\sum\limits_{k = 0}^{n-1}\varphi_{\pm}(f^k(x)) =   -\lim\limits_{n \to +\infty}\frac{1}{n}\sum\limits_{k = 1}^{n}\varphi_{\pm}(f^{-k}(x)).
\end{equation}

Lemma~\ref{lem:twopointslemma}, says that for \(\mu\)-a.e.\ \(x\) and any lift \(\tx\),
\[0 = \lim\limits_{n \to +\infty}\frac{1}{n}\tdist\big(\tf^n(\tx),\beta_+(\R)\big),\]
so, using the fact that $\beta_+$ is 1-Lipschitz, the definition \eqref{eq:trackingequation} of the rotation speed $\vartheta$ and $B_+(\beta_+(t)) = t$, one gets
\[\vartheta = \lim\limits_{n \to +\infty}\frac{1}{n}B_+(\tf^n(\tx)) = -\lim\limits_{n \to +\infty}\frac{1}{n}B_+(\tf^{-n}(\tx)) = \lim\limits_{n \to +\infty}\frac{1}{n}B_-(\tf^{-n}(\tx)),\]
which implies \(B_+ \neq B_-\), and hence \(\alpha(\tx) \neq \omega(\tx)\).
\end{proof}

\section{Equidistribution of tracking geodesics and minimal laminations --- proof of Theorems \ref{thm:equidistributiontheoremintro} and \ref{thm:ifsimpletheoremintro} \label{sec:equidistributionsection}}

\subsection{Proof of equidistribution}

For each \(x \in S\) consider a lift \(\tx\) and the corresponding lift \(\tgamma_\tx\) of \(\gamma_x\) satisfying the tracking Equation~\eqref{eq:trackingequation}.

Let \(T(x)\) be the difference between the projections of $\tx$ and $\tf(\tilde x)$ on $\tgamma_{\tilde x}$ measured by the parametrization of $\tgamma_\tx$ by $\R$, in other words
\[\tdist(\tf(\tx),\tgamma_\tx(\R)) = \tdist(\tf(\tx),\tgamma_\tx(T(x))).\]
Recall that by definition of the parametrization of $\tgamma_\tx$, one has
\[\tdist(\tx,\tgamma_\tx(\R)) = \tdist(\tx,\tgamma_\tx(0)).\]

Notice that \(T\) is measurable, and well defined by Equation~\eqref{eq:preferedliftequation}. We also have:

\begin{claim}
$T: S \to \mathbb{R}$ is bounded from above by $\widehat{L}_1 = \max_{x \in S} L_1(x)$. 
\end{claim}

\begin{proof}
Let $x \in S$, let $\tx$ a lift of $x$, and take Fermi coordinates $(t,r)$ with respect to the geodesic $\tgamma_{\tx}$. Note that the $t$-coordinates of $\tx$ and $\tf(\tx)$ are $0$ and $T(x)$, respectively. Recalling that in these coordinates we have that $ds^2 = dr^2 + \cosh(r)^2 dt^2$, we obtain that $L_1(x) = \tdist(\tf(x),\tx) \geq T(x)$.
\end{proof}

Defining \(T_n\) for \(n \in \N\) by
\[T_n(x) = T(x) + T(f(x)) + \cdots + T(f^{n-1}(x)),\]
we have
\[\tdist\big(\tf^n(\tx),\tgamma_{\tilde x}(\R)\big) = \tdist\big(\tf^n(\tx),\tgamma_{\tilde x}(T_n(x))\big).\]

From this observation, Birkhoff ergodic theorem and the tracking equation, we obtain the following.

\begin{remark}\label{rem:limitofT_n}
  For each \(\mu \in \M\), the rotation speed defined in \eqref{eq:trackingequation} satisfies
  \[\vartheta_\mu = \int_{S}T \dd \mu = \lim\limits_{n \to +\infty}\frac{1}{n}T_n(x),\]
  for \(\mu\)-a.e.\ \(x \in S\).
\end{remark}

Note that in particular, $\lim \limits_{n \to +\infty} T_n(x) = + \infty$ for \(\mu\)-a.e.\ \(x \in S\).

\begin{proof}[Proof of Theorem~\ref{thm:equidistributiontheoremintro}] 
\textbf{Equidistribution:}
Take a continuous function \(\varphi:\textnormal{T}^1S \to \R\), and define $\Mo\varphi: S \to \mathbb{R}$ by
\begin{equation}\label{eq:DefM}
(\Mo\varphi)(x) = \int_0^{T(x)}\varphi(\dgamma_x(t))\dd t.
\end{equation}
Given that the speed of $\dgamma_x$ is $\mu$-almost everywhere equal to 1, $\varphi$ is continuous and $T$ is bounded, we have that $\Mo\varphi \in \textnormal{L}^1(\mu)$. Let us define the linear operator
\begin{equation}\label{eq:DefI}
I(\varphi) := \int_{S}\Mo\varphi \dd \mu 
\end{equation}

Since \(\varphi\mapsto I(\varphi)\) is linear and positive, by the Riesz representation theorem there is some positive measure \(\nu_\mu\) on \(\textnormal{T}^1S\) such that
\[\int_{\textnormal{T}^1S} \varphi \dd \nu_\mu = \frac{I(\varphi)}{\vartheta_\mu},\]
for all continuous \(\varphi\). 

Note that (the last equality is Remark~\ref{rem:limitofT_n})
\[\nu_\mu\big(\textnormal{T}^1S\big) = \frac{I(1)}{\vartheta_\mu} = \frac{\int_{S}\Mo 1 \dd \mu}{\vartheta_\mu} = \frac{\int_{S}T \dd \mu}{\vartheta_\mu} = 1;\]
in other words $\nu_\mu$ is a probability measure. 

We then apply the ergodic theorem and obtain, for \(\mu\)-a.e.\ \(x \in S\),
\begin{align*}\label{eq:integralequation}
	I(\varphi) &= \lim\limits_{n \to +\infty}\frac{1}{n}\sum_{k=0}^{n-1}(\Mo\varphi)(f^k(x))\\ 
	&= \lim\limits_{n \to +\infty}\frac{1}{n}\int_{0}^{T_n(x)}\varphi(\dgamma_x(t))\dd t\\ 
	&\!\!\!\!\underset{Rk. \ref{rem:limitofT_n}}{=} \vartheta_\mu\lim_{n \to +\infty}\frac{1}{T_n(x)}\int_{0}^{T_n(x)}\varphi(\dgamma_x(t))\dd t\\ 
	& = \vartheta_\mu \lim_{n \to +\infty}\frac{1}{T_n(x)}\int_{T_n(f^{-n}(x))}^{0}\varphi(\dgamma_x(t))\dd t.
\end{align*}

Since \(T_{n+1}-T_n\) is uniformly bounded and $T_n$ tends to infinity as $n$ tends to infinity, this implies that, for \(\mu\)-a.e.\ \(x \in S\),
\begin{align*}
\int_{\mathrm{T}^1S} \varphi \dd \nu_\mu \ 
	&{=} \lim_{t \to +\infty}\frac{1}{t}\int_{0}^{t}\varphi(\dgamma_x(t))\dd t\\ 
	& = \lim_{t \to +\infty}\frac{1}{t}\int_{-t}^{0}\varphi(\dgamma_x(t))\dd t,
\end{align*}
hence \(\dot\gamma_x\) equidistributes to \(\nu_\mu\) for \(\mu\)-a.e. \(x\).

This fact also implies that $\nu_\mu$ is invariant under the geodesic flow.
\medskip

\textbf{Ergodicity of $\nu_\mu$:}
Suppose that the measure \(\nu_\mu\) is not ergodic. Then, there exists a measurable subset $A \subset \textnormal{T}^1 S$ invariant by the geodesic flow and of non-trivial $\nu_\mu$-measure.

The set of points $x\in S$ such that $\dot\gamma_x \in A$ is invariant by $f$ (because $\dot\gamma_x = \dot\gamma_{f(x)}$ up to reparametrization), then has measure 0 or 1. Up to changing $A$ to $S\setminus A$, we can assume it is 0.

The map $I$ defined in \eqref{eq:DefI} can be applied to any continuous map; let us extend it to indicator functions. Let $\lambda_\mu$ be the Borel probability measure defined on $\textnormal{T}^1S$ by $\lambda_\mu(B)=\int_S \Mo(\1_B)\dd\mu/\vartheta_\mu$ (the operator $\Mo$ is defined in \eqref{eq:DefM}); it follows from the following classical arguments that $\lambda_\mu=\nu_\mu$. 

Indeed, let $B\subset\mathrm{T}^1 M$ be a Borel set, and $\varep>0$. By the exterior regularity of the measure $\nu_\mu$, there exists an open set $O\subset T^1S$ containing $B$ such that $\nu_\mu(O)-\nu_\mu(B)<\varepsilon$. 
By the exterior regularity of the measure $\lambda_\mu$, there exists an open set $O'\subset S$ containing $B$ such that $\lambda_\mu(O')-\lambda_\mu(B)<\varepsilon$.
Let $U = O \cap O'$, then $0\leq \nu_{\mu}(U)-\nu_{\mu}(B)<\varepsilon$ and $0 \leq \lambda_\mu(U)-\lambda_\mu(B)<\varepsilon$.
Consider a non decreasing sequence $(\phi_n)_n$ of continuous maps that converges simply to $\1_U$ (obtained for example by playing with the distance to the complement of $U$). With the monotone convergence theorem, one gets
\[ \nu_\mu(U) = \int_{\textnormal{T}^1S} \1_U\dd\nu_\mu \underset{n\to+\infty}{\longleftarrow} \int_{\textnormal{T}^1S} \phi_n \dd \nu_\mu = \int_S \frac{\Mo\phi_n}{\vartheta_\mu} \dd\mu  \underset{n\to+\infty}{\longrightarrow}  \int_S \frac{\Mo \1_U}{\vartheta_\mu} \dd\mu = \lambda_\mu(U).\]
Hence, $|\nu_\mu(B)-\lambda(B)|<2 \varepsilon$ for any $\varep>0$, in other words
\[ \nu_\mu(B) = \lambda_\mu(B) = \int_S \Mo \1_B \dd\mu.\]

However, by hypothesis we have $\nu_\mu(A)\in (0,1)$; moreover for $\mu$-a.e.\ $x\in S$ one has $\dot\gamma_x \notin A$, hence (because $A$ is invariant under the geodesic flow) $(\Mo \1_A)(x)=0$ and so $\int_S \Mo \1_A \dd\mu = 0$. 
This is a contradiction with the above equality. 
\medskip

\textbf{Almost every tracking geodesic is included in the support of $\nu_{\mu}$:}

Let $A = \textnormal{supp}(\nu_{\mu})$. By the proof of ergodicity of $\nu_{\mu}$, we know that 
\[1 = \nu_{\mu}(A) = \lambda_\mu(A) = \frac{1}{\vartheta_{\mu}} \int_S (\Mo \1_A) (x) \dd \mu = \frac{1}{\vartheta_{\mu}} \int_S  \int_0^{T(x)} \1_A(\dgamma_x(t)) \dd t  \dd \mu. \]
Note that for every $x\in S$ we have $(\Mo \1_A) (x) \leq T(x)$, as we are integrating a function which is bounded by $1$. Furthermore, for every $x$, $(\Mo \1_A) (x)$ is either equal to 0 or to $T(x)$, because $A$ is invariant under the geodesic flow. We know by Remark~\ref{rem:limitofT_n}, that 
\[\frac{1}{\vartheta_{\mu}} \int_S T(x) \dd \mu = 1.\]
Given that these two integrals are equal, we obtain that for $\mu$-a.e.\ point, 
\[\int_0^{T(x)} \1_A(\dgamma_x(t)) \dd t = T(x),\]
from where we conclude that $\dgamma_x \in A$, as desired. 

Finally, let us check that \(\mu\)-a.e.\ \(x\), we have that \(\dgamma_x(\R)\) is dense in \( \textnormal{supp}(\nu_{\mu}) \). Take an open neighbourhood $V$ of a vector $v \in \textnormal{supp}(\nu_{\mu})$. Given $\nu_{\mu}(V)$ is positive, and that $\dgamma_x$ equidistributes to $\nu_{\mu}$, we have that $\dgamma_x$ must intersect $V$, which concludes the proof.
\end{proof}

\subsection{Minimal laminations for simple tracking geodesics}

We say a geodesic \(\gamma:\R \to S\) is \emph{simple} if it has no transverse self-intersections; that is: \(\gamma(t) = \gamma(s)\) implies \(\dgamma(t) = \dgamma(s)\). 

Recall that Theorem~\ref{thm:equidistributiontheoremintro} ensures that for \(\mu\)-a.e.\ \(x \in S\), we have
\[\textnormal{supp}(\nu_{\mu}) = \overline{\dgamma_x(\R)},\]

Moreover, again from Theorem~\ref{thm:equidistributiontheoremintro}, we get that \(\gamma_x\) is recurrent for \(\mu\)-a.e.\ \(x \in S\), by which we mean that for \(t\) there exists \(t_n \to +\infty\) such that
\[\dgamma_x(t) = \lim\limits_{n \to +\infty}\dgamma_x(t_n).\]

Theorem~\ref{thm:ifsimpletheoremintro} now follows immediately from:

\begin{lemma}[Recurrent simple geodesics yield minimal laminations]
  If \(\gamma:\R \to S\) is a recurrent simple geodesic then
  \[\Lambda = \overline{\gamma(\R)},\]
  is a minimal geodesic lamination.
\end{lemma}

\begin{proof}
The case where \(\gamma\) is periodic is obvious, so we assume from now on that \(\gamma\) is injective.

  It is immediate that \(\Lambda\) is a geodesic lamination. By \cite[corollary 4.7.2]{casson} to conclude that \(\Lambda\) is minimal it suffices to show that
  \begin{itemize}
    \item \(\Lambda\) is connected, and
    \item \(\Lambda\) has no isolated leaves.  
  \end{itemize}

  The first property is evident since the closure of a connected set is connected.

  We recall that the existence of an isolated leaf means there is some point \(y \in \Lambda\) with a neighborhood \(U\) homeomorphic to the unit disc \(\D\) in such a way that \(U \cap \Lambda\) is sent to the diameter \([-1,1]\) of \(\D\).   

  If such a pair \(y,U\) existed then, by recurrence of \(\gamma\), there would exist a sequence \(t_n \to +\infty\) such that \(\gamma(t_n) \in U\) for all \(n\). In particular one would have \(\gamma(t_n) \in \Lambda \cap U\) for all \(n\), and it follows that \(\gamma(t_n)\) belongs to the leaf of \(y\) in \(\Lambda\) for all \(n\). However, this would imply that \(\gamma\) is periodic.  Hence, \(\Lambda\) has no isolated leaves and therefore is minimal, concluding the proof.
\end{proof}

\section{Rotation vectors\label{sec:rotationsection}}

The purpose of this section is to relate the current work to the several different notions measuring rotation for homeomorphisms on closed hyperbolic surfaces.

\begin{itemize}
\item We give a necessary and sufficient condition for \(\M\) to be non-empty in terms of growth of the diameter of iterates of a fundamental domain (Proposition~\ref{prop:rotationcriteria}). This is the occasion to discuss the relation of tracking geodesics with the homotopical rotation vectors of \cite{pa}. 
\item We show that if there exists an ergodic measure with a non-zero homological rotation vector then this measure belongs to \(\M\) and its tracking geodesics ``converge'' to the corresponding homology vector (Proposition~\ref{prop:ErgMeasuresConstHomRotation}). 
\item Finally we prove in Proposition~\ref{prop:InterHomoImpliesInterGeod} that if two ergodic measures $\mu_1,\mu_2$ satisfy \(\rho_{H_1}(\mu_1) \wedge \rho_{H_1}(\mu_2)\neq 0 \), then for typical $x_1$ for $\mu_1$, $x_2$ for $\mu_2$, their tracking geodesics $\gamma_{x_1}$ and $\gamma_{x_2}$ intersect transversally, \emph{i.e.} are dynamically transverse, see Lemma~\ref{ClaimSpeedTransverseTraj} (moral: ``intersection in homology implies intersection of tracking geodesics''). 
\end{itemize}

Recall that the fundamental group $\Gamma$ of our surface $S$ acts by isometries on the universal cover $\tilde{S}$, in fact $\Gamma$ coincides with the group of deck transformations of $\tilde{S}$. Also, note that $\tilde{S}$ is a length space, i.e. the distance between two points can be defined as the infimum of lengths among curves joining those points.

\subsection{Condition for non-trivial rotation}

We consider a bounded fundamental domain \(D\subset \tilde S\) for the action of \(\Gamma\) on \(\tS \simeq \D\).  It is known that there exists at least one fixed point of \(\tf\) in \(D\) (this is Lefschetz fixed point theorem).  Therefore, if \(\M\) is non-empty then by Theorem~\ref{thm:trackinggeodesictheorem} the diameter of \(\tf^n(D)\) must grow linearly.  We show that this condition is in fact also sufficient for \(\M\) to be non-empty. This proposition was the initial motivation of the present work: it started from here!

\begin{proposition}\label{prop:rotationcriteria}
	Let \(D\) be a bounded fundamental domain for the action of \(\Gamma\) on \(\tS\).   If
	\[\limsup\limits_{n \to +\infty}\frac{1}{n}\diam(\tf^n(D)) = C > 0,\]
	then there exists \(\mu \in \M\) with \(\vartheta_\mu \ge C/2\). 
\end{proposition}

In particular, for every \(x \in S_T\) with tracking geodesic \(\gamma\), each lift \(\tx\) yields a rotation vector \((\alpha,\omega,v)\) in the sense of \cite{pa} (see Definition~\ref{Def:RotVect}),
with \(\alpha,\omega\) the limit points in $\partial \tilde S$ of the lift \(\tgamma\) of the tracking geodesic satisfying the tracking Equation~\eqref{eq:trackingequation}, and \(v = \vartheta(x)\). Hence, Proposition~\ref{prop:rotationcriteria} gives a criteria under which there are rotation vectors of $\tf$ that are realized by a sequence \(\tx_k\) belonging to a single \(\tf\)-orbit.

\begin{proof}
	We let \(\delta = \diam(D)\) and consider sequences \(x_k,y_k \in D\) and \(n_k \to +\infty\) such that
	\[\lim\limits_{k \to +\infty}\frac{1}{n_k}\tdist\big(\tf^{n_k}(\tx_k),\tf^{n_k}(\ty_k)\big) = C.\]
	
	By the triangle inequality we have
	\begin{align*}
		\tdist(\tf^{n_k}(\tx_k),\tf^{n_k}(\ty_k)) 
		&\le \tdist\big(\tf^{n_k}(\tx_k),\tx_k\big) + \tdist\big(\tx_k,\ty_k\big) + \tdist\big(\ty_k,\tf^{n_k}(\ty_k)\big) \\
		&\le \delta +   \tdist\big(\tf^{n_k}(\tx_k),\tx_k\big) + \tdist\big(\ty_k,\tf^{n_k}(\ty_k)\big),
	\end{align*}
	so (possibly exchanging \(x_k\) and \(y_k\) and taking a subsequence) we may assume without loss of generality that
	\begin{equation}\label{Eq:HypCroiss}
		d(\tx_k, \tf^{n_k}(\tx_k)) \ge \frac{C}{2}n_k-\delta,
	\end{equation}
	for all \(k\) large enough.
	
	Now, define 
	\[\mu_k = \frac{1}{n_k+1}\sum_{i=0}^{n_k} \delta_{f^i(x_k)}.\]
	By taking a subsequence if necessary, one can suppose that $\lim_{k \to +\infty}\mu_k = \mu$ for the weak-* topology. 
	
	It is immediate that \(\mu\) is \(f\)-invariant. From Kingman's subadditive ergodic theorem, for \(\mu\)-almost every \(x \in S\), one has ($L_n$ is defined in \eqref{eq:rotationdistanceequation})
	\[\vartheta(x) = \lim\limits_{n \to +\infty}\frac{1}{n}L_n(x),\]
and by dominated convergence theorem it follows that
\begin{equation}\label{eq:consDominated}
\int_{S} \vartheta(x) \dd\mu(x) = \lim\limits_{n \to +\infty}\frac{1}{n}\int_{S} L_n(x) \dd\mu(x).
\end{equation}
	
	We claim that
	\[\lim\limits_{n \to +\infty}\frac{1}{n}\int_{S}L_n(x) \dd\mu(x) \ge \frac{1}{2}C.\]
	From this it follows that some ergodic component of \(\mu\) yields the desired element of \(\M\) with rotation speed at least \(\frac{1}{2}C\).
	
	To prove the claim we observe that since \(L_m\) is continuous for each \(m\) we have 
	\begin{align*}
		\int_S L_m(x)\dd\mu(x) &= \lim\limits_{k \to +\infty}\int_S L_{m}(x)\dd \mu_k(x) 
		\\ &= \lim\limits_{k \to +\infty}\frac{1}{n_k+1}\int_S \tdist(\tx,\tf^m(\tx)) \dd\left(\sum_{i=0}^{n_k} \delta_{f^i(x_k)}\right)(x)
		\\ & = \lim\limits_{k \to +\infty}\frac{1}{n_k+1} \sum_{i=0}^{n_k} \tdist\big(\tf^i(\tx_k),\tf^{i+m}(\tx_k)\big).
	\end{align*}
	
	We separate the sum on the right into disjoint arithmetic progressions in the iterates of common difference $m$, plus a remainder which will be bounded:
	\[\int_S L_{m}(x)\dd \mu_k(x) 
	= \frac{1}{n_k+1} \left(\left(\sum_{j=0}^{m-1}\sum_{\ell=0}^{\lfloor n_k/m\rfloor-1} \tdist\big(\tf^{j+\ell m}(\tx_k),\tf^{j+(\ell+1) m}(\tx_k)\big)\right)-R_m\right),\]
	where $|R_m| \le m\max L_m\). The last sum, for it part, satisfies, with $|R'_m|\le 2\max_{1\le i\le m}L_i$ (by triangle inequality, the last inequality comes from \eqref{Eq:HypCroiss}),
	\begin{align*}
		\sum_{\ell=0}^{\lfloor n_k/m\rfloor-1} \tdist\big(\tf^{j+\ell m}(\tx_k),\tf^{j+(\ell+1) m}(\tx_k)\big) 
		& \ge \tdist\big(\tf^{j}(\tx_k),\tf^{j+\lfloor n_k/m\rfloor m}(\tx_k)\big) \\
		& \ge \tdist\big(\tx_k,\tf^{n_k}(\tx_k)\big) - R'_m\\
		& \ge \frac{C}{2} n_k-\delta - R'_m,
	\end{align*}
	
	Finally we obtain,
	\begin{align*}
		\int_{S}L_m(x)\dd\mu(x) &= \lim\limits_{k \to +\infty}\int_S L_{m}(x)\dd \mu_k(x)
		\\ & \ge \liminf\limits_{k \to +\infty}\left(\frac{1}{n_k+1} \sum_{j=0}^{m-1}\Big(\frac{C}{2} n_k-\delta - R'_m\Big)-\frac{R_m}{n_k+1}\right)
		\\ & \ge \liminf\limits_{k \to +\infty}\frac{m\frac{C}{2} n_k-m\delta - mR'_m-R_m}{n_k+1} = \frac{Cm}{2},
	\end{align*}
hence, by \eqref{eq:consDominated},
\[\vartheta_\mu = \int_{S} \vartheta(x) \dd\mu(x) \ge \frac{C}{2},\]
	which concludes the proof.
\end{proof}

\subsection{Homological rotation vectors and tracking geodesics}

We now relate tracking geodesics to the homological rotation vectors of \cite{MR88720} and \cite{pollicott}.

\begin{definition}
	Fix a bounded fundamental domain $D$. For every curve $\tilde{\gamma}: [0,t] \to \tS$, we define its \emph{homological class} $[\tgamma]_D = [a]$, where $a \in \Gamma$ is the only element such that $ba^{-1} \tgamma(t), b\tgamma(0) \in D$, for some other $b \in \Gamma$.     
\end{definition}

While this definition relies on the chosen fundamental domain, it is roughly independent for long curves: 

\begin{remark}\label{Rem:indpFond}
	Fix two bounded fundamental domains $D_1, D_2$. For every $\varepsilon > 0$, there exists $L > 0$ such that for any geodesic $\tgamma$ with length greater than $L$, we have that $$\frac{[\tgamma]_{D_1} - [\tgamma]_{D_2}}{L} < \varepsilon.$$ 
\end{remark}

This implies that the following notion does not depend on the choice of fundamental domain:

\begin{definition}[Homological class for geodesics]
	Given a parametrized geodesic $\tilde{\gamma} \subset \tS$ we define its \emph{homological class} $[\tgamma]$ as (when the limits are well defined)
	$$ [\tgamma] = \lim \limits_{t \to +\infty} \frac{1}{t} [\tgamma(t)] = \lim \limits_{t \to +\infty} \frac{1}{t}[\tgamma|_{[0,t]}],$$
	where the second equality will be a notation convention. 
\end{definition}

Proofs for the following lemma can be found in \cite[p.~140]{bridson}, and \cite{brodskiy}.

\begin{lemma}[\v{S}varc-Milnor Lemma]
	Let $\Gamma$ be a group which acts properly and cocompactly by isometries on a proper length space $X$. Then, $\Gamma$ is finitely generated (and inherits from it a metric by wordlength that is canonical up to quasi-isometry) and for every $x \in X$, the function $\gamma \mapsto \gamma x$ is a quasi-isometry.
\end{lemma}

\begin{proposition}[Ergodic measures' homological rotation is a.e.\ constant]\label{prop:ErgMeasuresConstHomRotation}
If \(\mu\) is an ergodic invariant measure for which \(\rho_{H_1}(\mu) \neq 0\), then \(\mu \in \M\) and furthermore
\[\rho_{H_1}(\mu) = [\tgamma_{\tx}],\]
for \(\mu\)-almost every \(x \in S\), where \(\tgamma_\tx\) is the lift of the tracking geodesic \(\gamma_x\) which satisfies \eqref{eq:trackingequation}.
\end{proposition}

\begin{proof}
Let $x$ be a typical point for $\mu$. Recall that $a_y$ was defined as the element of $\Gamma$ satisfying $a_y^{-1}\tf(\ty)\in D$.
By \v{S}varc-Milnor Lemma, the distance $\tdist(\tx,\tf^n(\tx))$ is, up to a constant, at least the word length of \(a_{f^{n-1}(x)}\cdots a_{x}\) with respect to some fixed finite symmetric generating set of \(\Gamma\). This is bounded from below by the norm of $[a_{f^{n-1}(x)}\dots a_{x}]$ which is roughly \(n\|\rho_{H_1}(\mu)\|\) by \eqref{eq:homologyequation}.
	
	Hence, we have
\[\vartheta_\mu = \lim\limits_{n \to +\infty}\frac{1}{n}\tdist(\tx, \tf^n(\tx)) \ge C\|\rho_{H_1}(\mu)\|\]
for some $C>0$, from which we deduce that \(\mu \in \M\).
	
	By Theorem~\ref{thm:trackinggeodesictheorem}, if \(x\) is typical with respect to \(\mu\), then \(\gamma_x\) has a lift \(\tgamma_x\) satisfying the tracking Equation~\eqref{eq:trackingequation}.
	It follows that
	\[\lim\limits_{t \to +\infty}\frac{1}{t}\left[\tgamma_x(t)\right] = \lim\limits_{n \to +\infty}\frac{1}{n}\left(\left[a_{f^{n-1}(x)}\dots a_{x}\right] + o(n)\right) = \rho_{H_1}(\mu),\]
	as claimed.
\end{proof}

\begin{proposition}\label{prop:InterHomoImpliesInterGeod}
	If \(\mu_1,\mu_2\) are ergodic invariant measures for which \(\rho_{H_1}(\mu_1) \wedge \rho_{H_1}(\mu_2)\neq 0 \), then for typical $x_1$ for $\mu_1$, $x_2$ for $\mu_2$, their tracking geodesics $\gamma_{x_1}$ and $\gamma_{x_2}$ intersect transversally.  
\end{proposition}

\begin{proof}
	Let us take a small geodesic arc $I$ which has $\gamma_{x_1}(0)$ in its interior and is transverse to $\gamma_{x_1}$. We moreover suppose that $I$ is not included in $\gamma_{x_2}$. From Theorem~\ref{thm:equidistributiontheoremintro} we have that for typical points, the geodesic $\dot{\gamma}_{x_1} \subset \textnormal{T}^1S$ is recurrent. Then, there exists a sequence $\{t_k\}_{k \in \N}$ of growing-to-infinity times, such that $\lim_{k \to +\infty} \dgamma_{x_1}(t_k) = \dgamma_{x_1}(0)$, and $\gamma_{x_1}(t_k) \in I$. 

Up to taking a subsequence we may complete each open geodesic $\gamma_{x_1} |_ {[0, t_k]}$ with a small arc $I_k \subset I$ of length less than $\frac{1}{2^k}$ (with speed near $\vartheta_{\mu_1}$), to obtain a simple closed curve $\gamma^1_k$. Let us first check that, setting
\[\lim_{t \to +\infty} \frac{1}{t}[\gamma^1_k|_{[0,t]}] := \rho_{H_1}(\gamma^1_k),\]
we have
\begin{equation*}\label{eq:ClosedGeodesicConvergeToTracking}
	\lim \limits_{k \to +\infty} \rho_{H_1}(\gamma^1_k)  = \lim \limits_{t \to +\infty} \frac{1}{t}[\gamma_{x_1}(t)] =  \rho_{H_1}(\mu_1).	
\end{equation*}
Indeed, the last of the equalities is true by Proposition~\ref{prop:ErgMeasuresConstHomRotation}. For the first one, parametrizing $\gamma^1_k$ by arclenght, we have
\[\rho_{H_1}(\gamma^1_k) = \frac{1}{t_k+\len(I_k)}[\gamma^1_k|_{[0,t_k+\len(I_k)]}] = \frac{[\gamma^1_k|_{[0,t_k]}]}{t_k+\len(I_k)} + \frac{[\gamma^1_k|_{[t_k,t_k+\len(I_k)]}]}{t_k+\len(I_k)}.\]
The numerator of the second fraction is bounded, as $\len(I_k)$, hence
\[\lim_{k \to +\infty}\rho_{H_1}(\gamma^1_k) = \lim_{k \to +\infty}\frac{[\gamma^1_k|_{[0,t_k]}]}{t_k} = \lim_{k \to +\infty}\frac{[\gamma_{x_1}|_{[0,t_k]}]}{t_k} = \rho_{H_1}(\mu_1).\]

This implies that for any $k$ large enough, we have
\[\big|\rho_{H_1}(\gamma_k^1)\wedge \rho_{H_1}(\mu_2)\big|\ge \frac{\big|\rho_{H_1}(\mu_1)\wedge \rho_{H_1}(\mu_2)\big|}{2} := \frac{r}{2}>0.\]	
Hence, decomposing $\gamma^1_k$ into a piece of $\gamma_{x_1}$ and $I_k$, one gets:
\[\liminf_{t\to+\infty}\frac{\Big| [\gamma_{x_1}|_{[0,t_k]}]\wedge [\gamma_{x_2}|_{[0,t]}]\Big| + \Big| I_k \wedge [\gamma_{x_2}|_{[0,t]}]\Big|}{t(t_k+\len(I_k))}
\ge \lim_{t\to+\infty}\frac{\Big| [\gamma_k^1|_{[0,t_k+\len(I_k)]}]\wedge [\gamma_{x_2}|_{[0,t]}]\Big|}{t(t_k+\len(I_k))}
\ge \frac{r}{2}.\]	
However, there exists some $r'>0$ such that, for any $t\ge 1$ and any $k\in\N$,
\[\big|I_k \wedge [\gamma_{x_2}|_{[0,t]}]\big| \le r't\]
(because the return time of $\gamma_{x_2}$ to $I_k$ is bounded from below by the geometry of the surface). This implies that
\[\liminf_{t\to+\infty}\frac{\Big| [\gamma_{x_1}|_{[0,t_k]}]\wedge [\gamma_{x_2}|_{[0,t]}]\Big|}{t(t_k+\len(I_k))}
\ge \frac{r}{2}-\frac{r'}{t_k+\len(I_k)}.\]	
Hence, for $k$ large enough we have 
\[\liminf_{t\to+\infty}\frac{\Big| [\gamma_{x_1}|_{[0,t_k]}]\wedge [\gamma_{x_2}|_{[0,t]}]\Big|}{t(t_k+\len(I_k))}>0,\]
which implies that $\gamma_{x_2}$ intersects $\gamma_{x_1}|_{[0,t_k]}$ transversely at a linear rate, which concludes the proof. 	
\end{proof}

\begin{coro}
	Let $\mu_1, \mu_2 \in \M$ such that their typical tracking geodesics belong to the same geodesic lamination. Then we have $\rho_{H_1}(\mu_1) \wedge \rho_{H_1}(\mu_2) = 0$.
\end{coro}

\section{Intersection of tracking geodesics --- proof of Theorem~\ref{thm:intertrackingimpliesentropy}\label{sec:forcing2}}

The aim of this section is to prove Theorem~\ref{thm:intertrackingimpliesentropy} of the introduction by means of the so-called \emph{forcing theory} of Le Calvez and Tal \cite{lecalveztalforcing,lct2,guiheneufforcing}. The proof strategy is inspired by the works of Lellouch \cite{lellouch} and Guihéneuf and Militon \cite{pa}. As premilimaries, we first introduce the notion of rotational topological horseshoe and some statements of forcing theory.

\subsection{Preliminaries on rotational horseshoes}

The following notions are thoroughly explained in \cite{zbMATH06864334}, \cite{lct2} and \cite[Section 9.2]{pa}. 

Given a homeomorphism $f$ of a surface $S$, we will say that $g$ is an \emph{extension of $f$} (and equivalently, we will also say that $f$ is a \emph{factor of $g$}) if there is a semiconjugacy from $g$ to $f$. We will say we  have a $m$-finite extension, when the fibres of the factor map are finite with cardinality uniformly bounded by $m\in\N$. 

\begin{definition}
	We will say a compact connected set $R \subset S$ is a \emph{rectangle} if it is homeomorphic to $[0,1]^2$ by a homeomorphism $h : [0,1]^2 \to h([0,1]^2) \subset S$. We will call \textit{sides} of $R$ the image of the sides of $[0,1]^2$ by $h$, the \textit{horizontal sides} being $R^- = h([0,1]\times \{0\})$ and $R^+ = h([0,1]\times \{1\})$. 
\end{definition}

\begin{definition}\label{def:MarkovianIntersection}
	Let $R_1, R_2 \subset S$ be two rectangles. We will say that $R_1 \cap R_2$ is a \textit{Markovian intersection} if there exists a homeomorphism $h$ from a neighbourhood of $R_1 \cup R_2$ to an open subset of $\R^2$, such that
	\begin{itemize}
		\item $h(R_2) = [0,1]^2$;
		\item Either $h(R_1^+) \subset \{(x,y) \in \R^2 : y > 1\}$ and $h(R_1^-) \subset \{(x,y) \in \R^2 : y < 0\}$; or $h(R_1^+) \subset \{(x,y) \in \R^2 : y < 0\}$ and $h(R_1^-) \subset \{(x,y) \in \R^2 : y > 1\}$;	
		\item $h(R_1) \subset \{(x,y) \in \R^2 \mid y > 1\} \cup [0,1]^2 \cup \{(x,y) \in \R^2 \mid y < 0\}$.
	\end{itemize}
\end{definition}

\begin{definition}[Rotational Horseshoe \cite{pa}]\label{def:rotationalhorseshoe}
	Let $f:S \to S$ be a homeomorphism of a hyperbolic surface $S$. We say that $f$ has a \emph{rotational horseshoe} associated to the covering transformations $U_1, \dots, U_r$ if there exists a rectangle $R \subset \tilde S$ and $j\ge 1$ such that for any $1 \le i \le r$, the intersections $\tilde f^j (R) \cap U_i(R)$ are Markovian. The set
\[X = \bigcap_{n\in\Z} f^{nj}(R)\]
is the rotational horseshoe and $j$ is its \emph{period}.
\end{definition}

The proof of the following result can be found in \cite{pa}.

\begin{proposition}\label{prop:topologicalhorseshoe}
Let $f \in \textnormal{Homeo}(S)$. Suppose that $f$ has a rotational horseshoe $X$ of period $j$ with associated covering transformations $U_1, \dots , U_r$ (with $r\ge 2$), which form a free group. Then
\begin{itemize}
	\item the map $f^j |_X$ has an $m$-finite extension $g: Y \to Y$, which on its turn is an extension of the Bernoulli shift $\sigma: \{1,\dots,r\}^{\mathbb{Z}} \to \{1,\dots,r\}^{\mathbb{Z}}$ in $r$ symbols.
In particular, the following diagram commutes:

\begin{center}
\begin{tikzpicture}[scale=1]
\node (A) at (0,0) {$X$};
\node (B) at (3,0) {$X$};
\node (C) at (0,1.5) {$Y$};
\node (D) at (3,1.5) {$Y$};
\node (E) at (0,3) {$\{1,\dots,r\}^\Z$};
\node (F) at (3,3) {$\{1,\dots,r\}^\Z$};

\draw[->,>=latex,shorten >=3pt, shorten <=3pt] (A) to node[midway, above]{$f^j$} (B);
\draw[->,>=latex,shorten >=3pt, shorten <=3pt] (C) to node[midway, above]{$g$} (D);
\draw[->,>=latex,shorten >=3pt, shorten <=3pt] (E) to node[midway, above]{$\sigma$} (F);
\draw[->,>=latex,shorten >=3pt, shorten <=3pt] (C) to node[midway, left]{$h$} (E);
\draw[->,>=latex,shorten >=3pt, shorten <=3pt] (D) to node[midway, right]{$h$} (F);
\draw[->,>=latex,shorten >=3pt, shorten <=3pt] (C) to node[midway, left]{$\pi$} (A);
\draw[->,>=latex,shorten >=3pt, shorten <=3pt] (D) to node[midway, right]{$\pi$} (B);
\end{tikzpicture}
\end{center}

\item The preimage of any $\sigma$-periodic sequence by the factor map $h : Y \to \{1,\dots,r\}^{\mathbb{Z}}$ contains a periodic point of $g$.  
\end{itemize}

\end{proposition}

\begin{remark}
If $f \in \textnormal{Homeo}(S)$ has a rotational horseshoe $X$ in $r$ symbols, then $h_{\textnormal{top}}(f) \geq \textnormal{log}(r)$. In particular, the existence of a topological horseshoe for $f$ implies positiveness of topological entropy of $f$.  
\end{remark}


\subsection{Preliminaries on forcing theory}

Given an isotopy $I = \{f_t\}_{t\in[0,1]}$ from the identity map to $f$, we define its fixed point set $\textnormal{Fix}(I) = \bigcap_{t\in[0,1]} \textnormal{Fix}(f_t)$, and denote its \emph{domain} $\Sigma = \textnormal{dom}(I) = S \backslash \textnormal{Fix}(I)$. Note that $\dom(I)$ is an oriented boundaryless surface, not necessarily closed, not necessarily connected.	

In this section we will consider an oriented surface $\Sigma$ without boundary, not necessarily closed or connected, and a non singular oriented topological foliation $\F$ on $\Sigma$. We will consider:
\begin{itemize}
	\item the universal covering space $\hat\Sigma$ of $\Sigma$;
	\item the covering projection $\hat\pi: \hat\Sigma\to \Sigma$;
	\item the group ${\mathcal G}$ of covering automorphisms;
	\item the lifted foliation $\hat{\F}$ on $\hat{\Sigma}$.
\end{itemize}

For every point $z\in\Sigma$, we denote $\phi_{z}$ the leaf of $\F$ that contains $z$. 
The complement of any simple injective proper path $\hat\alpha$ of $\hat\Sigma$ has two connected components, denoted by $L(\hat\alpha)$ and $R(\hat\alpha)$, chosen accordingly to some fixed orientation of $\hat\Sigma$ and the orientation of $\hat\alpha$.
Given a simple injective oriented proper path $\hat\alpha$ and $\hat z\in \hat\alpha$, we denote $\hat\alpha^+_{\hat z}$ and $\hat\alpha^-_{\hat z}$ the connected components of $\hat\alpha\setminus\{\hat z\}$, chosen accordingly to the orientation of $\hat\alpha$.

\paragraph{$\F$-transverse paths and $\F$-transverse intersections}

A path $\eta:J\to\Sigma$ is called \emph{positively transverse}\footnote{In the sequel, ``transverse'' will mean ``positively transverse''.} to $\F$ if it locally crosses each leaf of $\F$ from the right to the left. Observe that every lift  $\hat\eta:J\to\hat\Sigma$ of $\eta$ is positively transverse to $\hat{\F}$ and that for every $a<b$ in $J$: 
\begin{itemize}
	\item $\hat\eta|_{[a,b]}$ meets once every leaf $\hat\phi$ of $\hat \F$ such that $R(\hat\phi_{\hat \eta(a)})\subset R(\hat\phi)\subset R(\hat\phi_{\hat \eta(b)})$;
	\item  $\hat\eta|_{[a,b]}$ does not meet any other leaf.
\end{itemize}
We will say that two transverse paths $\hat\eta_1:J_1\to\hat\Sigma$ and $\hat\eta_2:J_2\to\hat\Sigma$ are \emph{equivalent} if they meet the same leaves of $\hat{\F}$. Two transverse paths $\eta_1:J_1\to\Sigma$ and $\eta_2:J_2\to\Sigma$ are \emph{equivalent} if there exists a lift  $\hat\eta_1:J_1\to\hat\Sigma$ of $\eta$ and a lift $\hat\eta_2:J_2\to\hat\Sigma$ of $\eta_2$ which are equivalent.

Let $\hat\eta_1:J_1\to \hat\Sigma$ and $\hat\eta_2:J_2\to \hat \Sigma$ be two transverse paths such that there exist $t_1\in J_1$ and $t_2\in J_2$ satisfying $\hat\eta_1(t_1)=\hat\eta_2(t_2)$. We will say that $\hat\eta_1$ and $\hat\eta_2$ have an \emph{$\hat{\F}$-transverse intersection} at $\hat\eta_1(t_1)=\hat\eta_2(t_2)$ (see Figure~\ref{Fig:extransverse}) if there exist $a_1, b_1\in J_1$ satisfying $a_1<t_1<b_1$  and $a_2, b_2\in J_2$ satisfying $a_2<t_2<b_2$ such that:
\begin{itemize}
	\item $\hat\phi_{\hat\eta_1(a_1)}\subset L(\hat\phi_{\hat\eta_2(a_2)}),\enskip \hat\phi_{\hat\eta_2(a_2)}\subset L(\hat\phi_{\widehat\eta_1(a_1)})$;
	
	\item $\hat\phi_{\widehat\eta_1(b_1)}\subset R(\hat\phi_{\hat\eta_2(b_2)}),\enskip \hat\phi_{\hat\eta_2(b_2)}\subset R(\hat\phi_{\widehat\eta_1(b_1)})$;
	
	\item every path joining $\hat\phi_{\hat\eta_1(a_1)}$ to $\hat\phi_{\hat\eta_1(b_1)}$ and every path joining $\hat\phi_{\hat\eta_2(a_2)}$ to $\hat\phi_{\hat\eta_2(b_2)}$ intersect. 
\end{itemize}

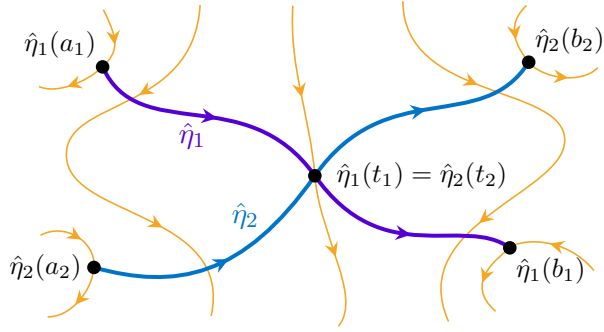
\begin{figure}
	\begin{center}
		
		\tikzset{every picture/.style={line width=0.6pt}} 
		
		\begin{tikzpicture}[x=0.75pt,y=0.75pt,yscale=-.9,xscale=.9]
			
			\draw [color={rgb, 255:red, 245; green, 166; blue, 35 }  ,draw opacity=1 ]   (261.43,50.73) .. controls (259.25,85.76) and (271.02,112.03) .. (273.64,145.74) .. controls (276.25,179.45) and (300.23,213.16) .. (286.72,229.8) ;
			\draw [shift={(266.37,101.91)}, rotate = 258.81] [fill={rgb, 255:red, 245; green, 166; blue, 35 }  ,fill opacity=1 ][line width=0.08]  [draw opacity=0] (8.04,-3.86) -- (0,0) -- (8.04,3.86) -- (5.34,0) -- cycle    ;
			\draw [shift={(285.23,191.05)}, rotate = 251.75] [fill={rgb, 255:red, 245; green, 166; blue, 35 }  ,fill opacity=1 ][line width=0.08]  [draw opacity=0] (8.04,-3.86) -- (0,0) -- (8.04,3.86) -- (5.34,0) -- cycle    ;
			\draw [color={rgb, 255:red, 245; green, 166; blue, 35 }  ,draw opacity=1 ]   (156.35,52.92) .. controls (164.64,69.56) and (165.07,76.12) .. (155.92,84.88) .. controls (146.76,93.64) and (130.19,100.2) .. (120.6,95.83) ;
			\draw [shift={(162.66,72.57)}, rotate = 263.3] [fill={rgb, 255:red, 245; green, 166; blue, 35 }  ,fill opacity=1 ][line width=0.08]  [draw opacity=0] (8.04,-3.86) -- (0,0) -- (8.04,3.86) -- (5.34,0) -- cycle    ;
			\draw [shift={(136.28,95.94)}, rotate = 340.46] [fill={rgb, 255:red, 245; green, 166; blue, 35 }  ,fill opacity=1 ][line width=0.08]  [draw opacity=0] (8.04,-3.86) -- (0,0) -- (8.04,3.86) -- (5.34,0) -- cycle    ;
			\draw [color={rgb, 255:red, 245; green, 166; blue, 35 }  ,draw opacity=1 ]   (391,50.6) .. controls (378.79,61.98) and (385.69,73.93) .. (392.23,82.25) .. controls (398.77,90.57) and (420.57,95.39) .. (430.6,85.32) ;
			\draw [shift={(385.1,69.68)}, rotate = 262.21] [fill={rgb, 255:red, 245; green, 166; blue, 35 }  ,fill opacity=1 ][line width=0.08]  [draw opacity=0] (8.04,-3.86) -- (0,0) -- (8.04,3.86) -- (5.34,0) -- cycle    ;
			\draw [shift={(414.19,90.98)}, rotate = 185.63] [fill={rgb, 255:red, 245; green, 166; blue, 35 }  ,fill opacity=1 ][line width=0.08]  [draw opacity=0] (8.04,-3.86) -- (0,0) -- (8.04,3.86) -- (5.34,0) -- cycle    ;
			\draw [color={rgb, 255:red, 245; green, 166; blue, 35 }  ,draw opacity=1 ]   (209.98,50.29) .. controls (206.49,112.03) and (133.24,101.96) .. (135.86,135.67) .. controls (138.48,169.38) and (228.73,165.88) .. (216.09,224.11) ;
			\draw [shift={(176.16,100.54)}, rotate = 329.51] [fill={rgb, 255:red, 245; green, 166; blue, 35 }  ,fill opacity=1 ][line width=0.08]  [draw opacity=0] (8.04,-3.86) -- (0,0) -- (8.04,3.86) -- (5.34,0) -- cycle    ;
			\draw [shift={(189.24,174.67)}, rotate = 208.12] [fill={rgb, 255:red, 245; green, 166; blue, 35 }  ,fill opacity=1 ][line width=0.08]  [draw opacity=0] (8.04,-3.86) -- (0,0) -- (8.04,3.86) -- (5.34,0) -- cycle    ;
			\draw [color={rgb, 255:red, 245; green, 166; blue, 35 }  ,draw opacity=1 ]   (322.47,48.1) .. controls (322.47,103.27) and (410.46,94.79) .. (412.2,125) .. controls (413.94,155.21) and (323.34,172.01) .. (347.76,227.17) ;
			\draw [shift={(364.02,94.87)}, rotate = 202.65] [fill={rgb, 255:red, 245; green, 166; blue, 35 }  ,fill opacity=1 ][line width=0.08]  [draw opacity=0] (8.04,-3.86) -- (0,0) -- (8.04,3.86) -- (5.34,0) -- cycle    ;
			\draw [shift={(364.67,172.23)}, rotate = 320.57] [fill={rgb, 255:red, 245; green, 166; blue, 35 }  ,fill opacity=1 ][line width=0.08]  [draw opacity=0] (8.04,-3.86) -- (0,0) -- (8.04,3.86) -- (5.34,0) -- cycle    ;
			\draw [color={rgb, 255:red, 84; green, 0; blue, 213 }  ,draw opacity=1 ][line width=1.5]    (155.92,84.88) .. controls (175.97,129.54) and (235.71,91.89) .. (273.64,145.74) .. controls (311.57,199.59) and (356.04,167.19) .. (381.77,186.02) ;
			\draw [shift={(219.54,113.55)}, rotate = 188.99] [fill={rgb, 255:red, 84; green, 0; blue, 213 }  ,fill opacity=1 ][line width=0.08]  [draw opacity=0] (8.75,-4.2) -- (0,0) -- (8.75,4.2) -- (5.81,0) -- cycle    ;
			\draw [shift={(327.1,178.63)}, rotate = 187.51] [fill={rgb, 255:red, 84; green, 0; blue, 213 }  ,fill opacity=1 ][line width=0.08]  [draw opacity=0] (8.75,-4.2) -- (0,0) -- (8.75,4.2) -- (5.81,0) -- cycle    ;
			\draw [color={rgb, 255:red, 0; green, 116; blue, 201 }  ,draw opacity=1 ][line width=1.5]    (151.12,196.96) .. controls (196.46,209.66) and (233.96,202.65) .. (273.64,145.74) .. controls (313.31,88.82) and (366.07,123.85) .. (392.23,82.25) ;
			\draw [shift={(225.84,192.34)}, rotate = 153.5] [fill={rgb, 255:red, 0; green, 116; blue, 201 }  ,fill opacity=1 ][line width=0.08]  [draw opacity=0] (8.75,-4.2) -- (0,0) -- (8.75,4.2) -- (5.81,0) -- cycle    ;
			\draw [shift={(335.43,108.69)}, rotate = 169.41] [fill={rgb, 255:red, 0; green, 116; blue, 201 }  ,fill opacity=1 ][line width=0.08]  [draw opacity=0] (8.75,-4.2) -- (0,0) -- (8.75,4.2) -- (5.81,0) -- cycle    ;
			\draw [color={rgb, 255:red, 245; green, 166; blue, 35 }  ,draw opacity=1 ]   (121.47,179.01) .. controls (134.99,171.57) and (152.43,182.51) .. (151.12,196.96) .. controls (149.81,211.41) and (137.17,224.98) .. (125.4,224.55) ;
			\draw [shift={(144.21,181.45)}, rotate = 211.25] [fill={rgb, 255:red, 245; green, 166; blue, 35 }  ,fill opacity=1 ][line width=0.08]  [draw opacity=0] (8.04,-3.86) -- (0,0) -- (8.04,3.86) -- (5.34,0) -- cycle    ;
			\draw [shift={(140.71,217.82)}, rotate = 311.55] [fill={rgb, 255:red, 245; green, 166; blue, 35 }  ,fill opacity=1 ][line width=0.08]  [draw opacity=0] (8.04,-3.86) -- (0,0) -- (8.04,3.86) -- (5.34,0) -- cycle    ;
			\draw [color={rgb, 255:red, 245; green, 166; blue, 35 }  ,draw opacity=1 ]   (428.42,198.71) .. controls (412.72,179.45) and (387.87,180.76) .. (381.77,186.02) .. controls (375.66,191.27) and (355.61,205.72) .. (373.92,221.48) ;
			\draw [shift={(403.94,183.63)}, rotate = 16.09] [fill={rgb, 255:red, 245; green, 166; blue, 35 }  ,fill opacity=1 ][line width=0.08]  [draw opacity=0] (8.04,-3.86) -- (0,0) -- (8.04,3.86) -- (5.34,0) -- cycle    ;
			\draw [shift={(366.62,205.34)}, rotate = 290.52] [fill={rgb, 255:red, 245; green, 166; blue, 35 }  ,fill opacity=1 ][line width=0.08]  [draw opacity=0] (8.04,-3.86) -- (0,0) -- (8.04,3.86) -- (5.34,0) -- cycle    ;
			\draw  [fill={rgb, 255:red, 0; green, 0; blue, 0 }  ,fill opacity=1 ] (270.32,145.74) .. controls (270.32,143.9) and (271.81,142.41) .. (273.64,142.41) .. controls (275.47,142.41) and (276.95,143.9) .. (276.95,145.74) .. controls (276.95,147.58) and (275.47,149.07) .. (273.64,149.07) .. controls (271.81,149.07) and (270.32,147.58) .. (270.32,145.74) -- cycle ;
			\draw  [fill={rgb, 255:red, 0; green, 0; blue, 0 }  ,fill opacity=1 ] (388.92,82.25) .. controls (388.92,80.41) and (390.4,78.92) .. (392.23,78.92) .. controls (394.06,78.92) and (395.55,80.41) .. (395.55,82.25) .. controls (395.55,84.09) and (394.06,85.58) .. (392.23,85.58) .. controls (390.4,85.58) and (388.92,84.09) .. (388.92,82.25) -- cycle ;
			\draw  [fill={rgb, 255:red, 0; green, 0; blue, 0 }  ,fill opacity=1 ] (378.45,186.02) .. controls (378.45,184.18) and (379.94,182.69) .. (381.77,182.69) .. controls (383.6,182.69) and (385.08,184.18) .. (385.08,186.02) .. controls (385.08,187.86) and (383.6,189.35) .. (381.77,189.35) .. controls (379.94,189.35) and (378.45,187.86) .. (378.45,186.02) -- cycle ;
			\draw  [fill={rgb, 255:red, 0; green, 0; blue, 0 }  ,fill opacity=1 ] (147.8,196.96) .. controls (147.8,195.12) and (149.29,193.63) .. (151.12,193.63) .. controls (152.95,193.63) and (154.44,195.12) .. (154.44,196.96) .. controls (154.44,198.8) and (152.95,200.29) .. (151.12,200.29) .. controls (149.29,200.29) and (147.8,198.8) .. (147.8,196.96) -- cycle ;
			\draw  [fill={rgb, 255:red, 0; green, 0; blue, 0 }  ,fill opacity=1 ] (152.6,84.88) .. controls (152.6,83.04) and (154.09,81.55) .. (155.92,81.55) .. controls (157.75,81.55) and (159.23,83.04) .. (159.23,84.88) .. controls (159.23,86.72) and (157.75,88.21) .. (155.92,88.21) .. controls (154.09,88.21) and (152.6,86.72) .. (152.6,84.88) -- cycle ;
			
			\draw (284.07,145.38) node [anchor=west] [inner sep=0.75pt]  [font=\small]  {$\hat{\eta }_{1}( t_{1}) =\hat{\eta }_{2}( t_{2})$};
			\draw (205.68,114.4) node [anchor=north] [inner sep=0.75pt]  [color={rgb, 255:red, 84; green, 0; blue, 213 }  ,opacity=1 ]  {$\hat{\eta }_{1}$};
			\draw (244.15,176.67) node [anchor=south east] [inner sep=0.75pt]  [color={rgb, 255:red, 0; green, 116; blue, 201 }  ,opacity=1 ]  {$\hat{\eta }_{2}$};
			\draw (394.23,78.85) node [anchor=south west] [inner sep=0.75pt]  [font=\small]  {$\hat{\eta }_{2}( b_{2})$};
			\draw (383.77,189.42) node [anchor=north west][inner sep=0.75pt]  [font=\small]  {$\hat{\eta }_{1}( b_{1})$};
			\draw (145.8,196.96) node [anchor=east] [inner sep=0.75pt]  [font=\small]  {$\hat{\eta }_{2}( a_{2})$};
			\draw (153.92,81.48) node [anchor=south east] [inner sep=0.75pt]  [font=\small]  {$\hat{\eta }_{1}( a_{1})$};

		\end{tikzpicture}
		
		\caption{Example of $\hat{\F}$-transverse intersection.\label{Fig:extransverse}}
	\end{center}
\end{figure}

A transverse intersection means that there is a ``crossing'' between the two paths naturally defined by $\hat\eta_1$ and $\hat\eta_2$ in the space of leaves of $\widehat{\F}$, which is a one-dimensional topological manifold, usually non Hausdorff.

\medskip
Now, let $\eta_1:J_1\to \Sigma$ and $\eta_2:J_2\to \Sigma$ be two transverse paths such that there exist $t_1\in J_1$ and $t_2\in J_2$ satisfying $\eta_1(t_1)=\eta_2(t_2)$. We say that $\eta_1$ and $\eta_2$ have an \emph{${\F}$-transverse intersection} at $\eta_1(t_1)=\eta_2(t_2)$ if, given $\hat\eta_1:J_1\to \hat \Sigma$ and $\hat\eta_2:J_2\to \hat \Sigma$ any two lifts of $\eta_1$ and $\eta_2$ such that $\hat\eta_1(t_1)=\hat\eta_2(t_2)$, we have that $\hat\eta_1$ and $\hat\eta_2$ have a ${\hat{\F}}$-transverse intersection at $\hat\eta_1(t_1)=\hat\eta_2(t_2)$, 
If $\eta_1=\eta_2$ one speaks of a \emph{$\F$-transverse self-intersection}. In this case, if $\widehat \eta_1$ is a lift of $\eta_1$, then there exists $T\in\mathcal G$ such that $\widehat\eta_1$ and $T\widehat\eta_1$ have a $\widehat{\F}$-transverse intersection at $\widehat\eta_1(t_1)=T\widehat\eta_1(t_2)$.

\paragraph{Recurrence, equivalence and accumulation} 

We will say a transverse path $\eta:\R\to \Sigma$ is \emph{positively recurrent} if, for every $a<b$, there exist $c<d$, with $b<c$, such that $\eta|_{[a,b]}$ and $\eta|_{[c,d]}$  are equivalent. Similarly $\eta$ is {\it negatively recurrent} if, for every $a<b$, there exist  $c<d$, with $d<a$,  such that $\eta|_{[a,b]}$ and $\eta|_{[c,d]}$  are equivalent. Finally $\eta$ is {\it recurrent} if it is both positively and negatively recurrent.
\medskip 

Two transverse paths $\eta_1:\R\to \Sigma$ and $\eta_2:\R\to \Sigma$ are {\it equivalent at $+\infty$} if there exists $a_1$ and $a_2$ in $\R$ such that $\eta_1{}|_{[a_1,+\infty)}$ and  $\eta_2{}|_{[a_2,+\infty)}$ are equivalent. Similarly $\eta_1$ and $\eta_2$ are {\it equivalent at $-\infty$} if there exists $b_1$ and $b_2$ in $\R$ such that $\eta_1{}|_{(-\infty,b_1]}$ and  $\eta_2{}|_{(-\infty,b_2]}$ are equivalent.
\medskip

\begin{figure}
\begin{center}

\tikzset{every picture/.style={line width=0.75pt}} 

\begin{tikzpicture}[x=0.75pt,y=0.75pt,yscale=-1,xscale=1]

\draw [color={rgb, 255:red, 208; green, 2; blue, 27 }  ,draw opacity=1 ]   (380,170) .. controls (383.33,138.17) and (377.33,65.5) .. (380,10) ;
\draw [shift={(380.14,94.49)}, rotate = 268.55] [fill={rgb, 255:red, 208; green, 2; blue, 27 }  ,fill opacity=1 ][line width=0.08]  [draw opacity=0] (8.04,-3.86) -- (0,0) -- (8.04,3.86) -- (5.34,0) -- cycle    ;
\draw [color={rgb, 255:red, 245; green, 166; blue, 35 }  ,draw opacity=1 ]   (355,170) .. controls (352.67,137.17) and (357,65.83) .. (355,10) ;
\draw [shift={(354.96,94.27)}, rotate = 271.05] [fill={rgb, 255:red, 245; green, 166; blue, 35 }  ,fill opacity=1 ][line width=0.08]  [draw opacity=0] (8.04,-3.86) -- (0,0) -- (8.04,3.86) -- (5.34,0) -- cycle    ;
\draw [color={rgb, 255:red, 245; green, 166; blue, 35 }  ,draw opacity=1 ]   (315,170) .. controls (318.33,138.17) and (320,46.83) .. (315,10) ;
\draw [shift={(318.11,93.76)}, rotate = 270.34] [fill={rgb, 255:red, 245; green, 166; blue, 35 }  ,fill opacity=1 ][line width=0.08]  [draw opacity=0] (8.04,-3.86) -- (0,0) -- (8.04,3.86) -- (5.34,0) -- cycle    ;
\draw [color={rgb, 255:red, 245; green, 166; blue, 35 }  ,draw opacity=1 ]   (265,170) .. controls (259.67,135.83) and (262.33,65.5) .. (265,10) ;
\draw [shift={(262.04,94.53)}, rotate = 270.94] [fill={rgb, 255:red, 245; green, 166; blue, 35 }  ,fill opacity=1 ][line width=0.08]  [draw opacity=0] (8.04,-3.86) -- (0,0) -- (8.04,3.86) -- (5.34,0) -- cycle    ;
\draw [color={rgb, 255:red, 245; green, 166; blue, 35 }  ,draw opacity=1 ]   (370,170) .. controls (372,137.83) and (368,51.5) .. (370,10) ;
\draw [shift={(370.01,94.42)}, rotate = 269.07] [fill={rgb, 255:red, 245; green, 166; blue, 35 }  ,fill opacity=1 ][line width=0.08]  [draw opacity=0] (8.04,-3.86) -- (0,0) -- (8.04,3.86) -- (5.34,0) -- cycle    ;
\draw [color={rgb, 255:red, 0; green, 99; blue, 217 }  ,draw opacity=1 ]   (200,120) .. controls (242.67,119.17) and (368,119.5) .. (375,10) ;
\draw [shift={(316.31,99.16)}, rotate = 154.81] [fill={rgb, 255:red, 0; green, 99; blue, 217 }  ,fill opacity=1 ][line width=0.08]  [draw opacity=0] (8.04,-3.86) -- (0,0) -- (8.04,3.86) -- (5.34,0) -- cycle    ;
\draw [color={rgb, 255:red, 0; green, 99; blue, 217 }  ,draw opacity=1 ]   (200,145) .. controls (243.67,141.83) and (373,147.83) .. (425,145) ;
\draw [shift={(316.02,144.99)}, rotate = 180.97] [fill={rgb, 255:red, 0; green, 99; blue, 217 }  ,fill opacity=1 ][line width=0.08]  [draw opacity=0] (8.04,-3.86) -- (0,0) -- (8.04,3.86) -- (5.34,0) -- cycle    ;

\draw (240.25,145.4) node [anchor=north] [inner sep=0.75pt]  [color={rgb, 255:red, 0; green, 99; blue, 217 }  ,opacity=1 ]  {$\eta _{2}$};
\draw (299,102.6) node [anchor=south east] [inner sep=0.75pt]  [color={rgb, 255:red, 0; green, 99; blue, 217 }  ,opacity=1 ]  {$\eta _{1}$};
\draw (386,96.88) node [anchor=west] [inner sep=0.75pt]  [color={rgb, 255:red, 208; green, 2; blue, 27 }  ,opacity=1 ]  {$\phi _{\eta _{2}( b_{2})}$};

\end{tikzpicture}

\caption{An accumulation configuration in $\R^2$: the transverse path $\eta_1$ {accumulates positively} on the transverse path $\eta_2$. The leaves are in orange, the limit leaf in red.\label{FigDefAccumulate}}
\end{center}
\end{figure}
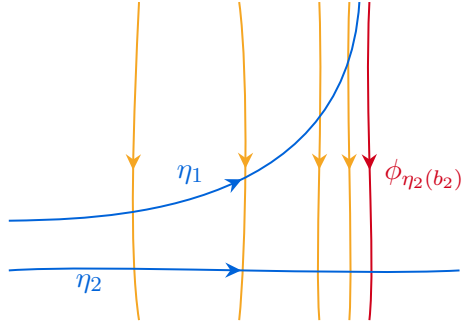

\label{DefAccumulate}
A transverse path $\eta_1:\R\to \Sigma$ \emph{accumulates positively} on the transverse path $\eta_2:\R\to \Sigma$ if there exist real numbers $a_1$ and $a_2<b_2$ such that  $\eta_1{}|_{[a_1,+\infty)}$ and $\eta_2{}|_{[a_2,b_2)}$ are equivalent (see Figure~\ref{FigDefAccumulate}). Similarly, $\eta_1$  {\it accumulates negatively} on $\eta_2$ if there exist real numbers $b_1$ and $a_2<b_2$ such that  $\eta_1{}|_{(-\infty,b_1]}$ and $\eta_2{}|_{(a_2,b_2]}$ are equivalent. 
Finally $\eta_1$ {\it accumulates} on $\eta_2$ if it accumulates  positively or negatively on $\eta_2$.

\paragraph{Strips}

We fix $T\in {\mathcal G}\setminus\{0\}$, and $\hat\alpha$ a $T$-invariant $\hat\F$-transverse path.
The set
\[\hat B=\{\hat z \in \hat\Sigma\,\vert\enskip \hat\phi_{\hat z} \cap \hat\alpha\not=\varnothing\}\]
is an $\hat{\F}$-saturated and $T$-invariant plane. We will call such a set a {\it strip} or a {\it $T$-strip}.
The boundary $\partial \hat B$ of $\hat B$ is a union of leaves (possibly empty) and can be written as $\partial\hat B=\partial\hat B^R\sqcup\partial\hat B^L$, where 
\[\partial\hat B^R=\partial\hat B\cap R(\hat\alpha)\,,\qquad \partial\hat B^L=\partial\hat B\cap L(\hat\alpha).\]
Let us state some facts that can be proven easily (see \cite{lct2} or \cite{lellouch}). Note first that if there is a $T$-invariant leaf $\hat \phi\subset \partial\hat B$, then the set $ \partial\hat B^R$ or $ \partial\hat B^L$ that contains $\hat\phi$ is reduced to this leaf.

Suppose now that $\hat\eta:\R\to\hat\Sigma$ is transverse to $\hat{\F}$ and that 
$$\big\{ t\in\R\,\vert\enskip \eta(t)\in\hat B\big\}=(a,b),$$ 
where $-\infty\leq a<b\leq\infty$. Say that
\begin{itemize}
	\item $\hat\eta$ {\it draws $\hat B$} if there exist $t<t'$ in $(a,b)$ such that $\hat\phi_{\hat\eta(t')}=T\hat\phi_{\hat\eta(t))}$.
\end{itemize}

If, moreover, we have $-\infty<a<b<+\infty$, say that:
\begin{itemize}
	\item $\hat\eta$ {\it crosses $\hat B$ from the left to the right} if $\hat\eta(a)\in \partial\hat B^L$ and $\hat\eta(b)\in \partial\hat B^R$;
	\item $\hat\eta$ {\it crosses $\hat B$ from the right to the left }if $\hat\eta(a)\in \partial\hat B^R$ and $\hat\eta(b)\in \partial\hat B^L$;
	\item $\hat\eta$ {\it visits $\hat B$ on the left} if $\hat\eta(a)\in \partial\hat B^L$ and $\hat\eta(b)\in \partial\hat B^L$;
	\item $\hat\eta$ {\it visits $\hat B$ on the right} if $\hat\eta(a)\in \partial\hat B^R$ and $\hat\eta(b)\in \partial\hat B^R$.
\end{itemize}
We will say that $\hat\eta$ {\it crosses} $\hat B$ if it crosses it from the right to the left or from the left to the right. Similarly, we will say that $\hat\eta$ {\it visits} $\hat B$ if it visits it on the right or on the left. Finally, observe that at least one of the following situations occurs (the two last assertions are not incompatible):
\begin{itemize}\label{ListPossibBand}
	\item $\hat\eta$ crosses $\hat B$;
	\item $\hat\eta$ visits $\hat B$;
	\item $\hat\eta$ is equivalent to $\hat \alpha$ at $+\infty$ or at $-\infty$;
	\item $\hat\eta$ accumulates on $\hat\alpha$ positively or negatively.
\end{itemize}

%

\paragraph{Brouwer foliations}

Let $\F$ be a singular foliation of a surface $S$; $\textnormal{Sing}(\F)$ the set of singularities of $\F$, and let us call $\textnormal{dom}(\F) := S \backslash \textnormal{Sing}(\F)$. 
The forcing theory grounds on the following result of existence of transverse foliations, which can be obtained as a combination of \cite{lecalvezfoliations} with \cite{bguin2016fixed}.

\begin{theorem}\label{thm:ThExistIstop}
	Let $S$ be a surface and $f\in\Homeo_0(S)$.
	Then there exists an isotopy $I$ linking $\Id$ to $f$, and a transverse topological oriented singular foliation $\F$ of $S$ with $\dom(\F) = S \backslash \textnormal{Fix}(I) = \Sigma$, such that:

	For any $z\in \Sigma$, there exists an $\F$-transverse path denoted by $\big(I_\F^t(z)\big)_{t\in[0,1]}$, linking $z$ to $f(z)$, that is homotopic in $\Sigma$, relative to its endpoints, to the arc $(I^t(z))_{t\in[0,1]}$.
\end{theorem}

This allows to define the path $I_{\F}^\Z (x)$ as the concatenation of the paths $\big(I_\F^t(f^n(z))\big)_{t\in[0,1]}$ for $n\in\Z$.

In the sequel we will need the following result about the local transversality of trajectories.

\begin{proposition}\label{LemLocalTransverse}
Let $\Sigma$ be a surface, $\F$ a singular foliation on $\Sigma$ and $f\in\Homeo_0(\Sigma)$.
Suppose that for any $x\in \Sigma\setminus\sing \F$ we are given an $\F$-transverse trajectory $I_\F(x)$ linking $x$ to $f(x)$ and homotopic relative to fixed points to a fixed isotopy $I$ from identity to $f$. 
Then for any neighbourhood $V_0\subset \Sigma$ of $\sing \F$, there exists a neighbourhood $U_0\subset V_0$ of $\sing \F$ such that for any $x\in U_0\setminus \sing \F$, there exists an $\F$-transverse trajectory $I'_\F(x)$ linking $x$ to $f(x)$, homotopic to $I_\F(x)$ and included in $V_0$. 
\end{proposition}

We postpone the proof of this proposition to Appendix~\ref{SecAppendix}.

\subsection{Intersection of tracking geodesics}\label{SubsecInter}

The goal of this subsection is to prove the following statement, which specifies Theorem~\ref{thm:intertrackingimpliesentropy} of the introduction. 

\begin{theorem}\label{TheoInterGeod}
Consider two $f$-invariant ergodic probability measures $\mu_1,\mu_2\in \M$ that are dynamically transverse (defined in Definition~\ref{def:dynatrans}). Then:
\begin{itemize}
\item The homeomorphism $f$ has a topological horseshoe (see Definition~\ref{def:rotationalhorseshoe}), in particular the topological entropy of $f$ is positive. 
\item For $\mu_1$-a.e.\ $x_1\in S$ and $\mu_2$-a.e.\ $x_2\in S$, the triples $\big(\alpha(\tilde x_1),\omega(\tilde x_2),\max(\vartheta_{\mu_1},\vartheta_{\mu_2})\big)$ and $\big(\alpha(\tilde x_2),\omega(\tilde x_1),\max(\vartheta_{\mu_1},\vartheta_{\mu_2})\big)$ are rotation vectors in the sense of \cite{pa} (Definition~\ref{Def:RotVect}).
\item For any $r\in H_1(S,\R)$ in the triangle spanned by $0, \rho_{H_1}(\mu_1), \rho_{H_1}(\mu_2)$ and any $\varepsilon>0$, there exists a periodic point $z\in S$ and four lifts $(\tilde z_j)_{1\le j\le 4}$ of $z$ in $\tilde S$ such that, denoting $d$ the distance on geodesics induced by some distance on $\Sp^1\times \Sp^1$ on the couples of their endpoints,
\begin{itemize}
\item $\rho_{H_1}(z)\in B(r,\varepsilon)$;
\item $d(\tgamma_{\wt x_1},\tgamma_{\wt z_1})<\varepsilon$;
\item $d(\tgamma_{\wt x_2},\tgamma_{\wt z_2})<\varepsilon$;
\item $d\big(\big(\alpha(\tilde x_1),\omega(\tilde x_2)\big),\, \tgamma_{\wt z_3}\big)<\varepsilon$;
\item $d\big(\big(\alpha(\tilde x_2),\omega(\tilde x_1)\big),\, \tgamma_{\wt z_4}\big)<\varepsilon$.
\end{itemize}
\end{itemize}
\end{theorem}

The last point expresses that any convex combination of $0, \rho_{H_1}(\mu_1), \rho_{H_1}(\mu_2)$ is accumulated by rotation vectors of periodic points of $f$, such that some lifts of the axes of these points accumulate on both tracking geodesics $\tgamma_{\wt x_1}$ and $\tgamma_{\wt x_2}$.

Let us start with a lemma stating that if $\mu_1$ and $\mu_2$ are {dynamically transverse}, then tracking geodesics for typical points of respectively $\mu_1$ and $\mu_2$ intersect transversally with an angle bounded away from 0.

\begin{lemma}\label{lem:DynTransThenAlmostEveryTrackGeodTrans}
	Let $\mu_1, \mu_2 \in \M$ be two dynamically transverse ergodic measures. Then there exists $\theta_0>0$ such that for $\mu_1$-a.e.\ $x_1$ and $\mu_2$-a.e.\ $x_2$, the geodesics $\gamma_{x_1}$ and $\gamma_{x_2}$ intersect transversally with angle of intersection bigger than $\theta_0$. In particular, given a lift $\tgamma_{\tx_1}$ of $\gamma_{x_1}$, there exists a lift $\tgamma_{\tx_2}$ of $\gamma_{x_2}$ such that $\tgamma_{\tx_1}$ and $\tgamma_{\tx_2}$ intersect transversally. 
\end{lemma}

\begin{proof}
Because $\mu_1$ and $\mu_2$ are dynamically transverse, there exist $v_1 \in \dotLambda_{\mu_1}, \ v_2 \in \dotLambda_{\mu_2}$ such that $\pi_S(v_1) = \pi_S(v_2)$ and $v_1 \neq v_2$. 
Provided $\varepsilon$ is sufficiently small, any two geodesics $\gamma_1$ passing $\varepsilon$-close to $v_1$, and $\gamma_2$ passing $\varepsilon$-close to $v_2$, must intersect transversally, with angle at least $\frac12 \angle(v_1,v_2):=\theta_0$. The proof is then a direct consequence of the density in $\textnormal{supp}(\nu_{\mu_1})$ of $\mu_1$-a.e.\ tracking geodesic (and the same for $\mu_2$), stated in Theorem~\ref{thm:equidistributiontheoremintro}.
\end{proof}

\begin{proof}[Proof of Theorem~\ref{TheoInterGeod}]
Let $\tilde x_1$ and $\tilde x_2$ be two points as in the statement of Lemma~\ref{lem:DynTransThenAlmostEveryTrackGeodTrans}: they are lifts of $x_1$, $x_2$ that are typical for respectively $\mu_1$ and $\mu_2$, and their tracking geodesics $\tgamma_{\tx_1}$ and $\tgamma_{\tx_2}$ intersect. Later on, we will make additional assumptions on these points, assumptions that will be typical for the measures $\mu_1$ and $\mu_2$.
Let $\F$ be a transverse foliation given by Theorem~\ref{thm:ThExistIstop} and $\tilde \F$ a lift of it to $\tilde S$. 
Let $\tilde I_{\tilde \F}^\Z(\tilde x_1)$ and $\tilde I_{\tilde \F}^\Z(\tilde x_2)$ be the transverse trajectories (given by Theorem~\ref{thm:ThExistIstop}) --- parametrized by $\R$ --- in $\tilde S$ associated to these orbits. We remark that here we lift to $\tilde S$ and not the universal cover of the domain of the isotopy.

\begin{lemma}\label{ClaimSpeedTransverseTraj}
Let $\mu_1, \mu_2 \in \M$ be two dynamically transverse ergodic measures. Then for $\mu_1$-a.e.\ $x_1\in S$, $\mu_2$-a.e.\ $x_2\in S$, for $i=1,2$, there exists an $\F$-transverse loop $\beta_{i+1}$ associated to a deck transformation $T$ whose axis is close enough to $\gamma_{x_{i+1}}$ so that it intersects $\gamma_{x_i}$ with angle at least $\theta_0/2$.
\end{lemma}

\begin{proof}
We fix a regular fundamental domain $D$ of $S$ in $\tilde S$ (its boundary has 0 $\mu_{i+1}$-measure) and define the map $x\in S\mapsto \tx\in\tilde S$ accordingly.
By Lusin theorem, there exists a set $A$ of continuity of the map $x\mapsto\tgamma_{\tx}$ of positive $\mu_{i+1}$-measure. If $x_{i+1}\in A$ is $\mu_{i+1}$-typical, then it has a well defined tracking geodesic $\tgamma_{\tx_{i+1}}$ that intersects some lift of the tracking geodesic $\gamma_{x_i}$ of a $\mu_i$-typical point $x_i$ with angle $\ge \theta_0$ (by Lemma~\ref{lem:DynTransThenAlmostEveryTrackGeodTrans}); moreover the point $x_{i+1}$ is recurrent in $A$. 
Let $W$ be a small open chart of the singular foliation $\F$ on $S$ around $x_{i+1}$, that is included in the projection of the interior of $D$.
Let $n>0$ (to be chosen large thereafter) such that $f^n(x_{i+1})\in A\cap W$. Let $T$ be the deck transformation such that if $\tilde W$ is the lift of $W$ satisfying $\tx_{i+1}\in\tilde W$, then $\tilde f^n(\tx_{i+1})\in T \tilde W$.

Let us explain why if $W$ is chosen small enough, and if $n$ is large enough, then the axis of the deck transformation $T$ is close to $\tgamma_{\tx_{i+1}}$. 
Indeed, on the one hand, if $n$ is large enough, then $\pr_{\tgamma_{\tx_{i+1}}}(\tilde f^n(\tx_{i+1}))$ is large; in particular $T\tx_{i+1}$, which belongs to the same fundamental domain as $\tilde f^n(\tx_{i+1})$, is close to $\omega(\tx_{i+1})$. 
On the other hand,
\begin{equation}\label{eqAxisBeta}
-\pr_{\tgamma_{T^{-1}\tf^n(\tx_{i+1})}}(T^{-1}\tx_{i+1}) = -\pr_{\tgamma_{\tf^n(\tx_{i+1})}}(\tx_{i+1}) = \pr_{\tgamma_{\tx_{i+1}}}(\tf^n(\tx_{i+1})) \text{ is large.}
\end{equation}
If $W$ is small enough, then by continuity of $x\mapsto\tgamma_{\tx}$ on $A$, we have that the (unparametrized) geodesics $\tgamma_{T^{-1}\tf^n(\tx_{i+1})}$ and $\tgamma_{\tx_{i+1}}$ are close; moreover $T^{-1}\tf^n(\tx_{i+1})$ and $\tx_{i+1}$ lie in the same fundamental domain; this implies that the parametrized geodesics $\tgamma_{T^{-1}\tf^n(\tx_{i+1})}$ and $\tgamma_{\tx_{i+1}}$ are close. By \eqref{eqAxisBeta} we deduce that $-\pr_{\tgamma_{\tx_{i+1}}}(T^{-1}\tx_{i+1})$ is large, in other words that $T^{-1}\tx_{i+1}$ is close to $\alpha(\tx_{i+1})$. Hence, $T$ maps $\tx_{i+1}$ close to $\omega(\tx_{i+1})$ and $T^{-1}$ maps $\tx_{i+1}$ close to $\alpha(\tx_{i+1})$; by elementary hyperbolic geometry we deduce that the axis of $T$ is close to $\tgamma_{\tx_{i+1}}$ for $n$ large enough.
\end{proof}

\begin{proposition}\label{LemSpeedTransverseTraj}
Let $\mu_1, \mu_2 \in \M$ be two dynamically transverse ergodic measures. Then for $i=1,2$, for $\mu_i$-a.e.\ $x_i\in S$, we have (we identify $\tgamma_{\tx_i}$ with $\R$ via its parametrization by arclength):
\[\frac{1}{t}\pr_{\tgamma_{\tx_i}}\Big(\tilde I^t_{\tilde \F}(\tx_i)\Big) \underset{t\to\pm\infty}\longrightarrow \vartheta_{\mu_i}.\]
\end{proposition}


\begin{proof}
Note that the proposition is true when $t$ is restricted to the set of integers: if $n\in\Z$, then $\tilde I^n_{\tilde \F}(\tx_i)=\tilde f^n(\tx_i)$, and the result follows from Theorem~\ref{thm:trackinggeodesictheorem}. Let us prove that the result remains true for arbitrary real numbers $t$. In the sequel, we identify $\{1,2\}$ with $\Z/2\Z$. The idea of the proof is depicted in Figure~\ref{fig:usebeta}.

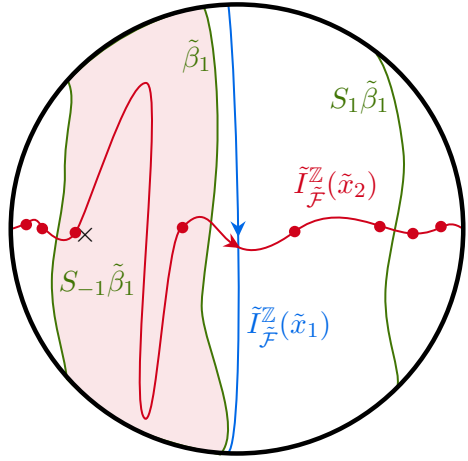
\begin{figure}
\begin{center}
\tikzset{every picture/.style={line width=0.75pt}} 

\begin{tikzpicture}[x=0.75pt,y=0.75pt,yscale=-1,xscale=1]

\draw [color={rgb, 255:red, 0; green, 104; blue, 230 }  ,draw opacity=1 ]   (324.41,37.6) .. controls (328.98,73.97) and (333.48,255.48) .. (324.41,263.48) ;
\draw [shift={(329.45,155.05)}, rotate = 269.15] [fill={rgb, 255:red, 0; green, 104; blue, 230 }  ,fill opacity=1 ][line width=0.08]  [draw opacity=0] (8.04,-3.86) -- (0,0) -- (8.04,3.86) -- (5.34,0) -- cycle    ;
\draw  [color={rgb, 255:red, 65; green, 117; blue, 5 }  ,draw opacity=1 ][fill={rgb, 255:red, 208; green, 2; blue, 27 }  ,fill opacity=0.1 ] (244.38,76.38) .. controls (261.38,58.38) and (285.38,44.38) .. (311.88,38.88) .. controls (315.85,44.79) and (320.23,78.78) .. (318.88,112.88) .. controls (317.5,147.51) and (309.86,183.3) .. (311.38,192.88) .. controls (312.88,210.88) and (332.57,251.19) .. (320.38,264.38) .. controls (283.88,256.38) and (267.38,250.38) .. (243.88,226.38) .. controls (242.55,212.49) and (236.19,197.65) .. (236.38,176.88) .. controls (236.52,160.47) and (243.18,138.48) .. (242.88,123.38) .. controls (242.4,99.75) and (249.38,87.38) .. (244.38,76.38) -- cycle ;
\draw [color={rgb, 255:red, 65; green, 117; blue, 5 }  ,draw opacity=1 ]   (398.38,62.88) .. controls (404.66,71.37) and (409.68,91.74) .. (411.88,112.38) .. controls (414.07,133.01) and (403.35,161.69) .. (402.88,183.38) .. controls (402.4,205.06) and (410.69,222.03) .. (408.88,230.38) ;
\draw  (253,154) node{$\times$}; 
\draw  [draw opacity=0][fill={rgb, 255:red, 208; green, 2; blue, 27 }  ,fill opacity=1 ] (245.25,152.75) .. controls (245.25,151.16) and (246.54,149.88) .. (248.13,149.88) .. controls (249.71,149.88) and (251,151.16) .. (251,152.75) .. controls (251,154.34) and (249.71,155.63) .. (248.13,155.63) .. controls (246.54,155.63) and (245.25,154.34) .. (245.25,152.75) -- cycle ;
\draw  [draw opacity=0][fill={rgb, 255:red, 208; green, 2; blue, 27 }  ,fill opacity=1 ] (298.75,150.25) .. controls (298.75,148.66) and (300.04,147.38) .. (301.63,147.38) .. controls (303.21,147.38) and (304.5,148.66) .. (304.5,150.25) .. controls (304.5,151.84) and (303.21,153.13) .. (301.63,153.13) .. controls (300.04,153.13) and (298.75,151.84) .. (298.75,150.25) -- cycle ;
\draw  [draw opacity=0][fill={rgb, 255:red, 208; green, 2; blue, 27 }  ,fill opacity=1 ] (354.75,152.25) .. controls (354.75,150.66) and (356.04,149.38) .. (357.63,149.38) .. controls (359.21,149.38) and (360.5,150.66) .. (360.5,152.25) .. controls (360.5,153.84) and (359.21,155.13) .. (357.63,155.13) .. controls (356.04,155.13) and (354.75,153.84) .. (354.75,152.25) -- cycle ;
\draw  [draw opacity=0][fill={rgb, 255:red, 208; green, 2; blue, 27 }  ,fill opacity=1 ] (397.25,149.75) .. controls (397.25,148.16) and (398.54,146.88) .. (400.13,146.88) .. controls (401.71,146.88) and (403,148.16) .. (403,149.75) .. controls (403,151.34) and (401.71,152.63) .. (400.13,152.63) .. controls (398.54,152.63) and (397.25,151.34) .. (397.25,149.75) -- cycle ;
\draw  [draw opacity=0][fill={rgb, 255:red, 208; green, 2; blue, 27 }  ,fill opacity=1 ] (413.75,153.25) .. controls (413.75,151.66) and (415.04,150.38) .. (416.63,150.38) .. controls (418.21,150.38) and (419.5,151.66) .. (419.5,153.25) .. controls (419.5,154.84) and (418.21,156.13) .. (416.63,156.13) .. controls (415.04,156.13) and (413.75,154.84) .. (413.75,153.25) -- cycle ;
\draw  [draw opacity=0][fill={rgb, 255:red, 208; green, 2; blue, 27 }  ,fill opacity=1 ] (427.75,149.75) .. controls (427.75,148.16) and (429.04,146.88) .. (430.63,146.88) .. controls (432.21,146.88) and (433.5,148.16) .. (433.5,149.75) .. controls (433.5,151.34) and (432.21,152.63) .. (430.63,152.63) .. controls (429.04,152.63) and (427.75,151.34) .. (427.75,149.75) -- cycle ;
\draw  [draw opacity=0][fill={rgb, 255:red, 208; green, 2; blue, 27 }  ,fill opacity=1 ] (228.75,150.75) .. controls (228.75,149.16) and (230.04,147.88) .. (231.63,147.88) .. controls (233.21,147.88) and (234.5,149.16) .. (234.5,150.75) .. controls (234.5,152.34) and (233.21,153.63) .. (231.63,153.63) .. controls (230.04,153.63) and (228.75,152.34) .. (228.75,150.75) -- cycle ;
\draw  [draw opacity=0][fill={rgb, 255:red, 208; green, 2; blue, 27 }  ,fill opacity=1 ] (220.75,148.75) .. controls (220.75,147.16) and (222.04,145.88) .. (223.63,145.88) .. controls (225.21,145.88) and (226.5,147.16) .. (226.5,148.75) .. controls (226.5,150.34) and (225.21,151.63) .. (223.63,151.63) .. controls (222.04,151.63) and (220.75,150.34) .. (220.75,148.75) -- cycle ;
\draw [color={rgb, 255:red, 208; green, 2; blue, 27 }  ,draw opacity=1 ]   (215.88,150.94) .. controls (225.12,148.77) and (226.63,144.19) .. (229.63,147.44) .. controls (232.63,150.69) and (241.88,162.56) .. (248.13,152.75) .. controls (254.38,142.94) and (273.63,77.03) .. (283.41,77.49) .. controls (293.2,77.95) and (275.02,227.87) .. (281.81,244.69) .. controls (288.6,261.5) and (292.88,151.44) .. (301.63,150.25) ;
\draw [color={rgb, 255:red, 208; green, 2; blue, 27 }  ,draw opacity=1 ]   (301.63,150.25) .. controls (321.25,128.44) and (317.63,182.25) .. (357.63,152.25) ;
\draw [shift={(330.44,160.53)}, rotate = 210.73] [fill={rgb, 255:red, 208; green, 2; blue, 27 }  ,fill opacity=1 ][line width=0.08]  [draw opacity=0] (8.04,-3.86) -- (0,0) -- (8.04,3.86) -- (5.34,0) -- cycle    ;
\draw [color={rgb, 255:red, 208; green, 2; blue, 27 }  ,draw opacity=1 ]   (357.63,152.25) .. controls (379.27,136.02) and (401.75,152.81) .. (416.63,153.25) .. controls (431.5,153.69) and (432.25,144.44) .. (441.75,150.94) ;
\draw  [line width=1.7]  (215.88,150.94) .. controls (215.88,88.56) and (266.44,38) .. (328.81,38) .. controls (391.19,38) and (441.75,88.56) .. (441.75,150.94) .. controls (441.75,213.31) and (391.19,263.88) .. (328.81,263.88) .. controls (266.44,263.88) and (215.88,213.31) .. (215.88,150.94) -- cycle ;

\draw (332,200) node [anchor=west] [inner sep=0.75pt]  [color={rgb, 255:red, 0; green, 104; blue, 230 }  ,opacity=1 ]  {$\tilde{I}_{\tilde\F}^{\Z}(\tilde{x}_{1})$};
\draw (378.21,141.17) node [anchor=south] [inner sep=0.75pt]  [color={rgb, 255:red, 208; green, 2; blue, 27 }  ,opacity=1 ]  {$\tilde{I}_{\tilde\F}^{\Z}(\tilde{x}_{2})$};
\draw (317.6,63.9) node [anchor=east] [inner sep=0.75pt]  [color={rgb, 255:red, 65; green, 117; blue, 5 }  ,opacity=1 ]  {$\tilde{\beta }_{1}$};
\draw (406.4,84.7) node [anchor=east] [inner sep=0.75pt]  [color={rgb, 255:red, 65; green, 117; blue, 5 }  ,opacity=1 ]  {$S_{1}\tilde{\beta }_{1}$};
\draw (238.38,176.88) node [anchor=west] [inner sep=0.75pt]  [color={rgb, 255:red, 65; green, 117; blue, 5 }  ,opacity=1 ]  {$S_{-1}\tilde{\beta }_{1}$};

\end{tikzpicture}
\caption{The idea of the proof of Proposition~\ref{LemSpeedTransverseTraj}: if the point $\tf^n(\tx_1)$ is close to a singularity of the foliation $\F$ (black cross), then the transverse trajectory $I_{\tilde\F}(\tx_1)$ stays in the domain between two curves $S_k\tilde\beta_1$ and $S_{k+1}\tilde\beta_1$.\label{fig:usebeta}}
\end{center}
\end{figure}

Let $\beta_{i+1}$ and $T$ the transverse loop and its deck transformation given by Lemma~\ref{ClaimSpeedTransverseTraj}. We suppose that the axis of $T$ intersects some lift of the tracking geodesic $\gamma_{x_i}$ of a $\mu_i$-typical point $x_i$ with angle $\ge \theta_0/2$.

By local modification of the transverse trajectory $\tilde I_{\tilde \F}^n(\tilde x_{i+1})$ in $\tilde W$, it is possible to create another transverse path $\tilde\beta_{i+1} : [0,1]\to \tilde S$ such that $\tilde\beta_{i+1}(1) = T\tilde\beta_{i+1}(0)$. We can concatenate $\tilde\beta_{i+1}$ with the sequence $(T^i\tilde\beta_{i+1})_i$ to get an extension $\tilde\beta_{i+1} : \R\to \tilde S$ that is $\tilde\F$-transverse and $T$-invariant. We denote by $\beta_{i+1}$ the projection of $\tilde\beta_{i+1}$ to $S$.

Let us apply Proposition~\ref{LemLocalTransverse} to $V_0$ equal to the complement of $\beta_{i+1}$: there exists $U_0\subset S$ such that if $x\in U_0\setminus\sing\F$, then $I_\F(x)$ does not meet $\beta_{i+1}$. Moreover, by local triviality of the foliation in a neighbourhood of a transverse trajectory in the universal cover of $S\setminus\sing \F$, we deduce that the transverse trajectories can be supposed locally bounded and hence bounded in the compact set $S\setminus U_0$: there exists $C>0$ such that if $x\in S\setminus U_0$, then $\diam(\tilde I_{\tilde\F}(\tx))<C$.

Recall that $x_i$ is a point that is $\mu_i$-typical, and that some lift of the geodesic $\gamma_{x_i}$ intersects $\tilde\beta_{i+1}$ with angle $\ge \theta_0/2$. In the sequel, we will suppose this lift is $\tgamma_{\tx_i}$ (this can be made by changing $\tilde\beta_{i+1}$ to some translate of it if necessary).
As the geodesic $\tgamma_{\tx_i}$ is recurrent (Theorem~\ref{thm:equidistributiontheoremintro}), it will cross translates of $\tilde\beta_{i+1}$ with positive frequency: there exists $\kappa>0$, a sequence $(t_m)_{m\in\Z}$ such that $t_m\sim_{\pm\infty} \kappa m$, and a sequence $(S_m)_{m\in\Z}$ of deck transformations such that $\tgamma_{\tx_i}(t_m)\in S_m \tilde\beta_{i+1}$, the angle of the intersection between $\tgamma_{\tx_i}$ and $S_m\tilde\beta_{i+1}$ being $\ge\theta_0/4$. This property on the angle of the intersection, together with the fact that $\tilde\beta_{i+1}$ stays at bounded distance from the axis of $T$, implies that 
\[
\inf\pr_{\tgamma_{\tx_i}}\big( S_m \tilde\beta_{i+1}\big) \underset{m\to\pm\infty}{\sim}
\sup \pr_{\tgamma_{\tx_i}}\big( S_m \tilde\beta_{i+1}\big)
\underset{m\to\pm\infty}{\sim} t_m.
\]
Moreover, by Theorem~\ref{thm:trackinggeodesictheorem}, we have that 
\begin{equation}\label{EqMN0}
\pr_{\tgamma_{\tx_i}}\big( \tilde f^n(\tx_i)\big) \underset{n\to\pm\infty}{\sim}
\vartheta_{\mu_i}n.
\end{equation} 
Hence, for any $n\in\N$ (using $t_m\sim \kappa m$), there exist $m_n,m'_n$ such that 
\begin{equation}\label{EqMN}
\left\{\begin{array}{l}
\sup \pr_{\tgamma_{\tx_i}}\big( S_{m_n} \tilde\beta_{i+1}\big) 
\le \pr_{\tgamma_{\tx_i}}\big( \tilde f^n(\tx_i)\big) \le
\inf \pr_{\tgamma_{\tx_i}}\big( S_{m'_n} \tilde\beta_{i+1}\big)\\
m_n\sim_{\pm\infty} m'_n \sim_{\pm\infty} \vartheta_{\mu_i}n.
\end{array}\right.
\end{equation} 

To conclude we have two cases to consider:
\begin{itemize}
\item Either $f^n(x_i)\in U_0$; in this case by definition of $U_0$ the transverse trajectory $\tilde I_{\tilde \F}(\tilde f^n(\tx_i))$ does not meet neither $S_{m_n} \tilde\beta_{i+1}$ nor $S_{m'_n} \tilde\beta_{i+1}$, in particular we have, by \eqref{EqMN}, 
\begin{multline*}
\vartheta_{\mu_i}n\underset{m\to\pm\infty}{\sim}
\sup_{t\in\R} \pr_{\tgamma_{\tx_i}}\big( S_{m_n} \tilde\beta_{i+1}(t)\big) \\
\le 
\inf_{s\in[0,1]}\pr_{\tgamma_{\tx_i}}\big(\tilde I_{\tilde \F}^s(\tilde f^n(\tx_i))\big) 
\le \sup_{s\in[0,1]}\pr_{\tgamma_{\tx_i}}\big(\tilde I_{\tilde \F}^s(\tilde f^n(\tx_i))\big) 
\le\\
\inf_{t\in\R} \pr_{\tgamma_{\tx_i}}\big( S_{m'_n} \tilde\beta_{i+1}(t)\big)\underset{m\to\pm\infty}{\sim}\vartheta_{\mu_i}n;
\end{multline*}
\item Or $f^n(x_i)\notin U_0$; in this case by the above argument, the diameter of the transverse trajectory $\tilde I_{\tilde \F}(\tilde f^n(\tx_i))$ is uniformly bounded by $C$, which implies that we have, by \eqref{EqMN0},
\[\tilde I_{\tilde \F}^s(\tilde f^n(\tx_i))\underset{m\to\pm\infty}{\sim}\vartheta_{\mu_i}n\]
uniformly in $s\in [0,1]$.
\end{itemize}
\end{proof}

There are two cases to consider, which are equivalent: either $\tilde I^\Z_{\tilde \F}(\tilde x_1)$ crosses $\tilde I^\Z_{\tilde \F}(\tilde x_2)$ from left to right or from right to left. We choose the first one for the proof, the other being identical.

\begin{lemma}\label{LemConsRecurGeod}
There exists $R_0>0$ such that given $R>R_0+1$, there exist deck transformations $T_1, T_1', T_2, T_2'$ satisfying the following (see Figure~\ref{FigConsRecurGeod}): 
\begin{itemize}
\item we have the inclusions: 
\[\pr_{\tilde\gamma_{\tilde x_2}}\Big(T_1\Big(\tilde I_{\tilde \F}^\Z(\tilde x_1) \cup T_2\tilde I_{\tilde \F}^\Z(\tilde x_2)\Big)\Big)\subset [-2R, -R],\]
\[\pr_{\tilde\gamma_{\tilde x_2}}\Big(\tilde I_{\tilde \F}^\Z(\tilde x_1) 
\cup T_2'\Big(\tilde I_{\tilde \F}^\Z(\tilde x_2)\cup T_1' \tilde I_{\tilde \F}^\Z(\tilde x_1)\Big)
\cup T_2\Big(\tilde I_{\tilde \F}^\Z(\tilde x_2)\cup T_1\tilde I_{\tilde \F}^\Z(\tilde x_1)\Big)\Big)\subset [-R_0,R_0],\]
\[\pr_{\tilde\gamma_{\tilde x_2}}\Big(T_1'\Big(\tilde I_{\tilde \F}^\Z(\tilde x_1)\cup T'_2 \tilde I_{\tilde \F}^\Z(\tilde x_2)\Big)\Big)\subset [R,2R]; \] 
\item we have the inclusions:  
\[\pr_{\tilde\gamma_{\tilde x_1}}\Big(T_2\Big(\tilde I_{\tilde \F}^\Z(\tilde x_2) \cup T_1\tilde I_{\tilde \F}^\Z(\tilde x_1)\Big)\Big)\subset [-2R, -R],\]
\[\pr_{\tilde\gamma_{\tilde x_1}}\Big(\tilde I_{\tilde \F}^\Z(\tilde x_2) 
\cup T_1\Big(\tilde I_{\tilde \F}^\Z(\tilde x_1)\cup T_2 \tilde I_{\tilde \F}^\Z(\tilde x_2)\Big)
\cup T_1'\Big(\tilde I_{\tilde \F}^\Z(\tilde x_1)\cup T'_2\tilde I_{\tilde \F}^\Z(\tilde x_2)\Big)\Big)\subset [-R_0,R_0],\]
\[\pr_{\tilde\gamma_{\tilde x_1}}\Big(T_2'\Big(\tilde I_{\tilde \F}^\Z(\tilde x_2)\cup T_1' \tilde I_{\tilde \F}^\Z(\tilde x_1)\Big)\Big)\subset [R,2R]. \]
\end{itemize}
\end{lemma}

In particular, the sets
\[T_1\Big(\tilde I_{\tilde \F}^\Z(\tilde x_1) \cup T_2\tilde I_{\tilde \F}^\Z(\tilde x_2)\Big),\]
\[\tilde I_{\tilde \F}^\Z(\tilde x_1) 
\cup T_2'\Big(\tilde I_{\tilde \F}^\Z(\tilde x_2)\cup T_1' \tilde I_{\tilde \F}^\Z(\tilde x_1)\Big)
\cup T_2\Big(\tilde I_{\tilde \F}^\Z(\tilde x_2)\cup T_1\tilde I_{\tilde \F}^\Z(\tilde x_1)\Big),\]
\[T_1'\Big(\tilde I_{\tilde \F}^\Z(\tilde x_1)\cup T'_2 \tilde I_{\tilde \F}^\Z(\tilde x_2)\Big) \] 
are pairwise disjoint, and cross $\tilde\gamma_{\tilde x_2}$ in an increasing order (and the same for the crossings of $\tilde\gamma_{\tilde x_1}$ with the other sets).

In the proof we will reuse the paths $\tilde\beta_i$ built in Lemma~\ref{ClaimSpeedTransverseTraj}. 

\begin{figure}
\begin{center}

\tikzset{every picture/.style={line width=0.75pt}} 

\begin{tikzpicture}[x=0.75pt,y=0.75pt,yscale=-1,xscale=1]

\clip (177.46,150.85) .. controls (177.46,76.65) and (237.61,16.5) .. (311.81,16.5) .. controls (386.01,16.5) and (446.17,76.65) .. (446.17,150.85) .. controls (446.17,225.06) and (386.01,285.21) .. (311.81,285.21) .. controls (237.61,285.21) and (177.46,225.06) .. (177.46,150.85) -- cycle ;

\draw [color={rgb, 255:red, 208; green, 2; blue, 27 }  ,draw opacity=1 ][fill={rgb, 255:red, 208; green, 2; blue, 27 }  ,fill opacity=1 ]   (165.67,150.92) -- (496.17,150.42) ;
\draw [shift={(335.92,150.66)}, rotate = 179.91] [fill={rgb, 255:red, 208; green, 2; blue, 27 }  ,fill opacity=1 ][line width=0.08]  [draw opacity=0] (10.72,-5.15) -- (0,0) -- (10.72,5.15) -- (7.12,0) -- cycle    ;
\draw [color={rgb, 255:red, 74; green, 144; blue, 226 }  ,draw opacity=1 ]   (308.67,8.92) -- (310.67,330.42) ;
\draw [shift={(309.7,174.67)}, rotate = 269.64] [fill={rgb, 255:red, 74; green, 144; blue, 226 }  ,fill opacity=1 ][line width=0.08]  [draw opacity=0] (10.72,-5.15) -- (0,0) -- (10.72,5.15) -- (7.12,0) -- cycle    ;
\draw [color={rgb, 255:red, 208; green, 2; blue, 27 }  ,draw opacity=1 ]   (237.17,33.42) .. controls (274.17,61.42) and (370.67,56.92) .. (405.17,18.42) ;
\draw [shift={(329.04,50.22)}, rotate = 174.63] [fill={rgb, 255:red, 208; green, 2; blue, 27 }  ,fill opacity=1 ][line width=0.08]  [draw opacity=0] (10.72,-5.15) -- (0,0) -- (10.72,5.15) -- (7.12,0) -- cycle    ;
\draw [color={rgb, 255:red, 208; green, 2; blue, 27 }  ,draw opacity=1 ]   (236.67,268.42) .. controls (273.67,237.42) and (379.17,237.92) .. (402.17,298.92) ;
\draw [shift={(332.75,250.32)}, rotate = 188.51] [fill={rgb, 255:red, 208; green, 2; blue, 27 }  ,fill opacity=1 ][line width=0.08]  [draw opacity=0] (10.72,-5.15) -- (0,0) -- (10.72,5.15) -- (7.12,0) -- cycle    ;
\draw [color={rgb, 255:red, 74; green, 144; blue, 226 }  ,draw opacity=1 ]   (184.67,96.42) .. controls (218.17,120.92) and (221.67,196.42) .. (174.17,227.42) ;
\draw [shift={(208.93,169.56)}, rotate = 276.01] [fill={rgb, 255:red, 74; green, 144; blue, 226 }  ,fill opacity=1 ][line width=0.08]  [draw opacity=0] (10.72,-5.15) -- (0,0) -- (10.72,5.15) -- (7.12,0) -- cycle    ;
\draw [color={rgb, 255:red, 74; green, 144; blue, 226 }  ,draw opacity=1 ]   (432.17,77.42) .. controls (395.67,106.92) and (404.67,207.92) .. (457.17,235.92) ;
\draw [shift={(412.9,166.94)}, rotate = 259.81] [fill={rgb, 255:red, 74; green, 144; blue, 226 }  ,fill opacity=1 ][line width=0.08]  [draw opacity=0] (10.72,-5.15) -- (0,0) -- (10.72,5.15) -- (7.12,0) -- cycle    ;
\draw [color={rgb, 255:red, 208; green, 2; blue, 27 }  ,draw opacity=1 ]   (177.67,129.42) .. controls (210.67,122.42) and (217.17,86.42) .. (197.17,66.42) ;
\draw [shift={(208.07,99.04)}, rotate = 103.2] [fill={rgb, 255:red, 208; green, 2; blue, 27 }  ,fill opacity=1 ][line width=0.08]  [draw opacity=0] (10.72,-5.15) -- (0,0) -- (10.72,5.15) -- (7.12,0) -- cycle    ;
\draw [color={rgb, 255:red, 74; green, 144; blue, 226 }  ,draw opacity=1 ]   (216.17,47.42) .. controls (237.67,72.92) and (289.67,53.42) .. (267.67,14.42) ;
\draw [shift={(252.94,57.25)}, rotate = 337.58] [fill={rgb, 255:red, 74; green, 144; blue, 226 }  ,fill opacity=1 ][line width=0.08]  [draw opacity=0] (10.72,-5.15) -- (0,0) -- (10.72,5.15) -- (7.12,0) -- cycle    ;
\draw [color={rgb, 255:red, 74; green, 144; blue, 226 }  ,draw opacity=1 ]   (352.17,280.92) .. controls (351.67,236.92) and (383.67,239.92) .. (408.17,249.42) ;
\draw [shift={(362.93,249.33)}, rotate = 328.49] [fill={rgb, 255:red, 74; green, 144; blue, 226 }  ,fill opacity=1 ][line width=0.08]  [draw opacity=0] (10.72,-5.15) -- (0,0) -- (10.72,5.15) -- (7.12,0) -- cycle    ;
\draw [color={rgb, 255:red, 208; green, 2; blue, 27 }  ,draw opacity=1 ]   (421.67,246.42) .. controls (392.17,212.42) and (405.17,184.42) .. (441.17,183.42) ;
\draw [shift={(407.52,200.87)}, rotate = 107.06] [fill={rgb, 255:red, 208; green, 2; blue, 27 }  ,fill opacity=1 ][line width=0.08]  [draw opacity=0] (10.72,-5.15) -- (0,0) -- (10.72,5.15) -- (7.12,0) -- cycle    ;
\draw[orange, ->, >=latex] (330,205) to[bend left] (285,210);
\draw[orange!80!black] (295,192) node{$\phi$};

\draw (269.38,147) node [anchor=south] [inner sep=0.75pt]  [font=\small,color={rgb, 255:red, 150; green, 0; blue, 0 }  ,opacity=1 ]  {$\tilde I_{\tilde \F}^{\Z}(\tilde  x_{2})$};
\draw (210.46,97.89) node [anchor=west] [inner sep=0.75pt]  [font=\small,color={rgb, 255:red, 150; green, 0; blue, 0 }  ,opacity=1 ]  {$T_{1} T_{2} \tilde I_{\tilde \F}^{\Z}(\tilde  x_{2})$};
\draw (332.96,52.9) node [anchor=north west][inner sep=0.75pt]  [font=\small,color={rgb, 255:red, 150; green, 0; blue, 0 }  ,opacity=1 ]  {$T_{2} \tilde I_{\tilde \F}^{\Z}(\tilde  x_{2})$};
\draw (262.74,250) node [anchor=south] [inner sep=0.75pt]  [font=\small,color={rgb, 255:red, 150; green, 0; blue, 0 }  ,opacity=1 ]  {$T'_2\tilde I_{\tilde \F}^{\Z}(\tilde  x_{2})$};
\draw (408,197.6) node [anchor=south east] [inner sep=0.75pt]  [font=\small,color={rgb, 255:red, 150; green, 0; blue, 0 }  ,opacity=1 ]  {$T'_{1} T'_{2} \tilde I_{\tilde \F}^{\Z}(\tilde  x_{2})$};
\draw (333.38,129.88) node [anchor=south] [inner sep=0.75pt]  [font=\small,color={rgb, 255:red, 0; green, 79; blue, 169 }  ,opacity=1 ]  {$\tilde I_{\tilde \F}^{\Z}(\tilde  x_{1})$};
\draw (200.59,204.28) node [anchor=west] [inner sep=0.75pt]  [font=\small,color={rgb, 255:red, 0; green, 79; blue, 169 }  ,opacity=1 ]  {$T_{1} \tilde I_{\tilde \F}^{\Z}(\tilde  x_{1})$};
\draw (253.88,61) node [anchor=north] [inner sep=0.75pt]  [font=\small,color={rgb, 255:red, 0; green, 79; blue, 169 }  ,opacity=1 ]  {$T_{2} T_{1} \tilde I_{\tilde \F}^{\Z}(\tilde  x_{1})$};
\draw (417,101.88) node [anchor=south east] [inner sep=0.75pt]  [font=\small,color={rgb, 255:red, 0; green, 79; blue, 169 }  ,opacity=1 ]  {$T'_{1} \tilde I_{\tilde \F}^{\Z}(\tilde  x_{1})$};
\draw (356.89,238.38) node [anchor=south] [inner sep=0.75pt]  [font=\small,color={rgb, 255:red, 0; green, 79; blue, 169 }  ,opacity=1 ]  {$T'_{2} T'_{1} \tilde I_{\tilde \F}^{\Z}(\tilde  x_{1})$};

\draw  [line width=3.75]  (177.46,150.85) .. controls (177.46,76.65) and (237.61,16.5) .. (311.81,16.5) .. controls (386.01,16.5) and (446.17,76.65) .. (446.17,150.85) .. controls (446.17,225.06) and (386.01,285.21) .. (311.81,285.21) .. controls (237.61,285.21) and (177.46,225.06) .. (177.46,150.85) -- cycle ;
\end{tikzpicture}
\caption{\label{FigConsRecurGeod}The configuration of Lemma~\ref{LemConsRecurGeod}. The orange arrow is an example of leaf.}
\end{center}
\end{figure}
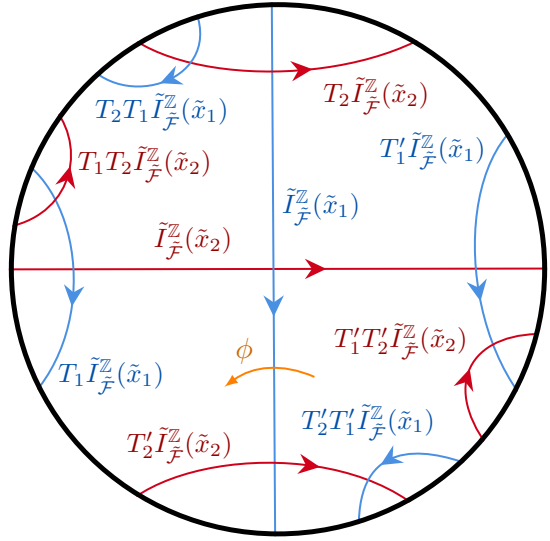

\begin{proof}
Denote $\theta_0$ the angle of intersection between $\tgamma_{\tx_1}$ and $\tgamma_{\tx_2}$. Let $\delta>0$ such that any geodesics $\tgamma_1, \tgamma_2$ of $\wt S$ at distance at most $2\delta$ from respectively $\tgamma_{\tx_1}$ and $\tgamma_{\tx_2}$ intersect transversally with an angle bigger than $\theta_0/2$. Recall that the distance on geodesics comes from a distance on $(\Sp^1)^2$ with the identification of a geodesic of $\tilde S$ with its two endpoints.

Let $\wt\beta_1$ and $\wt\beta_2$ be the paths built in the proof of Proposition~\ref{LemSpeedTransverseTraj}, such that the tracking geodesics of these paths are at distance at most $\delta$ from respectively $\tgamma_{\tx_1}$ and $\tgamma_{\tx_2}$. Let also $U_0\subset S$ be a neighbourhood of $\sing \F$ (obtained, as in the proof of Proposition~\ref{LemSpeedTransverseTraj}, as an application of Proposition~\ref{LemLocalTransverse}) such that if $x\in U_0\setminus\sing\F$, then $I_\F(x)$ does not meet $\beta_{1}$ nor $\beta_2$. Let $C>0$ be such that if $x\in S\setminus U_0$, then $\diam(\tilde I_{\tilde\F}(\tx))<C$.

Let $R_1>0$ such that, for $i=1,2$,
\[\pr_{\tgamma_{\tx_i}}\big(\{\tilde f^n(\tx_{i+1})\mid n\in\Z\}\cup \wt\beta_{i+1}\big)\subset [-R_1,R_1]\]
and $R_0>0$ such that 
\begin{equation}\label{eq:projsmallangle}
\pr_{\tgamma_i}\big(\pr_{\tgamma_{\tx_i}}^{-1}([-R_1,R_1])\big)\subset [-R_0,R_0]
\end{equation}
for any geodesic $\tgamma_i$ of $\wt S$ at distance at most $2\delta$ from $\tgamma_{\tx_i}$. In particular, we get
\begin{equation}\label{eq:projsmallangle2}
\pr_{\tgamma_i}\big(\{\tilde f^n(\tx_{i+1})\mid n\in\Z\}\cup \wt\beta_{i+1}\big)\subset [-R_0,R_0]
\end{equation}
For $i=1,2$, let $W_i$ be a neighbourhood of $\dot\tgamma_{\tx_i}$ in $\mathrm{T}^1 \tilde S$ such that if $\tgamma$ is a geodesic of $\tilde S$ satisfying $\dot\tgamma(0)\in W_i$, then the distance between $\tgamma$ and $\tgamma_{\tx_i}$ is smaller than $\delta$.

Let $\{\tilde y_0\} = \tgamma_{\tx_1}\cap\tgamma_{\tx_2}$ and let us parametrize $\tgamma_{\tx_1}$ and $\tgamma_{\tx_2}$ such that $\tgamma_{\tx_1}(0) = \tgamma_{\tx_2}(0) = \tilde y_0$.

For $i=1,2$, by typicality of the geodesic $\tgamma_{\tx_i}$ with respect to the ergodic measure $\nu_{\mu_i}$ (Theorem~\ref{thm:equidistributiontheoremintro}), there exist a sequence of times $(t_{i,k})_{k\in\Z}$, a sequence of deck transformations $(S_{i,k})_{k\in\Z}$ and $\kappa>0$ such that:
\begin{enumerate}[label=\alph*)]
\item $t_{i,k}\sim_{k\to\pm\infty}\kappa k$;
\item $t_{i,k+1} - t_{i,k}\ge 2R_0+2C$;
\item $S_{i,k}^{-1}(\dot\tgamma_{\tx_i}(t_{i,k}))\in W_i$.
\end{enumerate}
In the end of the proof we will take $T_i = S_{i,-k}$ and $T'_i = S_{i,k}$ for $k$ large enough. The reader is encouraged to follow the proof on Figure~\ref{FigConsRecurGeod}.

By applying \eqref{eq:projsmallangle2} to $\tgamma_i = S_{i,k}^{-1}(\tgamma_{\tx_i})$, condition c) implies that
\begin{equation}\label{eq:projJ}
\pr_{\tgamma_{\tx_i}}\Big(S_{i,k}\big(\{\tilde f^n(\tx_{i+1})\mid n\in\Z\}\cup \wt\beta_{i+1}\big)\Big)\subset [t_{i,k}-R_0,t_{i,k}+R_0].
\end{equation}
In particular, up to deleting the first terms of the sequences $(t_{i,k})_{k\in\Z}$ and $(S_{i,k})_{k\in\Z}$, we can suppose that for $k\neq 0$, we have  
\[\pr_{\tgamma_{\tx_{i+1}}}\Big(S_{i,k}\big(\{\tilde f^n(\tx_{i+1})\mid n\in\Z\}\cup \wt\beta_{i+1}\big)\Big)\subset [-R_1,R_1].\]
Using condition c) again to apply \eqref{eq:projsmallangle} to $\tgamma_i = S_{i,k}^{-1}\tgamma_{\tx_i}$ and to the previous inclusion, one gets that
\begin{equation}\label{eq:conscondc}
\pr_{\tgamma_{\tx_{i}}}\Big(S_{i,k}S_{i+1,\pm k}\big(\{\tilde f^n(\tx_{i})\mid n\in\Z\}\big)\Big)\subset [t_{i,k}-R_0,t_{i,k}+R_0].
\end{equation}

Now, let $k\ge 2$, and consider 
\[\tx \in S_{i,k}\bigcup_{n\in\Z}\big\{\tilde f^n(\tx_{i+1}),S_{i+1,k}\tilde f^n(\tx_{i}),S_{i+1,-k}\tilde f^n(\tx_{i})\big\}.\] 
Then, by \eqref{eq:projJ} and \eqref{eq:conscondc}, we have that 
\[\pr_{\tgamma_{\tx_i}}(\tx)\in [t_{i,k}-R_0,t_{i,k}+R_0],\]
in particular, by \eqref{eq:projJ} and condition b), $\tx$ belongs to the connected component of the complement of $S_{i,k-1}\tilde\beta_{i+1}$ containing $\omega(\tgamma_{\tx_i})$ in its boundary. 
We have two cases:
\begin{itemize}
\item Either the projection of $\tx$ on $S$ belongs to $U_0$. Then the transverse trajectory $\tilde I_{\tilde\F}(\tx)$ does not meet $S_{i,k-1}\tilde\beta_{i+1}$ and hence stays in the connected component of the complement of $S_{i,k-1}\tilde\beta_{i+1}$ containing $\omega(\tgamma_{\tx_i})$ in its boundary. In particular, 
\[\pr_{\tgamma_i}\big(\tilde I_{\tilde\F}(\tx)\big)\subset [t_{i,k-1}-R_0,+\infty);\]
\item Or the projection of $\tx$ on $S$ does not belong to $U_0$. Then the diameter of the transverse trajectory $\tilde I_{\tilde\F}(\tx)$ is smaller than $C$, and by hypothesis on the deck transformations $S_{i,k}$ we have also
\[\pr_{\tgamma_i}\big(\tilde I_{\tilde\F}(\tx)\big)\subset [t_{i,k-1}-R_0,+\infty).\]
\end{itemize}
The exact same reasoning holds for $k\le -2$. The conclusion of the lemma then follows from the fact that $t_{i,k}\sim_{k\to\pm\infty}\kappa k$ and condition a), by taking $T_i = S_{i,-k}$ and $T'_i = S_{i,k}$.
\end{proof}

For $i=1,2$ and $t\in\R$, let $\tilde\phi_i^t$ be the leaf of $\tilde\F$ passing by $\tilde I^t_{\tilde \F}(\tilde x_i)$. We denote by $\alpha(\tilde\phi_i^t)$ and $\omega(\tilde\phi_i^t)$ its alpha and omega limits in the closed disk $\overline{\tilde S} \simeq \overline{\D}$.

\begin{lemma}\label{LemaSubTrajecSimple}
There is an orientable and simple transverse trajectory $\tilde I_2$ in $\tilde S$ that is made of pieces of $T_2\tilde I^\Z_{\tilde \F}(\tilde x_2)$, such that $\lim_{t\to \pm\infty} \tilde I_2(t) = T_2\alpha(\tilde x_2)$ and $\lim_{t\to \mp\infty} \tilde I_2(t) = T_2\omega(\tilde x_2)$.
\end{lemma}

\begin{figure}
\begin{center}

\tikzset{every picture/.style={line width=0.75pt}} 

\begin{tikzpicture}[x=0.75pt,y=0.75pt,yscale=-1,xscale=1]

\draw [color={rgb, 255:red, 208; green, 2; blue, 27 }  ,draw opacity=1 ]   (382,159.63) .. controls (495.66,131.23) and (542.3,182.8) .. (471.9,181.6) .. controls (401.5,180.4) and (446.74,140.94) .. (465.1,138.4) .. controls (483.46,135.86) and (542.78,156.93) .. (557.83,151.99) ;
\draw [shift={(468.89,152.58)}, rotate = 186.72] [fill={rgb, 255:red, 208; green, 2; blue, 27 }  ,fill opacity=1 ][line width=0.08]  [draw opacity=0] (8.04,-3.86) -- (0,0) -- (8.04,3.86) -- (5.34,0) -- cycle    ;
\draw [shift={(434.28,162.97)}, rotate = 91] [fill={rgb, 255:red, 208; green, 2; blue, 27 }  ,fill opacity=1 ][line width=0.08]  [draw opacity=0] (8.04,-3.86) -- (0,0) -- (8.04,3.86) -- (5.34,0) -- cycle    ;
\draw [shift={(514.98,146.62)}, rotate = 192.84] [fill={rgb, 255:red, 208; green, 2; blue, 27 }  ,fill opacity=1 ][line width=0.08]  [draw opacity=0] (8.04,-3.86) -- (0,0) -- (8.04,3.86) -- (5.34,0) -- cycle    ;
\draw  [line width=1.5]  (381.44,151.99) .. controls (381.44,102.91) and (420.93,63.11) .. (469.64,63.11) .. controls (518.35,63.11) and (557.83,102.91) .. (557.83,151.99) .. controls (557.83,201.08) and (518.35,240.88) .. (469.64,240.88) .. controls (420.93,240.88) and (381.44,201.08) .. (381.44,151.99) -- cycle ;
\draw [color={rgb, 255:red, 208; green, 2; blue, 27 }  ,draw opacity=1 ]   (52.94,150.99) .. controls (108,139.5) and (132.67,176.83) .. (79.33,171.17) .. controls (26,165.5) and (81.67,130.17) .. (120,136.83) .. controls (158.33,143.5) and (173.67,176.17) .. (114.67,170.17) .. controls (55.67,164.17) and (143.25,143.22) .. (174,142.17) .. controls (204.75,141.11) and (236.33,180.83) .. (169,170.17) .. controls (101.67,159.5) and (272,152.83) .. (214.67,165.17) .. controls (157.33,177.5) and (213.74,149.11) .. (229.33,150.99) ;
\draw [shift={(102.38,155.35)}, rotate = 206.17] [fill={rgb, 255:red, 208; green, 2; blue, 27 }  ,fill opacity=1 ][line width=0.08]  [draw opacity=0] (8.04,-3.86) -- (0,0) -- (8.04,3.86) -- (5.34,0) -- cycle    ;
\draw [shift={(75.72,143.5)}, rotate = 153.55] [fill={rgb, 255:red, 208; green, 2; blue, 27 }  ,fill opacity=1 ][line width=0.08]  [draw opacity=0] (8.04,-3.86) -- (0,0) -- (8.04,3.86) -- (5.34,0) -- cycle    ;
\draw [shift={(152.95,163.16)}, rotate = 276.5] [fill={rgb, 255:red, 208; green, 2; blue, 27 }  ,fill opacity=1 ][line width=0.08]  [draw opacity=0] (8.04,-3.86) -- (0,0) -- (8.04,3.86) -- (5.34,0) -- cycle    ;
\draw [shift={(125.49,150.08)}, rotate = 165.3] [fill={rgb, 255:red, 208; green, 2; blue, 27 }  ,fill opacity=1 ][line width=0.08]  [draw opacity=0] (8.04,-3.86) -- (0,0) -- (8.04,3.86) -- (5.34,0) -- cycle    ;
\draw [shift={(207.92,166.35)}, rotate = 276.19] [fill={rgb, 255:red, 208; green, 2; blue, 27 }  ,fill opacity=1 ][line width=0.08]  [draw opacity=0] (8.04,-3.86) -- (0,0) -- (8.04,3.86) -- (5.34,0) -- cycle    ;
\draw [shift={(191.33,158.88)}, rotate = 176.91] [fill={rgb, 255:red, 208; green, 2; blue, 27 }  ,fill opacity=1 ][line width=0.08]  [draw opacity=0] (8.04,-3.86) -- (0,0) -- (8.04,3.86) -- (5.34,0) -- cycle    ;
\draw [shift={(199.2,159.55)}, rotate = 151.19] [fill={rgb, 255:red, 208; green, 2; blue, 27 }  ,fill opacity=1 ][line width=0.08]  [draw opacity=0] (8.04,-3.86) -- (0,0) -- (8.04,3.86) -- (5.34,0) -- cycle    ;
\draw  [line width=1.5]  (52.94,150.99) .. controls (52.94,101.91) and (92.43,62.11) .. (141.14,62.11) .. controls (189.85,62.11) and (229.33,101.91) .. (229.33,150.99) .. controls (229.33,200.08) and (189.85,239.88) .. (141.14,239.88) .. controls (92.43,239.88) and (52.94,200.08) .. (52.94,150.99) -- cycle ;
\draw [color={rgb, 255:red, 255; green, 133; blue, 233 }  ,draw opacity=1 ]   (54.54,162.54) .. controls (53.86,166.43) and (83.95,178.05) .. (105.67,169.86) .. controls (126.81,173.76) and (146.43,171.19) .. (151.95,165.86) .. controls (161.95,171.95) and (199.57,177.1) .. (208.14,167.48) .. controls (219.19,165.38) and (222.33,164.81) .. (228.23,161.62) ;
\draw [shift={(74.51,171.33)}, rotate = 8.81] [fill={rgb, 255:red, 255; green, 133; blue, 233 }  ,fill opacity=1 ][line width=0.08]  [draw opacity=0] (8.04,-3.86) -- (0,0) -- (8.04,3.86) -- (5.34,0) -- cycle    ;
\draw [shift={(124.8,171.74)}, rotate = 358.82] [fill={rgb, 255:red, 255; green, 133; blue, 233 }  ,fill opacity=1 ][line width=0.08]  [draw opacity=0] (8.04,-3.86) -- (0,0) -- (8.04,3.86) -- (5.34,0) -- cycle    ;
\draw [shift={(175.67,172.09)}, rotate = 5.13] [fill={rgb, 255:red, 255; green, 133; blue, 233 }  ,fill opacity=1 ][line width=0.08]  [draw opacity=0] (8.04,-3.86) -- (0,0) -- (8.04,3.86) -- (5.34,0) -- cycle    ;
\draw [shift={(213.91,166.37)}, rotate = 347.35] [fill={rgb, 255:red, 255; green, 133; blue, 233 }  ,fill opacity=1 ][line width=0.08]  [draw opacity=0] (8.04,-3.86) -- (0,0) -- (8.04,3.86) -- (5.34,0) -- cycle    ;
\draw [color={rgb, 255:red, 245; green, 166; blue, 35 }  ,draw opacity=1 ]   (453.9,90.4) .. controls (404.39,119.73) and (416.04,191.81) .. (442.78,206.31) .. controls (469.53,220.82) and (475.83,173.55) .. (484.3,167.2) ;
\draw [shift={(419.84,149.67)}, rotate = 274.82] [fill={rgb, 255:red, 245; green, 166; blue, 35 }  ,fill opacity=1 ][line width=0.08]  [draw opacity=0] (6.25,-3) -- (0,0) -- (6.25,3) -- cycle    ;
\draw [shift={(471.1,193.27)}, rotate = 120.35] [fill={rgb, 255:red, 245; green, 166; blue, 35 }  ,fill opacity=1 ][line width=0.08]  [draw opacity=0] (6.25,-3) -- (0,0) -- (6.25,3) -- cycle    ;
\draw  [draw opacity=0][fill={rgb, 255:red, 0; green, 0; blue, 0 }  ,fill opacity=1 ] (423.8,184.82) .. controls (423.8,183.74) and (424.67,182.86) .. (425.76,182.86) .. controls (426.84,182.86) and (427.72,183.74) .. (427.72,184.82) .. controls (427.72,185.9) and (426.84,186.78) .. (425.76,186.78) .. controls (424.67,186.78) and (423.8,185.9) .. (423.8,184.82) -- cycle ;

\draw (515.8,141.49) node [anchor=south] [inner sep=0.75pt]  [color={rgb, 255:red, 208; green, 2; blue, 27 }  ,opacity=1 ]  {$\tilde{\alpha }$};
\draw (143.93,137) node [anchor=south] [inner sep=0.75pt]  [color={rgb, 255:red, 208; green, 2; blue, 27 }  ,opacity=1 ]  {$T_{2}\tilde{I}_{F}^{Z}\left(\tilde{x}_{2}\right)$};
\draw (151.96,173) node [anchor=north] [inner sep=0.75pt]  [color={rgb, 255:red, 195; green, 99; blue, 177 }  ,opacity=1 ]  {$\tilde{I}_{2}$};
\draw (424,184.11) node [anchor=north east] [inner sep=0.75pt]    {$\tilde{z}$};
\draw (442.5,100.07) node [anchor=north west][inner sep=0.75pt]  [color={rgb, 255:red, 229; green, 143; blue, 0 }  ,opacity=1 ]  {$\tilde{\phi }$};

\end{tikzpicture}

\caption{Left: an example of construction of the path $\tilde I_2$. Right: Illustration of Lemma~\ref{LemEndleaves}.\label{FigI2}}
\end{center}
\end{figure}
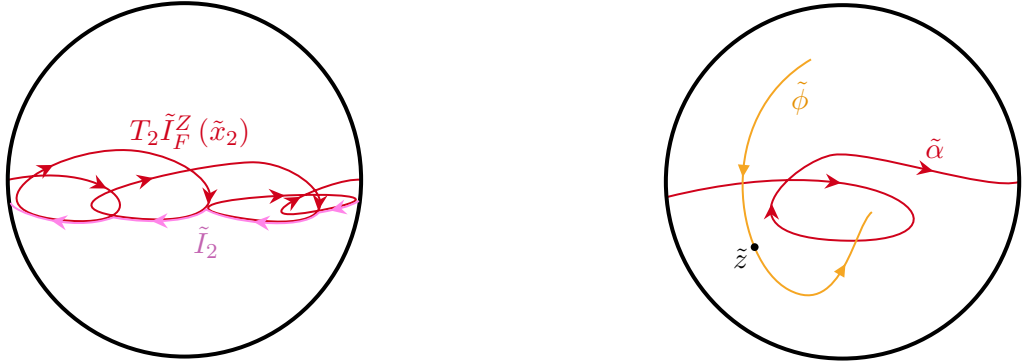

Note that the lemma does not rule out the possibility that $\lim_{t\to +\infty} \tilde I_2(t) = T_2\alpha(\tilde x_2)$ and $\lim_{t\to -\infty} \tilde I_2(t) = T_2\omega(\tilde x_2)$ (see Figure~\ref{FigI2}, left).

\begin{proof}
As the trajectory of $T_2\tilde x_2$ is proper in $\tilde S$, for any $n\in\N$ there exists $R_n>0$ such that if $|t|\ge R_n$, then $T_2\tilde I^t_{\tilde \F}(\tilde x_2)\notin T_2\tilde I^{[-n,n]}_{\tilde \F}(\tilde x_2)$.

Up to modifying the trajectory $T_2\tilde I^\Z_{\tilde \F}(\tilde x_2)$ to an $\F$-equivalent one, one can suppose that the self-intersections of $T_2\tilde I^\Z_{\tilde \F}(\tilde x_2)$ are discrete. Hence, the number of self-intersections of $T_2\tilde I^{[-R_n,R_n]}_{\tilde \F}(\tilde x_2)$ is finite.

Let us apply the following algorithm. Set $\tilde I^0_2 = T_2\tilde I^{\Z}_{\tilde \F}(\tilde x_2)$, and define the paths $\tilde I^n_2$ inductively.
For any $n\ge 0$, consider (if they exist, \emph{i.e.}\ if $\tilde I^n_2|_{[-R_n,R_n]}$ is not simple) $-R_n\le t<t'\le R_n$ such that $\tilde I^n_2(t) = \tilde I^n_2(t')$ and that $\tilde I^n_2|_{(t,t')}$ is simple. We then replace the path $\tilde I^n_2$ by the concatenation $\tilde I^n_2|_{(-\infty,t]}\tilde I^n_2|_{[t',+\infty)}$. This new path is oriented as the previous one, and the number of self-intersections of its restriction to $[-R_n,R_n]$ strictly decreased. Hence, iterating this process, it terminates in finite time.

The obtained path $\tilde I^{n+1}_2$ is simple in restriction to\footnote{The careful reader will have noticed the abuse of notation here: the path $\tilde I^{n+1}_2$ is defined in a subset of $\R$ that is made of a finite union of intervals; in the sequel we will sometimes identify this set with $\R$ in a natural way.} $[-R_n,R_n]$, and links a (Euclidean) neighbourhood of $\alpha(\tilde x_2)$ to a neighbourhood of $\omega(\tilde x_2)$ (these neighbourhoods being arbitrarily small as $n$ goes to infinity). Moreover, by definition of $R_n$, for any $n\in\N$ one has $\tilde I^n_2|_{[-n,n]} = \tilde I^{n+1}_2|_{[-n,n]}$. This implies that the paths $\tilde I^n_2$ converge simply as $n$ goes to infinity. The limit path $\tilde I_2$ satisfies the conclusion of the lemma.
\end{proof}

Note that as the trajectory $\tilde I^\Z_{\tilde \F}(\tilde x_2)$ is proper and satisfies $\lim_{t\to-\infty}\tilde I^t_{\tilde \F}(\tilde x_2)  = \alpha(\tilde x_2)$ and $\lim_{t\to+\infty}\tilde I^t_{\tilde \F}(\tilde x_2) = \omega(\tilde x_2) \neq \alpha(\tilde x_2)$, one can define $L(\tilde I^\Z_{\tilde \F}(\tilde x_2))$ and $R(\tilde I^\Z_{\tilde \F}(\tilde x_2))$ as the connected components of the complement of $\tilde I^\Z_{\tilde \F}(\tilde x_2)$ having for respective intersection with $\partial \Hy^2$ the intervals $(\omega(\tilde x_2), \alpha(\tilde x_2))$ and $(\alpha(\tilde x_2), \omega(\tilde x_2))$.

We split the rest of the proof depending whether the following condition is satisfied or not (see Figure~\ref{fig:casesInterGeod}).

\begin{enumerate}[label=(\Alph{enumi}), ref=(\Alph{enumi}), start=3]
\item \label{PropertyC} There exists $t_0\in\R$ such that for any $t<t_0$, the leaf $\tilde \phi_1^t$ meets both $L(T_2\tilde I^\Z_{\tilde \F}(\tilde x_2))$ and $R(T_2\tilde I^\Z_{\tilde \F}(\tilde x_2))$.
\end{enumerate}

\begin{figure}
\begin{center}
\tikzset{every picture/.style={line width=0.75pt}} 

\begin{tikzpicture}[x=0.75pt,y=0.75pt,yscale=-1,xscale=1]

\draw [color={rgb, 255:red, 245; green, 166; blue, 35 }  ,draw opacity=1 ]   (112.1,197.9) .. controls (121.3,154.3) and (139.13,121.79) .. (188.21,120.61) ;
\draw [shift={(132.47,146.75)}, rotate = 309.43] [fill={rgb, 255:red, 245; green, 166; blue, 35 }  ,fill opacity=1 ][line width=0.08]  [draw opacity=0] (6.25,-3) -- (0,0) -- (6.25,3) -- cycle    ;
\draw [color={rgb, 255:red, 245; green, 166; blue, 35 }  ,draw opacity=1 ]   (106.43,196.4) .. controls (111.63,148) and (119.65,91.18) .. (161.25,89.58) ;
\draw [shift={(116.56,136.34)}, rotate = 287.28] [fill={rgb, 255:red, 245; green, 166; blue, 35 }  ,fill opacity=1 ][line width=0.08]  [draw opacity=0] (6.25,-3) -- (0,0) -- (6.25,3) -- cycle    ;
\draw [color={rgb, 255:red, 245; green, 166; blue, 35 }  ,draw opacity=1 ]   (109.4,197.43) .. controls (118.6,153.83) and (136.49,108.19) .. (185.58,107.01) ;
\draw [shift={(129.48,141.03)}, rotate = 302.94] [fill={rgb, 255:red, 245; green, 166; blue, 35 }  ,fill opacity=1 ][line width=0.08]  [draw opacity=0] (6.25,-3) -- (0,0) -- (6.25,3) -- cycle    ;
\draw [color={rgb, 255:red, 245; green, 166; blue, 35 }  ,draw opacity=1 ]   (107.43,196.74) .. controls (114.9,151.43) and (124.18,97.75) .. (173.26,96.57) ;
\draw [shift={(121.58,137.25)}, rotate = 293.35] [fill={rgb, 255:red, 245; green, 166; blue, 35 }  ,fill opacity=1 ][line width=0.08]  [draw opacity=0] (6.25,-3) -- (0,0) -- (6.25,3) -- cycle    ;
\draw [color={rgb, 255:red, 0; green, 104; blue, 230 }  ,draw opacity=1 ]   (350.12,85.8) .. controls (360.92,110.45) and (357.41,250.92) .. (350.12,263.56) ;
\draw [shift={(356.92,177.99)}, rotate = 270.4] [fill={rgb, 255:red, 0; green, 104; blue, 230 }  ,fill opacity=1 ][line width=0.08]  [draw opacity=0] (8.04,-3.86) -- (0,0) -- (8.04,3.86) -- (5.34,0) -- cycle    ;
\draw [color={rgb, 255:red, 208; green, 2; blue, 27 }  ,draw opacity=1 ]   (265.36,174.99) .. controls (291.81,182.45) and (391.38,191.51) .. (441.75,174.99) ;
\draw [shift={(357.11,184.26)}, rotate = 180.87] [fill={rgb, 255:red, 208; green, 2; blue, 27 }  ,fill opacity=1 ][line width=0.08]  [draw opacity=0] (8.04,-3.86) -- (0,0) -- (8.04,3.86) -- (5.34,0) -- cycle    ;
\draw  [line width=1.5]  (265.36,174.99) .. controls (265.36,125.91) and (304.84,86.11) .. (353.55,86.11) .. controls (402.26,86.11) and (441.75,125.91) .. (441.75,174.99) .. controls (441.75,224.08) and (402.26,263.88) .. (353.55,263.88) .. controls (304.84,263.88) and (265.36,224.08) .. (265.36,174.99) -- cycle ;
\draw [color={rgb, 255:red, 0; green, 104; blue, 230 }  ,draw opacity=1 ]   (546.78,86.47) .. controls (557.59,111.11) and (554.08,251.59) .. (546.78,264.23) ;
\draw [shift={(553.58,178.66)}, rotate = 270.4] [fill={rgb, 255:red, 0; green, 104; blue, 230 }  ,fill opacity=1 ][line width=0.08]  [draw opacity=0] (8.04,-3.86) -- (0,0) -- (8.04,3.86) -- (5.34,0) -- cycle    ;
\draw [color={rgb, 255:red, 208; green, 2; blue, 27 }  ,draw opacity=1 ]   (462.02,175.66) .. controls (488.48,183.12) and (588.05,192.17) .. (638.42,175.66) ;
\draw [shift={(553.77,184.92)}, rotate = 180.87] [fill={rgb, 255:red, 208; green, 2; blue, 27 }  ,fill opacity=1 ][line width=0.08]  [draw opacity=0] (8.04,-3.86) -- (0,0) -- (8.04,3.86) -- (5.34,0) -- cycle    ;
\draw  [line width=1.5]  (462.02,175.66) .. controls (462.02,126.57) and (501.51,86.78) .. (550.22,86.78) .. controls (598.93,86.78) and (638.42,126.57) .. (638.42,175.66) .. controls (638.42,224.75) and (598.93,264.54) .. (550.22,264.54) .. controls (501.51,264.54) and (462.02,224.75) .. (462.02,175.66) -- cycle ;
\draw [color={rgb, 255:red, 0; green, 104; blue, 230 }  ,draw opacity=1 ]   (154.12,85.8) .. controls (164.92,110.45) and (161.41,250.92) .. (154.12,263.56) ;
\draw [shift={(160.92,177.99)}, rotate = 270.4] [fill={rgb, 255:red, 0; green, 104; blue, 230 }  ,fill opacity=1 ][line width=0.08]  [draw opacity=0] (8.04,-3.86) -- (0,0) -- (8.04,3.86) -- (5.34,0) -- cycle    ;
\draw [color={rgb, 255:red, 208; green, 2; blue, 27 }  ,draw opacity=1 ]   (69.36,174.99) .. controls (95.81,182.45) and (195.38,191.51) .. (245.75,174.99) ;
\draw [shift={(161.11,184.26)}, rotate = 180.87] [fill={rgb, 255:red, 208; green, 2; blue, 27 }  ,fill opacity=1 ][line width=0.08]  [draw opacity=0] (8.04,-3.86) -- (0,0) -- (8.04,3.86) -- (5.34,0) -- cycle    ;
\draw  [line width=1.5]  (69.36,174.99) .. controls (69.36,125.91) and (108.84,86.11) .. (157.55,86.11) .. controls (206.26,86.11) and (245.75,125.91) .. (245.75,174.99) .. controls (245.75,224.08) and (206.26,263.88) .. (157.55,263.88) .. controls (108.84,263.88) and (69.36,224.08) .. (69.36,174.99) -- cycle ;
\draw [color={rgb, 255:red, 208; green, 2; blue, 27 }  ,draw opacity=1 ]   (107.28,179.24) -- (107.28,183.4) ;
\draw [color={rgb, 255:red, 245; green, 166; blue, 35 }  ,draw opacity=1 ]   (271.69,187.56) .. controls (273.19,156.31) and (322.69,88.81) .. (361.19,90.06) ;
\draw [shift={(300.5,129.12)}, rotate = 309.98] [fill={rgb, 255:red, 245; green, 166; blue, 35 }  ,fill opacity=1 ][line width=0.08]  [draw opacity=0] (6.25,-3) -- (0,0) -- (6.25,3) -- cycle    ;
\draw [color={rgb, 255:red, 245; green, 166; blue, 35 }  ,draw opacity=1 ]   (281.94,190.06) .. controls (287.37,154.35) and (327.03,100.03) .. (365.89,97.46) ;
\draw [shift={(308.64,134.68)}, rotate = 309.52] [fill={rgb, 255:red, 245; green, 166; blue, 35 }  ,fill opacity=1 ][line width=0.08]  [draw opacity=0] (6.25,-3) -- (0,0) -- (6.25,3) -- cycle    ;
\draw [color={rgb, 255:red, 245; green, 166; blue, 35 }  ,draw opacity=1 ]   (300.44,193.81) .. controls (306.15,154.38) and (324.46,113.18) .. (369.03,110.6) ;
\draw [shift={(317.35,142.69)}, rotate = 303.57] [fill={rgb, 255:red, 245; green, 166; blue, 35 }  ,fill opacity=1 ][line width=0.08]  [draw opacity=0] (6.25,-3) -- (0,0) -- (6.25,3) -- cycle    ;
\draw [color={rgb, 255:red, 245; green, 166; blue, 35 }  ,draw opacity=1 ]   (315.94,191.81) .. controls (321.65,152.38) and (327.32,130.32) .. (371.89,127.75) ;
\draw [shift={(326.77,150.03)}, rotate = 303.09] [fill={rgb, 255:red, 245; green, 166; blue, 35 }  ,fill opacity=1 ][line width=0.08]  [draw opacity=0] (6.25,-3) -- (0,0) -- (6.25,3) -- cycle    ;
\draw [color={rgb, 255:red, 245; green, 166; blue, 35 }  ,draw opacity=1 ]   (480.58,122.03) .. controls (497.15,132.03) and (514.29,97.46) .. (558.86,94.89) ;
\draw [shift={(514.83,109.63)}, rotate = 329.63] [fill={rgb, 255:red, 245; green, 166; blue, 35 }  ,fill opacity=1 ][line width=0.08]  [draw opacity=0] (6.25,-3) -- (0,0) -- (6.25,3) -- cycle    ;
\draw [color={rgb, 255:red, 245; green, 166; blue, 35 }  ,draw opacity=1 ]   (470.01,140.03) .. controls (486.58,150.03) and (523.43,116.6) .. (568,114.03) ;
\draw [shift={(514.19,129.25)}, rotate = 336.84] [fill={rgb, 255:red, 245; green, 166; blue, 35 }  ,fill opacity=1 ][line width=0.08]  [draw opacity=0] (6.25,-3) -- (0,0) -- (6.25,3) -- cycle    ;
\draw [color={rgb, 255:red, 245; green, 166; blue, 35 }  ,draw opacity=1 ]   (490.3,190.6) .. controls (497.15,158.6) and (527.44,136.6) .. (572.01,134.03) ;
\draw [shift={(517.43,150.99)}, rotate = 327.41] [fill={rgb, 255:red, 245; green, 166; blue, 35 }  ,fill opacity=1 ][line width=0.08]  [draw opacity=0] (6.25,-3) -- (0,0) -- (6.25,3) -- cycle    ;

\draw (404.51,178.89) node [anchor=south] [inner sep=0.75pt]  [color={rgb, 255:red, 208; green, 2; blue, 27 }  ,opacity=1 ]  {$T_2\tilde{I}_{\tilde\F}^{\Z}(\tilde{x}_{2})$};
\draw (356.54,228.86) node [anchor=west] [inner sep=0.75pt]  [color={rgb, 255:red, 0; green, 104; blue, 230 }  ,opacity=1 ]  {$\tilde{I}_{\tilde\F}^{\Z}(\tilde{x}_{1})$};
\draw (601.18,179.06) node [anchor=south] [inner sep=0.75pt]  [color={rgb, 255:red, 208; green, 2; blue, 27 }  ,opacity=1 ]  {$T_2\tilde{I}_{\tilde\F}^{\Z}(\tilde{x}_{2})$};
\draw (553.96,227.39) node [anchor=west] [inner sep=0.75pt]  [color={rgb, 255:red, 0; green, 104; blue, 230 }  ,opacity=1 ]  {$\tilde{I}_{\tilde\F}^{\Z}(\tilde{x}_{1})$};
\draw (209.01,178.64) node [anchor=south] [inner sep=0.75pt]  [color={rgb, 255:red, 208; green, 2; blue, 27 }  ,opacity=1 ]  {$T_2\tilde{I}_{\tilde\F}^{\Z}(\tilde{x}_{2})$};
\draw (161.4,227.43) node [anchor=west] [inner sep=0.75pt]  [color={rgb, 255:red, 0; green, 104; blue, 230 }  ,opacity=1 ]  {$\tilde{I}_{\tilde\F}^{\Z}(\tilde{x}_{1})$};
\draw (105.28,182.64) node [anchor=north east] [inner sep=0.75pt]  [color={rgb, 255:red, 208; green, 2; blue, 27 }  ,opacity=1 ]  {$s_{0}$};

\end{tikzpicture}

\caption{The three different cases in the proof of Theorem~\ref{TheoInterGeod}: Condition~\ref{PropertyC} holds and $s_0>-\infty$ (left), Condition~\ref{PropertyC} holds and $s_0=-\infty$ (middle) and Condition~\ref{PropertyC} does not hold (right). The leaves are in orange.\label{fig:casesInterGeod}}
\end{center}
\end{figure}
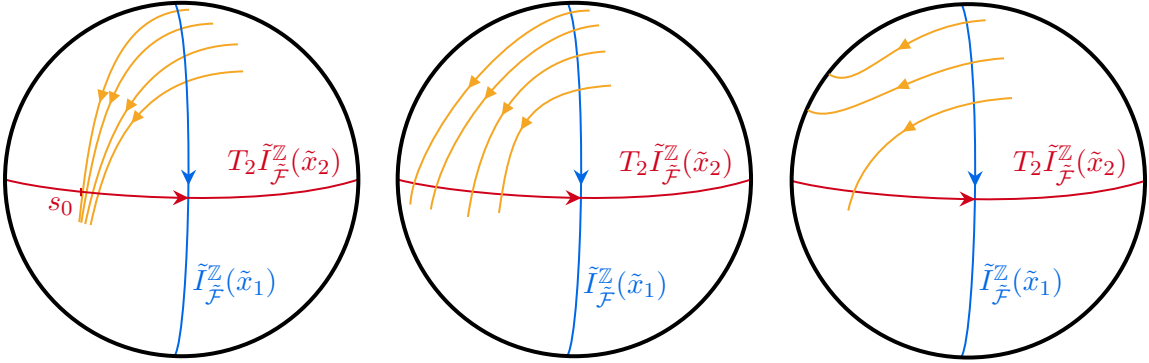

\begin{lemma}\label{LemSimpleOrbit}
If Condition~\ref{PropertyC} holds, then $\tilde I^{(-\infty,t_0]}_{\tilde \F}(\tilde x_1)$ is simple. Moreover, we have $\lim_{t\to -\infty} \tilde I_2(t) = T_2\alpha(\tilde x_2)$ and $\lim_{t\to +\infty} \tilde I_2(t) = T_2\omega(\tilde x_2)$.
\end{lemma}

Note that the last conclusion of this lemma consists in ruling out the case where $\lim_{t\to +\infty} \tilde I_2(t) = T_2\alpha(\tilde x_2)$ and $\lim_{t\to -\infty} \tilde I_2(t) = T_2\omega(\tilde x_2)$ left possible by Lemma~\ref{LemaSubTrajecSimple}.

During the proof we will need the following lemma of \cite{pa}. We will denote $\overline{\tilde S} = \tilde S \sqcup \partial\tilde S$.

\begin{lemma}[\cite{pa}, Lemma 10.10.2] \label{LemEndleaves}
Let $\tilde \alpha\subset\tilde S$ be a proper transverse trajectory, and $\tilde\phi$ a leaf of $\tilde{\F}$ that contains some point $\tilde{z}$ of $R(\tilde \alpha)$.
If the half-leaf $\tilde\phi_{\tilde z}^+$ meets $\tilde \alpha$, then the $\omega$-limit set of $\tilde\phi$ in $\overline{\tilde S}$ does not meet $\partial \tilde S$; moreover, $\tilde\phi_{\tilde z}^+$ does not meet $L(\tilde \alpha)$, and the intersection $\overline{\tilde\phi_{\tilde z}^+}\cap \overline{R(\tilde \alpha)}\subset \overline{\tilde S}$ is a segment of $\tilde\phi_{\tilde z}^+$.
\end{lemma}

Note that the same lemma holds by replacing left with right and positive with negative. See Figure~\ref{FigI2} for an example of the configuration of Lemma~\ref{LemEndleaves}.

\begin{proof}[Proof of Lemma~\ref{LemSimpleOrbit}]
First, suppose that  $\lim_{t\to +\infty} \tilde I_2(t) = \alpha(\tilde x_2)$ and $\lim_{t\to -\infty} \tilde I_2(t) = \omega(\tilde x_2)$. This implies that $L(T_2\tilde I^\Z_{\tilde \F}(\tilde x_2))\subset R(\tilde I_2)$ and $R(T_2\tilde I^\Z_{\tilde \F}(\tilde x_2))\subset L(\tilde I_2)$. Applying Lemma~\ref{LemEndleaves} to both $T_2\tilde I^\Z_{\tilde \F}(\tilde x_2)$ and $\tilde I_2$ implies that picking one point of $\tilde \phi_1^t$ inside $R(T_2\tilde I^\Z_{\tilde \F}(\tilde x_2))$, both half-leaves of $\tilde\phi_1^t$ defined by this point of the leaf stay in the complement of $L(T_2\tilde I^\Z_{\tilde \F}(\tilde x_2))$. This contradicts Hypothesis \ref{PropertyC}.

Hence, $\lim_{t\to -\infty} \tilde I_2(t) = \alpha(\tilde x_2)$ and $\lim_{t\to +\infty} \tilde I_2(t) = \omega(\tilde x_2)$.
The path $\tilde I_2$ separates $L(T_2\tilde I^\Z_{\tilde \F}(\tilde x_2))$ from $R(T_2\tilde I^\Z_{\tilde \F}(\tilde x_2))$. Hence, for $t<t_0$, the leaf $\tilde \phi_1^t$ meets the path $\tilde I_2$. As this path is simple and $\tilde\F$-transverse, it can only meet it once. This allows us to define the holonomy map $\varphi : (-\infty,t_0]\to \R$ by
\begin{equation}\label{eq:Ip2}
\tilde I_2(\varphi(t))\in \tilde \phi_1^t.
\end{equation}
This map $\varphi$ is locally continuous and locally strictly monotonic (as a holonomy), hence it is globally continuous and strictly monotonic. This strict monotonicity prevents $\tilde I^{(-\infty,t_0]}_{\tilde \F}(\tilde x_1)$ from having self-intersections.
\end{proof}

The last conclusion of Lemma~\ref{LemSimpleOrbit} implies that the holonomy map $\varphi$ defined by \eqref{eq:Ip2} is increasing.
Let 
\[s_0 = \inf\varphi\big((-\infty,t_0]\big).\]
We now split the proof into two cases, depending whether $s_0>-\infty$ or $s_0=-\infty$ (see Figure~\ref{fig:casesInterGeod}).

\subsubsection{If \ref{PropertyC} holds and $s_0>-\infty$}

\begin{figure}
\begin{center}
\tikzset{every picture/.style={line width=0.75pt}} 

\begin{tikzpicture}[x=0.75pt,y=0.75pt,yscale=-1,xscale=1]

\draw [color={rgb, 255:red, 0; green, 104; blue, 230 }  ,draw opacity=1 ]   (505.12,89.3) .. controls (515.92,113.95) and (512.41,254.42) .. (505.12,267.06) ;
\draw [shift={(511.92,181.49)}, rotate = 270.4] [fill={rgb, 255:red, 0; green, 104; blue, 230 }  ,fill opacity=1 ][line width=0.08]  [draw opacity=0] (8.04,-3.86) -- (0,0) -- (8.04,3.86) -- (5.34,0) -- cycle    ;
\draw [color={rgb, 255:red, 208; green, 2; blue, 27 }  ,draw opacity=1 ]   (420.36,178.49) .. controls (446.81,185.95) and (546.38,195.01) .. (596.75,178.49) ;
\draw [shift={(512.11,187.76)}, rotate = 180.87] [fill={rgb, 255:red, 208; green, 2; blue, 27 }  ,fill opacity=1 ][line width=0.08]  [draw opacity=0] (8.04,-3.86) -- (0,0) -- (8.04,3.86) -- (5.34,0) -- cycle    ;
\draw  [line width=1.5]  (420.36,178.49) .. controls (420.36,129.41) and (459.84,89.61) .. (508.55,89.61) .. controls (557.26,89.61) and (596.75,129.41) .. (596.75,178.49) .. controls (596.75,227.58) and (557.26,267.38) .. (508.55,267.38) .. controls (459.84,267.38) and (420.36,227.58) .. (420.36,178.49) -- cycle ;
\draw [color={rgb, 255:red, 245; green, 166; blue, 35 }  ,draw opacity=1 ]   (505.12,267.06) .. controls (454.25,194.71) and (580.25,180.21) .. (505.12,89.3) ;
\draw [shift={(516.09,182.24)}, rotate = 301.51] [fill={rgb, 255:red, 245; green, 166; blue, 35 }  ,fill opacity=1 ][line width=0.08]  [draw opacity=0] (6.25,-3) -- (0,0) -- (6.25,3) -- cycle    ;
\draw [color={rgb, 255:red, 245; green, 166; blue, 35 }  ,draw opacity=1 ]   (505.12,267.06) .. controls (450.75,212.71) and (555.25,162.71) .. (505.12,89.3) ;
\draw [shift={(504.31,183.62)}, rotate = 294.31] [fill={rgb, 255:red, 245; green, 166; blue, 35 }  ,fill opacity=1 ][line width=0.08]  [draw opacity=0] (6.25,-3) -- (0,0) -- (6.25,3) -- cycle    ;
\draw [color={rgb, 255:red, 245; green, 166; blue, 35 }  ,draw opacity=1 ]   (505.12,267.06) .. controls (434.25,213.27) and (544.25,155.27) .. (505.12,89.3) ;
\draw [shift={(492.4,184.51)}, rotate = 295.09] [fill={rgb, 255:red, 245; green, 166; blue, 35 }  ,fill opacity=1 ][line width=0.08]  [draw opacity=0] (6.25,-3) -- (0,0) -- (6.25,3) -- cycle    ;
\draw [color={rgb, 255:red, 245; green, 166; blue, 35 }  ,draw opacity=1 ]   (505.12,267.06) .. controls (466.75,204.71) and (594.25,161.21) .. (505.12,89.3) ;
\draw [shift={(525.57,179.27)}, rotate = 298.31] [fill={rgb, 255:red, 245; green, 166; blue, 35 }  ,fill opacity=1 ][line width=0.08]  [draw opacity=0] (6.25,-3) -- (0,0) -- (6.25,3) -- cycle    ;
\draw [color={rgb, 255:red, 245; green, 166; blue, 35 }  ,draw opacity=1 ]   (505.12,267.06) .. controls (487.25,196.21) and (599.25,163.21) .. (505.12,89.3) ;
\draw [shift={(534.13,178.54)}, rotate = 296.28] [fill={rgb, 255:red, 245; green, 166; blue, 35 }  ,fill opacity=1 ][line width=0.08]  [draw opacity=0] (6.25,-3) -- (0,0) -- (6.25,3) -- cycle    ;
\draw  [fill={rgb, 255:red, 74; green, 74; blue, 74 }  ,fill opacity=0.1 ][line width=1.5]  (68.43,157.1) .. controls (68.8,68.8) and (172.93,106.6) .. (218.43,107.1) .. controls (263.93,107.6) and (368.3,68.3) .. (368.43,157.1) .. controls (368.55,245.89) and (271.3,205.8) .. (218.43,207.1) .. controls (165.55,208.39) and (68.05,245.39) .. (68.43,157.1) -- cycle ;
\draw  [fill={rgb, 255:red, 255; green, 255; blue, 255 }  ,fill opacity=1 ][line width=1.5]  (158.29,163.77) .. controls (146.81,167.89) and (135.02,170.31) .. (118.7,163.57) .. controls (128.68,147.05) and (148.57,146.94) .. (158.29,163.77) -- cycle ;
\draw [line width=1.5]    (108.4,157) .. controls (123.48,171.72) and (154.48,171.39) .. (168.4,157) ;

\draw  [fill={rgb, 255:red, 255; green, 255; blue, 255 }  ,fill opacity=1 ][line width=1.5]  (318.49,164.37) .. controls (307.01,168.49) and (295.22,170.91) .. (278.9,164.17) .. controls (288.88,147.65) and (308.77,147.54) .. (318.49,164.37) -- cycle ;
\draw [line width=1.5]    (268.6,157.6) .. controls (283.68,172.32) and (314.68,171.99) .. (328.6,157.6) ;

\draw [color={rgb, 255:red, 0; green, 104; blue, 230 }  ,draw opacity=1 ] [dash pattern={on 0.84pt off 2.51pt}]  (138.6,167.8) .. controls (142.87,168.05) and (142.87,215.33) .. (137.83,215.65) ;
\draw [color={rgb, 255:red, 168; green, 104; blue, 0 }  ,draw opacity=1 ][fill={rgb, 255:red, 139; green, 87; blue, 42 }  ,fill opacity=0.1 ]   (505.12,267.06) .. controls (436.78,214.48) and (506.38,163.29) .. (505.12,89.3) ;
\draw [shift={(479.96,186.21)}, rotate = 287.79] [fill={rgb, 255:red, 168; green, 104; blue, 0 }  ,fill opacity=1 ][line width=0.08]  [draw opacity=0] (6.25,-3) -- (0,0) -- (6.25,3) -- cycle    ;
\draw [color={rgb, 255:red, 168; green, 104; blue, 0 }  ,draw opacity=1 ][fill={rgb, 255:red, 139; green, 87; blue, 42 }  ,fill opacity=0.1 ]   (505.12,267.06) .. controls (540.38,172.88) and (579.51,149.68) .. (505.12,89.3) ;
\draw [shift={(542.9,177.94)}, rotate = 287.03] [fill={rgb, 255:red, 168; green, 104; blue, 0 }  ,fill opacity=1 ][line width=0.08]  [draw opacity=0] (6.25,-3) -- (0,0) -- (6.25,3) -- cycle    ;
\draw [color={rgb, 255:red, 168; green, 104; blue, 0 }  ,draw opacity=1 ]   (124.79,165.76) .. controls (114.79,162.6) and (91.7,210.07) .. (106.95,213.32) ;
\draw [color={rgb, 255:red, 168; green, 104; blue, 0 }  ,draw opacity=1 ] [dash pattern={on 0.84pt off 2.51pt}]  (124.79,165.76) .. controls (133.46,168.93) and (118.92,215.96) .. (106.95,213.32) ;
\draw [color={rgb, 255:red, 168; green, 104; blue, 0 }  ,draw opacity=1 ]   (158.29,163.77) .. controls (168.51,156.22) and (159.58,213.72) .. (151.46,214.6) ;
\draw [color={rgb, 255:red, 168; green, 104; blue, 0 }  ,draw opacity=1 ] [dash pattern={on 0.84pt off 2.51pt}]  (156.73,164.9) .. controls (151.82,168.93) and (143.58,215.6) .. (151.46,214.6) ;
\draw [color={rgb, 255:red, 245; green, 166; blue, 35 }  ,draw opacity=1 ]   (142.08,167.36) .. controls (141.01,167.33) and (135.15,188.48) .. (124.15,194.9) .. controls (113.15,201.33) and (106.82,213.07) .. (112.11,214.21) ;
\draw [color={rgb, 255:red, 245; green, 166; blue, 35 }  ,draw opacity=1 ]   (152.84,165.63) .. controls (157.84,164.77) and (150.11,190.53) .. (141.87,195.62) .. controls (133.63,200.71) and (124,214.5) .. (125.82,214.79) ;
\draw [color={rgb, 255:red, 245; green, 166; blue, 35 }  ,draw opacity=1 ]   (133.96,167.5) .. controls (129.98,167.08) and (132.68,181.07) .. (123.68,185.5) .. controls (114.68,189.93) and (102.98,212.39) .. (108.27,213.54) ;
\draw [color={rgb, 255:red, 245; green, 166; blue, 35 }  ,draw opacity=1 ]   (157.01,164.41) .. controls (162.01,163.56) and (157.01,198.64) .. (149.39,203.21) .. controls (141.77,207.79) and (140.65,215.06) .. (139.79,215.54) ;
\draw [color={rgb, 255:red, 245; green, 166; blue, 35 }  ,draw opacity=1 ]   (146.42,168.06) .. controls (148.36,167.94) and (143.11,187.9) .. (131.46,195.26) .. controls (119.82,202.62) and (115.44,213.99) .. (117.39,214.21) ;
\draw [color={rgb, 255:red, 208; green, 2; blue, 27 }  ,draw opacity=1 ]   (169.03,186.75) .. controls (227.43,159.95) and (183.37,109.25) .. (117,126.6) .. controls (50.63,143.95) and (94.23,222.55) .. (169.03,186.75) -- cycle ;
\draw [shift={(179.16,131.57)}, rotate = 33.6] [fill={rgb, 255:red, 208; green, 2; blue, 27 }  ,fill opacity=1 ][line width=0.08]  [draw opacity=0] (10.72,-5.15) -- (0,0) -- (10.72,5.15) -- (7.12,0) -- cycle    ;
\draw [shift={(100.42,186.85)}, rotate = 218.8] [fill={rgb, 255:red, 208; green, 2; blue, 27 }  ,fill opacity=1 ][line width=0.08]  [draw opacity=0] (10.72,-5.15) -- (0,0) -- (10.72,5.15) -- (7.12,0) -- cycle    ;
\draw  [draw opacity=0][fill={rgb, 255:red, 139; green, 87; blue, 42 }  ,fill opacity=0.1 ] (124.79,165.76) .. controls (136.15,168.52) and (143.15,168.02) .. (158.29,163.77) .. controls (169.28,156.15) and (158.93,213.92) .. (151.46,214.6) .. controls (147.15,215.52) and (113.4,217.02) .. (106.95,213.32) .. controls (91.15,207.77) and (116.15,161.52) .. (124.79,165.76) -- cycle ;
\draw [color={rgb, 255:red, 0; green, 104; blue, 230 }  ,draw opacity=1 ]   (138.6,167.8) .. controls (132.01,168.05) and (132.01,216.33) .. (137.83,215.65) ;
\draw [shift={(133.57,194.95)}, rotate = 270.46] [fill={rgb, 255:red, 0; green, 104; blue, 230 }  ,fill opacity=1 ][line width=0.08]  [draw opacity=0] (8.04,-3.86) -- (0,0) -- (8.04,3.86) -- (5.34,0) -- cycle    ;

\draw (461.51,179.89) node [anchor=south] [inner sep=0.75pt]  [color={rgb, 255:red, 208; green, 2; blue, 27 }  ,opacity=1 ]  {$\tilde{I}_{\tilde\F}^{\Z}(\tilde{x}_{2})$};
\draw (517,241) node [anchor=west] [inner sep=0.75pt]  [color={rgb, 255:red, 0; green, 104; blue, 230 }  ,opacity=1 ]  {$\tilde{I}_{\tilde\F}^{\Z}(\tilde{x}_{1})$};
\draw (200.01,137.56) node [anchor=west] [inner sep=0.75pt]  [color={rgb, 255:red, 208; green, 2; blue, 27 }  ,opacity=1 ]  {$I_{\F}^{\Z}( x_{2})$};
\draw (137.83,219.05) node [anchor=north] [inner sep=0.75pt]  [color={rgb, 255:red, 0; green, 104; blue, 230 }  ,opacity=1 ]  {$I_{\F}^{\Z}( x_{1})$};
\draw (163.43,200.41) node [anchor=west] [inner sep=0.75pt]  [color={rgb, 255:red, 139; green, 87; blue, 42 }  ,opacity=1 ]  {$B$};
\draw (475,227.61) node [anchor=east] [inner sep=0.75pt]  [color={rgb, 255:red, 139; green, 87; blue, 42 }  ,opacity=1 ]  {$\tilde{\phi}_0$};

\end{tikzpicture}

\caption{Lellouch's example \cite{lellouch} on the genus-2 surface (left) and on the universal cover $\wt S$: the trajectories $\I_{\tilde \F}^\Z(\tilde x_1)$ and $\I_{\tilde \F}^\Z(\tilde x_2)$ have no $\F$-transverse intersection, as $\I_{\tilde \F}^\Z(\tilde x_1)$ is equivalent to a subpath of $\I_{\tilde \F}^\Z(\tilde x_2)$.}\label{Figlellouch}
\end{center}
\end{figure}
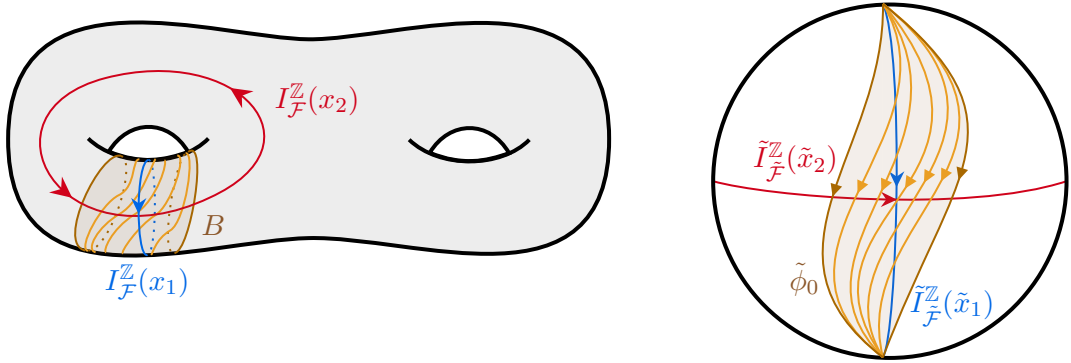

If the infimum $s_0$ is finite, then it is not attained (because $\varphi$ is an open map). By monotonicity, this implies that $\tilde I^\Z_{\tilde \F}(\tilde x_1)$ accumulates in $\tilde I_2$ at the point $\tilde I_2(s_0)$ (in the sense of the definition page~\pageref{DefAccumulate}). 
As $\tilde I_2$ is made of a locally finite number of pieces of $T_2\tilde I^\Z_{\tilde \F}(\tilde x_2)$, this implies that $\tilde I^\Z_{\tilde \F}(\tilde x_1)$ accumulates in $T_2\tilde I^\Z_{\tilde \F}(\tilde x_2)$. An example of such configuration is depicted in Figure~\ref{Figlellouch}.

By \cite[Proposition 3.2 and Proposition 4.17]{guiheneuf2023area}, there exists a transverse simple loop $A_*\subset S$, with associated deck transformation $T$, with the following properties:
\begin{enumerate}
\item The set $B$ of leaves met by $A_*$ is an open annulus of $S$.
\item The path $I^\Z_{\F}(x_1)$ stays in $B$ and is $\F$-equivalent to the natural lift of $A_*$.
\item The path $I^\Z_{\F}(x_2)$ is not included in $B$. 

More precisely, denote $\tilde B$ the lift of $B$ to $\tilde S$ containing $\tilde I_{\tilde\F}^\Z(\tilde x_1)$, and $\tilde \phi_0$ the limit leaf of the accumulation. 
Then $\tilde \phi_0\subset \partial \tilde B^R$, and $\tilde B\subset L(\tilde \phi)$ for every $\tilde \phi\subset \partial \tilde B^R$. 
\end{enumerate}

These properties allow us to apply the arguments of \cite[Section 3.4]{lellouch}. Indeed, this section treats a setting that is identical to ours, with the additional assumption that the homological rotation vectors of the measures $\mu_1$ and $\mu_2$ have non-trivial intersection in homology. The only place this assumption is used is in the proof of \cite[Lemme 3.4.3]{lellouch}, that can be replaced here with point 3. of the above properties\footnote{Equivalently, one could use \cite[Proposition 4.17]{guiheneuf2023area} to get in the setting of \cite[Section 3.4]{lellouch}, but it would somehow use Lellouch's arguments twice.}.
Hence, by \cite[Section 3.4]{lellouch}, we get the existence of a topological horseshoe, and thus the positivity of topological entropy.
\medskip

Let us now explain how to get new rotation vectors in the sense of \cite{pa} (\emph{i.e.}\ the second point of the theorem). 

By \cite[Lemme 2.2.13]{lellouch}, there exists a bounded neighbourhood $\tilde W$ of $\tilde x_1$, two increasing sequences $(r_n)_{n\in\Z}$, $(s_n)_{n\in\Z}$ of integers with $r_0=s_0=0$, such that for any $n\in\Z$, one has $\tilde f^{r_n}(\tilde x_1)\in T^{s_n}(\tilde W)$; moreover the sets $T^{s_n}(\tilde W)$ go towards $\alpha(\tilde x_1)$ with speed $\vartheta_{\mu_1}$ when $n$ goes to $-\infty$, and the sets $T^{s_n}(\tilde W)$ go towards $\omega(\tilde x_1)$ with speed $\vartheta_{\mu_1}$ when $n$ goes to $+\infty$. In particular, denoting  $\tilde \alpha$ a natural lift of $A_*$ to $\tilde S$ that is included in $\tilde B$, this implies that $\alpha(\tilde\alpha) = \alpha(\tilde x_1)$ and $\omega(\tilde\alpha) = \omega(\tilde x_1)$. 

Using the recurrence of the trajectory of $x_2$, we know that there exists a sequence $(S_n)_{n\in\Z}$ of deck transformations and an increasing sequence $(k_n)_{n\in\Z}$ of integers with $k_1>0$ and $k_{-1}<0$ such that for any $i<0$ and $j>0$, the trajectory $\tilde \I^{[k_i,k_j]}_{\tilde\F}(\tilde x_2)$ crosses both $S_i\tilde\phi_0$ and $S_j\tilde \phi_0$. 
Moreover, using as before the local continuity of $\tilde y\mapsto\tilde\gamma_{\tilde y}$ on a positive measure subset by Lusin theorem and the fact that $\tilde x_2$ tracks a geodesic (Theorem~\ref{thm:trackinggeodesictheorem}), we can suppose that for $j\to+\infty$, the sets $S_j\tilde\alpha$ converge towards $\omega(\tilde x_2)$ (for Hausdorff distance in the Poincaré closed disk associated to the Euclidean distance), and the quantities 
\begin{equation}\label{eq:projSAlpha}
\tdist\big(\inf\pr_{\tilde\gamma_{\tx_2}}(S_j\tilde\alpha),\, \sup\pr_{\tilde\gamma_{\tx_2}}(S_j\tilde\alpha)\big)
\qquad \text{and} \qquad
\tdist\big(\sup\pr_{\tilde\gamma_{\tx_2}}(S_j\tilde\alpha),\, \inf\pr_{\tilde\gamma_{\tx_2}}(S_j\tilde\alpha)\big)
\end{equation}
are equivalent to $k_j\vartheta_{\mu_2}$ (and a similar statements holds at $-\infty$). By taking a subsequence if necessary, we can suppose that the sets $(S_j\tilde B)_j$ (that are pairwise equal or disjoint) cross the geodesic $\tgamma_{\tx_2}$ in an increasing way.

A slight modification of \cite[Lemme 3.4.9]{lellouch} (for taking into account negative times, replacing the leaves $\tilde\phi$ and $S_m\tilde\phi$ of \cite{lellouch} by respectively $S_{-i}\tilde\phi_0$ and $S_j\tilde \phi_0$) leads to the following version of \cite[Lemme 3.4.9]{lellouch}. 

\begin{lemma}
There exists $K>0$ and $n_0>0$ such that for any $i,j\in\N$ and $n\ge n_0$, there exists $\tilde y_{i,j,n}\in\tilde S$ such that the transverse path $\tilde\alpha_{i,j,n} := \tilde I_{\tilde\F}^{K+r_n-k_{-i}+k_j}(\tilde y_{i,j,n})$ links $S_{-i}\tilde\phi_0$ to $T^{s_n-1}S_j\tilde \phi_0$.
\end{lemma}

In particular, the transverse path $\tilde I_{\tilde\F}^{K+r_n-k_{-i}+k_1}(\tilde y_{i,1,n})$ links $S_{-i}\tilde\phi_0$ to $T^{s_n-1}S_1\tilde \phi_0$.

All this construction is equivariant, in particular by the deck transformation $T$: the transverse path $T^{-s_n+1}\tilde I_{\tilde\F}^{K+r_n-k_{-1}+k_j}(\tilde y_{1,j,n})$ links $T^{-s_n+1}S_{-1}\tilde\phi_0$ to $S_j\tilde \phi_0$.
\medskip

Suppose first that $\vartheta_{\mu_1}\ge \vartheta_{\mu_2}$, and let $\{\tilde x_0\} = \tgamma_{\tx_1}\cap\tgamma_{\tx_2}$. Then, for any $i\ge 1$, consider some integer $n_i\ge n_0$ such that 
\begin{equation}\label{Eq1Wins}
S_{n_i}\vartheta_{\mu_1}\ge i (-k_{-i})\vartheta_{\mu_2},
\end{equation}
and set $\tilde y_i:=\tilde y_{i,1,n_i}$ so that 
$\tilde I^{[0,K+r_{n_i}-k_{-i}+k_1]}_{\tilde\F}(\tilde y_i)$ starts on $S_{-i}\tilde\phi_0$ and ends at $T^{s_{n_i}-1}S_1\tilde \phi_0$. In particular, $\tilde y_i \notin R(S_{-i}\tilde B)$ and $\tilde f^{K+r_{n_i}-k_{-i}+k_1}(\tilde y_i)\notin L(T^{s_{n_i}-1}S_1\tilde B)$. This implies, using~\eqref{eq:projSAlpha}, that 
\[\lim_{i\to +\infty} \tilde y_i = \alpha(\tilde x_2)
\qquad \text{and}\qquad 
\lim_{i\to +\infty} \tilde f^{K+r_{n_i}-k_{-i}+k_1}(\tilde y_i) = \omega(\tilde x_1).\]
Moreover, we have 
\[\tdist \big(\tx_0,\, \tilde f^{K+r_{n_i}-k_{-i}+k_1}(\tilde y_i) \big) \ge
\tdist \big(\tx_0,\, T^{s_{n_i}-1}(\tilde B) \big) \underset{i\to+\infty}{\sim}
r_{n_i}\vartheta_{\mu_1},\]
with, by \eqref{Eq1Wins}
\[\frac{r_{n_i}\vartheta_{\mu_1}}{K+r_{n_i}-k_{-i}+k_1}\underset{i\to+\infty}{\longrightarrow}\vartheta_{\mu_1}.\]
Using elementary hyperbolic geometry in $\tilde S$ (\textit{e.g.}\ \cite[Claim 4.4]{pa}), we can translate these estimates from the geodesic arcs from $\alpha(\tgamma_{\tx_2})$ to $\tx_0$ and $\tx_0$ to $\omega(\tx_1)$, to the geodesic from $\alpha(\tx_2)$ to $\omega(\tx_1)$. We get that there exists  some $\vartheta \ge \vartheta_{\mu_1}$ such that $(\alpha(\tilde x_2),\omega(\tilde x_1),\vartheta)$ is a rotation vector in the sense of \cite{pa}, hence, using \cite[Theorem A]{pa}, we get that $(\alpha(\tilde x_2),\omega(\tilde x_1),\vartheta_{\mu_1})$ is a rotation vector in the sense of \cite{pa}.

Similarly, for any $j\in\N$, consider $n'_j\ge n_0$ such that 
\begin{equation*}
jS_{n_j}\vartheta_{\mu_1}\le -k_{-j}\vartheta_{\mu_2},
\end{equation*}
and $\tilde y'_j\:=T^{s_{n'_j}-1}\tilde y_{1,j,n'_j}$ so that 
$\tilde I^{[0,K+r_{n'_j}-k_{-1}+k_j]}_{\tilde\F}(\tilde y'_j)$ starts on $T^{-s_{n'_j}+1}S_{-1}\tilde\phi_0$ and finishes on $S_j\tilde \phi_0$. In particular, $\tilde y_i \notin R(S_{-1}\tilde B)$ and $\tilde f^{K+r_{n'_j}-k_{-1}+k_j}(\tilde y'_j)\notin L(T^{r_{n_i}-1}S_j\tilde B)$. This implies, using~\eqref{eq:projSAlpha}, that 
\[\lim_{j\to +\infty} \tilde y'_j = \alpha(\tilde x_1)
\qquad \text{and}\qquad 
\lim_{j\to +\infty} \tilde f^{K+r_{n'_j}-k_{-1}+k_j}(\tilde y'_j) = \omega(\tilde x_2).\]
Like before, we moreover deduce the speed of the rotation vector: $(\alpha(\tilde x_1),\beta(\tilde x_2),\vartheta_{\mu_1})$ is a rotation vector in the sense of \cite{pa}.

A similar reasoning shows that if $\vartheta_{\mu_1}\le \vartheta_{\mu_2}$, then both triples  $(\alpha(\tilde x_1),\omega(\tilde x_2),\vartheta_{\mu_2})$ and $(\alpha(\tilde x_2),\omega(\tilde x_1),\vartheta_{\mu_2})$ are rotation vectors in the sense of \cite{pa}.
\medskip

Finally, we prove the third point of the theorem.
By \cite[Proposition 3.4.13.]{lellouch}\footnote{Or, again, one could use \cite[Proposition 4.17]{guiheneuf2023area} to get in the setting of \cite[Section 3.4]{lellouch}.}, we get that any point in the triangle spanned by $0,\rho_{H_1}(\mu_1),\rho_{H_1}(\mu_2)$ is accumulated by homological rotation vectors periodic orbits of $f$.  In particular, for $t\in (0,1)$ and $\varepsilon\in (0, \min(t,1-t))$, there exists a periodic point $z$ for $f$ such that $\rho_{H_1}(z)\in B(t\rho_{H_1}(\mu_1)+(1-t)\rho_{H_1}(\mu_2), \varepsilon)$, which moreover satisfies $\rho_{H_1}(z)\wedge \rho_{H_1}(\mu_i)\neq 0$ for $i=1,2$. 
By Proposition~\ref{prop:InterHomoImpliesInterGeod}, this implies that for $i=1,2$, there is a lift $\tilde z$ of $z$ and a lift $\tilde x_i$ of a typical point $x_i$ for $\mu_i$ such that the geodesics $\tgamma_{\tilde z}$ and $\tgamma_{\tilde x_i}$ intersect. To get the fact that the tracking geodesic of $\tilde z$ can be supposed to be as close as we want to the tracking geodesic of $\tx_1$ or to the one of $\tx_2$, one can reason as in the end of the proof in the case of Paragraph~\ref{Subsub:other} (this case of the end of Paragraph~\ref{Subsub:other} is similar although a bit more difficult because the deck transformations associated to the trajectory of $x_1$ are not multiples of a single deck transformation $T$).
\medskip

One can apply the same reasoning for the behaviour of the leaves $\tilde\phi_1^t$ for $t$ in a neighbourhood of $+\infty$, and the behaviour of the leaves $\tilde\phi_2^t$ for $t$ in a neighbourhood of $\pm \infty$.
If in one of these cases Condition~\ref{PropertyC} holds and $|s_0|<+\infty$, then the theorem is proved. Otherwise, one has that in all cases, either Condition~\ref{PropertyC} does not hold, or $|s_0|=+\infty$. We treat this case in the following paragraph.

\subsubsection{The other cases}\label{Subsub:other}

We now treat the remaining cases, \emph{i.e.}\ for all four intersections 
\begin{equation}\label{EqPossibleInter}
\begin{array}{cc}
\tilde I^\Z_{\tilde \F}(\tilde x_1)\cap T_2'\tilde I^\Z_{\tilde \F}(\tilde x_2),\qquad&
\tilde I^\Z_{\tilde \F}(\tilde x_1)\cap T_2\tilde I^\Z_{\tilde \F}(\tilde x_2),\\ \tilde I^\Z_{\tilde \F}(\tilde x_2)\cap T_1 \tilde I^\Z_{\tilde \F}(\tilde x_1),\qquad& \tilde I^\Z_{\tilde \F}(\tilde x_2)\cap T_1'\tilde I^\Z_{\tilde \F}(\tilde x_1),
\end{array}
\end{equation}
either Condition~\ref{PropertyC} does not hold, or Condition~\ref{PropertyC} holds and $|s_0|=+\infty$ (see Figure~\ref{fig:OtherCases}).

Consider the first of these intersections. If Condition~\ref{PropertyC} holds and $|s_0|=+\infty$, then for any $t$ small enough the leaf $\tilde\phi_1^t$ meets $T_2T_1\tilde I^\Z_{\tilde \F}(\tilde x_1)$ before it meets $T_2\tilde I^\Z_{\tilde \F}(\tilde x_2)$. Applying again Lemma~\ref{LemEndleaves}, we deduce that the leaf $\tilde\phi_1^t$ is disjoint from the connected component of the complement of $T_2T_1\tilde I^\Z_{\tilde \F}(\tilde x_1) \cup T_2\tilde I^\Z_{\tilde \F}(\tilde x_2)$ containing $\omega(\tilde x_1)$. The same conclusion is true if Condition~\ref{PropertyC} does not hold.

Similar conclusions hold for all of the four intersections of \eqref{EqPossibleInter}. This allows to apply \cite[Lemma 10.7.3]{pa} to conclude that the paths $\tilde I^\Z_{\tilde \F}(\tilde x_1)$ and $\tilde I^\Z_{\tilde \F}(\tilde x_2)$ intersect $\F$-transversally at some point $\tilde I^{t_1}_{\tilde \F}(\tilde x_1) = \tilde I^{t_2}_{\tilde \F}(\tilde x_2)$.
In particular, we get two recurrent transverse trajectories which intersect $\F$-transversally; this implies \cite{lecalveztalforcing, lct2} the existence of a topological horseshoe for $f$ (this proves the first point of Theorem~\ref{TheoInterGeod}). 
\medskip

\begin{figure}
\begin{center}
\tikzset{every picture/.style={line width=0.75pt}} 

\begin{tikzpicture}[x=0.75pt,y=0.75pt,yscale=-.9,xscale=.9]

\draw [color={rgb, 255:red, 0; green, 104; blue, 230 }  ,draw opacity=1 ]   (174.61,42.17) .. controls (186.88,69.75) and (182.89,227) .. (174.61,241.15) ;
\draw [shift={(182.33,145.36)}, rotate = 270.44] [fill={rgb, 255:red, 0; green, 104; blue, 230 }  ,fill opacity=1 ][line width=0.08]  [draw opacity=0] (8.04,-3.86) -- (0,0) -- (8.04,3.86) -- (5.34,0) -- cycle    ;
\draw [color={rgb, 255:red, 208; green, 2; blue, 27 }  ,draw opacity=1 ]   (83.47,169.52) .. controls (123.05,161.2) and (216.82,160.88) .. (274.25,170.8) ;
\draw [shift={(182.89,163.43)}, rotate = 180.7] [fill={rgb, 255:red, 208; green, 2; blue, 27 }  ,fill opacity=1 ][line width=0.08]  [draw opacity=0] (8.04,-3.86) -- (0,0) -- (8.04,3.86) -- (5.34,0) -- cycle    ;
\draw [color={rgb, 255:red, 245; green, 166; blue, 35 }  ,draw opacity=1 ]   (101.02,91.49) .. controls (100.07,70.77) and (168.13,45.1) .. (187.18,45.39) ;
\draw [shift={(134.3,61.65)}, rotate = 335.22] [fill={rgb, 255:red, 245; green, 166; blue, 35 }  ,fill opacity=1 ][line width=0.08]  [draw opacity=0] (6.25,-3) -- (0,0) -- (6.25,3) -- cycle    ;
\draw [color={rgb, 255:red, 245; green, 166; blue, 35 }  ,draw opacity=1 ]   (107.63,93.65) .. controls (108,79.14) and (145.42,57.26) .. (189.55,54.38) ;
\draw [shift={(139.25,66.32)}, rotate = 337.89] [fill={rgb, 255:red, 245; green, 166; blue, 35 }  ,fill opacity=1 ][line width=0.08]  [draw opacity=0] (6.25,-3) -- (0,0) -- (6.25,3) -- cycle    ;
\draw [color={rgb, 255:red, 245; green, 166; blue, 35 }  ,draw opacity=1 ]   (128.11,98.91) .. controls (134.41,79.74) and (164.98,65.82) .. (188.53,64.37) ;
\draw [shift={(149.5,76.3)}, rotate = 332.21] [fill={rgb, 255:red, 245; green, 166; blue, 35 }  ,fill opacity=1 ][line width=0.08]  [draw opacity=0] (6.25,-3) -- (0,0) -- (6.25,3) -- cycle    ;
\draw [color={rgb, 255:red, 208; green, 2; blue, 27 }  ,draw opacity=1 ]   (98.07,84.45) .. controls (132.46,100.12) and (207.41,95.64) .. (262.57,88.6) ;
\draw [shift={(183.4,94.91)}, rotate = 179.04] [fill={rgb, 255:red, 208; green, 2; blue, 27 }  ,fill opacity=1 ][line width=0.08]  [draw opacity=0] (8.04,-3.86) -- (0,0) -- (8.04,3.86) -- (5.34,0) -- cycle    ;
\draw [color={rgb, 255:red, 0; green, 104; blue, 230 }  ,draw opacity=1 ]   (128.74,55.31) .. controls (141.98,76.65) and (115.51,110.66) .. (89.99,97.67) ;
\draw [shift={(121.31,92.09)}, rotate = 311.69] [fill={rgb, 255:red, 0; green, 104; blue, 230 }  ,fill opacity=1 ][line width=0.08]  [draw opacity=0] (8.04,-3.86) -- (0,0) -- (8.04,3.86) -- (5.34,0) -- cycle    ;
\draw [color={rgb, 255:red, 0; green, 104; blue, 230 }  ,draw opacity=1 ]   (458.61,41.5) .. controls (470.88,69.09) and (466.89,226.33) .. (458.61,240.48) ;
\draw [shift={(466.33,144.69)}, rotate = 270.44] [fill={rgb, 255:red, 0; green, 104; blue, 230 }  ,fill opacity=1 ][line width=0.08]  [draw opacity=0] (8.04,-3.86) -- (0,0) -- (8.04,3.86) -- (5.34,0) -- cycle    ;
\draw [color={rgb, 255:red, 208; green, 2; blue, 27 }  ,draw opacity=1 ]   (367.47,168.85) .. controls (407.05,160.54) and (500.82,160.22) .. (558.25,170.13) ;
\draw [shift={(466.89,162.76)}, rotate = 180.7] [fill={rgb, 255:red, 208; green, 2; blue, 27 }  ,fill opacity=1 ][line width=0.08]  [draw opacity=0] (8.04,-3.86) -- (0,0) -- (8.04,3.86) -- (5.34,0) -- cycle    ;
\draw [color={rgb, 255:red, 245; green, 166; blue, 35 }  ,draw opacity=1 ]   (417.33,53.17) .. controls (423.67,60.17) and (452.13,44.44) .. (471.18,44.73) ;
\draw [shift={(439.8,51.37)}, rotate = 343.5] [fill={rgb, 255:red, 245; green, 166; blue, 35 }  ,fill opacity=1 ][line width=0.08]  [draw opacity=0] (6.25,-3) -- (0,0) -- (6.25,3) -- cycle    ;
\draw [color={rgb, 255:red, 245; green, 166; blue, 35 }  ,draw opacity=1 ]   (410.33,73.83) .. controls (425,78.17) and (429.42,56.59) .. (473.55,53.71) ;
\draw [shift={(436.77,63.13)}, rotate = 334.73] [fill={rgb, 255:red, 245; green, 166; blue, 35 }  ,fill opacity=1 ][line width=0.08]  [draw opacity=0] (6.25,-3) -- (0,0) -- (6.25,3) -- cycle    ;
\draw [color={rgb, 255:red, 245; green, 166; blue, 35 }  ,draw opacity=1 ]   (412.11,98.24) .. controls (418.41,79.07) and (448.98,65.16) .. (472.53,63.7) ;
\draw [shift={(433.5,75.63)}, rotate = 332.21] [fill={rgb, 255:red, 245; green, 166; blue, 35 }  ,fill opacity=1 ][line width=0.08]  [draw opacity=0] (6.25,-3) -- (0,0) -- (6.25,3) -- cycle    ;
\draw [color={rgb, 255:red, 208; green, 2; blue, 27 }  ,draw opacity=1 ]   (382.07,83.78) .. controls (416.46,99.45) and (491.41,94.97) .. (546.57,87.94) ;
\draw [shift={(467.4,94.24)}, rotate = 179.04] [fill={rgb, 255:red, 208; green, 2; blue, 27 }  ,fill opacity=1 ][line width=0.08]  [draw opacity=0] (8.04,-3.86) -- (0,0) -- (8.04,3.86) -- (5.34,0) -- cycle    ;
\draw [color={rgb, 255:red, 0; green, 104; blue, 230 }  ,draw opacity=1 ]   (412.74,54.64) .. controls (425.98,75.98) and (399.51,109.99) .. (373.99,97.01) ;
\draw [shift={(405.31,91.42)}, rotate = 311.69] [fill={rgb, 255:red, 0; green, 104; blue, 230 }  ,fill opacity=1 ][line width=0.08]  [draw opacity=0] (8.04,-3.86) -- (0,0) -- (8.04,3.86) -- (5.34,0) -- cycle    ;
\draw  [line width=1.5]  (362.36,141.34) .. controls (362.36,86.4) and (407.2,41.85) .. (462.51,41.85) .. controls (517.83,41.85) and (562.67,86.4) .. (562.67,141.34) .. controls (562.67,196.29) and (517.83,240.83) .. (462.51,240.83) .. controls (407.2,240.83) and (362.36,196.29) .. (362.36,141.34) -- cycle ;
\draw  [line width=1.5]  (78.36,142.01) .. controls (78.36,87.06) and (123.2,42.52) .. (178.51,42.52) .. controls (233.83,42.52) and (278.67,87.06) .. (278.67,142.01) .. controls (278.67,196.96) and (233.83,241.5) .. (178.51,241.5) .. controls (123.2,241.5) and (78.36,196.96) .. (78.36,142.01) -- cycle ;
\draw  [draw opacity=0][fill={rgb, 255:red, 126; green, 211; blue, 33 }  ,fill opacity=0.1 ] (373.99,97.01) .. controls (410.38,16.56) and (516.75,32.94) .. (546.57,87.94) .. controls (511.38,93.81) and (443.75,97.19) .. (405.67,90.83) .. controls (400,96.56) and (388,104.31) .. (373.99,97.01) -- cycle ;
\draw  [draw opacity=0][fill={rgb, 255:red, 126; green, 211; blue, 33 }  ,fill opacity=0.1 ] (89.99,97.67) .. controls (126.37,17.23) and (232.75,33.6) .. (262.57,88.6) .. controls (227.37,94.48) and (159.75,97.85) .. (121.67,91.5) .. controls (116,97.23) and (104,104.98) .. (89.99,97.67) -- cycle ;

\draw (230.84,95.03) node [anchor=north] [inner sep=0.75pt]  [color={rgb, 255:red, 208; green, 2; blue, 27 }  ,opacity=1 ]  {$T_{2}\tilde{I}_{\tilde\F}^{\Z}(\tilde{x}_{2})$};
\draw (185.29,202.31) node [anchor=west] [inner sep=0.75pt]  [color={rgb, 255:red, 0; green, 104; blue, 230 }  ,opacity=1 ]  {$\tilde{I}_{\tilde\F}^{\Z}(\tilde{x}_{1})$};
\draw (133.64,168) node [anchor=north] [inner sep=0.75pt]  [color={rgb, 255:red, 208; green, 2; blue, 27 }  ,opacity=1 ]  {$\tilde{I}_{\tilde\F}^{\Z}(\tilde{x}_{2})$};
\draw (86,102) node [anchor=north west][inner sep=0.75pt]  [color={rgb, 255:red, 0; green, 104; blue, 230 }  ,opacity=1 ]  {$T_{2} T_{1}\tilde{I}_{\tilde\F}^{\Z}(\tilde{x}_{1})$};
\draw (514.84,94.36) node [anchor=north] [inner sep=0.75pt]  [color={rgb, 255:red, 208; green, 2; blue, 27 }  ,opacity=1 ]  {$T_{2}\tilde{I}_{\tilde\F}^{\Z}(\tilde{x}_{2})$};
\draw (469.29,201.64) node [anchor=west] [inner sep=0.75pt]  [color={rgb, 255:red, 0; green, 104; blue, 230 }  ,opacity=1 ]  {$\tilde{I}_{F}^{Z}(\tilde{x}_{1})$};
\draw (417.64,168) node [anchor=north] [inner sep=0.75pt]  [color={rgb, 255:red, 208; green, 2; blue, 27 }  ,opacity=1 ]  {$\tilde{I}_{F}^{Z}(\tilde{x}_{2})$};
\draw (371,100) node [anchor=north west][inner sep=0.75pt]  [color={rgb, 255:red, 0; green, 104; blue, 230 }  ,opacity=1 ]  {$T_{2} T_{1}\tilde{I}_{\tilde\F}^{\Z}(\tilde{x}_{1})$};

\end{tikzpicture}
\caption{In both cases of Paragraph~\ref{Subsub:other}, for $t$ small enough the leaves $\tilde\phi_1^t$ stay in the (light green) region that is the complement of the connected component of $T_2\big(\tilde{I}_{\tilde\F}^{\Z}(\tilde{x}_{2})\cup T_{1}\tilde{I}_{\tilde\F}^{\Z}(\tilde{x}_{1})\big)^\complement$ containing $\omega(\tilde x_1)$. Left: Condition~\ref{PropertyC} holds and $s_0=-\infty$. Right: Condition~\ref{PropertyC} does not hold.\label{fig:OtherCases}}
\end{center}
\end{figure}
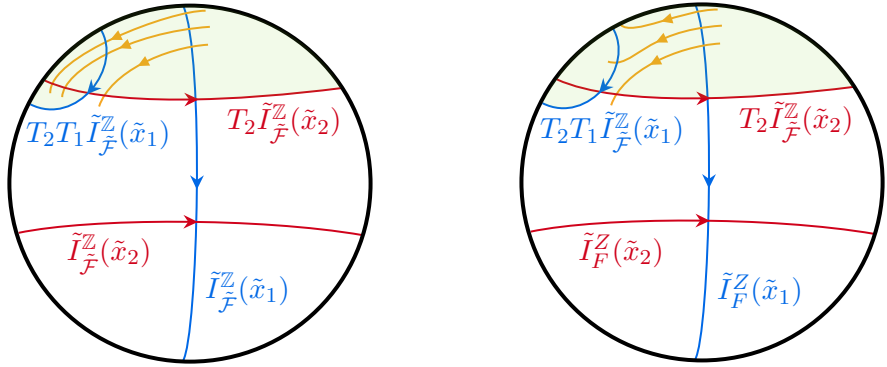

Let us get more precise implications of this property, about the new rotation vectors that are created.


Let us set notations for the transverse intersection we just obtained. Let $n_1\le t_1\le  n_1'$, $n_2\le t_2\le n_2'$, with $n_1,n'_1,n_2,n'_2\in\Z$, be such that:
\begin{itemize}
\item $\tilde I_{\tilde \F}^{n_1'-n_1}(\tilde f^{n_1}(\tilde x_1))$ and $\tilde I_{\tilde \F}^{n_2'-n_2}(\tilde f^{n_2}(\tilde x_2))$ intersect $\tilde\F$-transversally at $\tilde I_{\tilde \F}^{t_1}(\tilde x_1) = \tilde I_{\tilde \F}^{t_2}(\tilde x_2)$;
\item for $i=1,2$, the whole leaf $\tilde\phi_{\tilde f^{n_i}(\tilde x_i)}$ is disjoint from the connected component of the complement of $T_{i+1}\big(\tilde I_{\tilde \F}^\Z(\tilde x_{i+1}) \cup T_i\tilde I_{\tilde \F}^\Z(\tilde x_i)\big)$ containing $\omega(\tilde x_i)$ (the indices are computed modulo 2);
\item for $i=1,2$, the whole leaf $\tilde\phi_{\tilde f^{n'_i}(\tilde x_i)}$ is disjoint from the connected component of the complement of $T_{i+1}'\big(\tilde I_{\tilde \F}^\Z(\tilde x_{i+1}) \cup T'_i\tilde I_{\tilde \F}^\Z(\tilde x_i)\big)$ containing $\alpha(\tilde x_i)$.
\end{itemize}

For $i=1,2$, let $\tilde U_i$ be a neighbourhood of $\tilde f^{n_i}(\tilde x_i)$ such that for any $\tilde y\in \tilde U_i$, the path $\tilde I_{\tilde \F}^{n'_i-n_i}(\tilde x_i)$ is equivalent to a subpath of $\tilde I_{\tilde \F}^{n'_i-n_i+2}(\tilde f^{-1}(\tilde y))$ (\emph{e.g.}\ \cite[Lemma 17]{lecalveztalforcing}).

As $x_i$ is typical, by Lusin theorem, there is a subset $\tilde A_i$ of $\tilde U_i$ of $\mu_i$-positive measure on which the (tracking geodesic) map $\tilde y\mapsto\tilde\gamma_{\tilde y}$ is continuous. 

By typicality of the point $\tilde x_1$, the fact that it is tracked by $\tilde\gamma_{\tilde x_1}$, and that it has homological rotation vector $\rho_{H_1}(\mu_1)$, Birkhoff theorem implies that there exists a sequence $(S_{1,j})_{j\in\N}$ of deck transformations, as well as a sequence of times $(m_{1,j})_{j\in\N}$, with $S_{1,0}=\Id$ and $m_{1,0}=0$, and $\eta>0$, such that:
\begin{enumerate}[label=(\alph*${}_1$)]
\item $m_{1,j}\underset{j\to+\infty}{\sim} -j\eta$;
\item $\tilde f^{m_{1,j}}(\tilde x_1)\in S_{1,j}(\tilde A_1)$;
\item $\frac{[S_{1,j}]_{H_1}}{m_{1,j}} = \rho_{H_1}(\mu_1)$;
\item \[\lim_{j\to+\infty}\inf\frac{\pr_{\tgamma_{\tx_1}}(S_{1,j}\tilde A_1)}{m_{1,j}} = \lim_{j\to+\infty}\sup\frac{\pr_{\tgamma_{\tx_1}}(S_{1,j}\tilde A_1)}{m_{1,j}} = \vartheta_{\mu_1};\]
\item $L\big(S_{1,j+1}T'_2I_{\tilde \F}^\Z(\tilde x_2)\big) \subset L\big(S_{1,j}T_2I_{\tilde \F}^\Z(\tilde x_2)\big)$ (this is a consequence of the above property about the angles between tracking geodesics being bigger than $\theta_0$, and Lemma~\ref{LemConsRecurGeod});
\item $R\big(S_{1,j}^{-1}T_2I_{\tilde \F}^\Z(\tilde x_2)\big)\subset R\big(T_2'I_{\tilde \F}^\Z(\tilde x_2)\big)$  (this is obtained by using the fact that the axis of $S_{1,j}^{-1}$ is close to $\tgamma_{\tx_2}$, by reasoning as in the proof of Lemma~\ref{ClaimSpeedTransverseTraj}, and the fact that the translation length of $S_{1,j}$ is arbitrarily long).  
\end{enumerate}

Similarly, by typicality of the point $\tilde x_2$, the fact that it is tracked by $\tilde\gamma_{\tilde x_2}$, and that it has homological rotation vector $\rho_{H_1}(\mu_2)$, Birkhoff theorem implies that there exists a sequence $(S_{2,j'})_{j'\in\N}$ of deck transformations, as well as a sequence of times $(m_{2,j'})_{j'\in\N}$, with $S_{2,0}=\Id$ and $m_{2,0}=0$, and $\eta>0$, such that:
\begin{enumerate}[label=(\alph*${}_2$)]
\item $m_{2,j'}\underset{j'\to+\infty}{\sim} j\eta$;
\item $\tilde f^{m_{2,j'}}(\tilde x_2)\in S_{2,j'}(\tilde A_2)$;
\item $\frac{[S_{2,j'}]_{H_1}}{m_{2,j'}} = \rho_{H_1}(\mu_2)$;
\item \[\lim_{j'\to+\infty}\inf\frac{\pr_{\tgamma_{\tx_2}}(S_{2,j'}\tilde A_2)}{m_{2,j'}} = \lim_{j'\to+\infty}\sup\frac{\pr_{\tgamma_{\tx_2}}(S_{2,j'}\tilde A_2)}{m_{2,j'}} = \vartheta_{\mu_2};\]
\item $L\big(S_{2,j'+1}T_1\tilde I_{\tilde \F}^\Z(\tilde x_1)\big)\subset L\big(S_{2,j'}T_1'\tilde I_{\tilde \F}^\Z(\tilde x_1)\big)$;
\item $R\big(S_{2,j}^{-1}T_1'I_{\tilde \F}^\Z(\tilde x_1)\big) \subset R\big(T_1I_{\tilde \F}^\Z(\tilde x_1)\big)$.
\end{enumerate}

By the fundamental proposition \cite[Proposition 20]{lecalveztalforcing} of the forcing theory, for $j,j'\in\N^*$, there exists $\tilde y_{j,j'}\in\tilde S$ such that the transverse path (see Figure~\ref{FigFinPreuveOuf}, left)
\[\tilde\alpha_{j,j'}:= \tilde I_{\tilde \F}^{-m_{1,j}+m_{2,j'}+n_1'-n_1+n_2'-n_2+2}(\tilde y_{j,j'})\]
is $\tilde \F$-equivalent to the path $\tilde I_{\tilde \F}^\Z(\tilde x_1)|_{[m_{1,j}-1,t_1]}I_{\tilde \F}^\Z(\tilde x_2)|_{[t_2,m_{2,j'}+1]}$.

These paths satisfy:
\begin{enumerate}[label=\arabic*)]
\item For any neighbourhoods $O_1,O_2$ of respectively $\alpha(\tilde x_1)$, $\omega(\tilde x_2)$, there exists $J_0\in\N$ such that if $j,j'\ge J_0$, then $\tilde y_{j,j'}\in O_1$ and $\tilde f^{-m_{1,j}+m_{2,j'}+n_1'-n_1+n_2'-n_2+2}(\tilde y_{j,j'})\in O_2$ (by b) and d));
\item If $j'$ is fixed, then $\pr_{\tilde \gamma_{\tilde x_1}}(\tilde y_{j,j'})\underset{j\to+\infty}{\sim} m_{1,j}\vartheta_{\mu_1}$ (by (b${}_1$), (b${}_2$), (d${}_1$) and (d${}_2$));
\item If $j$ is fixed, then $\pr_{\tilde \gamma_{\tilde x_2}}(\tilde f^{-m_{1,j}+m_{2,j'}+n_1'-n_1+n_2'-n_2}(\tilde y_{j,j'}))\underset{j'\to+\infty}{\sim} m_{2,j'}\vartheta_{\mu_2}$ (by (b${}_1$), (b${}_2$), (d${}_1$) and (d${}_2$));
\item The paths $\tilde\alpha_{j,j'}$ and $S_{2,j'}S_{1,j}^{-1}\tilde\alpha_{j,j'}$ intersect $\F$-transversally. This comes from the facts that the path $S_{1,j}\tilde I_{\tilde \F}^{n'_1-n_1}(\tilde x_1)$ is equivalent to a subpath of $\tilde I_{\tilde \F}^{n'_1-n_1+2}(\tilde y_{j,j'})$, and the path $S_{2,j'}\tilde I_{\tilde \F}^{n'_2-n_2}(\tilde x_2)$ is equivalent to a subpath of $\tilde I_{\tilde \F}^{n'_2-n_2+2}(\tilde f^{-m_{1,j}+m_{2,j'}+n_1'-n_1}(\tilde y_{j,j'}))$.
\end{enumerate}

The first three properties show that the rotation vector $(\alpha(\tilde x_1), \omega(\tilde x_2), \max(\vartheta_{\mu_1}, \vartheta_{\mu_2}))$ is a rotation vector of $f$ in the sense of \cite{pa} (see \cite[Claim 4.4]{pa}). Symmetrical arguments show that  $(\alpha(\tilde x_2), \omega(\tilde x_1), \max(\vartheta_{\mu_1}, \vartheta_{\mu_2}))$ is also a rotation vector of $f$ in the sense of \cite{pa}.
This proves the second point of the theorem.

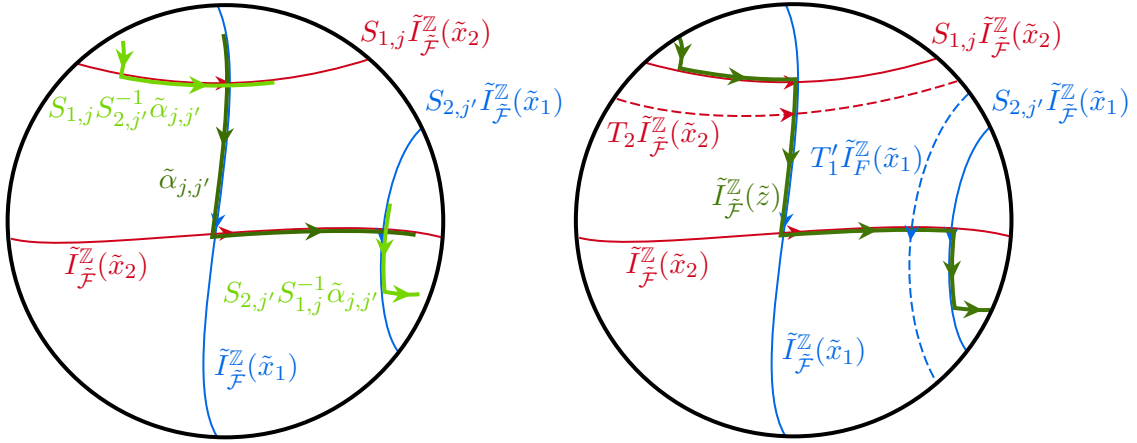
\begin{figure}
\begin{center}

\tikzset{every picture/.style={line width=0.75pt}} 

\begin{tikzpicture}[x=0.75pt,y=0.75pt,yscale=-.97,xscale=.97]

\draw [color={rgb, 255:red, 0; green, 104; blue, 230 }  ,draw opacity=1 ]   (451.58,33.3) .. controls (474.85,82.01) and (427.64,216.74) .. (451.58,257.44) ;
\draw [shift={(451.15,149.72)}, rotate = 277.48] [fill={rgb, 255:red, 0; green, 104; blue, 230 }  ,fill opacity=1 ][line width=0.08]  [draw opacity=0] (8.04,-3.86) -- (0,0) -- (8.04,3.86) -- (5.34,0) -- cycle    ;
\draw [color={rgb, 255:red, 208; green, 2; blue, 27 }  ,draw opacity=1 ]   (345.58,154.78) .. controls (379.2,164.18) and (515.4,139.05) .. (568.17,154.42) ;
\draw [shift={(460.49,151.62)}, rotate = 176.11] [fill={rgb, 255:red, 208; green, 2; blue, 27 }  ,fill opacity=1 ][line width=0.08]  [draw opacity=0] (8.04,-3.86) -- (0,0) -- (8.04,3.86) -- (5.34,0) -- cycle    ;
\draw [color={rgb, 255:red, 208; green, 2; blue, 27 }  ,draw opacity=1 ]   (379.35,63.99) .. controls (412.97,73.4) and (466.75,82.29) .. (530.77,61.47) ;
\draw [shift={(458.68,74.09)}, rotate = 179.01] [fill={rgb, 255:red, 208; green, 2; blue, 27 }  ,fill opacity=1 ][line width=0.08]  [draw opacity=0] (8.04,-3.86) -- (0,0) -- (8.04,3.86) -- (5.34,0) -- cycle    ;
\draw [color={rgb, 255:red, 0; green, 104; blue, 230 }  ,draw opacity=1 ]   (556.91,97.5) .. controls (533.31,119.83) and (530.04,194.05) .. (544.93,213.5) ;
\draw [shift={(536.43,157.29)}, rotate = 275.27] [fill={rgb, 255:red, 0; green, 104; blue, 230 }  ,fill opacity=1 ][line width=0.08]  [draw opacity=0] (8.04,-3.86) -- (0,0) -- (8.04,3.86) -- (5.34,0) -- cycle    ;
\draw [color={rgb, 255:red, 65; green, 117; blue, 5 }  ,draw opacity=1 ][line width=1.5]    (396.78,50.3) .. controls (399.19,54.99) and (398.91,59.35) .. (397.33,66.96) .. controls (417.5,71.04) and (437.64,73.31) .. (456.83,72.63) .. controls (456.92,96.38) and (451.67,133.69) .. (449.25,153.63) .. controls (480.33,151.13) and (517.25,149.96) .. (538.58,151.38) .. controls (537.15,162.67) and (536.85,178.25) .. (538.83,191.29) .. controls (547.43,192.13) and (549.69,192.59) .. (557.31,192.23) ;
\draw [shift={(397.98,63.54)}, rotate = 275.46] [fill={rgb, 255:red, 65; green, 117; blue, 5 }  ,fill opacity=1 ][line width=0.08]  [draw opacity=0] (9.91,-4.76) -- (0,0) -- (9.91,4.76) -- (6.58,0) -- cycle    ;
\draw [shift={(431.92,71.99)}, rotate = 185.35] [fill={rgb, 255:red, 65; green, 117; blue, 5 }  ,fill opacity=1 ][line width=0.08]  [draw opacity=0] (9.91,-4.76) -- (0,0) -- (9.91,4.76) -- (6.58,0) -- cycle    ;
\draw [shift={(453.54,118.53)}, rotate = 276.24] [fill={rgb, 255:red, 65; green, 117; blue, 5 }  ,fill opacity=1 ][line width=0.08]  [draw opacity=0] (9.91,-4.76) -- (0,0) -- (9.91,4.76) -- (6.58,0) -- cycle    ;
\draw [shift={(498.99,150.99)}, rotate = 178.35] [fill={rgb, 255:red, 65; green, 117; blue, 5 }  ,fill opacity=1 ][line width=0.08]  [draw opacity=0] (9.91,-4.76) -- (0,0) -- (9.91,4.76) -- (6.58,0) -- cycle    ;
\draw [shift={(537.53,176.54)}, rotate = 269.07] [fill={rgb, 255:red, 65; green, 117; blue, 5 }  ,fill opacity=1 ][line width=0.08]  [draw opacity=0] (9.91,-4.76) -- (0,0) -- (9.91,4.76) -- (6.58,0) -- cycle    ;
\draw [shift={(553.11,192.35)}, rotate = 182.3] [fill={rgb, 255:red, 65; green, 117; blue, 5 }  ,fill opacity=1 ][line width=0.08]  [draw opacity=0] (9.91,-4.76) -- (0,0) -- (9.91,4.76) -- (6.58,0) -- cycle    ;
\draw [color={rgb, 255:red, 208; green, 2; blue, 27 }  ,draw opacity=1 ] [dash pattern={on 3.75pt off 1.5pt}]  (364.93,82.79) .. controls (408.51,97.01) and (476.06,94.46) .. (540.07,73.64) ;
\draw [shift={(456.79,90.56)}, rotate = 175.74] [fill={rgb, 255:red, 208; green, 2; blue, 27 }  ,fill opacity=1 ][line width=0.08]  [draw opacity=0] (8.04,-3.86) -- (0,0) -- (8.04,3.86) -- (5.34,0) -- cycle    ;
\draw [color={rgb, 255:red, 0; green, 104; blue, 230 }  ,draw opacity=1 ] [dash pattern={on 3.75pt off 1.5pt}]  (546.05,80.07) .. controls (515.79,107.93) and (503.79,185.07) .. (529.5,229.93) ;
\draw [shift={(515.99,156.33)}, rotate = 275.58] [fill={rgb, 255:red, 0; green, 104; blue, 230 }  ,fill opacity=1 ][line width=0.08]  [draw opacity=0] (8.04,-3.86) -- (0,0) -- (8.04,3.86) -- (5.34,0) -- cycle    ;
\draw  [line width=1.5]  (343.86,145.77) .. controls (343.86,83.87) and (394.04,33.7) .. (455.94,33.7) .. controls (517.85,33.7) and (568.03,83.87) .. (568.03,145.77) .. controls (568.03,207.66) and (517.85,257.83) .. (455.94,257.83) .. controls (394.04,257.83) and (343.86,207.66) .. (343.86,145.77) -- cycle ;
\draw [color={rgb, 255:red, 0; green, 104; blue, 230 }  ,draw opacity=1 ]   (159.17,33.8) .. controls (182.44,82.51) and (135.23,217.24) .. (159.17,257.94) ;
\draw [shift={(158.74,150.22)}, rotate = 277.48] [fill={rgb, 255:red, 0; green, 104; blue, 230 }  ,fill opacity=1 ][line width=0.08]  [draw opacity=0] (8.04,-3.86) -- (0,0) -- (8.04,3.86) -- (5.34,0) -- cycle    ;
\draw [color={rgb, 255:red, 208; green, 2; blue, 27 }  ,draw opacity=1 ]   (53.17,155.28) .. controls (86.79,164.68) and (222.98,139.55) .. (275.76,154.92) ;
\draw [shift={(168.08,152.12)}, rotate = 176.11] [fill={rgb, 255:red, 208; green, 2; blue, 27 }  ,fill opacity=1 ][line width=0.08]  [draw opacity=0] (8.04,-3.86) -- (0,0) -- (8.04,3.86) -- (5.34,0) -- cycle    ;
\draw [color={rgb, 255:red, 208; green, 2; blue, 27 }  ,draw opacity=1 ]   (86.94,64.49) .. controls (120.56,73.9) and (174.34,82.79) .. (238.36,61.97) ;
\draw [shift={(166.27,74.59)}, rotate = 179.01] [fill={rgb, 255:red, 208; green, 2; blue, 27 }  ,fill opacity=1 ][line width=0.08]  [draw opacity=0] (8.04,-3.86) -- (0,0) -- (8.04,3.86) -- (5.34,0) -- cycle    ;
\draw [color={rgb, 255:red, 0; green, 104; blue, 230 }  ,draw opacity=1 ]   (264.5,98) .. controls (240.9,120.33) and (237.63,194.55) .. (252.52,214) ;
\draw [shift={(244.02,157.79)}, rotate = 275.27] [fill={rgb, 255:red, 0; green, 104; blue, 230 }  ,fill opacity=1 ][line width=0.08]  [draw opacity=0] (8.04,-3.86) -- (0,0) -- (8.04,3.86) -- (5.34,0) -- cycle    ;
\draw [color={rgb, 255:red, 58; green, 108; blue, 0 }  ,draw opacity=1 ][line width=1.5]    (162.45,49.23) .. controls (168,75.41) and (160.21,125.9) .. (156.73,154.41) .. controls (184.58,151.67) and (239.64,149.59) .. (261.36,153.41) ;
\draw [shift={(162.38,107.37)}, rotate = 275.3] [fill={rgb, 255:red, 58; green, 108; blue, 0 }  ,fill opacity=1 ][line width=0.08]  [draw opacity=0] (9.91,-4.76) -- (0,0) -- (9.91,4.76) -- (6.58,0) -- cycle    ;
\draw [shift={(214.33,151.4)}, rotate = 178.9] [fill={rgb, 255:red, 58; green, 108; blue, 0 }  ,fill opacity=1 ][line width=0.08]  [draw opacity=0] (9.91,-4.76) -- (0,0) -- (9.91,4.76) -- (6.58,0) -- cycle    ;
\draw  [line width=1.5]  (51.44,146.27) .. controls (51.44,84.37) and (101.63,34.2) .. (163.53,34.2) .. controls (225.44,34.2) and (275.62,84.37) .. (275.62,146.27) .. controls (275.62,208.16) and (225.44,258.33) .. (163.53,258.33) .. controls (101.63,258.33) and (51.44,208.16) .. (51.44,146.27) -- cycle ;
\draw [color={rgb, 255:red, 117; green, 213; blue, 0 }  ,draw opacity=1 ][line width=1.5]    (110.79,52.79) .. controls (112.05,65.21) and (110.54,64.79) .. (109.91,71.32) .. controls (136.65,77.11) and (169.36,76.63) .. (188.73,74.32) ;
\draw [shift={(110.55,67.1)}, rotate = 276.89] [fill={rgb, 255:red, 117; green, 213; blue, 0 }  ,fill opacity=1 ][line width=0.08]  [draw opacity=0] (9.91,-4.76) -- (0,0) -- (9.91,4.76) -- (6.58,0) -- cycle    ;
\draw [shift={(154.58,75.9)}, rotate = 181.42] [fill={rgb, 255:red, 117; green, 213; blue, 0 }  ,fill opacity=1 ][line width=0.08]  [draw opacity=0] (9.91,-4.76) -- (0,0) -- (9.91,4.76) -- (6.58,0) -- cycle    ;
\draw [color={rgb, 255:red, 117; green, 213; blue, 0 }  ,draw opacity=1 ][line width=1.5]    (248.27,137.14) .. controls (245.69,151.64) and (243.87,166.53) .. (245.45,182.86) .. controls (254.95,184.45) and (256.33,184.63) .. (263.4,184.1) ;
\draw [shift={(244.97,165.04)}, rotate = 273.36] [fill={rgb, 255:red, 117; green, 213; blue, 0 }  ,fill opacity=1 ][line width=0.08]  [draw opacity=0] (9.91,-4.76) -- (0,0) -- (9.91,4.76) -- (6.58,0) -- cycle    ;
\draw [shift={(259.41,184.34)}, rotate = 182.38] [fill={rgb, 255:red, 117; green, 213; blue, 0 }  ,fill opacity=1 ][line width=0.08]  [draw opacity=0] (9.91,-4.76) -- (0,0) -- (9.91,4.76) -- (6.58,0) -- cycle    ;

\draw (391.24,157.63) node [anchor=north] [inner sep=0.75pt]  [color={rgb, 255:red, 208; green, 2; blue, 27 }  ,opacity=1 ]  {$\tilde{I}_{\tilde\F}^{\Z}(\tilde{x}_{2})$};
\draw (448.38,214.42) node [anchor=west] [inner sep=0.75pt]  [color={rgb, 255:red, 0; green, 104; blue, 230 }  ,opacity=1 ]  {$\tilde{I}_{\tilde\F}^{\Z}(\tilde{x}_{1})$};
\draw (451.4,134.31) node [anchor=east] [inner sep=0.75pt]  [color={rgb, 255:red, 65; green, 117; blue, 5 }  ,opacity=1 ]  {$\tilde{I}_{\tilde\F}^{\Z}(\tilde{z})$};
\draw (420.27,89.98) node [anchor=north east] [inner sep=0.75pt]  [color={rgb, 255:red, 208; green, 2; blue, 27 }  ,opacity=1 ]  {$T_{2}\tilde{I}_{\tilde\F}^{\Z}(\tilde{x}_{2})$};
\draw (525,62.5) node [anchor=south west] [inner sep=0.75pt]  [color={rgb, 255:red, 208; green, 2; blue, 27 }  ,opacity=1 ]  {$S_{1,j}\tilde{I}_{\tilde\F}^{\Z}(\tilde{x}_{2})$};
\draw (524.56,113.36) node [anchor=east] [inner sep=0.75pt]  [color={rgb, 255:red, 0; green, 104; blue, 230 }  ,opacity=1 ]  {$T'_{1}\tilde{I}_{F}^{\Z}(\tilde{x}_{1})$};
\draw (556.03,85) node [anchor=west] [inner sep=0.75pt]  [color={rgb, 255:red, 0; green, 104; blue, 230 }  ,opacity=1 ]  {$S_{2,j'}\tilde{I}_{\tilde\F}^{\Z}(\tilde{x}_{1})$};
\draw (102.33,158.13) node [anchor=north] [inner sep=0.75pt]  [color={rgb, 255:red, 208; green, 2; blue, 27 }  ,opacity=1 ]  {$\tilde{I}_{\tilde\F}^{\Z}(\tilde{x}_{2})$};
\draw (156.14,222.75) node [anchor=west] [inner sep=0.75pt]  [color={rgb, 255:red, 0; green, 104; blue, 230 }  ,opacity=1 ]  {$\tilde{I}_{\tilde\F}^{\Z}(\tilde{x}_{1})$};
\draw (231.23,61.5) node [anchor=south west] [inner sep=0.75pt]  [color={rgb, 255:red, 208; green, 2; blue, 27 }  ,opacity=1 ]  {$S_{1,j}\tilde{I}_{\tilde\F}^{\Z}(\tilde{x}_{2})$};
\draw (264.62,85) node [anchor=west] [inner sep=0.75pt]  [color={rgb, 255:red, 0; green, 104; blue, 230 }  ,opacity=1 ]  {$S_{2,j'}\tilde{I}_{\tilde\F}^{\Z}(\tilde{x}_{1})$};
\draw (158.11,126.51) node [anchor=east] [inner sep=0.75pt]  [color={rgb, 255:red, 65; green, 117; blue, 5 }  ,opacity=1 ]  {$\tilde{\alpha }_{j,j'}$};
\draw (244.61,185) node [anchor=east] [inner sep=0.75pt]  [color={rgb, 255:red, 117; green, 213; blue, 0 }  ,opacity=1 ]  {$S_{2,j'} S_{1,j}^{-1}\tilde{\alpha }_{j,j'}$};
\draw (153.41,77.08) node [anchor=north east] [inner sep=0.75pt]  [color={rgb, 255:red, 117; green, 213; blue, 0 }  ,opacity=1 ]  {$S_{1,j} S_{2,j'}^{-1}\tilde{\alpha }_{j,j'}$};

\end{tikzpicture}

\caption{The path $\tilde{\alpha }_{j,j'}$ and its translates $S_{2,j'} S_{1,j}^{-1}\tilde{\alpha }_{j,j'}$ and $S_{1,j} S_{2,j'}^{-1}\tilde{\alpha }_{j,j'}$ with whom it has transverse intersections (left), and the trajectory of the point $\tilde z$ for the end of the proof of the theorem (right).\label{FigFinPreuveOuf}}
\end{center}
\end{figure}
\medskip

By \cite[Theorem M]{lct2}, property 4) implies the existence of a topological rotational horseshoe relative to the deck transformation $S_{2,j'}S_{1,j}^{-1}$. More precisely, for any rational number $p/q\in (0,1]$ written in a irreducible way, there exists a point $\tilde z\in \tilde S$ such that $\tilde f^{q(-m_{1,j}+m_{2,j'}+n_1'-n_1+n_2'-n_2+2)}(\tilde z) = (S_{2,j'}S_{1,j}^{-1})^p(\tilde z)$ (see Figure~\ref{FigFinPreuveOuf}, right).
In particular, the homological rotation vector of $\tilde z$ is equal to 
\[\frac{p}{q} \frac{[S_{2,j'}]_{H_1(S)}-[S_{1,j}]_{H_1(S)}}{-m_{1,j}+m_{2,j'}+n_1'-n_1+n_2'-n_2+2}
\ \underset{j,j'\to+\infty}{\sim}\ 
 \frac{p}{q} \frac{[S_{2,j'}]_{H_1}-[S_{1,j}]_{H_1}}{-m_{1,j}+m_{2,j'}}.\]
Let $\lambda\in \R_+^*$, and pick $j=n$, $j'=\lfloor\lambda n\rfloor$, we get that these rotation vectors are equivalent to (recall that $-m_{1,j}\sim \eta j \sim m_{2,j}$, and use properties (c${}_1$) and (c${}_2$))
\[ \frac{p}{q} \left(\frac{[S_{2,j'}]_{H_1}}{(1+\lambda^{-1})m_{2,j}} + \frac{-[S_{1,j}]_{H_1}}{-(1+\lambda)m_{1,j}}\right)
\ \underset{j,j'\to+\infty}{\sim}\ 
 \frac{p}{q} \left(\frac{\lambda}{1+\lambda}\rho_{H_1}(\mu_1) + \frac{1}{1+\lambda}\rho_{H_1}(\mu_2)\right),\]
which implies that the whole triangle spanned by $0, \rho_{H_1}(\mu_1), \rho_{H_1}(\mu_2)$ is accumulated by rotation vectors of periodic orbits.
\medskip

Finally, let us locate the endpoints in the boundary at infinity of the trajectory of $\tilde z$. Equivalently, we want to locate the endpoints of the axis of $S_{2,j'}S_{1,j}^{-1}$. For this we consider $j,j'\ge 2$.

Let us first show that for any $k\ge 0$, we have
\begin{equation}\label{eq:inclusions1}
\left(S_{2,j'}S_{i,j}^{-1}\right)^k\big(L(S_{2,1}\tgamma_{\tx_1})\big) \subset L(S_{2,1}\tgamma_{\tx_1}),
\end{equation}
In other words that $L(S_{2,1}\tgamma_{\tx_1})$ is stable under $S_{2,j'}S_{i,j}^{-1}$. Note that, using (e${}_1$)
\[L(S_{2,1}\tgamma_{\tx_1})
\subset L(T_1'\tgamma_{\tx_1}) 
\subset R(T_2\tgamma_{\tx_2})
\subset R(S_{1,j}T'_2\tgamma_{\tx_2})
\subset L(S_{1,j}T_1\tgamma_{\tx_1})\]
so
\[S_{1,j}^{-1}L(S_{2,1}\tgamma_{\tx_1})
\subset L(T_1\tgamma_{\tx_1}).\]
Moreover (because $j'\ge 2$, and by (e${}_2$)), 
\[S_{2,j'}(L(T_1\tgamma_{\tx_1}))\subset S_{2,1}(L(T'_1\tgamma_{\tx_1})).\]
Therefore,
\[S_{2,j'}S_{1,j}^{-1}(L(S_{2,2}\tgamma_{\tx_1}))
\subset S_{2,2}(L(T'_1\tgamma_{\tx_1})).\] 

Similarly, one can show that for any $k\ge 0$, we have
\begin{equation*}
\left(S_{2,j'}S_{i,j}^{-1}\right)^{-k}\big(L(S_{1,2}\tgamma_{\tx_2})\big) \subset L(S_{1,2}\tgamma_{\tx_2}).
\end{equation*}
Combined with \eqref{eq:inclusions1}, this shows that $\alpha(\tilde z) \subset L(S_{1,2}^{-1}\tgamma_{\tx_2})$ and $ \omega(\tilde z)\subset L(S_{2,2}\tgamma_{\tx_1})$.
Hence, $\alpha(S_{1,1}^{-1}\tilde z) \subset L(S_{1,1}^{-1}S_{1,2}\tgamma_{\tx_2})\subset L(T_2\tgamma_{\tx_2})$ (by (e${}_1$)) and $ \omega(S_{1,1}^{-1}\tilde z)\subset L(S_{1,1}^{-1}S_{2,2}\tgamma_{\tx_1}) \subset R(T_2'\tgamma_{\tx_2})$ (by (f${}_1$) and (e${}_2$)). As the number $R$ in Lemma~\ref{LemConsRecurGeod} is arbitrary, this proves that the tracking geodesic of $S_{1,1}^{-1}\tilde z$ can be chosen arbitrarily close to $\tgamma_{\tx_1}$. Similarly, one can show that the tracking geodesic of $S_{2,1}\tilde z$ can be chosen arbitrarily close to $\tgamma_{\tx_2}$.
Moreover, the tracking geodesic of $\tilde z$ can be chosen arbitrarily close to the geodesic $\big(\alpha(\tilde x_1),\omega(\tilde x_2)\big)$, and the tracking geodesic of $S_{1,j}^{-1}\tilde z$ can be chosen arbitrarily close to the geodesic $\big(\alpha(\tilde x_2),\omega(\tilde x_1)\big)$.
This ends the proof of the theorem.
\end{proof}

\section{The shape of rotation sets --- proof of Theorems~\ref{thm:ShapeRotationSet}, \ref{thm:DecompRotSetIntro} and Corollary~\ref{thm:maintheorem}\label{sec:maintheoremsection}}

Let us recall the statement of Theorem~\ref{thm:DecompRotSetIntro}. 

\begin{theorem}\label{thm:DecompRotSet}
Let $f\in\Homeo_0(S)$. Then there exists a decomposition 
\[\M = \bigsqcup_{i\in I} \cl_i = \bigsqcup_{i\in I^1} \cl_i \sqcup \bigsqcup_{i\in I^+} \cl_i\]
of the set of ergodic measures with positive rotation speed into equivalence classes such that, denoting for $i\in I$ (see Theorem~\ref{thm:equidistributiontheoremintro} for the set $\Lambda_\mu$),
\[\rho_i = \big\{\rho_{H_1}(\mu)\mid \mu\in\cl_i\big\},\qquad 
V_i = \operatorname{Span}(\rho_i),\qquad
\Lambda_i = \bigcup_{\mu\in\cl_i} \Lambda_\mu,\]
we have:
\begin{enumerate}
\item For every $i \in I^1$, 

\begin{itemize}
	\item $\Lambda_i$ is a minimal geodesic lamination (hence contains all the tracking geodesics associated to any of the ergodic measures of $\cl_i$). 
	\item $\rho_{I^1} = \bigcup \limits_{i \in I^1}\rho_i$ is included in a union of at most $3g-3$ lines of $H_1(S,\R)$.
\end{itemize}

\item If $I^+ \neq \varnothing$, then $f$ has a topological horseshoe (and in particular, positive topological entropy), and for every $i \in I^+$,

\begin{itemize}
\item The linear subspace $V_i$ has a basis formed by elements of $H_1(S,\Z)$;
\item The set $\overline{\rho_i}$ is a convex set containing $0$;
\item We have $\operatorname{int}_{V_i}(\overline{\rho_i})= \operatorname{int}_{V_i}(\rho_i)$ (in other words, $\rho_i$ is convex up to the fact that elements of $\partial_{V_i}(\rho_i)\setminus\operatorname{extrem}(\rho_i)$ can be in the complement of $\rho_i$);
\item Every element of $\operatorname{int}_{V_i}(\rho_i) \cap H_1(S,\Q)$ is the rotation vector of some $f$-periodic orbit (because $V_i$ has a rational basis, such elements are dense in $\operatorname{int}_{V_i}(\rho_i)$). 
\end{itemize} 

\item For $i,j\in I$, $i\neq j$, for $v_i\in V_i$ and $v_j\in V_j$, we have $v_i\wedge v_j = 0$. Thus, $\mathrm{span}(\rho_{I^{1}})$ is a totally isotropic subspace of $H_1(S,\mathbb{R})$. Moreover, if $i,j\in I^1$, then $\Lambda_i\cap\Lambda_j = \varnothing$.
\item $\card I^1\le 3g-3$ and $\card I^+\le 2g-2$.
\end{enumerate}
\end{theorem}

Note that this theorem implies that $\rho^{\textnormal{erg}}_{H_1}(f) = \bigcup_{i\in I} \rho_i$ (see Definition \ref{def:ErgHomRotationSet}), so $\rho^{\textnormal{erg}}_{H_1}(f)$ is a finite union of convex sets containing 0, together with a set contained in a $g$-dimensional subspace and included in a finite number of one-dimensional subspaces of $H_1(S,\R)$.  

Note also that the periodic orbits that we get in order to prove that a dense subset of $\overline{\rho_i}$ ($i\in I^+$) is realised by periodic orbits are stable, in other words the rotation vectors of these periodic orbits persist under small perturbations of $f$.

Recall that for any finite set $F\subset \inte_{V_i}(\rho_i)$, denoting $R = \conv(\{0\}, F)$, there exists a constant $M>0$ such that for any $p/q\in R\cap H_1(S,\Q)$, with $p\in H_1(S,\Z)$ and $q\in\N^*$, there exists a periodic point of period dividing $Mq$ and with rotation vector $p/q$. We cannot hope having a universal bound on the constant $M$, as shown by the exemple of Figure~\ref{FigExNotExact}.

As depicted in Figure~\ref{FigEx2cvx}, it is possible that for $i,j\in I^+$, $i\neq j$, we have $V_i = V_j$ and $\rho_i\neq \rho_j$. In this specific example, the intersection $\rho^{\textnormal{erg}}_{H_1}(f) \cap V_i$ is made of 4 pieces, two that are convex sets of dimension 2 and two included in segments.
\medskip

The fact that Theorem~\ref{thm:DecompRotSet} implies Theorem~\ref{thm:ShapeRotationSet} and Corollary~\ref{thm:maintheorem} is trivial.

\subsection{Preliminary tools of Nielsen-Thurston Theory}

Let us first recall some definitions of Nielsen-Thurston theory.

Assume that $h\in\Homeo(\Sigma)$ is a homeomorphism of a compact surface $\Sigma$, possibly with boundary or punctures. We call $h$ \emph{periodic} if there exists $n > 0$ such that $h^n = \Id_\Sigma$. We call $h$ \emph{pseudo-Anosov} if there exist $h$-invariant measurable foliations with associated uniformly expanding transverse measures (see \cite{thurston} for details). These two types of homeomorphism are distinct, and in particular periodic homeomorphisms have zero topological entropy while pseudo-Anosov homeomorphisms have nonzero topological entropy. Given a finite $h$-invariant set $F$ (in other words, $F$ is a finite union of $h$-periodic orbits), we call $h$ \emph{pseudo-Anosov relative to $F$} if $h|_{\Sigma\setminus F}$ is pseudo-Anosov.

The key theorem of Nielsen-Thurston theory is the following (\emph{e.g.}\ \cite{Travaux}):

\begin{theorem}[Nielsen-Thurston classification]\label{TheoNTClass}
	Every homeomorphism $f\in\Homeo(\Sigma)$ is isotopic to a homeomorphism $h\in\Homeo(\Sigma)$ such that:
	\begin{enumerate}[label=(\roman*)]
		\item $h$ leaves invariant a finite family (possibly empty) of disjoint simple closed curves $C_1, \dots, C_n$ on $\Sigma$;
		\item No curve $C_i$ is homotopic to a boundary curve of $S$;
		\item We can decompose $\Sigma = \bigcup_{j=1}^d \Sigma_j$, where $\Sigma_1,\dots, \Sigma_d$ are closed surfaces with disjoint interiors obtained by cutting the surface $S$ along the curves $\{C_1, \dots , C_n\}$;
		\item For each $1 \le j \le  d$, there exists $N\ge 1$ such that the map $h^N|_{\Sigma_j}$ is a homeomorphism that is either periodic or pseudo-Anosov.
	\end{enumerate}
\end{theorem} 

In \cite{llibremackay}, Llibre and MacKay prove that if $S=\T^2$ and if $f\in \Homeo_0(S)$ has a rotation set with nonempty interior, then $f$ is isotopic to a pseudo-Anosov relative to a finite set. This result was later generalized to other contexts, with the same proof strategy, in \cite{pollicott, MR1334719, matsumoto, boyland2}. Let us state the first of these results.

\begin{theorem}[\cite{pollicott}, Theorem 2]
	Let $S$ be a compact closed surface of genus $g \ge 2$, and $f\in \Homeo_0(S)$. Assume that there exist $2g+ 1$ periodic points $x_1,\dots,x_{2g+1}$
	whose rotation vectors $\rho_1,\dots,\rho_{2g+1}\in H_1(S,\R)$ do not lie on a hyperplane, then $f$ is isotopic to some $h\in \Homeo_0(S)$ having an invariant set $S_1\subset S$ that is a closed surface with boundary, such that $h|_{S_1}$ is pseudo-Anosov relative to a finite set. 
\end{theorem}

We will adapt these proofs in order to deal with second item of the proof of Theorem~\ref{thm:DecompRotSet}.

\subsection{The proof}

Recall that the partition of $\M$ into equivalence classes $\cl_i$ was defined at the beginning of Section~\ref{SecIntroRot}.

\begin{proof}[Proof of Theorem~\ref{thm:DecompRotSet}]
Let us study the classes $\cl_i$. We start by defining the two complementary sets
\begin{align*}
I^1 & = \big\{i \in I \mid\forall  \mu_1,\mu_2\in \cl_i, \ \Lambda_{\mu_1} = \Lambda_{\mu_2}\big\},\\
I^+ & = \big\{i \in I \mid \exists\mu_1,\mu_2\in \cl_i \text{ that are dynamically transverse}\big\}.
\end{align*}

Fix $i\in I$. Let us now deal with each of the four items of the statement. 

\medskip

\noindent \textbf{Item 1.} Suppose that $i\in I^1$ and let us prove the two required properties.

\smallskip

\noindent \emph{--- $\Lambda_i$ is a minimal geodesic lamination.} 
By Theorem~\ref{TheoInterGeod}, all geodesics of $\Lambda_i$ are simple. Indeed, if they are not, then Theorem~\ref{TheoInterGeod} produces a periodic orbit whose tracking geodesic intersects one tracking geodesic of $\mu$. Hence, the uniform measure on this periodic orbit belongs to $\cl_i$, and hence $\Lambda_i$ is reduced to a single closed geodesic. By hypothesis, this geodesic is not simple. But by \cite[Theorem E]{pa}, this geodesic crosses the tracking geodesic of another $f$-periodic point, which is a contradiction. Incidentally, the same reasoning also proves that $I^1\cap I^+=\varnothing$.

 By Theorem~\ref{thm:ifsimpletheoremintro}, the set \(\Lambda_i \) is a minimal geodesic lamination and
\[\overline{\gamma_x(\R)} = \Lambda_i\]
for any $\mu\in \cl_i$ and \(\mu\)-a.e.\ \(x \in S\). 

\smallskip
	
\noindent \emph{--- $\rho_{I^1}$ is included in a union of at most $3g-3$ lines.} 
Note that the union of the minimal laminations $\{\Lambda_i\}_{i \in I^1}$ is itself a (possibly non-minimal) lamination $\Lambda^1$. By Proposition \ref{prop:ErgMeasuresConstHomRotation}, it is enough to bound the amount of homological rotation vectors associated to geodesics in this lamination. By Birkhoff's ergodic theorem, this is equivalent to bounding the amount of ergodic transverse measures for the lamination. By \cite{zbMATH06272204}, this number is bounded by $3g-3$. 

%
%
	

\bigskip

\noindent \textbf{Item 2.} Suppose that there exist two measures of $\cl_i$ that are dynamically transverse. 

\smallskip

\noindent \emph{--- If $I^+ \neq \varnothing$, then $f$ has a topological horseshoe.}
\smallskip 

This is a direct consequence of  Theorem~\ref{TheoInterGeod}. 

\smallskip

\noindent \emph{--- $\overline{\rho_i}$ is a convex subset of $V_i$ with nonempty interior, containing 0.} 

\begin{lemma}\label{LemCombineClass}
Let $\mu_1,\mu_2\in \cl_i$ and suppose that $\Lambda_{\mu_1} \neq \Lambda_{\mu_2}$. Then, any point of the convex hull of $0,\rho_{\mu_1}, \rho_{\mu_2}$ is accumulated by rotation vectors of periodic measures that belong to $\cl_i$. 
\end{lemma}

\begin{proof}
There exist $\tau_1,\dots,\tau_m\in\M$ such that $\tau_1=\mu_1$, $\tau_m=\mu_2$ and for all $1\le i<m$, the measures $\tau_i$ and $\tau_{i+1}$ are dynamically transverse.

Let $\lambda_1,\lambda_2\in [0,1]$ such that $\lambda_1+\lambda_2\le 1$, and $\varep>0$. We want to show that the ball of center $\lambda_1\rho_{H_1}(\mu_1) + \lambda_2\rho_{H_1}(\mu_2)$ and radius $\varep$ intersects $\rho_i = \{\rho(\mu)\mid\mu\in\cl_i\}$.
By Theorem~\ref{TheoInterGeod}, for $\tau_2$-almost any $x_2$, there exists an $f$-periodic point $z_2$, such that $\rho_{H_1}(z_2) \in B(\rho_{H_1}(\tau_1),\varep)$ and that one lift of the tracking geodesic for $z_2$ is close to some tracking geodesic $\tgamma_{\tx_2}$. As by hypothesis $\tau_3$-almost any $x_3$ has a lift such that $\tgamma_{\tx_2}\cap \tgamma_{\tx_3}\neq\varnothing$, the tracking geodesic for $z_2$ intersects the tracking geodesic $\tgamma_{\tx_3}$.

Applying now Theorem~\ref{TheoInterGeod} to the uniform measure on the periodic orbit of $z_2$ and the measure $\tau_3$, for $\tau_3$-almost any $x_3$, one obtains an $f$-periodic point $z_3$ such that $\rho_{H_1}(z_3)\in B(\rho_{H_1}(\tau_1),2\varep)$ and that the tracking geodesic for $z_3$ is close to some tracking geodesic $\tgamma_{\tx_3}$.

Iterating this process, for $\tau_{m-1}$-almost any $x_{m-1}$, one obtains an $f$-periodic point $z_{m-1}$ such that $\rho_{H_1}(z_{m-1})\in B(\rho_{H_1}(\tau_1),({m-1})\varep)$ and that the tracking geodesic for $z_{m-1}$ is close to some tracking geodesic $\tgamma_{\tx_{m-1}}$.

Finally, applying Theorem~\ref{TheoInterGeod} to the uniform measure on the periodic orbit of $z_{m-1}$ and the measure $\tau_m$, one obtains a periodic point $z_m$ for $f$, such that $$\rho_{H_1}(z_m)\in B\big(\lambda_1\rho_{H_1}(\tau_1)+\lambda_2\rho_{H_1}(\mu_2),m\varep\big)$$ 
and such that the tracking geodesic for $z_m$ is close to some tracking geodesic $\tgamma_{\tx_m}$ (and in particular the uniform measure on the orbit of $z_m$ belongs to $\cl_i$). Hence, any point of the convex hull of $0,\rho_{\mu_1}, \rho_{\mu_2}$ is accumulated by rotation vectors of periodic measures that belong to $\cl_i$. Note that this convex hull can be trivial: we can very well have $\rho_{H_1}(\mu) = 0$ for every $\mu \in \cl_i$. See \textit{e.g.}\ Figure \ref{fig:HomologicallyTrivialHorseshoe}.
\end{proof}

This shows that the set $\overline{\rho_i}$ is convex, contains 0 and spans a linear vector space having a basis made of elements of $H_1(S,\Z)$. In particular, $\overline{\rho_i}$ has nonempty interior in the vector space it spans.
\medskip

\noindent \emph{--- There is a dense subset of $\overline{\rho_i}$ made of rotation vectors of periodic orbits belonging to $\cl_i$ (and hence $V_i$ is a rational subspace of $H_1(S,\R)$), and $\operatorname{int}_{V_i}(\overline{\rho_i})\subset\rho_i$.}

\smallskip

Let us show that $\inte_{V_i}\overline{\rho_i}$ is contained in $\rho_i$. By the fact that $\rho_i$ is dense in $\conv (\rho_i)$ (this is the above fact about the denseness of rational elements in $\conv (\rho_i)$) this amounts to show that $\inte_{V_i} \conv (\rho_i) \subset\rho_i$. We do this by means of arguments due to Llibre and MacKay \cite{llibremackay}. These arguments are now considered as folklore results (for instance, Boyland attributes them to a conversation with Franks \cite[Theorem 11.9]{boyland2}).

We start by adapting the proof of \cite[Theorem 2]{pollicott} to get the following:

\begin{claim}\label{ClaimpA}
Suppose that $F$ is a finite $f^q$-invariant subset of $S$, made of fixed points of $f^q$ belonging to a single class $\cl_i$, with $\card F\ge 2$. Suppose also that for $x,y\in F$, $x\neq y$, the geodesics $\gamma_x$ and $\gamma_y$ are distinct and intersect transversally.
Then, $f^q|_{S\setminus F}$ is isotopic to a homeomorphism having a pseudo-Anosov component on some sub-surface $\Sigma_{j_0}$ such that $F\subset\overline{\Sigma_{j_0}}$.
\end{claim}

\begin{proof}
Set $\Sigma = S\setminus F$ and apply Nielsen-Thurston classification (Theorem~\ref{TheoNTClass}) to $f^q|_\Sigma$ to get $h\in\Homeo(\Sigma)$. 
Consider one simple closed curve $C$ appearing in the decomposition \textit{(i)} of Theorem~\ref{TheoNTClass}. In the proof of \cite[Theorem 2]{pollicott}, Pollicott shows that $C$ cannot be contractible in $S$. 
Indeed if it were, then it would bound a disk in $S$. This disk cannot be contractible in $\Sigma$ nor contain a single point of $F$ by point \textit{(ii)} of Theorem~\ref{TheoNTClass}. 
It cannot contain two different points of $F$ either, as these points have different tracking geodesics and hence the lifts of their orbits do not stay at finite distance under the action of $f^q$ on pair of points.

Let $\Sigma_{j_0}$ be one component of the decomposition of Theorem~\ref{TheoNTClass} such that there exists $x\in F\cap \overline{\Sigma_{j_0}}$. 
The homeomorphism $h$ extends naturally to a homeomorphism of $S$ homotopic to the identity, having the points of $F$ as fixed points and with the same rotational behaviour as under $f$. 
The set $\Sigma_{j_0}$ lifts to some set $\widetilde\Sigma_{j_0}$ of $\tilde S$, whose boundary is made of lifts of $F$ together with lifts $\tilde C_i$ of some of the curves $C_i$. Each curve $\tilde C_i$ stays at a finite distance from some geodesic of $\tilde S$. 
The (extension to $S$ of the) homeomorphism $h$ lifts canonically to a homeomorphism $\tilde h$ of $\tilde S$; $\tilde h$ leaves invariant each of the $\tilde C_i$, the set $\widetilde\Sigma_{j_0}$ and hence each of the connected components of $\tilde S\setminus \widetilde\Sigma_{j_0}$.
Recall that $x\in F\cap \overline{\Sigma_{j_0}}$ and suppose that there exists $y\in F\setminus \overline{\Sigma_{j_0}}$. Let us take $\tilde x \in \widetilde\Sigma_{j_0}$ a lift of $x$ and $\tilde y$ a lift of $y$ such that $\tgamma_{\tilde y}$ and $\tgamma_\tx$ intersect transversally (this exists by hypothesis). 
Both $\tgamma_{\tilde y}$ and $\tgamma_{\tilde x}$ cannot cross any of the geodesics that stay at a finite distance of some of the $\tilde C_i$ because $\widetilde\Sigma_{j_0}$ and the connected components of the complement of $\widetilde\Sigma_{j_0}$ are $\tilde h$-invariant. In other words, both endpoints of $\tgamma_{\tilde x}$ are points of $\overline{\widetilde\Sigma_{j_0}}\cap\partial \tilde S$, and both endpoints of $\tgamma_{\tilde y}$ are points of $\overline{\tilde S \setminus\widetilde\Sigma_{j_0}}\cap\partial \tilde S$.
This contradicts the fact that $\tgamma_{\tilde y}$ and $\tgamma_\tx$ intersect transversally. 
This implies that $F\subset \overline\Sigma_i$.

Finally, note that $h|_{\Sigma_i}$ cannot be periodic, as $\overline{\Sigma_i}$ contains two different points with different rotational behaviour. This shows, by Theorem~\ref{TheoNTClass}, that $h|_{\Sigma_i}$ is pseudo-Anosov.
\end{proof}

It then suffices to apply the arguments of \cite{zbMATH00009916} in the higher genus case. These arguments are based on Nielsen-Thurston theory, and in particular Handel's semi-conjugation result \cite{zbMATH03921585} for pseudo-Anosov maps, that can be replaced here by Boyland's result \cite{zbMATH01408389} for pseudo-Anosov components\footnote{Alternatively, Militon signalled to us that it is possible to consider the cover of the surface $S$ whose $\pi_1$ is the subgroup of $\pi_1(S)$ generated by the closed geodesics associated to the periodic orbits of $F$. In this cover, the lift of $S\setminus F$ is endowed with the lift of $f|_{S\setminus F}$ which is homotopic to a pseudo-Anosov map, and one can apply Handel's result on this surface.}.

We then get the following.

\begin{lemma}\label{LemNonemInte}
If $f\in\Homeo_0(S)$ has a pseudo-Anosov component relative to a finite subset $F$ of $S$, then any element of $\inte(\conv(\rho(f|_F)))$ is realized as the rotation vector of an $f$-ergodic measure.
\end{lemma} 

With a similar proof, one gets the following consequence: for any finite set $F\subset \inte_{V_i}(\rho_i)$, denoting $R = \conv(\{0\}, F)$, there exists a constant $M>0$ such that for any $p/q\in R\cap H_1(S,\Q)$, with $p\in H_1(S,\Z)$ and $q\in\N^*$, there exists a periodic point of period dividing $Mq$ and with rotation vector $p/q$ (see also \cite[Theorem 4]{matsumoto}). 
\medskip

Let us now finish the analysis for Item 2. Consider one class $\cl_i$ with $i\in I^+$, and $r\in\inte_{V_i} \conv (\rho_i)$. We have already proved that this implies that $r\in\inte_{V_i} \conv (r_1,\dots,r_k)$, with each $r_j\in\rho_i$ that is realised by some $f$-periodic orbit belonging to $\cl_i$.
By repeated use of Theorem~\ref{TheoInterGeod} as in the above discussion, we can approximate each $r_j$ by some $r'_j\in\rho_i$ that is realised by some $f$-periodic point $x_j$ belonging to $\cl_i$, and with the additional property that for any $j\neq j'$ one has that $\gamma_{x_j}$ and $\gamma_{x_{j'}}$ intersect transversally. If the $r'_j$ approximate well enough the $r_j$, then one still has $r\in\inte_{V_i} \conv (r'_1,\dots,r'_k)$.

Take $q$ the lcm of the periods of the $x_j$; then $F = \{x_1,\dots,x_k\}$ is a set of fixed points for $f^q$, such that the tracking geodesics of two of these points intersect transversally. It allows to apply Claim~\ref{ClaimpA}: $f^q|_{S\setminus F}$ has a pseudo-Anosov component on some sub-surface $\Sigma_{j_0}$ such that $F\subset\overline{\Sigma_{j_0}}$. Then, Lemma~\ref{LemNonemInte} implies that $r$ is realised as the rotation vector of an $f$-ergodic measure.

\medskip
\noindent{\textbf{Item 3.}} Given $V_i = \textnormal{span}(\rho_i), \ V_j = \textnormal{span}(\rho_j)$, it suffices to prove that for every $\mu_i \in \mathcal{N}_i, \mu_j \in \mathcal{N}_j$ and defining $v_i = \rho_{H_1}(\mu_i), v_j = \rho_{H_1}(\mu_j)$, we have that $v_i \wedge v_j = 0$. Assume that $v_i \wedge v_j \neq 0$. By Proposition \ref{prop:InterHomoImpliesInterGeod} we obtain that typical tracking geodesics of $\mu_i$ and $\mu_j$ intersect transversally, which by Theorem \ref{thm:equidistributiontheoremintro} implies that $\mu_i$ and $\mu_j$ are dynamically transverse, which then yields that they belong to the same equivalence class. 


Now, let us restrict to the case $i, j \in I^1$. Recall that by Theorem \ref{thm:equidistributiontheoremintro}, we have that $\Lambda_{\mu_i}, \Lambda_{\mu_j}$ are minimal geodesic laminations. Therefore, whenever $\Lambda_{\mu_i} \cap \Lambda_{\mu_j} \neq \varnothing$, we obtain that $\Lambda_{\mu_i} = \Lambda_{\mu_j}$, which then implies that $i = j$. This finishes the study of item 3.  

\medskip

\noindent \textbf{Item 4.} \emph{$\card I^1\le 3g-3$ and $\card I^+\le 2g-2$.}
\smallskip

Let us first bound the number of classes bearing only simple tracking geodesics, \emph{i.e.}\ those from $I^1$.

Consider $n$ pairwise disjoint geodesic laminations $\Lambda_1,\dots,\Lambda_n$ on $S$. Each of these laminations contains a minimal geodesic lamination, so one can suppose that each $\Lambda_i$ is minimal. As these laminations are pairwise disjoint and compact, they are all at positive distance one to the other; let us call $d$ the minimum of these distances. 
	
	For each $i$, consider a geodesic $\gamma_i\subset \Lambda_i$. As $\dot\gamma_i$ is recurrent (in $\mathrm{T}^1S$), there exist times $t_k\to+\infty$ such that $\lim_{k\to+\infty}\dist (\dot\gamma_i(0),\dot\gamma_i(t_k))=0$. By the periodic shadowing lemma, this implies that there exists a closed geodesic $\alpha_i$ staying in the $d/3$-neighbourhood of $\gamma_i$. Hence, we obtain a collection of $n$ pairwise disjoint closed geodesics on $S$: it is a classical result that this implies $n\le 3g-3$.
\medskip
	
Now, let us bound the number of classes in $I^+$.


We may rewrite the set of classes $I = I^1 \sqcup I^+$ as $I = I_{\textnormal{s}}\sqcup I_{\text{m}}\sqcup I^+$, with 
\begin{itemize}
	\item $I_{\textnormal{s}}$ the set of classes reduced to a simple closed geodesic,
	\item $I_{\text{m}}$ is the set of classes associated to a minimal geodesic lamination that is not a simple closed geodesic.
\end{itemize}

Let us define, for any $i\in\I, i\notin I_{\textnormal{s}}$, a surface $S_i\subset S$ associated to the class 
$\cl_i$ of ergodic measures. We will use the vocabulary of \cite[Chapter 4]{casson}.

\begin{itemize}

\item If $i\in I_{\text{m}}$ (\emph{i.e.}~$\Lambda_i$ is a minimal geodesic lamination that is not reduced to a simple closed geodesic), then set $S_i \subset S$ to be equal to the closure of the union of the ideal polygons which are principal regions of $\Lambda_i$, together with the crowns associated to the other principal regions of $\Lambda_i$. In other words, $S_i$ is the closure of the complement of all the cores of the principal regions associated to $\Lambda_i$.
\item If $i\in I^+$, (\emph{i.e.} $\Lambda_i$ is not a geodesic lamination), then  consider the lift $\tilde \Lambda_i$ of the set $\Lambda_i$ to $\tilde S$. Set $\tilde S_i$ as the closure of the convex hull of a connected component $\tilde\Lambda_i^0$ of $\tilde \Lambda_i$ containing a tracking geodesic, and $S_i$ the projection of $\tilde S_i$ on the surface $S$.
\end{itemize}

Note that in the case $i\in I^+$, any two tracking geodesics of $\Lambda_i$ intersect up to taking a chain of geodesics belonging to the lamination; this implies that the definition of $S_i$ does not depend on the tracking geodesic used to define the connected component of $\tilde \Lambda_i$ (and hence, to define $S_i$).  Note also that (because any tracking geodesic is dense in $\Lambda_i$), the set $\tilde\Lambda_i$ is equal to  the closure of the set of all translates of $\tilde\Lambda_i^0$ by deck transformations.

\begin{lemma}
For any $i\in I, i \notin I_{\textnormal{s}}$, the set $S_i$ is a closed surface whose boundary is made of a finite union of simple closed geodesics.
\end{lemma}

\begin{proof}
%
%
%

If $i \in I_m$, by \cite[Lemma 4.4]{casson} there are finitely many cores of the principal regions, and these are either simple closed geodesics or compact connected surfaces with geodesic boundary, and thus are at a positive distance from each other. By construction, we have that $\partial S_i$ is equal to the union of the boundaries of the nontrivial (\emph{i.e.}~different from a simple closed geodesic) cores of the principal regions of $\Lambda_i$, which concludes the proof for this case.
\medskip

Suppose now that $i\in I^+$. Denote $\tilde\Lambda_i^0$ a connected component of $\tilde \Lambda_i$ containing a tracking geodesic.
Note that the boundary of $\tilde S_i$ is made of geodesics. Too see this, observe that the complement of $\tilde S_i$ is made of the union of the convex hulls in $\tilde S$ of the open intervals contained in $\partial\tilde S\setminus \overline{\tilde\Lambda_i^0}$ (the proof of the equality between these two sets uses the fact that in dimension 2, the convex hull of a connected set is equal to the union of segments having both their endpoints in the set).

Let us prove that the projection of $\partial \tilde S_i$ on $S$ is a geodesic lamination. Let $(\alpha_1,\beta_1)$ and $(\alpha_2,\beta_2)$ be two geodesics of respectively $\partial \tilde S_i$ and $T\partial \tilde S_i$, with $T$ a deck transformation of $\tilde S\to S$, and suppose that $(\alpha_1,\beta_1)$ and $(\alpha_2,\beta_2)$ intersect transversally. Then there exists a path $\gamma_1\subset\tilde\Lambda_1^0$ having $\alpha_1$ and $\beta_1$ in its closure, and a path $\gamma_2\subset T\tilde\Lambda_1^0$ having $\alpha_2$ and $\beta_2$ in its closure. The transverse intersection between $(\alpha_1,\beta_1)$ and $(\alpha_2,\beta_2)$ forces $\gamma_1$ and $\gamma_2$ to intersect, hence $\tilde\Lambda_1^0\cap T\tilde\Lambda_1^0 \neq\emptyset$ and so $\tilde\Lambda_1^0= T\tilde\Lambda_1^0$. This implies that $(\alpha_1,\beta_1)$ and $(\alpha_2,\beta_2)$ are two boundary components of  $\tilde S_i$ that intersect transversally, a contradiction.

Let us prove that there is no point of $\partial\tilde S$ that is the endpoint of two different geodesics $(\alpha,\beta_1)$ and $(\alpha,\beta_2)$ of $\partial \tilde S_i$. Indeed, if it was the case, then $\alpha$ would also be the endpoint of a tracking geodesic $\tilde\gamma$ of $\tilde\Lambda_i^0$ (because $\alpha$ is accumulated by tracking geodesics of $\tilde\Lambda_i^0$ and these tracking geodesics have to lie between both boundary geodesics), and we get a contradiction by using the fact that the tracking geodesic $\tilde\gamma$ crosses other tracking geodesics of $\cl_i$ arbitrarily close to $\alpha$, and with angle bounded away from zero\footnote{This comes from the fact that a tracking geodesic is typical for the measure $\nu_\mu$, see Theorem~\ref{thm:equidistributiontheoremintro}.}: these new tracking geodesics would eventually cross the two geodesics $(\alpha,\beta_1)$ and $(\alpha,\beta_2)$ of $\partial \tilde S_i$.

Hence, $\partial\tilde S_i$ is a geodesic lamination such that no point of $\partial\tilde S$ is the endpoint of two different geodesics of $\partial \tilde S_i$. By \cite[Lemma 4.4]{casson}, we deduce that $\partial S_i$ is a finite union of pairwise disjoint closed geodesics.
\end{proof}


\begin{lemma}
For any $i,j\in I, i,j \notin I_{\textnormal{s}}$, $i\neq j$, the interiors of $S_i$ and $S_j$ are disjoint.
\end{lemma}

\begin{proof}
Suppose first that $i\in I^+$. Denote $\tilde \Lambda_i$ a lift of $\Lambda_i$ to $\tilde S$, and $\tilde\Lambda_i^0$ a connected component of $\tilde \Lambda_i$ containing a tracking geodesic. Suppose that there is a tracking geodesic $\tilde\gamma$ of $\tilde\Lambda_j$ intersecting the interior of $\tilde S_i$. Then both connected components of the complement of $\tilde\gamma$ in $\tilde S$ intersect $\tilde\Lambda_i^0$; by connectedness of $\tilde\Lambda_i^0$ this means that $\tilde\gamma$ meets $\tilde\Lambda_i^0$, a contradiction. 

If $j\in I^+$, this proves directly that the interiors of $S_i$ and $S_j$ are disjoint: in this case, any choice of connected component $\tilde\Lambda_j^0$ forces $\tilde\Lambda_j^0$ to belong to a complementary region of $\tilde S_i$, and these regions are convex. If not, one has $j\in I_{\text{m}}$. For the interiors of $S_i$ and $S_j$ not to be disjoint, it would force $\tilde\Lambda_i^0$ to be included either in an ideal polygon or a crown associated to $\tilde\Lambda_j$. Both cases can be easily seen to be impossible. 

Finally, suppose that $i,j\in I_{\text{m}}$. Remind that we have proved that $\Lambda_i$ and $\Lambda_j$ stay at positive distance (\emph{i.e.} there exists $\delta>0$ such that if $z_1\in\lambda_1, z_2\in\Lambda_2$, we have $d(z_1,z_2)>\delta$). This prevents a geodesic of $\Lambda_i$ to be included in an ideal polygon or a crown associated to $\Lambda_j$ (and reciprocally), and implies that the surfaces $S_i$ and $S_j$ are disjoint.
\end{proof}

This allows to get the sharp bound. Let $S_0$ be the closure of the complement of all the $S_i$ for $i\in I, i \notin I_{\textnormal{s}}$. Note that $\chi(S_0)\le 0$ (where $\chi$ denotes the Euler characteristic) because it does not contain any component homeomorphic to a disk or a sphere, and that $\chi(S_i)\le -1$ for any $i\notin I_{\mathrm{s}}$. Moreover, given that each of the connected components of the boundary of $S_0$ and every $S_i$ are simple closed geodesic (with Euler characteristic equal to 0), we obtain that
\[2-2g = \chi(S) = \chi(S_0) + \sum_{i\in I\setminus I_{\mathrm{s}}} \chi(S_i) \le -\card(I^+)-\card ( I_{\text{m}}).\]
In particular, $\card(I^+)\le 2g-2$ and $\card ( I_{\text{m}})\le 2g-2$, which concludes the proof.
\end{proof}

\begin{remark}
We could have been a bit more careful in the definition of the surfaces $S_i$, by avoiding taking closures. This would have allowed to take into account the classes of measures whose associated laminations are simple closed geodesics, which would have been disjoint from these "new" surfaces $S_i$.
\end{remark}

\begin{remark}
Note that by construction, if $i\in I_{\text{m}}$, then $\Lambda_i$ is a minimal and filling lamination of $S_i$. In particular, $S_i$ cannot be homeomorphic to a pair of pants, and $\chi(S_i)\ge -1$ iff $S_i$ is homeomorphic to a torus minus a disk. As there are at most $g$ such surfaces included in $S$, this implies that $\card ( I_{\text{m}})\le 3g/2-1$. As we do not know if there is some minimal filling lamination of the 4-punctured sphere associated to a class of measures, we do not know if this bound is sharp.
\end{remark}



\section{Examples and proof of Proposition~\ref{thm:nonconstanttrackingtheorem}\label{sec:examplessection}}

Throughout this section we will explain various examples and in particular we will prove Proposition~\ref{thm:nonconstanttrackingtheorem}. 

\begin{figure}
\begin{center}
\tikzset{every picture/.style={line width=1.2pt}} 

\begin{tikzpicture}[x=0.75pt,y=0.75pt,yscale=-1.8,xscale=1.8]

\draw [color={rgb, 255:red, 0; green, 0; blue, 0 }  ,draw opacity=1 ][fill={rgb, 255:red, 155; green, 155; blue, 155 }  ,fill opacity=0.25 ]   (300,150) .. controls (279.8,149.4) and (269.4,159.4) .. (250,160) .. controls (202.2,161) and (201.4,90.6) .. (250,90) .. controls (262.02,89.89) and (280.97,100.03) .. (300,100) .. controls (329.8,99) and (339.8,90.2) .. (350,90) .. controls (397.85,90.57) and (401,161) .. (350,160) .. controls (340.2,160.2) and (321.4,150.2) .. (300,150) -- cycle ;
\draw [draw opacity=0][fill={rgb, 255:red, 255; green, 255; blue, 255 }  ,fill opacity=1 ]   (245.18,126.71) .. controls (252.92,118.04) and (260.58,120.63) .. (265.09,126.71) .. controls (258.75,130.96) and (251.42,131.29) .. (245.18,126.71) -- cycle ;
\draw    (240,122.66) .. controls (249.86,132.59) and (260.29,132.3) .. (270,122.66) ;
\draw    (245.18,126.71) .. controls (252.04,118.78) and (259.95,119.78) .. (265.09,126.71) ;

\draw [draw opacity=0][fill={rgb, 255:red, 255; green, 255; blue, 255 }  ,fill opacity=1 ]   (335.18,126.81) .. controls (342.92,118.14) and (350.58,120.73) .. (355.09,126.81) .. controls (348.75,131.06) and (341.42,131.39) .. (335.18,126.81) -- cycle ;
\draw    (330,122.76) .. controls (339.86,132.69) and (350.29,132.4) .. (360,122.76) ;
\draw    (335.18,126.81) .. controls (342.04,118.88) and (349.95,119.88) .. (355.09,126.81) ;

\draw  [color={rgb, 255:red, 155; green, 155; blue, 155 }  ,draw opacity=0.69 ][line width=10]  (230.2,125.3) .. controls (230,148.9) and (269.4,154.5) .. (300,125) .. controls (330.6,95.5) and (370,100.4) .. (370,125) .. controls (370,149.6) and (329,154.1) .. (300,125) .. controls (271,95.9) and (230.4,101.7) .. (230.2,125.3) -- cycle ;
\draw  [color={rgb, 255:red, 208; green, 2; blue, 27 }  ,draw opacity=1 ] (230.2,125.3) .. controls (230,148.9) and (269.4,154.5) .. (300,125) .. controls (330.6,95.5) and (370,100.4) .. (370,125) .. controls (370,149.6) and (329,154.1) .. (300,125) .. controls (271,95.9) and (230.4,101.7) .. (230.2,125.3) -- cycle ;
\draw [color={rgb, 255:red, 208; green, 2; blue, 27 }  ,draw opacity=1 ]   (233.92,114.29) .. controls (231.04,118.35) and (231.34,118.67) .. (230.82,121.4) ;
\draw [shift={(230.13,124.31)}, rotate = 285.78] [fill={rgb, 255:red, 208; green, 2; blue, 27 }  ,fill opacity=1 ][line width=0.08]  [draw opacity=0] (8.04,-3.86) -- (0,0) -- (8.04,3.86) -- (5.34,0) -- cycle    ;
\draw [color={rgb, 255:red, 208; green, 2; blue, 27 }  ,draw opacity=1 ]   (367.67,115.54) .. controls (369.15,118.26) and (369.29,119.67) .. (369.58,122.04) ;
\draw [shift={(370,125)}, rotate = 260.19] [fill={rgb, 255:red, 208; green, 2; blue, 27 }  ,fill opacity=1 ][line width=0.08]  [draw opacity=0] (8.04,-3.86) -- (0,0) -- (8.04,3.86) -- (5.34,0) -- cycle    ;
\draw [color={rgb, 255:red, 74; green, 144; blue, 226 }  ,draw opacity=1 ]   (263.27,133.93) .. controls (281.06,127.92) and (284.23,115.27) .. (263.69,115.44) ;
\draw [shift={(261,115.53)}, rotate = 356.75] [fill={rgb, 255:red, 74; green, 144; blue, 226 }  ,fill opacity=1 ][line width=0.08]  [draw opacity=0] (8.04,-3.86) -- (0,0) -- (8.04,3.86) -- (5.34,0) -- cycle    ;
\draw [color={rgb, 255:red, 74; green, 144; blue, 226 }  ,draw opacity=1 ]   (337.67,133.4) .. controls (323.66,129.96) and (316.34,117.77) .. (336.15,115.52) ;
\draw [shift={(339.13,115.27)}, rotate = 176.75] [fill={rgb, 255:red, 74; green, 144; blue, 226 }  ,fill opacity=1 ][line width=0.08]  [draw opacity=0] (8.04,-3.86) -- (0,0) -- (8.04,3.86) -- (5.34,0) -- cycle    ;

\draw (279,122) node [anchor=west] [inner sep=0.75pt]  [color={rgb, 255:red, 74; green, 144; blue, 226 }  ,opacity=1 ]  {$a$};
\draw (322,122) node [anchor=east] [inner sep=0.75pt]  [color={rgb, 255:red, 74; green, 144; blue, 226 }  ,opacity=1 ]  {$c$};

\end{tikzpicture}

\caption{In this example, the homeomorphism $f$ is the time one of a time dependent vector field which is identically zero outside the grey neighbourhood of the red curve. The red curve represents a periodic orbit of $f$. One easily sees that the homological rotation set of $f$ must be included in $\langle [a]_{H_1}, [c]_{H_1}\rangle$, which is a totally isotropic subspace of $H_1(S,\R)$ for $\wedge$. However, by \cite[Theorem E]{pa}, the closure $\overline{\rho_{H_1}^{\textnormal{erg}}(f)}$ of the ergodic rotation set has nonempty interior.\label{Fig:Exuncount}}
\end{center}
\end{figure}
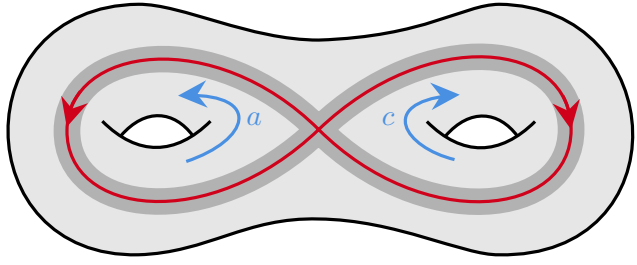

\begin{figure}[ht]
	\centering
	\def\svgscale{0.6}
	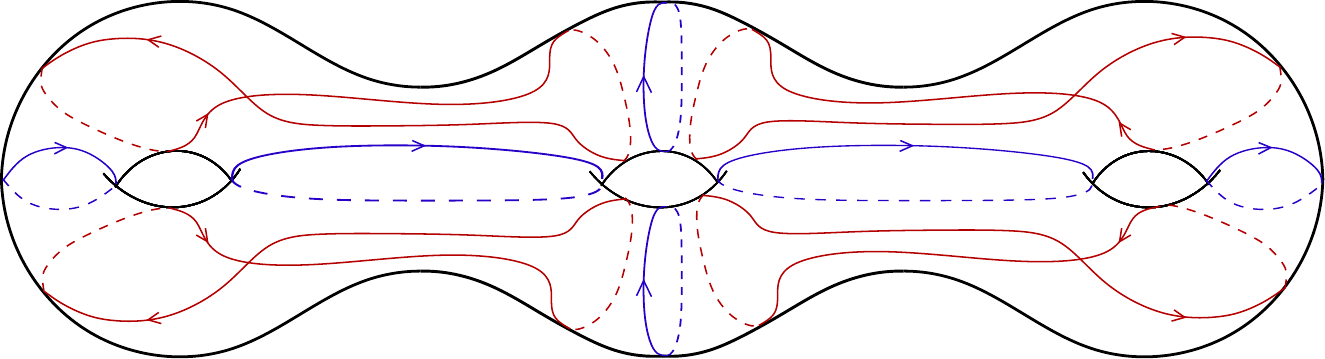
	
	\caption{In this example, the homeomorphism $f$ is supported on sufficiently small neighbourhoods of the ten depicted closed curves, and is built as a time dependent vector field which is identically zero outside these neighbourhoods, as in Figure \ref{Fig:Exuncount}. We then obtain ten classes for $\M$, the blue curves representing the rotation of measures in $I^1$, and the red ones representing the rotation of measures in $I^+$.	
	\label{fig:TenClassesTritorus}} 

\end{figure}

\begin{figure}[ht]
	\centering
	\def\svgscale{0.6}
	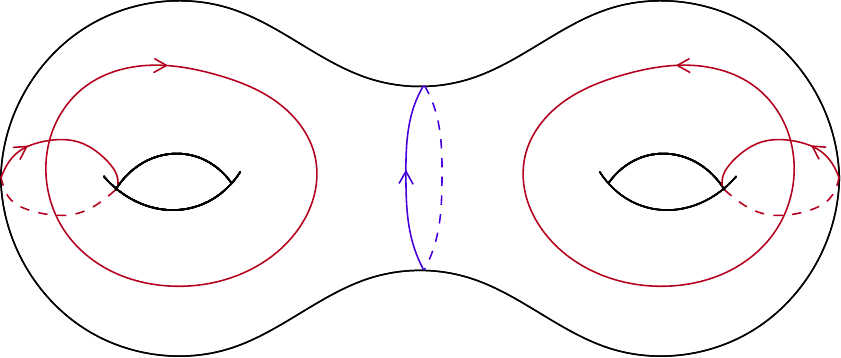
	
	\caption{As in the previous examples of the section, the homeomorphism $f$ is built as the product of point-push homeomorphisms, three for this case. We have two classes in $\M$ which belong to $I^+$, and one class which belongs to $I^1$. Moreover, the closure of the set $\rho_{H_1}^{\textnormal{erg}}(f)$ is the union of two two-dimensional convex sets, each of them generated by a horseshoe represented with the red pairs of transversally intersecting curves. This example is due to Matsumoto \cite[Proposition 3.2]{matsumoto}.\label{fig:CrissCrossBitorus}}   
\end{figure}

\begin{figure}[h]
	\centering
	\def\svgscale{0.6}
	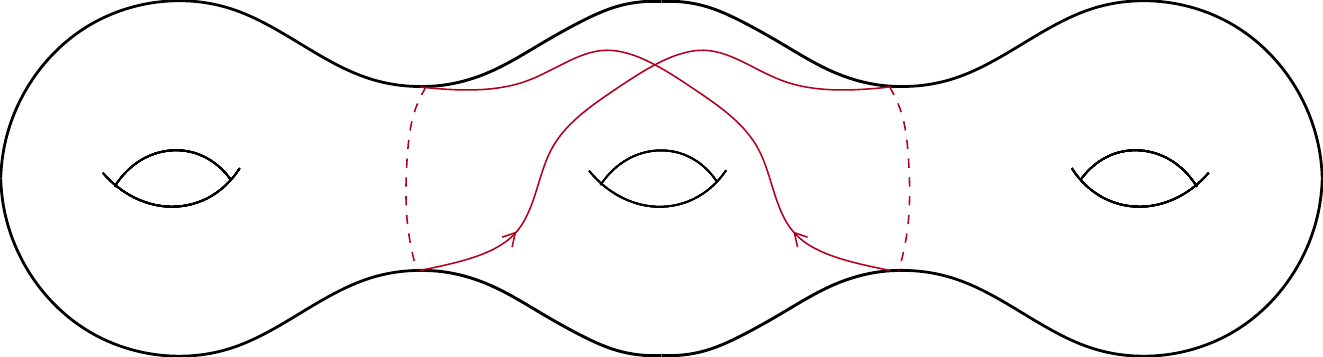
	
	\caption{The red curve is homologically trivial, as it is the concatenation of two boundaries of subsurfaces. As in Figure \ref{Fig:Exuncount}, the homeomorphism $f$ is built time as a time dependent vector field supported in a small neighbourhood of the red curve, which yields on its turn a periodic orbit of $f$ whose trajectory by the isotopy is said red curve. This is called a \emph{point-push}. For this example, the set $\rho_{H_1}^{\textnormal{erg}}(f)$ is trivial.}  
	\label{fig:HomologicallyTrivialHorseshoe}
\end{figure}

\begin{figure}
\begin{center}

\tikzset{every picture/.style={line width=0.75pt}} 

\begin{tikzpicture}[x=0.75pt,y=0.75pt,yscale=-1.25,xscale=1.25]

\draw  [fill={rgb, 255:red, 74; green, 74; blue, 74 }  ,fill opacity=0.1 ][line width=1.5]  (90,100) .. controls (90.37,11.71) and (192.75,60.04) .. (240,60) .. controls (287.25,59.96) and (332.25,59.54) .. (380,60) .. controls (427.75,60.46) and (528.75,11.04) .. (530,100) .. controls (531.25,188.96) and (429.75,139.46) .. (380,140) .. controls (330.25,140.54) and (288.25,139.46) .. (240,140) .. controls (191.75,140.54) and (89.63,188.29) .. (90,100) -- cycle ;
\draw  [fill={rgb, 255:red, 255; green, 255; blue, 255 }  ,fill opacity=1 ][line width=1.5]  (169.89,105.85) .. controls (158.41,109.97) and (146.62,112.39) .. (130.3,105.66) .. controls (140.28,89.13) and (160.17,89.03) .. (169.89,105.85) -- cycle ;
\draw [line width=1.5]    (120,99.08) .. controls (135.08,113.81) and (166.08,113.47) .. (180,99.08) ;

\draw  [fill={rgb, 255:red, 255; green, 255; blue, 255 }  ,fill opacity=1 ][line width=1.5]  (489.89,105.85) .. controls (478.41,109.97) and (466.62,112.39) .. (450.3,105.66) .. controls (460.28,89.13) and (480.17,89.03) .. (489.89,105.85) -- cycle ;
\draw [line width=1.5]    (440,99.08) .. controls (455.08,113.81) and (486.08,113.47) .. (500,99.08) ;

\draw [color={rgb, 255:red, 189; green, 16; blue, 224 }  ,draw opacity=1 ]   (240,60) .. controls (230.28,60.08) and (229.61,140.08) .. (240,140) ;
\draw [color={rgb, 255:red, 189; green, 16; blue, 224 }  ,draw opacity=1 ] [dash pattern={on 0.84pt off 2.51pt}]  (240,60) .. controls (251.61,60.08) and (250.94,139.75) .. (240,140) ;
\draw [color={rgb, 255:red, 189; green, 16; blue, 224 }  ,draw opacity=1 ]   (377.54,60) .. controls (367.82,60.08) and (367.15,140.08) .. (377.54,140) ;
\draw [color={rgb, 255:red, 189; green, 16; blue, 224 }  ,draw opacity=1 ] [dash pattern={on 0.84pt off 2.51pt}]  (377.54,60) .. controls (389.15,60.08) and (388.49,139.75) .. (377.54,140) ;
\draw  [color={rgb, 255:red, 208; green, 2; blue, 27 }  ,draw opacity=1 ][fill={rgb, 255:red, 208; green, 2; blue, 27 }  ,fill opacity=0.1 ] (140,120) .. controls (145.28,122.42) and (151.28,121.75) .. (160,120) .. controls (161.61,126.08) and (161.61,132.42) .. (160,140) .. controls (153.28,142.42) and (147.94,142.75) .. (140,140) .. controls (137.28,132.75) and (138.61,126.42) .. (140,120) -- cycle ;
\draw  [color={rgb, 255:red, 65; green, 117; blue, 5 }  ,draw opacity=1 ][fill={rgb, 255:red, 65; green, 117; blue, 5 }  ,fill opacity=0.1 ] (244.33,113.33) .. controls (252,113.88) and (255.4,114.08) .. (264.33,113.33) .. controls (262.8,118.48) and (262.4,125.08) .. (264.33,133.33) .. controls (257,133.48) and (252.8,133.68) .. (244.33,133.33) .. controls (241.61,126.08) and (242.94,119.75) .. (244.33,113.33) -- cycle ;
\draw [color={rgb, 255:red, 208; green, 2; blue, 27 }  ,draw opacity=1 ][line width=1.5]    (128.28,121.75) .. controls (141.94,128.75) and (197.48,127.03) .. (182.6,95.64) .. controls (167.72,64.25) and (84.06,88.92) .. (106.33,120.75) .. controls (128.61,152.58) and (238.11,126.72) .. (270.61,128.75) ;
\draw [color={rgb, 255:red, 65; green, 117; blue, 5 }  ,draw opacity=1 ][line width=1.5]    (238.94,120.08) .. controls (282.94,118.75) and (262.84,131.18) .. (300.51,128.84) ;
\draw  [color={rgb, 255:red, 65; green, 117; blue, 5 }  ,draw opacity=1 ][fill={rgb, 255:red, 65; green, 117; blue, 5 }  ,fill opacity=0.1 ] (275.93,112.53) .. controls (283.6,113.08) and (287,113.28) .. (295.93,112.53) .. controls (294.4,117.68) and (294,124.28) .. (295.93,132.53) .. controls (288.6,132.68) and (284.4,132.88) .. (275.93,132.53) .. controls (273.21,125.28) and (274.54,118.95) .. (275.93,112.53) -- cycle ;
\draw [color={rgb, 255:red, 65; green, 117; blue, 5 }  ,draw opacity=1 ][line width=1.5]    (270.26,119.28) .. controls (314.26,117.95) and (294.27,130.32) .. (331.94,127.99) ;
\draw  [color={rgb, 255:red, 65; green, 117; blue, 5 }  ,draw opacity=1 ][fill={rgb, 255:red, 65; green, 117; blue, 5 }  ,fill opacity=0.1 ] (307.73,112.33) .. controls (315.4,112.88) and (318.8,113.08) .. (327.73,112.33) .. controls (326.2,117.48) and (325.8,124.08) .. (327.73,132.33) .. controls (320.4,132.48) and (316.2,132.68) .. (307.73,132.33) .. controls (305.01,125.08) and (306.34,118.75) .. (307.73,112.33) -- cycle ;
\draw [color={rgb, 255:red, 65; green, 117; blue, 5 }  ,draw opacity=1 ][line width=1.5]    (302.06,119.08) .. controls (346.06,117.75) and (328.84,128.89) .. (366.51,126.56) ;
\draw  [color={rgb, 255:red, 65; green, 117; blue, 5 }  ,draw opacity=1 ][fill={rgb, 255:red, 65; green, 117; blue, 5 }  ,fill opacity=0.1 ] (344.53,111.73) .. controls (352.2,112.28) and (355.6,112.48) .. (364.53,111.73) .. controls (363,116.88) and (362.6,123.48) .. (364.53,131.73) .. controls (357.2,131.88) and (353,132.08) .. (344.53,131.73) .. controls (341.81,124.48) and (343.14,118.15) .. (344.53,111.73) -- cycle ;
\draw  [color={rgb, 255:red, 0; green, 97; blue, 212 }  ,draw opacity=1 ][fill={rgb, 255:red, 0; green, 97; blue, 212 }  ,fill opacity=0.1 ] (449.17,119.6) .. controls (454.44,122.02) and (460.44,121.35) .. (469.17,119.6) .. controls (470.78,125.68) and (470.78,132.02) .. (469.17,139.6) .. controls (462.44,142.02) and (457.11,142.35) .. (449.17,139.6) .. controls (446.44,132.35) and (447.78,126.02) .. (449.17,119.6) -- cycle ;
\draw [color={rgb, 255:red, 65; green, 117; blue, 5 }  ,draw opacity=1 ][line width=1.5]    (339.14,118.48) .. controls (383.14,117.15) and (435.23,139.82) .. (472.9,137.48) ;
\draw  [color={rgb, 255:red, 121; green, 141; blue, 0 }  ,draw opacity=1 ][fill={rgb, 255:red, 121; green, 141; blue, 0 }  ,fill opacity=0.1 ] (244.33,70.93) .. controls (252,71.48) and (255.4,71.68) .. (264.33,70.93) .. controls (262.8,76.08) and (262.4,82.68) .. (264.33,90.93) .. controls (257,91.08) and (252.8,91.28) .. (244.33,90.93) .. controls (241.61,83.68) and (242.94,77.35) .. (244.33,70.93) -- cycle ;
\draw [color={rgb, 255:red, 121; green, 141; blue, 0 }  ,draw opacity=1 ][line width=1.5]    (238.94,77.68) .. controls (282.94,76.35) and (263.41,88.32) .. (301.08,85.99) ;
\draw  [color={rgb, 255:red, 121; green, 141; blue, 0 }  ,draw opacity=1 ][fill={rgb, 255:red, 121; green, 141; blue, 0 }  ,fill opacity=0.1 ] (275.93,70.13) .. controls (283.6,70.68) and (287,70.88) .. (295.93,70.13) .. controls (294.4,75.28) and (294,81.88) .. (295.93,90.13) .. controls (288.6,90.28) and (284.4,90.48) .. (275.93,90.13) .. controls (273.21,82.88) and (274.54,76.55) .. (275.93,70.13) -- cycle ;
\draw [color={rgb, 255:red, 121; green, 141; blue, 0 }  ,draw opacity=1 ][line width=1.5]    (270.54,76.88) .. controls (314.54,75.55) and (296.27,88.32) .. (333.94,85.99) ;
\draw  [color={rgb, 255:red, 121; green, 141; blue, 0 }  ,draw opacity=1 ][fill={rgb, 255:red, 121; green, 141; blue, 0 }  ,fill opacity=0.1 ] (307.73,69.93) .. controls (315.4,70.48) and (318.8,70.68) .. (327.73,69.93) .. controls (326.2,75.08) and (325.8,81.68) .. (327.73,89.93) .. controls (320.4,90.08) and (316.2,90.28) .. (307.73,89.93) .. controls (305.01,82.68) and (306.34,76.35) .. (307.73,69.93) -- cycle ;
\draw [color={rgb, 255:red, 121; green, 141; blue, 0 }  ,draw opacity=1 ][line width=1.5]    (302.34,76.68) .. controls (346.34,75.35) and (329.7,88.32) .. (367.37,85.99) ;
\draw  [color={rgb, 255:red, 121; green, 141; blue, 0 }  ,draw opacity=1 ][fill={rgb, 255:red, 121; green, 141; blue, 0 }  ,fill opacity=0.1 ] (344.53,69.33) .. controls (352.2,69.88) and (355.6,70.08) .. (364.53,69.33) .. controls (363,74.48) and (362.6,81.08) .. (364.53,89.33) .. controls (357.2,89.48) and (353,89.68) .. (344.53,89.33) .. controls (341.81,82.08) and (343.14,75.75) .. (344.53,69.33) -- cycle ;
\draw [color={rgb, 255:red, 0; green, 97; blue, 212 }  ,draw opacity=1 ][line width=1.5]    (340.07,75.88) .. controls (391.03,74.48) and (398.3,131.82) .. (457.57,132.28) .. controls (516.83,132.75) and (523.85,105.2) .. (507.83,85.75) .. controls (491.81,66.3) and (467.73,68.05) .. (448.33,75.25) .. controls (428.93,82.45) and (385.97,127.08) .. (477.97,126.08) ;
\draw [color={rgb, 255:red, 121; green, 141; blue, 0 }  ,draw opacity=1 ][line width=1.5]    (131.77,129.88) .. controls (283.37,131.28) and (171.37,83.28) .. (268.57,86.68) ;
\draw  [draw opacity=0][fill={rgb, 255:red, 0; green, 0; blue, 0 }  ,fill opacity=1 ] (148.19,125.17) .. controls (148.19,124.23) and (148.95,123.48) .. (149.89,123.48) .. controls (150.82,123.48) and (151.58,124.23) .. (151.58,125.17) .. controls (151.58,126.11) and (150.82,126.86) .. (149.89,126.86) .. controls (148.95,126.86) and (148.19,126.11) .. (148.19,125.17) -- cycle ;
\draw  [draw opacity=0][fill={rgb, 255:red, 0; green, 0; blue, 0 }  ,fill opacity=1 ] (249.91,119.93) .. controls (249.91,118.99) and (250.67,118.24) .. (251.6,118.24) .. controls (252.54,118.24) and (253.3,118.99) .. (253.3,119.93) .. controls (253.3,120.87) and (252.54,121.63) .. (251.6,121.63) .. controls (250.67,121.63) and (249.91,120.87) .. (249.91,119.93) -- cycle ;
\draw  [draw opacity=0][fill={rgb, 255:red, 0; green, 0; blue, 0 }  ,fill opacity=1 ] (282.28,119.56) .. controls (282.28,118.62) and (283.04,117.86) .. (283.98,117.86) .. controls (284.91,117.86) and (285.67,118.62) .. (285.67,119.56) .. controls (285.67,120.49) and (284.91,121.25) .. (283.98,121.25) .. controls (283.04,121.25) and (282.28,120.49) .. (282.28,119.56) -- cycle ;
\draw  [draw opacity=0][fill={rgb, 255:red, 0; green, 0; blue, 0 }  ,fill opacity=1 ] (314.53,119.18) .. controls (314.53,118.24) and (315.29,117.49) .. (316.23,117.49) .. controls (317.16,117.49) and (317.92,118.24) .. (317.92,119.18) .. controls (317.92,120.12) and (317.16,120.88) .. (316.23,120.88) .. controls (315.29,120.88) and (314.53,120.12) .. (314.53,119.18) -- cycle ;
\draw  [draw opacity=0][fill={rgb, 255:red, 0; green, 0; blue, 0 }  ,fill opacity=1 ] (352.05,119.06) .. controls (352.05,118.12) and (352.81,117.36) .. (353.75,117.36) .. controls (354.68,117.36) and (355.44,118.12) .. (355.44,119.06) .. controls (355.44,119.99) and (354.68,120.75) .. (353.75,120.75) .. controls (352.81,120.75) and (352.05,119.99) .. (352.05,119.06) -- cycle ;
\draw  [draw opacity=0][fill={rgb, 255:red, 0; green, 0; blue, 0 }  ,fill opacity=1 ] (457.57,132.28) .. controls (457.57,131.35) and (458.33,130.59) .. (459.26,130.59) .. controls (460.2,130.59) and (460.96,131.35) .. (460.96,132.28) .. controls (460.96,133.22) and (460.2,133.98) .. (459.26,133.98) .. controls (458.33,133.98) and (457.57,133.22) .. (457.57,132.28) -- cycle ;
\draw  [draw opacity=0][fill={rgb, 255:red, 0; green, 0; blue, 0 }  ,fill opacity=1 ] (351.65,86.11) .. controls (351.65,85.17) and (352.41,84.41) .. (353.35,84.41) .. controls (354.28,84.41) and (355.04,85.17) .. (355.04,86.11) .. controls (355.04,87.04) and (354.28,87.8) .. (353.35,87.8) .. controls (352.41,87.8) and (351.65,87.04) .. (351.65,86.11) -- cycle ;
\draw  [draw opacity=0][fill={rgb, 255:red, 0; green, 0; blue, 0 }  ,fill opacity=1 ] (281.1,85.57) .. controls (281.1,84.64) and (281.86,83.88) .. (282.79,83.88) .. controls (283.73,83.88) and (284.49,84.64) .. (284.49,85.57) .. controls (284.49,86.51) and (283.73,87.27) .. (282.79,87.27) .. controls (281.86,87.27) and (281.1,86.51) .. (281.1,85.57) -- cycle ;
\draw  [draw opacity=0][fill={rgb, 255:red, 0; green, 0; blue, 0 }  ,fill opacity=1 ] (315.44,85.81) .. controls (315.44,84.87) and (316.2,84.11) .. (317.13,84.11) .. controls (318.07,84.11) and (318.83,84.87) .. (318.83,85.81) .. controls (318.83,86.74) and (318.07,87.5) .. (317.13,87.5) .. controls (316.2,87.5) and (315.44,86.74) .. (315.44,85.81) -- cycle ;
\draw  [draw opacity=0][fill={rgb, 255:red, 0; green, 0; blue, 0 }  ,fill opacity=1 ] (252.18,86.68) .. controls (252.18,85.74) and (252.94,84.99) .. (253.87,84.99) .. controls (254.81,84.99) and (255.57,85.74) .. (255.57,86.68) .. controls (255.57,87.62) and (254.81,88.38) .. (253.87,88.38) .. controls (252.94,88.38) and (252.18,87.62) .. (252.18,86.68) -- cycle ;
\draw    (190.51,107.42) .. controls (218.82,82.81) and (181.19,57.13) .. (141.6,67.2) ;
\draw [shift={(139.18,67.86)}, rotate = 343.76] [fill={rgb, 255:red, 0; green, 0; blue, 0 }  ][line width=0.08]  [draw opacity=0] (10.72,-5.15) -- (0,0) -- (10.72,5.15) -- (7.12,0) -- cycle    ;
\draw    (517.9,95.84) .. controls (514.85,74.94) and (491.75,54.24) .. (452.77,64.51) ;
\draw [shift={(450.36,65.18)}, rotate = 343.76] [fill={rgb, 255:red, 0; green, 0; blue, 0 }  ][line width=0.08]  [draw opacity=0] (10.72,-5.15) -- (0,0) -- (10.72,5.15) -- (7.12,0) -- cycle    ;

\draw (145.17,139) node [anchor=north east] [inner sep=0.75pt]  [color={rgb, 255:red, 208; green, 2; blue, 27 }  ,opacity=1 ]  {$R_{L}$};
\draw (473.83,141.72) node [anchor=west] [inner sep=0.75pt]  [color={rgb, 255:red, 0; green, 97; blue, 212 }  ,opacity=1 ]  {$R_{R}$};
\draw (230.8,71.62) node [anchor=east] [inner sep=0.75pt]  [color={rgb, 255:red, 189; green, 16; blue, 224 }  ,opacity=1 ]  {$\gamma _{L}$};
\draw (384.06,70) node [anchor=west] [inner sep=0.75pt]  [color={rgb, 255:red, 189; green, 16; blue, 224 }  ,opacity=1 ]  {$\gamma _{R}$};
\draw (186.26,70.5) node [anchor=south west] [inner sep=0.75pt]    {$a$};
\draw (492.73,64.5) node [anchor=south west] [inner sep=0.75pt]    {$c$};
\draw (139.17,81.46) node [anchor=south east] [inner sep=0.75pt]  [color={rgb, 255:red, 208; green, 2; blue, 27 }  ,opacity=1 ]  {$f( R_{L})$};
\draw (448.38,75.27) node [anchor=south east] [inner sep=0.75pt]  [color={rgb, 255:red, 0; green, 97; blue, 212 }  ,opacity=1 ]  {$f( R_{R})$};

\end{tikzpicture}

\caption{The images of the two rectangles $R_L$ and $R_R$ are depicted with thick lines, all the intersections between rectangles are markovian. The black dots are contractible fixed points of $f$. One can build the homeomorphism $f$ such that the intersections $f^n(\gamma_L)\cap\gamma_R$ are empty for any $|n|\le n_0$ (on this example, $n_0=5$). Hence, any periodic point realizing any nonzero homology vector collinear to $[ac]_{H_1}$ must have period bigger than $2n_0$. However, one can see by studying the discrete dynamics associated to the Markovian intersections that for any rational number $p/q\in [0,1)$, the homology vector $p/q\,[ac]_{H_1}$ is realized by some periodic orbit. In particular, $1/2\,[ac]_{H_1}$ is realized by some periodic orbit but by no periodic orbit of period 2: Franks' ``exactness'' property for periods in the torus case \cite{MR0958891} does not hold on higher genus.
\label{FigExNotExact}}
\end{center}
\end{figure}
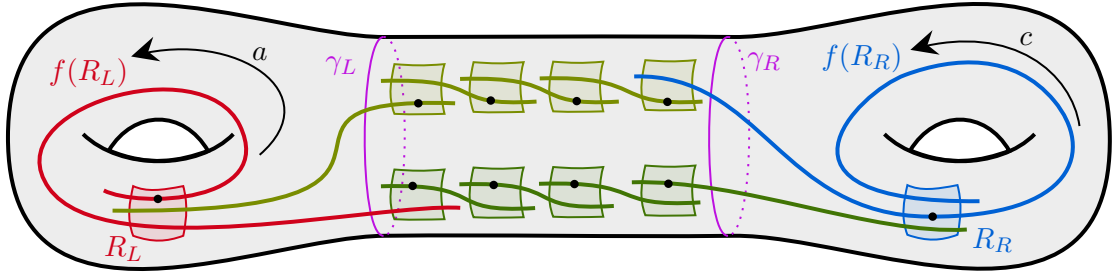

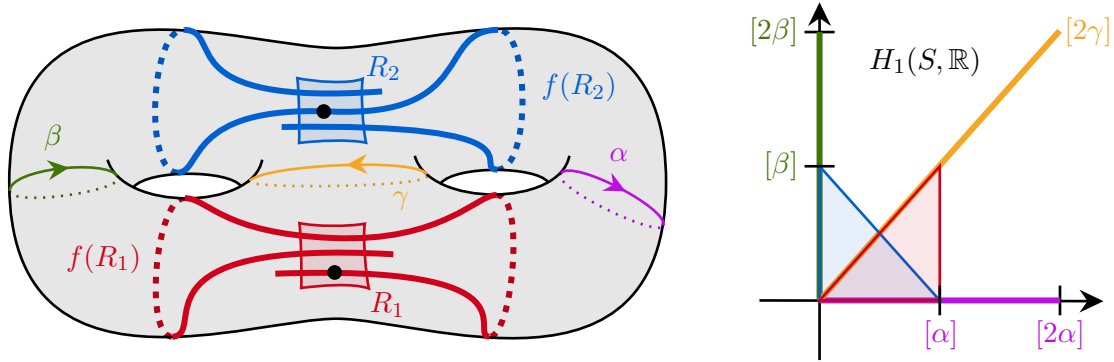
\begin{figure}
\begin{center}

\tikzset{every picture/.style={line width=1pt}} 

\begin{tikzpicture}[x=0.75pt,y=0.75pt,yscale=-1.5,xscale=1.5]

\draw [color={rgb, 255:red, 208; green, 2; blue, 27 }  ,draw opacity=1 ][line width=2.25]  [dash pattern={on 2.53pt off 3.02pt}]  (324.46,163.7) .. controls (334.39,163.09) and (329.8,211.58) .. (322.09,211.18) ;
\draw [color={rgb, 255:red, 0; green, 96; blue, 209 }  ,draw opacity=1 ][line width=2.25]  [dash pattern={on 2.53pt off 3.02pt}]  (324.82,155.15) .. controls (336.3,154.75) and (335.5,107.86) .. (325.02,108.26) ;
\draw [color={rgb, 255:red, 189; green, 16; blue, 224 }  ,draw opacity=1 ] [dash pattern={on 0.84pt off 2.51pt}]  (346.05,156.3) .. controls (341.78,160.57) and (379.14,178.57) .. (380.06,173.88) ;
\draw [color={rgb, 255:red, 208; green, 2; blue, 27 }  ,draw opacity=1 ][line width=2.25]  [dash pattern={on 2.53pt off 3.02pt}]  (220.81,164.69) .. controls (210.92,164.5) and (206.96,211.38) .. (216.65,211.18) ;
\draw [color={rgb, 255:red, 65; green, 117; blue, 5 }  ,draw opacity=1 ] [dash pattern={on 0.84pt off 2.51pt}]  (162.53,162.25) .. controls (162.83,168.04) and (203.22,161.79) .. (198.23,157.21) ;
\draw [color={rgb, 255:red, 0; green, 96; blue, 209 }  ,draw opacity=1 ][line width=2.25]  [dash pattern={on 2.53pt off 3.02pt}]  (218.42,156.15) .. controls (206.95,156.15) and (209.3,107.47) .. (220.57,108.26) ;
\draw [color={rgb, 255:red, 245; green, 166; blue, 35 }  ,draw opacity=1 ] [dash pattern={on 0.84pt off 2.51pt}]  (242.38,158.87) .. controls (238.29,161.82) and (306.28,161.26) .. (301.29,156.68) ;
\draw  [fill={rgb, 255:red, 0; green, 0; blue, 0 }  ,fill opacity=0.1 ] (271.63,114.34) .. controls (302.23,114.08) and (381.18,82.7) .. (380.92,159.74) .. controls (380.67,236.77) and (300.18,205.39) .. (271.63,205.13) .. controls (243.09,204.87) and (162.6,236.25) .. (162.34,159.74) .. controls (162.08,83.22) and (241.03,114.6) .. (271.63,114.34) -- cycle ;
\draw [draw opacity=0][fill={rgb, 255:red, 255; green, 255; blue, 255 }  ,fill opacity=1 ]   (240.09,160.45) .. controls (231.33,165.64) and (211.81,166.58) .. (203.74,160.92) .. controls (209.24,154.91) and (235.92,155.19) .. (240.09,160.45) -- cycle ;
\draw    (203.74,160.92) .. controls (209.24,154.91) and (235.21,155.31) .. (239.39,160.57) ;
\draw    (195.08,151.8) .. controls (199.76,169.09) and (243.31,169.09) .. (246.51,151.8) ;

\draw [draw opacity=0][fill={rgb, 255:red, 255; green, 255; blue, 255 }  ,fill opacity=1 ]   (342.32,159.15) .. controls (333.55,164.34) and (314.03,165.28) .. (305.96,159.62) .. controls (311.46,153.61) and (338.14,153.9) .. (342.32,159.15) -- cycle ;
\draw    (305.96,159.62) .. controls (311.46,153.61) and (337.44,154.02) .. (341.61,159.27) ;
\draw    (297.3,150.51) .. controls (301.98,167.8) and (345.53,167.8) .. (348.73,150.51) ;

\draw  [color={rgb, 255:red, 208; green, 2; blue, 27 }  ,draw opacity=1 ][fill={rgb, 255:red, 208; green, 2; blue, 27 }  ,fill opacity=0.1 ] (258.9,173.07) .. controls (264.44,174.27) and (275.2,174.25) .. (280.66,173.27) .. controls (280.02,177.46) and (280.86,189.64) .. (282.04,195.02) .. controls (276.7,193.03) and (265.63,194.02) .. (259.49,194.82) .. controls (259.69,189.83) and (259.44,176.87) .. (258.9,173.07) -- cycle ;
\draw [color={rgb, 255:red, 208; green, 2; blue, 27 }  ,draw opacity=1 ][line width=2.25]    (250.88,190.03) .. controls (340.13,186.84) and (314.57,210.39) .. (322.09,211.18) ;
\draw [color={rgb, 255:red, 208; green, 2; blue, 27 }  ,draw opacity=1 ][line width=2.25]    (220.81,164.69) .. controls (228.72,164.89) and (232.87,176.47) .. (269.87,178.06) .. controls (306.86,179.66) and (318.54,163.07) .. (324.46,163.7) ;
\draw [color={rgb, 255:red, 208; green, 2; blue, 27 }  ,draw opacity=1 ][line width=2.25]    (216.65,211.18) .. controls (228.97,210.27) and (200.5,178.31) .. (290.17,183.54) ;
\draw  [draw opacity=0][fill={rgb, 255:red, 0; green, 0; blue, 0 }  ,fill opacity=1 ] (268.2,189.62) .. controls (268.2,188.26) and (269.29,187.16) .. (270.64,187.16) .. controls (271.99,187.16) and (273.08,188.26) .. (273.08,189.62) .. controls (273.08,190.98) and (271.99,192.08) .. (270.64,192.08) .. controls (269.29,192.08) and (268.2,190.98) .. (268.2,189.62) -- cycle ;
\draw  [color={rgb, 255:red, 0; green, 96; blue, 209 }  ,draw opacity=1 ][fill={rgb, 255:red, 0; green, 96; blue, 209 }  ,fill opacity=0.1 ] (258.2,124.45) .. controls (263.74,125.65) and (275.49,126.56) .. (279.96,124.65) .. controls (279.32,128.84) and (279.91,140.74) .. (281.1,146.13) .. controls (275.72,144.92) and (264.92,145.4) .. (258.79,146.2) .. controls (258.99,141.21) and (258.74,128.24) .. (258.2,124.45) -- cycle ;
\draw [color={rgb, 255:red, 0; green, 96; blue, 209 }  ,draw opacity=1 ][line width=2.25]    (220.57,108.26) .. controls (234.49,108.68) and (214.43,134.1) .. (286.43,128.91) ;
\draw [color={rgb, 255:red, 0; green, 96; blue, 209 }  ,draw opacity=1 ][line width=2.25]    (218.42,156.15) .. controls (226.86,155.97) and (220.6,133.58) .. (269.46,135.65) .. controls (318.32,137.73) and (304.2,108.28) .. (325.02,108.26) ;
\draw [color={rgb, 255:red, 0; green, 96; blue, 209 }  ,draw opacity=1 ][line width=2.25]    (253,140.84) .. controls (343.4,139.44) and (315.72,154.75) .. (324.82,155.15) ;
\draw  [draw opacity=0][fill={rgb, 255:red, 0; green, 0; blue, 0 }  ,fill opacity=1 ] (264.58,135.65) .. controls (264.58,134.29) and (265.67,133.19) .. (267.02,133.19) .. controls (268.37,133.19) and (269.46,134.29) .. (269.46,135.65) .. controls (269.46,137.01) and (268.37,138.11) .. (267.02,138.11) .. controls (265.67,138.11) and (264.58,137.01) .. (264.58,135.65) -- cycle ;
\draw [color={rgb, 255:red, 189; green, 16; blue, 224 }  ,draw opacity=1 ]   (346.05,156.3) .. controls (349.95,151.57) and (381.57,168.04) .. (380.06,173.88) ;
\draw [shift={(367.51,162.44)}, rotate = 207.25] [fill={rgb, 255:red, 189; green, 16; blue, 224 }  ,fill opacity=1 ][line width=0.08]  [draw opacity=0] (7.14,-3.43) -- (0,0) -- (7.14,3.43) -- (4.74,0) -- cycle    ;
\draw [color={rgb, 255:red, 65; green, 117; blue, 5 }  ,draw opacity=1 ]   (162.53,162.25) .. controls (162.53,154.77) and (194.44,151.57) .. (198.23,157.21) ;
\draw [shift={(182.26,154.49)}, rotate = 173.27] [fill={rgb, 255:red, 65; green, 117; blue, 5 }  ,fill opacity=1 ][line width=0.08]  [draw opacity=0] (7.14,-3.43) -- (0,0) -- (7.14,3.43) -- (4.74,0) -- cycle    ;
\draw    (432,219) -- (432,102) ;
\draw [shift={(432,99)}, rotate = 90] [fill={rgb, 255:red, 0; green, 0; blue, 0 }  ][line width=0.08]  [draw opacity=0] (8.04,-3.86) -- (0,0) -- (8.04,3.86) -- (5.34,0) -- cycle    ;
\draw    (412,199) -- (523.5,199.24) ;
\draw [shift={(526.5,199.25)}, rotate = 180.13] [fill={rgb, 255:red, 0; green, 0; blue, 0 }  ][line width=0.08]  [draw opacity=0] (8.04,-3.86) -- (0,0) -- (8.04,3.86) -- (5.34,0) -- cycle    ;
\draw    (512,194) -- (512,204) ;
\draw    (427,109) -- (437,109) ;
\draw [color={rgb, 255:red, 65; green, 117; blue, 5 }  ,draw opacity=1 ][line width=2.25]    (432,109) -- (432,199) ;
\draw [color={rgb, 255:red, 189; green, 16; blue, 224 }  ,draw opacity=1 ][line width=2.25]    (432,199) -- (512,199) ;
\draw [color={rgb, 255:red, 245; green, 166; blue, 35 }  ,draw opacity=1 ]   (242.38,158.87) .. controls (247.53,153.54) and (297.51,151.03) .. (301.29,156.68) ;
\draw [shift={(274.39,153.58)}, rotate = 357.3] [fill={rgb, 255:red, 245; green, 166; blue, 35 }  ,fill opacity=1 ][line width=0.08]  [draw opacity=0] (7.14,-3.43) -- (0,0) -- (7.14,3.43) -- (4.74,0) -- cycle    ;
\draw [color={rgb, 255:red, 245; green, 166; blue, 35 }  ,draw opacity=1 ][line width=2.25]    (432,199) -- (511.88,108.59) ;
\draw  [color={rgb, 255:red, 0; green, 96; blue, 209 }  ,draw opacity=1 ][fill={rgb, 255:red, 0; green, 96; blue, 209 }  ,fill opacity=0.1 ] (472,199) -- (432,199) -- (432,154) -- cycle ;
\draw  [color={rgb, 255:red, 208; green, 2; blue, 27 }  ,draw opacity=1 ][fill={rgb, 255:red, 208; green, 2; blue, 27 }  ,fill opacity=0.1 ] (472,154) -- (472,199) -- (432,199) -- cycle ;
\draw    (427,154) -- (437,154) ;
\draw    (472,194) -- (472,204) ;

\draw (365.2,153.39) node [anchor=south] [inner sep=0.75pt]  [color={rgb, 255:red, 189; green, 16; blue, 224 }  ,opacity=1 ]  {$\alpha $};
\draw (176.99,149.8) node [anchor=south] [inner sep=0.75pt]  [color={rgb, 255:red, 65; green, 117; blue, 5 }  ,opacity=1 ]  {$\beta $};
\draw (281.27,195.23) node [anchor=north west][inner sep=0.75pt]  [color={rgb, 255:red, 208; green, 2; blue, 27 }  ,opacity=1 ]  {$R_{1}$};
\draw (279.94,126.3) node [anchor=south west] [inner sep=0.75pt]  [color={rgb, 255:red, 0; green, 96; blue, 209 }  ,opacity=1 ]  {$R_{2}$};
\draw (338.39,127.45) node [anchor=west] [inner sep=0.75pt]  [color={rgb, 255:red, 0; green, 96; blue, 209 }  ,opacity=1 ]  {$f( R_{2})$};
\draw (207.34,184.96) node [anchor=east] [inner sep=0.75pt]  [color={rgb, 255:red, 208; green, 2; blue, 27 }  ,opacity=1 ]  {$f( R_{1})$};
\draw (472.38,204.4) node [anchor=north] [inner sep=0.75pt]  [color={rgb, 255:red, 189; green, 16; blue, 224 }  ,opacity=1 ]  {$[ \alpha ]$};
\draw (512,204.4) node [anchor=north] [inner sep=0.75pt]  [color={rgb, 255:red, 189; green, 16; blue, 224 }  ,opacity=1 ]  {$[ 2\alpha ]$};
\draw (426,154) node [anchor=east] [inner sep=0.75pt]  [color={rgb, 255:red, 65; green, 117; blue, 5 }  ,opacity=1 ]  {$[ \beta ]$};
\draw (426,109) node [anchor=east] [inner sep=0.75pt]  [color={rgb, 255:red, 65; green, 117; blue, 5 }  ,opacity=1 ]  {$[ 2\beta ]$};
\draw (467.15,118.94) node    {$H_{1}( S,\mathbb{R})$};
\draw (292.66,161.29) node [anchor=north] [inner sep=0.75pt]  [color={rgb, 255:red, 245; green, 166; blue, 35 }  ,opacity=1 ]  {$\gamma $};
\draw (530.5,109) node [anchor=east] [inner sep=0.75pt]  [color={rgb, 255:red, 245; green, 166; blue, 35 }  ,opacity=1 ]  {$[2\gamma] $};

\end{tikzpicture}
\caption{The right figure represents the ergodic rotation set of $f\in\Homeo_0(S)$ represented on the left. The images of the two rectangles $R_1$ and $R_2$ are depicted with thick lines, all the intersections are markovian. The homeomorphism $f$ also has periodic orbits rotating like $[2\alpha]$, $[2\beta]$ and $[2\gamma]$, and his support is included in a small neighbourhood of $\alpha\cup\beta\cup\gamma\cup R_1 \cup R_2\cup f(R_1)\cup f(R_2)$. The two black dots are contractible fixed points of $f$. The precise construction of the Markovian intersections (using arguments of Kwapisz \cite{zbMATH00120193}) is depicted in Figure~\ref{FigEx2cvx2}.\label{FigEx2cvx}}
\end{center}
\end{figure}

\begin{figure}[!ht]
\begin{center}

\tikzset{every picture/.style={line width=0.75pt}} 

\begin{tikzpicture}[x=0.75pt,y=0.75pt,yscale=-.9,xscale=.88]

\draw  [draw opacity=0][fill={rgb, 255:red, 16; green, 213; blue, 0 }  ,fill opacity=0.1 ] (10,10) -- (330,10) -- (330,120) -- (10,120) -- cycle ;
\draw  [draw opacity=0][fill={rgb, 255:red, 74; green, 144; blue, 226 }  ,fill opacity=0.1 ] (10,120) -- (95.85,120) -- (95.85,140) -- (10,140) -- cycle ;
\draw  [draw opacity=0][fill={rgb, 255:red, 144; green, 19; blue, 254 }  ,fill opacity=0.1 ] (244.15,120) -- (330,120) -- (330,140) -- (244.15,140) -- cycle ;
\draw  [draw opacity=0][fill={rgb, 255:red, 208; green, 2; blue, 27 }  ,fill opacity=0.1 ] (127.07,120) -- (212.93,120) -- (212.93,140) -- (127.07,140) -- cycle ;
\draw  [draw opacity=0][fill={rgb, 255:red, 226; green, 255; blue, 9 }  ,fill opacity=0.1 ] (10,140) -- (330,140) -- (330,250) -- (10,250) -- cycle ;
\draw   (10,10) -- (330,10) -- (330,250) -- (10,250) -- cycle ;
\draw  [fill={rgb, 255:red, 255; green, 255; blue, 255 }  ,fill opacity=1 ] (95.85,110) -- (127.07,110) -- (127.07,150) -- (95.85,150) -- cycle ;
\draw  [fill={rgb, 255:red, 255; green, 255; blue, 255 }  ,fill opacity=1 ] (212.93,110) -- (244.15,110) -- (244.15,150) -- (212.93,150) -- cycle ;
\draw [color={rgb, 255:red, 128; green, 128; blue, 128 }  ,draw opacity=1 ][fill={rgb, 255:red, 226; green, 255; blue, 9 }  ,fill opacity=0.5 ]   (173.9,220) -- (173.9,210) -- (166.1,210) -- (166.1,220) -- (25.61,220) -- (25.61,210) -- (17.8,210) -- (17.8,240) -- (322.2,240) -- (322.2,210) -- (314.39,210) -- (314.39,220) -- cycle ;
\draw [color={rgb, 255:red, 128; green, 128; blue, 128 }  ,draw opacity=1 ][fill={rgb, 255:red, 16; green, 213; blue, 0 }  ,fill opacity=0.5 ]   (25.61,50) -- (17.8,50) -- (17.8,20) -- (26.39,20) -- (322.2,20) -- (322.2,50) -- (314.39,50) -- (314.39,40) -- (173.9,40) -- (173.9,50) -- (166.1,50) -- (166.1,40) -- (25.61,40) -- cycle ;
\draw  [color={rgb, 255:red, 128; green, 128; blue, 128 }  ,draw opacity=1 ][fill={rgb, 255:red, 74; green, 144; blue, 226 }  ,fill opacity=0.5 ] (17.8,50) -- (25.61,50) -- (25.61,210) -- (17.8,210) -- cycle ;
\draw  [color={rgb, 255:red, 128; green, 128; blue, 128 }  ,draw opacity=1 ][fill={rgb, 255:red, 144; green, 19; blue, 254 }  ,fill opacity=0.5 ] (314.39,50) -- (322.2,50) -- (322.2,210) -- (314.39,210) -- cycle ;
\draw  [color={rgb, 255:red, 128; green, 128; blue, 128 }  ,draw opacity=1 ][fill={rgb, 255:red, 208; green, 2; blue, 27 }  ,fill opacity=0.5 ] (128.63,65.63) .. controls (178.81,66.2) and (172.15,86.13) .. (172.34,126.88) .. controls (172.54,167.63) and (188.09,192.21) .. (219.37,192.13) .. controls (250.64,192.04) and (275.55,181.93) .. (275,130.75) .. controls (274.45,79.57) and (251.56,80.53) .. (228.15,80.13) .. controls (204.73,79.72) and (196.15,97.77) .. (196.15,128.38) .. controls (196.15,158.98) and (203.91,148.3) .. (204.15,127.88) .. controls (204.38,107.45) and (204.73,92.75) .. (228.54,92.63) .. controls (252.34,92.5) and (264.67,97.33) .. (264.83,129.13) .. controls (264.99,160.93) and (253.12,178.13) .. (227.17,177.88) .. controls (201.22,177.63) and (178.55,180.68) .. (178.39,120.88) .. controls (178.23,61.07) and (166.45,60.45) .. (166.1,50) .. controls (169.43,49.91) and (172.13,49.91) .. (173.9,50) .. controls (174.4,60.45) and (184.99,66.9) .. (184.52,119.1) .. controls (184.05,171.3) and (200.63,170.38) .. (228.15,170.63) .. controls (255.66,170.88) and (257.45,150.82) .. (257.73,130.1) .. controls (258,109.38) and (248.83,100.38) .. (228.85,100.5) .. controls (208.87,100.63) and (209.69,112.2) .. (209.02,135.38) .. controls (208.36,158.55) and (191.27,174.93) .. (191.27,133.13) .. controls (191.27,91.32) and (197.4,73.43) .. (228.15,72.63) .. controls (258.9,71.82) and (281.96,69.72) .. (282,128.75) .. controls (282.04,187.78) and (258.33,200.64) .. (219.56,201.13) .. controls (180.79,201.61) and (166.76,172.35) .. (166.29,132.13) .. controls (165.82,91.9) and (175.18,74.94) .. (128.63,75.13) .. controls (82.09,75.31) and (66.31,80.75) .. (66,119.75) .. controls (65.69,158.75) and (83.89,188.82) .. (110.68,188.3) .. controls (137.48,187.78) and (143.86,156.42) .. (144.17,130.82) .. controls (144.49,105.22) and (135.49,111.45) .. (135.8,130.45) .. controls (136.11,149.45) and (137.95,165.47) .. (112.71,165.3) .. controls (87.48,165.13) and (79.44,146) .. (79.31,130.5) .. controls (79.18,115) and (87.77,90.19) .. (111.57,89.86) .. controls (135.38,89.52) and (159.59,94.89) .. (159.07,115.73) .. controls (158.55,136.56) and (173.97,199.73) .. (173.9,210) .. controls (171.14,210.09) and (168.72,210.09) .. (166.1,210) .. controls (166.03,200.09) and (154.5,139.86) .. (154.11,119.36) .. controls (153.72,98.86) and (134.25,95.46) .. (112.24,95.86) .. controls (90.23,96.26) and (85.13,118.78) .. (85.24,130.9) .. controls (85.34,143.02) and (88.18,157.58) .. (112.24,157.75) .. controls (136.31,157.92) and (131.29,147.8) .. (131.29,127.3) .. controls (131.29,106.8) and (148.12,100) .. (148.43,125) .. controls (148.74,150) and (144.84,195.83) .. (111.15,195.5) .. controls (77.46,195.17) and (59.6,158.23) .. (59,119.25) .. controls (58.4,80.27) and (78.46,65.05) .. (128.63,65.63) -- cycle ;
\draw  [draw opacity=0][fill={rgb, 255:red, 16; green, 213; blue, 0 }  ,fill opacity=0.1 ] (350,10) -- (670,10) -- (670,120) -- (350,120) -- cycle ;
\draw  [draw opacity=0][fill={rgb, 255:red, 74; green, 144; blue, 226 }  ,fill opacity=0.1 ] (350,120) -- (435.85,120) -- (435.85,140) -- (350,140) -- cycle ;
\draw  [draw opacity=0][fill={rgb, 255:red, 144; green, 19; blue, 254 }  ,fill opacity=0.1 ] (584.15,120) -- (670,120) -- (670,140) -- (584.15,140) -- cycle ;
\draw  [draw opacity=0][fill={rgb, 255:red, 208; green, 2; blue, 27 }  ,fill opacity=0.1 ] (467.07,120) -- (552.93,120) -- (552.93,140) -- (467.07,140) -- cycle ;
\draw  [draw opacity=0][fill={rgb, 255:red, 226; green, 255; blue, 9 }  ,fill opacity=0.1 ] (350,140) -- (670,140) -- (670,250) -- (350,250) -- cycle ;
\draw   (350,10) -- (670,10) -- (670,250) -- (350,250) -- cycle ;
\draw  [fill={rgb, 255:red, 255; green, 255; blue, 255 }  ,fill opacity=1 ] (435.85,110) -- (467.07,110) -- (467.07,150) -- (435.85,150) -- cycle ;
\draw  [fill={rgb, 255:red, 255; green, 255; blue, 255 }  ,fill opacity=1 ] (552.93,110) -- (584.15,110) -- (584.15,150) -- (552.93,150) -- cycle ;
\draw [color={rgb, 255:red, 128; green, 128; blue, 128 }  ,draw opacity=1 ][fill={rgb, 255:red, 226; green, 255; blue, 9 }  ,fill opacity=0.5 ]   (513.9,220) -- (513.9,210) -- (506.1,210) -- (506.1,220) -- (365.61,220) -- (365.61,210) -- (357.8,210) -- (357.8,240) -- (662.2,240) -- (662.2,210) -- (654.39,210) -- (654.39,220) -- cycle ;
\draw [color={rgb, 255:red, 128; green, 128; blue, 128 }  ,draw opacity=1 ][fill={rgb, 255:red, 16; green, 213; blue, 0 }  ,fill opacity=0.5 ]   (365.61,50) -- (357.8,50) -- (357.8,20) -- (366.39,20) -- (662.2,20) -- (662.2,50) -- (654.39,50) -- (654.39,40) -- (513.9,40) -- (513.9,50) -- (506.1,50) -- (506.1,40) -- (365.61,40) -- cycle ;
\draw  [color={rgb, 255:red, 128; green, 128; blue, 128 }  ,draw opacity=1 ][fill={rgb, 255:red, 74; green, 144; blue, 226 }  ,fill opacity=0.5 ] (357.8,50) -- (365.61,50) -- (365.61,210) -- (357.8,210) -- cycle ;
\draw  [color={rgb, 255:red, 128; green, 128; blue, 128 }  ,draw opacity=1 ][fill={rgb, 255:red, 144; green, 19; blue, 254 }  ,fill opacity=0.5 ] (654.39,50) -- (662.2,50) -- (662.2,210) -- (654.39,210) -- cycle ;
\draw  [color={rgb, 255:red, 128; green, 128; blue, 128 }  ,draw opacity=1 ][fill={rgb, 255:red, 208; green, 2; blue, 27 }  ,fill opacity=0.5 ] (593.81,204.54) .. controls (632.58,203.87) and (640.6,181.83) .. (639.04,129.17) .. controls (637.48,76.5) and (608.34,64.17) .. (566.72,63.83) .. controls (525.09,63.5) and (500.37,104.5) .. (500.37,128.5) .. controls (500.37,152.5) and (504.28,159.17) .. (518.07,158.5) .. controls (531.85,157.83) and (540.18,147.83) .. (539.66,130.83) .. controls (539.14,113.83) and (547.98,96.17) .. (569.06,96.17) .. controls (590.13,96.17) and (600.02,100.17) .. (600.54,135.17) .. controls (601.06,170.17) and (574.52,178.5) .. (522.23,178.17) .. controls (469.93,177.83) and (472.8,151.83) .. (473.32,125.5) .. controls (473.84,99.17) and (469.93,98.17) .. (450.94,97.5) .. controls (431.95,96.83) and (423.89,106.17) .. (423.89,134.83) .. controls (423.89,163.5) and (468.11,181.83) .. (514.16,182.83) .. controls (560.21,183.83) and (613.54,174.5) .. (614.59,134.17) .. controls (615.63,93.83) and (591.69,91.17) .. (568.8,90.83) .. controls (545.9,90.5) and (533.93,113.5) .. (533.67,131.83) .. controls (533.41,150.17) and (511.56,168.83) .. (511.82,134.5) .. controls (512.08,100.17) and (533.93,70.83) .. (569.58,70.17) .. controls (605.22,69.5) and (634.62,87.83) .. (633.58,135.83) .. controls (632.54,183.83) and (620.57,198.83) .. (528.21,199.17) .. controls (435.85,199.5) and (378.36,198.83) .. (378.1,134.17) .. controls (377.84,69.5) and (389.02,56.5) .. (445.74,58.17) .. controls (502.46,59.83) and (506.28,59.92) .. (506.1,50) .. controls (509.16,50.38) and (512.64,50.08) .. (513.9,50) .. controls (513.96,69.46) and (498.29,71.83) .. (443.92,72.17) .. controls (389.54,72.5) and (386.16,82.5) .. (386.94,137.17) .. controls (387.72,191.83) and (432.99,197.17) .. (527.17,195.83) .. controls (621.35,194.5) and (627.33,177.83) .. (627.07,134.83) .. controls (626.81,91.83) and (600.54,77.83) .. (566.72,77.83) .. controls (532.89,77.83) and (520.67,107.17) .. (518.85,130.83) .. controls (517.02,154.5) and (528.21,144.07) .. (528.21,128.83) .. controls (528.21,113.6) and (535.76,83.17) .. (569.32,83.5) .. controls (602.88,83.83) and (621.87,91.17) .. (621.35,135.5) .. controls (620.83,179.83) and (573.74,191.5) .. (514.42,190.17) .. controls (455.11,188.83) and (415.04,172.5) .. (414.26,136.5) .. controls (413.48,100.5) and (425.71,87.83) .. (452.76,88.5) .. controls (479.82,89.17) and (482.16,90.83) .. (482.16,122.17) .. controls (482.16,153.5) and (471.76,171.17) .. (522.75,171.17) .. controls (573.74,171.17) and (592.73,172.5) .. (592.21,138.17) .. controls (591.69,103.83) and (584.15,104.83) .. (565.41,105.17) .. controls (546.68,105.5) and (546.42,111.83) .. (546.68,131.17) .. controls (546.94,150.5) and (540.7,165.83) .. (517.8,166.17) .. controls (494.91,166.5) and (494.91,160.5) .. (493.87,128.5) .. controls (492.83,96.5) and (513.12,55.83) .. (566.46,55.5) .. controls (619.79,55.17) and (646.59,64.5) .. (647.11,130.17) .. controls (647.63,195.83) and (643.38,211.56) .. (595.25,212.23) .. controls (547.12,212.9) and (513.96,201.15) .. (513.9,210) .. controls (510.84,209.77) and (508.96,209.83) .. (506.1,210) .. controls (505.32,199.46) and (555.05,205.21) .. (593.81,204.54) -- cycle ;
\draw    (121.8,126.3) .. controls (121.8,115.3) and (110.9,115.15) .. (104.9,120.15) .. controls (99.32,124.8) and (97.11,136.15) .. (108.98,137.52) ;
\draw [shift={(111.9,137.65)}, rotate = 183.7] [fill={rgb, 255:red, 0; green, 0; blue, 0 }  ][line width=0.08]  [draw opacity=0] (8.04,-3.86) -- (0,0) -- (8.04,3.86) -- (5.34,0) -- cycle    ;
\draw    (239.4,126.3) .. controls (239.4,115.3) and (228.5,115.15) .. (222.5,120.15) .. controls (216.92,124.8) and (214.71,136.15) .. (226.58,137.52) ;
\draw [shift={(229.5,137.65)}, rotate = 183.7] [fill={rgb, 255:red, 0; green, 0; blue, 0 }  ][line width=0.08]  [draw opacity=0] (8.04,-3.86) -- (0,0) -- (8.04,3.86) -- (5.34,0) -- cycle    ;
\draw    (462.47,128.3) .. controls (462.47,117.3) and (451.57,117.15) .. (445.57,122.15) .. controls (439.99,126.8) and (437.78,138.15) .. (449.64,139.52) ;
\draw [shift={(452.57,139.65)}, rotate = 183.7] [fill={rgb, 255:red, 0; green, 0; blue, 0 }  ][line width=0.08]  [draw opacity=0] (8.04,-3.86) -- (0,0) -- (8.04,3.86) -- (5.34,0) -- cycle    ;
\draw    (580.07,128.3) .. controls (580.07,117.3) and (569.17,117.15) .. (563.17,122.15) .. controls (557.59,126.8) and (555.38,138.15) .. (567.24,139.52) ;
\draw [shift={(570.17,139.65)}, rotate = 183.7] [fill={rgb, 255:red, 0; green, 0; blue, 0 }  ][line width=0.08]  [draw opacity=0] (8.04,-3.86) -- (0,0) -- (8.04,3.86) -- (5.34,0) -- cycle    ;

\draw (42.05,128.69) node  [color={rgb, 255:red, 0; green, 102; blue, 218 }  ,opacity=1 ]  {$R_{L}$};
\draw (298.44,128.69) node  [color={rgb, 255:red, 117; green, 0; blue, 218 }  ,opacity=1 ]  {$R_{R}$};
\draw (249.52,68.48) node [anchor=south west] [inner sep=0.75pt]  [color={rgb, 255:red, 208; green, 2; blue, 27 }  ,opacity=1 ]  {$f( R_{M})$};
\draw (112.23,127) node    {$a$};
\draw (229.8,127) node    {$b$};
\draw (452.89,129) node    {$a$};
\draw (570.47,129) node    {$b$};

\end{tikzpicture}

\begin{tikzpicture}
\node[draw,circle,fill=gray!10] (1) at (0,0) {\color{rgb, 255:red, 0; green, 102; blue, 218 } {$R_L$}};
\node[draw,circle,fill=gray!10] (2) at (2.7,0) {\color{rgb, 255:red, 208; green, 2; blue, 27 }{$R_M$}};
\node[draw,circle,fill=gray!10] (3) at (5.4,0) {\color{rgb, 255:red, 117; green, 0; blue, 218 }{$R_R$}};

\draw[->, >=stealth] (2) to[looseness=1, out= -190, in=10] node[midway,above]{$0$} (1) ;
\draw[->, >=stealth] (2) to[looseness=1, out= -170, in=-10] node[midway, below]{$0$} (1) ;
\draw[->, >=stealth] (2) to[looseness=1, out= 10, in=-190] node[midway,above]{$0$} (3) ;
\draw[->, >=stealth] (2) to[looseness=1, out= -10, in=-170] node[midway, below]{$0$} (3) ;

\draw[->, >=stealth] (2) to [looseness=14, out= -150, in=-130]node[midway, below]{$[a]$} (2) ;
\draw[->, >=stealth] (2) to [looseness=14, out= -120, in=-100]node[midway, below]{$[a]$} (2) ;
\draw[->, >=stealth] (2) to [looseness=14, out= -60, in=-80]node[midway, below]{$[b]$} (2) ;
\draw[->, >=stealth] (2) to [looseness=14, out= -30, in=-50]node[midway, below]{$[b]$} (2) ;

\draw[->, >=stealth] (2) to [looseness=14, out= 30, in=50]node[midway, above]{$0$} (2) ;
\draw[->, >=stealth] (2) to [looseness=14, out= 80, in=100]node[midway, above]{$0$} (2) ;
\draw[->, >=stealth] (2) to [looseness=14, out= 120, in=140]node[midway, above]{$0$} (2) ;

\draw[->, >=stealth] (1) to [looseness=5, out= -120, in=-60]node[midway, below]{$0$} (1) ;
\draw[->, >=stealth] (3) to [looseness=5, out= -120, in=-60]node[midway, below]{$0$} (3) ;

\end{tikzpicture}
\hfill 
\begin{tikzpicture}
\node[draw,circle,fill=gray!10] (1) at (0,0) {\color{rgb, 255:red, 0; green, 102; blue, 218 } {$R_L$}};
\node[draw,circle,fill=gray!10] (2) at (2.7,0) {\color{rgb, 255:red, 208; green, 2; blue, 27 }{$R_M$}};
\node[draw,circle,fill=gray!10] (3) at (5.4,0) {\color{rgb, 255:red, 117; green, 0; blue, 218 }{$R_R$}};

\draw[->, >=stealth] (2) to[looseness=1, out= -190, in=10] node[midway,above]{$0$} (1) ;
\draw[->, >=stealth] (2) to[looseness=1, out= -170, in=-10] node[midway, below]{$0$} (1) ;
\draw[->, >=stealth] (2) to[looseness=1, out= 10, in=-190] node[midway,above]{$0$} (3) ;
\draw[->, >=stealth] (2) to[looseness=1, out= 40, in=-220] node[midway, above]{$0$} (3) ;
\draw[->, >=stealth] (2) to[looseness=1, out= -50, in=-130] node[midway,below]{$[a]$} (3) ;
\draw[->, >=stealth] (2) to[looseness=1, out= -10, in=-170] node[midway, below]{$[a]$} (3) ;

\draw[->, >=stealth] (2) to [looseness=14, out= -140, in=-120]node[midway, below]{$[b]$} (2) ;
\draw[->, >=stealth] (2) to [looseness=14, out= -100, in=-80]node[midway, below]{$[b]$} (2) ;

\draw[->, >=stealth] (2) to [looseness=14, out= 50, in=70]node[midway, above]{$0$} (2) ;
\draw[->, >=stealth] (2) to [looseness=14, out= 100, in=120]node[midway, above]{$[ab]$} (2) ;
\draw[->, >=stealth] (2) to [looseness=14, out= 140, in=160]node[midway, above]{$[ab]$} (2) ;

\draw[->, >=stealth] (1) to [looseness=5, out= -120, in=-60]node[midway, below]{$0$} (1) ;
\draw[->, >=stealth] (3) to [looseness=5, out= -120, in=-60]node[midway, below]{$0$} (3) ;
\end{tikzpicture}

\caption{The adaptation of the arguments of \cite{zbMATH00120193} for building two homeomorphisms of a pair of pants, the first one (left) having (ergodic) rotation set equal to $\conv(0, [a], [b])$ and the second one (right) having (ergodic) rotation set equal to $\conv(0, [ab], [b])$. The union of the light-coloured sets is a pair of pants that is mapped to the union dark-coloured sets, hence strictly into itself.
These can be used to build the example of Figure~\ref{FigEx2cvx}. All the intersections between rectangles $R_L$, $R_M$ and $R_R$ and their images are Markovian. The bottom diagrams represent the transition diagrams of the discrete dynamics associated to the different Markovian intersections. Adapting the arguments of \cite{zbMATH00120193}, one can ensure that the rotation set of the resulting homeomorphisms are reduced to the one of the restriction to the maximal invariant set spanned by the rectangles.
\label{FigEx2cvx2}}
\end{center}
\end{figure}
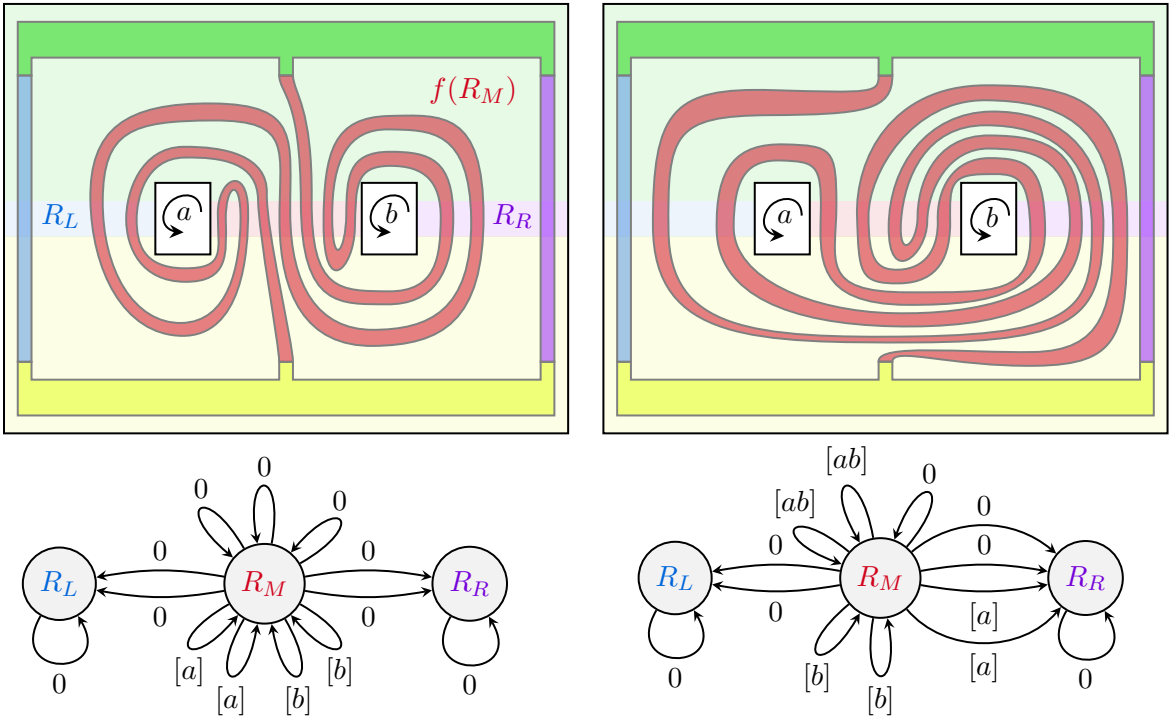

Let us summarize the various examples we will meet in this section.
\begin{enumerate}[label=\alph*)]
\item In Figure~\ref{Fig:Exuncount} we give an example of a homeomorphism $f$ of a genus 2 surface having a single class $\cl_i$ in $I^+$ (and at least two classes in $I^1$), but such that both $\cl_i$ and $\rho_{H_1}^{\text{erg}}(f)$ generate the same 2-dimensional totally isotropic subspace of $H_1(S,\R)$.
\item In Figure~\ref{fig:TenClassesTritorus} we give an example of a homeomorphism $f$ of a genus 3 surface having 4 classes in $I^+$ and 6 classes in $I^1$. This example can be easily generalized to a homeomorphism of a closed surface of genus $g$, having $2g-2$ classes in $I^+$ and $3g-3$ classes in $I^1$. This, in particular, proves that our bound on the number of classes is sharp.
\item In Figure~\ref{fig:CrissCrossBitorus} we give an example, due to Matsumoto, of a homeomorphism $f$ of a genus 2 surface having 2 classes in $I^+$ and 1 class in $I^1$. The ergodic rotation set is made of two pieces that are 2-dimensional ``almost convex'' sets containing 0, spanning two 2-dimensional vector subspaces of $H_1(S,\R)$ that are orthogonal for $\wedge$ and in direct sum.
\item In Figure~\ref{fig:HomologicallyTrivialHorseshoe} we give an example of a homeomorphism $f$ of a genus 3 surface having a single class $\cl_i$ belonging to $I^+$ (and two classes in $I^1$, but each one having rotation set $\{0\}$), but such that $\rho_{H_1}^{\text{erg}}(f)=\{0\}$.
\item In Figure~\ref{FigExNotExact} we give an example of a homeomorphism $f$ of a genus 2 surface for which Franks' ``exactness'' property for the periods of the periodic points does not hold.
\item In Figure~\ref{FigEx2cvx} we give an example of a homeomorphism $f$ of a genus 2 surface whose ergodic rotation set is included in a plane, but is not convex: it is made of the union of two 2-dimensional sets together with three sets included in lines. In particular, it shows that one cannot hope to have a decomposition of $H_1(S,\R)$ in a direct sum of vector subspaces such that the rotation set is included in the union of these vector subspaces, and the intersection of the rotation set with each of these subspaces is ``almost convex''. The precise construction of this homeomorphism, adapting arguments of Kwapisz in a pair of pants, is made in Figure~\ref{FigEx2cvx2}.
\item In Paragraph~\ref{SubSecOxt}, and in particular Proposition~\ref{prop:OxtobyAutomorphism}, building on arguments of Oxtoby and Stepanov, we get an example of smooth diffeomorphism having an invariant measure for which there are uncountably many tracking geodesics.
\item In Paragraph~\ref{SubsecHandel}, and in particular Proposition~\ref{prop:HandelExample}, using a sketch of Handel, we give an example of a topological flow for which there exists a minimal filling geodesic lamination $\Lambda$, such that the tracking geodesic of every orbit (except for a finite number of them) equidistributes to a measure supported in $\Lambda$.
\end{enumerate}

Note also the example of \cite[Figure 1]{guiheneuf2023hyperbolic} showing that there are homeomorphisms of a closed surface $S$ of genus 2 with rotation set $\rho_{H_1}(f)$ having nonempty interior but with $\rho_{H_1}^\textrm{erg}(f)$ included in the union of two planes of $H_1(S,\R)$.

\subsection{Oxtoby Automorphism: uncountably many tracking geodesics}\label{SubSecOxt}

We will prove the following.

\begin{proposition}\label{prop:OxtobyAutomorphism}
	Let $S_g$ be a closed hyperbolic surface. There exists a $\mathcal{C} ^{\infty}$-diffeomorphism of $S_g$, homotopic to the identity, and having uncountably many tracking geodesics belonging to at most two minimal geodesic laminations. 
\end{proposition}

The example we will build is based on the idea of Stepanov flows in $\mathbb{T}^2$, which are introduced in \cite{stepanov}. A $\mathcal{C}^0$ version of this example can be found in \cite[Section 7.1]{pa}.

\subsubsection*{Stepanov Flows}

\begin{definition}[Stepanov flow]\label{def:stepanovflow}
	A flow $\phi^t$ in $\T^2$ is a \emph{Stepanov flow} if there exists an irrational number $\alpha$ such that
	\begin{itemize}
		\item Every orbit belongs to a line $y = \alpha x + k$;
		\item There exists exactly one fixed point $p_0$ of $\phi^t$. 
	\end{itemize}
Fixing $t_0\neq 0$, we will say that $\phi^{t_0}$ is a \emph{Stepanov diffeomorphism}. 
\end{definition}

\begin{theorem}[\cite{oxtoby}]\label{thm:oxtoby}
	Let $f$ be a Stepanov diffeomorphism in $\T^2$. Then, exactly one of the following is true:
	
	\begin{enumerate}
		\item The only $f$-invariant probability is $\delta_{p_0}$,
		\item There exists exactly one $f$-ergodic Borel probability $\mu$ with $\mu(p_0) = 0$.  
	\end{enumerate}
\end{theorem}

In particular, if we take the constant vector field $V = (\alpha_1,\alpha_2)$, where $\alpha_1, \alpha_2$ are rationally independent irrational numbers with quotient $\alpha$, and then \textit{puncture the flow} scaling by a non-negative function $\kappa: \T^2 \to \R$ which vanishes at a single point $p_0$, then the flow $\phi$ induced by $\kappa V$ preserves the measure 

\begin{equation}\label{eq:invmeasurestepanovflow}
	\mu(E) = \iint_E \frac{\dd \textnormal{Leb}}{\kappa(x,y)}
\end{equation}


We will be in case 2. of Theorem~\ref{thm:oxtoby} if and only if this measure is finite. Name $f = \phi^1$, and note that if $F$ denotes a lift of $f$ to $\R^2$, the rotation set $\rho(F)$ (defined in \cite{zbMATH04084609}) is a segment with irrational slope $\alpha$, containing $0$, the other endpoint being
\begin{equation}\label{eq:torusrotationequation}
	v = (v_1,v_2) = \rho_{\mu}(f) := \int_{\mathbb{T}^2} F(z) - z \dd\mu
\end{equation}
By the ergodic theorem, for Lebesgue almost every point $z = \T ^2 \backslash \{p_0\}$, we have that
\begin{equation}\label{eq:lineardisplacementtorusflow}
\lim \limits_{n \to +\infty} \frac{F^{n}(z)-z}{n} = \rho_{\mu}(f)
\end{equation}

Oxtoby also built the following example in \cite{oxtoby}:

\begin{theorem}[Oxtoby analytic example]\label{thm:OxtobyExample}
	There exists an analytic, Lebesgue ergodic flow in $\T^2$ with exactly one fixed point. 
\end{theorem}

\begin{proof}[Sketch of the proof.]
Let $\alpha\in[0,1]$ be an irrational number. Take the vector field $V = (X,Y)$ in $\T^2$ defined as follows:
\[\left\{\begin{array}{ll}
X(x,y)\!\!\!\! & =  \alpha(1-\textnormal{cos}(2\pi(x-y))) + (1-\alpha)(1-\textnormal{cos}(2\pi y)) \\
Y(x,y)\!\!\!\! & = \alpha(1-\textnormal{cos}(2\pi(x-y)))
\end{array}\right.\]
	
	The flow $\phi$ induced by $V$ preserves the Lebesgue measure because \(X_x + Y_y = 0\). Moreover, the flow $\phi$ is topologically conjugate to a Stepanov flow by the homeomorphism $h$ given by 
	\[h(x,y) = \left ( x, y + \frac{\alpha \sin(2\pi (x-y))}{2 \pi} + \frac{(1-\alpha)(\sin(2\pi y))}{2 \pi}\right ).\]
	
	Let $f = \phi^1$. By Theorem~\ref{thm:oxtoby}, we conclude that the Lebesgue measure is $f$-ergodic. 	
\end{proof}

Because of this result and \eqref{eq:torusrotationequation}, we obtain that any Oxtoby example $f$ has $\rho_{\textnormal{Leb}}(f) = v \neq 0$, and its rotation set $\rho(f)$ is a segment of vertices $0$ and $v$.  

\begin{proof}[Proof or Proposition~\ref{prop:OxtobyAutomorphism}]

Let us start with an Oxtoby automorphism $f$ of the torus $\T^2$, from Theorem~\ref{thm:OxtobyExample}. 

\subsubsection*{Step 1. Blowing it up and gluing together}

The idea is now to explode the fixed point into a circle of directions. By Stark's result in \cite[Theorem 6.1]{stark} we know that if we start with a $C^{\infty}$ diffeomorphism $f$ of $\R^2$, with $0$ as a fixed point, we obtain another $C ^{\infty}$ diffeomorphism $\widehat{f}$ of $\R^2 \backslash \{\textnormal{B}(0,1)\}$, by conjugating outside the closed ball, and extending to the boundary $\partial B(0,1)$ as 
\[\hat{f}(p) = \frac{\textnormal{D}f \lvert _0 (p)}{\| \textnormal{D}f \lvert _0 (p).   \|}\]

This technique induces a blow up $\widehat{f}$ of the Oxtoby example, which is a $\mathcal{C}^{\infty}$ diffeomorphism of $\widehat{\Sigma}$, a punctured torus with boundary. Using Seeley's version of the Whitney extension theorem (see \cite{seeley}, and Whitney's result in \cite{whitney}) in local coordinates, we may extend $\widehat{f}$ to a $\mathcal{C}^{\infty}$ diffeorphism with the same name, of the surface obtained by gluing a closed annulus $\overline{\mathbb{A}} \simeq \mathbb{S}^1 \times [0,1]$ to $\Sigma$, through one of its boundaries $\mathbb{S}^1 \times \{0\}$. 
Using a bump function $\kappa'$ we may modify $\widehat{f}$ inside the interior of $\mathbb{A}$, such that $\widehat{f}$ is the identity map in a neighbourhood of $\partial_1\mathbb{A} = \mathbb{S}^1 \times \{1\}$. We may then glue another $(g-1)$-torus minus a disk through $\partial_1\mathbb{A}$, and extend $\widehat{f}$ as the identity, to obtain a $\mathcal{C}^{\infty}$ automorphism of $S_g$, which we will still call $\widehat{f}$.  
\end{proof}

\subsubsection*{Step 2. Describing the tracking geodesics}

Let us understand the rotation of $\widehat{f}$. Recall that $\widehat{f}$ is the identity outside the gluing of a punctured torus with an annulus (this last one being homologically trivial in our surface $S_g$).

Let us take $\{a_1, b_1, \dots , a_{g}, b_{g}\}$ the canonical generator of the fundamental group of $S = S_g$. For the sake of simplicity let us call $a = a_1, b = b_1$. 

Restrict $f$ to $\Sigma = \T^2 \backslash \{0\} \subset S$, and note that it is uniquely ergodic. Let us check that the ergodic measure $\mu$ has positive rotation speed when seen in $S$ (we already know that to be true we looking at the torus). More precisely, taking $\tilde S$ the universal covering of $S$ and $\tilde f$ the respective lift of $\widehat{f}|_{\Sigma}$ which commutes with the covering transformations, we will prove that 

\begin{claim}
	For $\mu$-a.e. $\tilde z \in \tilde S$ we have that 
	
	\begin{equation}\label{eq:LinearEscapeOxtoby}
		\lim \limits_{n \to \pm \infty} \frac{1}{n} \wt{\mathrm{dist}}(\tilde{f}^n(\tilde{z}),\tilde{z})) = \vartheta > 0.
	\end{equation}
\end{claim}


\begin{proof}

Note that the fundamental group $\pi_1(\Sigma)$ is free, and generated by the canonical generator set $\{[a], [b]\}$ of $\pi_1(\T^2)$ given by the vectors $a = (1,0), \ b = (0,1)$, which also generates $\pi_1(\Sigma)$. Now, we have that there exist $\vartheta_1, \vartheta_2 > 0$ such that for $\mu$-a.e. $z \in \Sigma$ (precisely for every $z$ outside the two rays respectively coming from and going to $0$ by the dynamics), we have that  

\begin{equation}\label{eq:WitnessCurves}
	\lim \limits_{n \to \infty} \frac{\phi^t(z)|_{t \in [0,n]} \wedge [a]}{n} = v_1 > 0, \ \lim \limits_{n \to \infty} \frac{\phi^t(z)|_{t \in [0,n]} \wedge [b]}{n} = v_2 > 0,	
\end{equation}

and similarly for the past orbits, where $\wedge$ denotes the intersection number between curves. Note that the quotient $\vartheta_1 / \vartheta_2$ equals the slope of the original flow in the torus.

Fix a fundamental domain $D \in \tilde S$. Given $\tilde f$ commutes with the covering transformations which are isometries of $\tilde S$, it is enough to check \eqref{eq:LinearEscapeOxtoby} for points $\tilde z$ in $D$. 

By the limits in \eqref{eq:WitnessCurves}, we have that up to a sublinear quantity, $\tilde{f}^n(\tilde z)$ belongs to $T(D)$ where the covering transformation $T$ is written as a word of length $(\vartheta_1 + \vartheta_2) n$ when using the generator $\{(a,b)\}$ (in other words this word has positive linear length growth rate). The \v{S}varc-Milnor Lemma then states that the desired linear rate from \eqref{eq:LinearEscapeOxtoby} is, up to a positive constant, the same as the one for the word, which concludes the proof.  
\end{proof}

Given that these generic points $\tilde z$ go to the boundary of $\tilde \Sigma$ at a positive linear rate and because $\widehat{f}$ is uniformly continuous, we may apply Proposition 7 in \cite{lessa}, and obtain that any of these generic points $z$ has a tracking geodesic $\gamma_z$. Given that any two nearby flow lines go through different sides of a lift of the singularity, and by \v{S}varc-Milnor Lemma, we obtain that if we take two points $z, z'$ in different flow lines, then the tracking geodesics separate at a linear rate. 

Moreover, note that we may describe the tracking geodesic of a point using the inersections with the already taken canonical $\pi_1$ generator curves in the torus, $a = (1,0), b=(0,1)$. Given an orbit by the flow with tracking geodesic $\gamma_z$, me may associate to it a Sturmian sequence $\xi_z \in \{a,b\}^{\Z}$ by counting the intersection with the chosen curves $a,b$ (it is best to look at this in the universal covering of the torus, without any puncture). We obtain what is usually called a \textit{cutting sequence} with the square grid, given by lines of slope $\alpha$. 

The resulting sequences are equivalent to the set of Sturmian sequences with density of $b$'s equal to $\frac{1}{\alpha+1}$, which may be coded by the rotation $\textnormal{R}_{\frac{1}{\alpha+1}}$ of angle $\frac{1}{\alpha+1}$ in $\Sp^1$. Note that this map is \textit{almost surjective}, the only sequence not having a preimage being the one corresponding to the orbit of $\textnormal{R}_{\frac{1}{\alpha+1}}(0)$. Given that the tracking geodesics of two points in the same flow line are mapped to the same sequence (up to a shift), we obtain that $f$ has uncountably many tracking geodesics, which by Corollary~\ref{thm:maintheorem} belong to a minimal geodesic lamination $\Lambda$. 

We have then proved that the tracking geodesics of $\widehat{f}$ are the elements of $\Lambda$, and potentially the $\pi_1$ generator for the annulus we used in the construction during Step 1, which yields its own minimal lamination. Thus, $\widehat{f}$ is the desired $\mathcal{C}^{\infty}$ automorphism.
  




\subsection{Handel's flow: filling lamination containing all the rotation}\label{SubsecHandel}

\begin{proposition}\label{prop:HandelExample}
	For every closed hyperbolic surface $S_g$, there exists a topological flow $\phi$ with $\phi^1 = f$, and a minimal filling geodesic lamination $\Lambda$, such that the tracking geodesic of every orbit (up to finitely many), equidistributes to a measure supported in $\Lambda$. Moreover, every tracking geodesic for $f$ has the same rotation speed.
\end{proposition}

Let us before give some tools for building orientable minimal filling geodesic laminations, which we need for the construction. 

\subsubsection*{Pseudo-Anosov examples with orientable stable foliations}

By Thurston classification theorem, we know that if $f$ is a pseudo-Anosov homeomorphism, then there exists a pair of singular foliations $\mathcal F_s, \mathcal{F}_u$ (respectively, the stable and unstable foliations of $f$), with respective transverse measures $\mu_s, \mu_u$, and some $\lambda > 1$ such that
\begin{equation}\label{eq:PseudoAnosovFoliations}
	f \cdot (\F_s, \mu_s) = (\F_s, \lambda^{-1}\mu_s), \ \  f \cdot (\F_u, \mu_u) = (\F_u, \lambda\mu_u).	
\end{equation}
We will build a family of examples on closed surfaces, for which the stable foliations are minimal and orientable. A brief explanation can be found at \cite[Section 14.1]{farb}.

Start with any closed surface $S_g$, and note that there exists an orientation preserving branched covering $\widehat{\pi}: S_g \to \T ^2$ with exactly $2g-2$ singularities, each of them having degree $2$. Take for example the quotient by an orientation preserving involution (see Figure \ref{fig:torusbranchedcovering}).

\begin{figure}[ht]
	\centering
	
	\def\svgscale{0.5}
\begingroup%
  \makeatletter%
  \providecommand\color[2][]{%
    \errmessage{(Inkscape) Color is used for the text in Inkscape, but the package 'color.sty' is not loaded}%
    \renewcommand\color[2][]{}%
  }%
  \providecommand\transparent[1]{%
    \errmessage{(Inkscape) Transparency is used (non-zero) for the text in Inkscape, but the package 'transparent.sty' is not loaded}%
    \renewcommand\transparent[1]{}%
  }%
  \providecommand\rotatebox[2]{#2}%
  \newcommand*\fsize{\dimexpr\f@size pt\relax}%
  \newcommand*\lineheight[1]{\fontsize{\fsize}{#1\fsize}\selectfont}%
  \ifx\svgwidth\undefined%
    \setlength{\unitlength}{764.91954894bp}%
    \ifx\svgscale\undefined%
      \relax%
    \else%
      \setlength{\unitlength}{\unitlength * \real{\svgscale}}%
    \fi%
  \else%
    \setlength{\unitlength}{\svgwidth}%
  \fi%
  \global\let\svgwidth\undefined%
  \global\let\svgscale\undefined%
  \makeatother%
  \begin{picture}(1,0.74187918)%
    \lineheight{1}%
    \setlength\tabcolsep{0pt}%
    \put(0,0){\includegraphics[width=\unitlength,page=1]{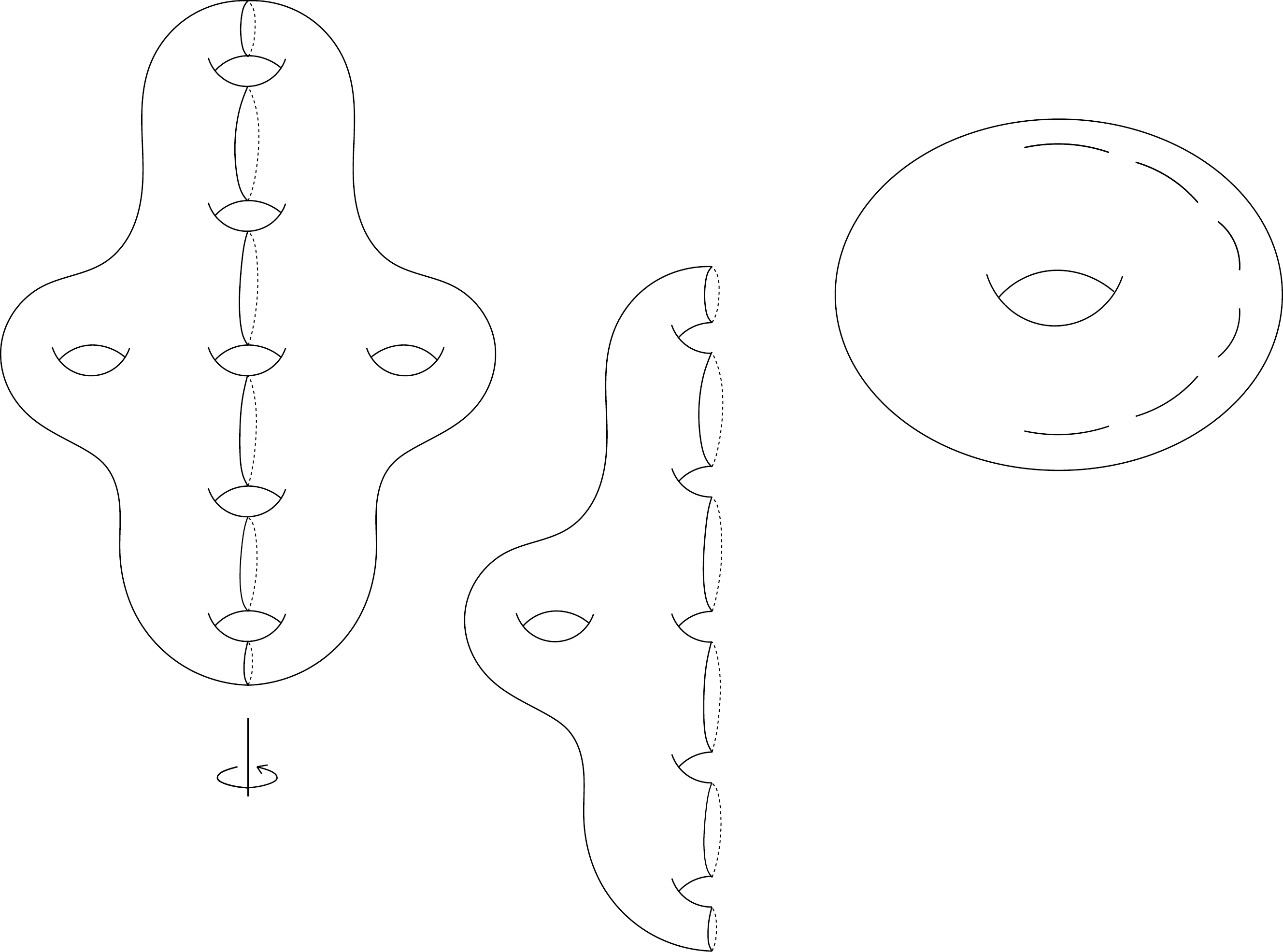}}%
    \put(0.48321715,0.63448262){\makebox(0,0)[lt]{\lineheight{1.25}\smash{\begin{tabular}[t]{l}$\widehat{\pi}$\end{tabular}}}}%
    \put(0,0){\includegraphics[width=\unitlength,page=2]{torusbranchedcovering.pdf}}%
  \end{picture}%
\endgroup%

	\caption{An example of the torus branched covering for $g = 7$. Each red endpoint of each blue seam is a critical value of $\widehat{\pi}$.}
	\label{fig:torusbranchedcovering}
\end{figure}

Take a linear Anosov diffeomorphism $A$ of $\T ^2$ (it suffices to take any integer matrix of determinant equal to 1 and trace greater than 2), and note that it will have a stable foliation $\mathcal{F}_s$ consisting of parallel lines to the projection of the subspace generated by the eigenvector associated to the smaller eigenvalue $\lambda_-$, which is less than 1. Note that this foliation is minimal because these lines have irrational slope.

Recall that every rational point is periodic for $A$. Take $\phi$ conjugate and isotopic to a power of $A$, such that it fixes the projection of the critical points. Take $\widehat{\phi}$ a lift of $\phi$ to $S_g$, and note that it preserves a stable singular foliation $\widehat{\mathcal{F}}_s$ (resp. unstable singular foliation $\wh \F_u$), whose singularities are the critical points of $\widehat{\pi}$, each of them having four prongs (see Figure \ref{fig:bitorusstablefoliation}). Moreover, the orientability of $\mathcal{F}_s$ and the fact that our branched covering preserves orientation, together imply that $\widehat{\mathcal{F}}_s$ is also orientable (resp. $\wh \F_u$). Note that we obtain natural foliations $\wh {\F}_s$, $\wh{\F}_u$ in the lift, with respective transverse measures $\hat{\mu}_s$, $\hat \mu_u$, such that they altogether hold an analogous of \eqref{eq:PseudoAnosovFoliations}, we then get the lifted dynamics $\hat{f}$ is a pseudo-Anosov map in the covering space $S_g$.  

\begin{figure}[h]
	\centering
	
	\def\svgscale{0.52}
\begingroup%
  \makeatletter%
  \providecommand\color[2][]{%
    \errmessage{(Inkscape) Color is used for the text in Inkscape, but the package 'color.sty' is not loaded}%
    \renewcommand\color[2][]{}%
  }%
  \providecommand\transparent[1]{%
    \errmessage{(Inkscape) Transparency is used (non-zero) for the text in Inkscape, but the package 'transparent.sty' is not loaded}%
    \renewcommand\transparent[1]{}%
  }%
  \providecommand\rotatebox[2]{#2}%
  \newcommand*\fsize{\dimexpr\f@size pt\relax}%
  \newcommand*\lineheight[1]{\fontsize{\fsize}{#1\fsize}\selectfont}%
  \ifx\svgwidth\undefined%
    \setlength{\unitlength}{685.23834949bp}%
    \ifx\svgscale\undefined%
      \relax%
    \else%
      \setlength{\unitlength}{\unitlength * \real{\svgscale}}%
    \fi%
  \else%
    \setlength{\unitlength}{\svgwidth}%
  \fi%
  \global\let\svgwidth\undefined%
  \global\let\svgscale\undefined%
  \makeatother%
  \begin{picture}(1,0.2701037)%
    \lineheight{1}%
    \setlength\tabcolsep{0pt}%
    \put(0,0){\includegraphics[width=\unitlength,page=1]{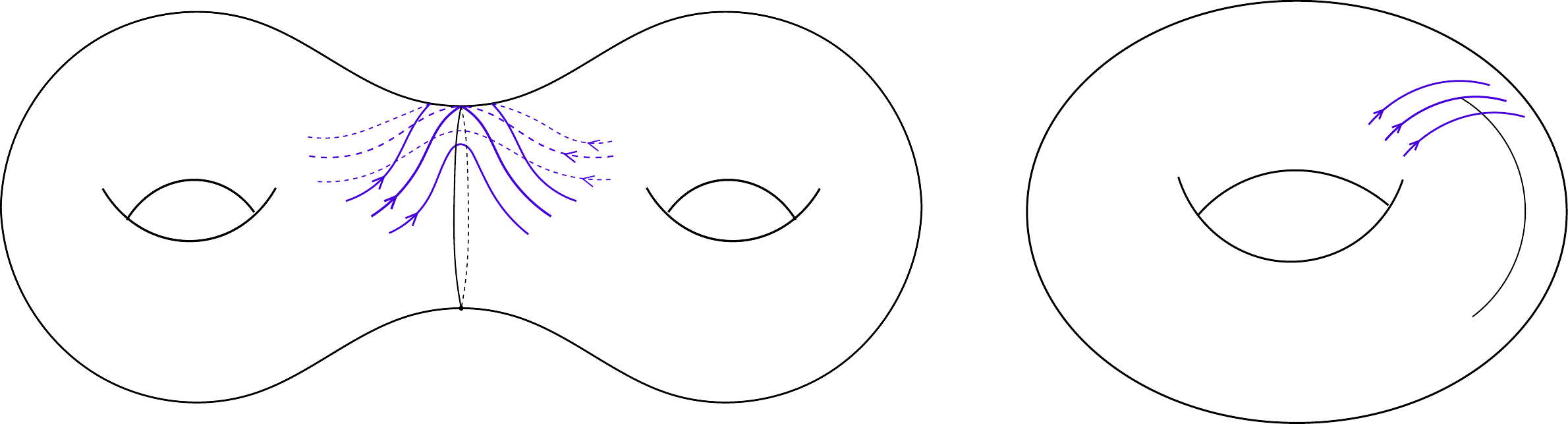}}%
    \put(0.61815407,0.22331062){\makebox(0,0)[lt]{\lineheight{1.25}\smash{\begin{tabular}[t]{l}$\widehat{\pi}$\end{tabular}}}}%
    \put(0.90931197,0.1418795){\color[rgb]{0,0,0.78431373}\makebox(0,0)[lt]{\lineheight{1.25}\smash{\begin{tabular}[t]{l}$\mathcal{F}_s$\end{tabular}}}}%
    \put(0.1875882,0.08494848){\color[rgb]{0.2745098,0,0.8627451}\makebox(0,0)[lt]{\lineheight{1.25}\smash{\begin{tabular}[t]{l}$\widehat{\mathcal{F}}_s$\end{tabular}}}}%
    \put(0,0){\includegraphics[width=\unitlength,page=2]{bitorusstablefoliation.pdf}}%
  \end{picture}%
\endgroup%

	\caption{For the bitorus case, there are exactly two points where the stable foliation is singular.}
	\label{fig:bitorusstablefoliation}
\end{figure}

In a similar fashion we can take branched coverings from closed hyperbolic surfaces to other closed hyperbolic surfaces, and lift the dynamics of pseudo-Anosov maps to obtain new pseudo-Anosov maps in these surfaces, which do not belong to the family we just built, but also have minimal orientable stable foliations.

\subsection*{Laminations obtained by foliation straightening}

\begin{definition}[Curve straightening]
	Let a $S$ be a closed hyperbolic surface, and let $\eta \subset S$ be a curve whose lift $\wt \eta$ to $\D$ is proper and lands in the boundary at $\alpha_{\wt \eta}$ and $\omega_{\wt \eta}$. The \emph{straightening} of $\eta$ is the projection $\gamma$ of the geodesic $\wt \gamma$ having the same endpoints as $\wt \eta$
\end{definition}

The proof of the following statement can be found at \cite[Theorem 1.47]{calegari}).

\begin{theorem}\label{thm:calegari}
	Given $f$ be a pseudo-Anosov homeomorphism of a closed surface $S$, there exists a filling minimal lamination $\Lambda$, such that the straightening of $f(\Lambda)$ is $\Lambda$  
\end{theorem}

\begin{proof}[Sketch of the proof.] 
	Choose a non-periodic closed curve homotopy class $[\gamma]$. Fixing a length $L > 0$, the amount of classes that have a representative of length smaller than $L$ is finite, by the \v Svarc-Milnor Lemma. This implies that the length of $\gamma_n = f^n(\gamma)$ grows to infinity with $n$. 
	
	The lamination $\Lambda$ appears as a minimal sublamination of the limit of a subsequence $\{\gamma_{n_k}\}$ in the Hausdorff topology, which will not be a simple closed curve since $f$ is not reducible. The fact that the obtained lamination is filling is a little bit tricky, the key idea being that otherwise there would exist a power of $f$ fixing a subsurface with non trivial fundamental group, which would imply that some boundary homotopy class is invariant for a -maybe greater- power of $f$, which on its turn contradicts the fact of $f$ being pseudo-Anosov. 
\end{proof}	

Further details on the following statement can be found at \cite[Sections 1.7 to 1.9]{calegari}. 

\begin{remark}
	The lamination from Theorem~\ref{thm:calegari} appears as the straightening of the stable or unstable singular foliations given by $f$. 
\end{remark}

Given that the examples of pseudo-Anosov homeomorphisms we built have orientable stable foliations, we obtain that their straghtenings are filling minimal laminations which are also orientable, thus setting the conditions for the examples we will now build, where we roughly \textit{follow} the lamination.

\begin{proof}[Proof of Proposition~\ref{prop:HandelExample}]

The idea for this example comes from a brief mention in a Handel's preprint \cite{handel}. 

Start with an oriented minimal filling geodesic lamination $\Lambda$. Note that each component of $\Lambda^{\complement}$ is an ideal geodesic polygon $P$ with an even number of sides $s$, because the lamination is oriented. Moreover, by Gauss-Bonnet's Theorem any ideal triangle has area $\pi$, and every ideal polygon has $s - 2$ disjoint ideal triangles inside; we then know that we have finitely many components $P_n$ in the complement of $\Lambda$, each of them having finitely many $s_n$ sides. 

Take a vector field $V$ in $\Lambda$, of unit tangent vectors respecting the lamination's orientation. Note that this induces a flow in $\Lambda$. We will now extend this flow to the rest of the surface $S$. 

For each of the ideal polygons, choose a point $z_n$ in its interior to be an $s_n$-pronged singularity, with each prong going to one different ideal vertex of the polygon. Orient the flow in the prongs such that it is compatible with the flow orientation on the sides of the polygon (that is, alternating the orientation for the flow when changing prongs). Each $P_n$ is now divided into $s_n$ ideal bigons with one geodesic side (the other one has the singularity in the middle, see Figure \ref{fig:IdealPolygon}). Take a filling set of parallel lines to that side, and extend the flow with unit tangent vectors. To make it continuous at the singularity, simple take a non-negative bump function $\kappa_n$ which only vanishes at $z_n$ and is equal to 1 outside a sufficiently small neighbourhood $U_n$ of $z_n$, to scale the vector field which defines the flow. 

\begin{figure}[h]
	\centering
	
	\def\svgscale{0.52}
\begingroup%
  \makeatletter%
  \providecommand\color[2][]{%
    \errmessage{(Inkscape) Color is used for the text in Inkscape, but the package 'color.sty' is not loaded}%
    \renewcommand\color[2][]{}%
  }%
  \providecommand\transparent[1]{%
    \errmessage{(Inkscape) Transparency is used (non-zero) for the text in Inkscape, but the package 'transparent.sty' is not loaded}%
    \renewcommand\transparent[1]{}%
  }%
  \providecommand\rotatebox[2]{#2}%
  \newcommand*\fsize{\dimexpr\f@size pt\relax}%
  \newcommand*\lineheight[1]{\fontsize{\fsize}{#1\fsize}\selectfont}%
  \ifx\svgwidth\undefined%
    \setlength{\unitlength}{387.14531365bp}%
    \ifx\svgscale\undefined%
      \relax%
    \else%
      \setlength{\unitlength}{\unitlength * \real{\svgscale}}%
    \fi%
  \else%
    \setlength{\unitlength}{\svgwidth}%
  \fi%
  \global\let\svgwidth\undefined%
  \global\let\svgscale\undefined%
  \makeatother%
  \begin{picture}(1,0.82412988)%
    \lineheight{1}%
    \setlength\tabcolsep{0pt}%
    \put(0,0){\includegraphics[width=\unitlength,page=1]{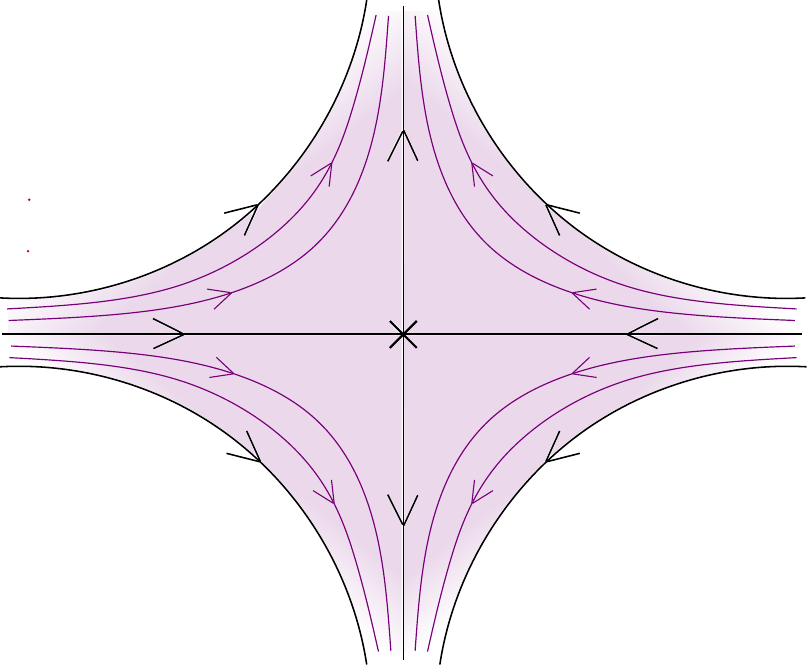}}%
    \put(0.5257177,0.42625442){\makebox(0,0)[lt]{\lineheight{1.25}\smash{\begin{tabular}[t]{l}$z_n$\end{tabular}}}}%
    \put(0.40842236,0.31039707){\color[rgb]{0.47058824,0,0.47058824}\makebox(0,0)[lt]{\lineheight{1.25}\smash{\begin{tabular}[t]{l}$P_n$\end{tabular}}}}%
  \end{picture}%
\endgroup%

	\caption{Sketch of the flow lines inside an ideal polygon $P_n$ with four sides, the sides of this polygon are geodesics in $\Lambda$. Every point in the interior of this polygon, except for $z_n$, is wandering by the time-one flow map.}
	\label{fig:IdealPolygon}
\end{figure}

Note that every orbit $\eta$ except for the singularities and their prongs, stays outside the neighbourhoods $U_n$ except for at most a finite amount of time. Therefore, any of these orbits is tracked by the geodesic side $\gamma$ of the ideal bigon it belongs to. Given that the distance in the universal covering to that geodesic side is decreasing and tends to 0, we conclude that the rotation speed holds $\vartheta(\eta) = 1$ for each of these orbits. 

We naturally have an ergodic measure $\delta_{z_n}$ for every fixed point $z_n$. Every other point $z \in \Lambda^{\complement}$ is wandering, and therefore every other ergodic measure $\mu$ is supported in $\Lambda$, having every point with a tracking geodesic equidisitributing to a measure supported in the same lamination $\Lambda$, and having speed equal to 1.  
\end{proof}

Note that we have proven that, given a orientable minimal filling geodesic lamination $\Lambda$ in $S_g$, there exists a homeomorphism of $S_g$ for which $\Lambda$ encodes all the rotational information of the homeomorphism (even for the wandering points).   

\begin{remark}
By building a rotational horsehoe in an invariant annulus, it is easy to get examples of smooth diffeomorphisms with invariant ergodic measures having positive metric entropy and belonging to $\cl_i$, with $i\in I^1$ (from Theorem~\ref{thm:DecompRotSetIntro}).

By Denjoy-Rees technique, for any closed hyperbolic surface $S$, one can build a homeomorphism of $S$ without topological horseshoe but having an ergodic measure with positive metric entropy and with nonzero rotation vector belonging to an irrational direction. In particular, at least in low regularity, there exist examples of homeomorphisms without topological horseshoes and having measures of $\cl_i$, with $i\in I^1$ (from Theorem~\ref{thm:DecompRotSetIntro}), with positive metric entropy and whose tracking geodesics are not closed (or also, with rotation vector not belonging to $\R H_1(S,\Z)$).

Indeed, consider an ergodic measure $\mu$ of the homeomorphism $f$ of Proposition~\ref{prop:HandelExample} supported in the minimal geodesic lamination associated to $f$, and an $f$-invariant set $A$ of total $\mu$-measure and null $\nu$-measure for any other $f$-ergodic invariant measure $\nu$ (this set $A$ is given by the ergodic decomposition theorem). The example is then obtained by applying \cite[Theorem 1.3]{BCLR} to the product map $h : A\times C \to A\times C;	 (x,c)\mapsto (f(x),h_0(c))$, where $C$ is a Cantor set and $h_0$ is a homeomorphism of $C$ with positive topological entropy. To our knowledge, it is an open question if this kind of examples can exist in $C^2$ regularity, or (for example) if mechanisms ``\emph{\`a la Katok}'' \cite{MR573822} prevent the existence of measures of positive entropy belonging to a class associated to a minimal lamination that is not a closed geodesic.
\end{remark}

\appendix

\section{Proof of Proposition~\ref{LemLocalTransverse}}\label{SecAppendix}

Let us recall Proposition~\ref{LemLocalTransverse}.

\begin{proposition*}[\ref{LemLocalTransverse}]
Let $\Sigma$ be a surface, $\F$ a singular foliation on $\Sigma$ and $f\in\Homeo_0(\Sigma)$. We denote by $\sing \F$ the set of singularities of $\F$.
Suppose that for any $x\in \Sigma\setminus\sing \F$ we are given an $\F$-transverse trajectory $I_\F(x)$ linking $x$ to $f(x)$ and homotopic relative to fixed points to a fixed isotopy $I$ from the identity to $f$. 
Then, for any neighbourhood $V_0\subset \Sigma$ of $\sing \F$, there exists a neighbourhood $U_0\subset V_0$ of $\sing \F$ such that for any $x\in U_0\setminus \sing \F$, there exists an $\F$-transverse trajectory $I'_\F(x)$ linking $x$ to $f(x)$, homotopic to $I_\F(x)$ (in $\dom\F$) and included in $V_0$. 
\end{proposition*}

\begin{proof}[Proof of Proposition~\ref{LemLocalTransverse}]
First, notice that it is sufficient to prove the proposition in the case where $\Sigma\setminus\sing\F$ is connected. Indeed, otherwise, consider each connected component of $\Sigma\setminus\sing\F$; the transverse trajectories $I_\F(x)$ stay in each of these connected components. By replacing $\Sigma$ by the compactification of $\Sigma\setminus\sing\F$ corresponding to the collapsing of each connected component of $\sing\F$, we can suppose that $\sing\F$ is inessential and that each of its connected components is reduced to a singleton.

Suppose that $\Sigma$ is the sphere and that $\sing\F$ is a single point of $\Sigma$. In other words, $\Sigma\setminus\sing\F$ is a plane that is foliated by the foliation $\hat \F:=\F|_{\Sigma\setminus\sing\F}$ which has no singularity. Hence, $\hat f:=f|_{\Sigma\setminus\sing\F}$ is a Brouwer homeomorphism, and $V_0$ projects to a neighbourhood of infinity in this plane. It is not very difficult to perturb the homeomorphism $\hat f$ in a compact set $\hat K\subset \Sigma\setminus\sing\F\simeq \R^2$ such that the obtained homeomorphism $\hat g$ has a single fixed point $\wh x_0\in\hat K$, and such that every leaf of $\hat\F$ is mapped by $\hat g$ to the closure of its left.
More precisely, the foliation $\hat \F|_{\R^2\setminus\{\wh{x}_0\}}$ is transverse for the homeomorphism $\hat g|_{\R^2\setminus\{\wh{x}_0\}}$, and for any $\wh x\in \R^2\setminus\{\wh{x}_0\}$ there is an $\wh\F$-transverse trajectory $I_{\wh \F}^{\hat g}$ linking $\wh x$ to $\wh g(\wh x)$, that can be supposed to coincide with $I_\F$ if $x\notin \hat K$. Hence, one can apply \cite[Proposition~3.4]{zbMATH05518893} to this homeomorphism and the neighbourhood $V_0\setminus \hat K$; it gives a neighbourhood $U_0$ of infinity in $\R^2\setminus\{\wh{x_0}\}$: for any $\wh x\in U_0$, the transverse trajectory $I_{\wh\F}^{\hat g}(x)$ for $g$ stays in $V_0\setminus \hat K$. But this transverse trajectory has been supposed to coincide with $I_{\wh \F}$; this proves the proposition for $f$ in this case.
\bigskip

As remarked before, \cite[Proposition~3.4]{zbMATH05518893} proves directly the proposition when $\sing\F$ is a single point of $\Sigma$ in the case $\Sigma$ is not the sphere.

It remains to treat the case where $\sing\F$ is not a single point of $\Sigma$.
Treating separately each connected component of $V_0$, we can suppose $V_0$ is connected, contains at least one singularity and does not contain any singularity in its boundary. By taking a smaller $V_0$ if necessary, one can suppose that it does not contain genus (recall that $\sing\F$ was supposed to be totally disconnected) and it is relatively compact.
If $\Sigma\setminus V_0$ is not connected, one can replace $V_0$ by some smaller neighbourhood of $\sing\F$ whose complement is connected.

Hence, we are reduced to the case where $\sing\F$ is not a single point of $\Sigma$ and is totally disconnected, and $V_0$ is a topological disk of $\Sigma$ containing at least one singularity of $\sing\F$. Finally, we can suppose that the boundary of $V_0$ is an embedded circle.
\bigskip

From now on, we follow and adapt the proof of \cite[Proposition~3.4]{zbMATH05518893} that treats the case where $V_0$ contains a single singularity of $\sing\F$.

Let us set, for $i\in\N$, $V_i = \bigcap_{|n|\le i} f^n(V_0)$. These sets are neighbourhoods of $V_0\cap\sing\F$ and are not necessarily connected, but each of their connected components are topological disks. Let $\widehat{\dom\F}$ be the universal cover of $\Sigma\setminus\sing\F$. The foliation $\F$, the homeomorphism $f$, the isotopy $I$ and the sets $V_i$ lift naturally to $\widehat{\dom\F}$ to respectively $\widehat\F$, $\widehat f$, $\widehat I$ and $\widehat{V_i}$. By the above remark, the connected components of the sets $\widehat{V_i}$ are simply connected.

We prove the proposition by contradiction and suppose its conclusion is false. This means that there exists a sequence $(z_n)_n$ of points of $V_0\setminus\sing\F$ with $\lim_{n\to+\infty} d(z_n,\sing \F) = 0$ and such that for any $n$, there is no transverse path from $z_n$ to $f(z_n)$ that is homotopic to $I(z)$ and included in $V_0$. Up to taking a subsequence, one can suppose that the $z_n$ tend to a single point $z_\infty\in \sing\F$, and that all the $z_n$ belong to the connected component $V_1^0$ of $V_1$ containing $z_\infty$. We denote by $\wh z_n$ the lifts of the $z_n$ belonging to $\wh V_0$. 

Note that given $\wh z\in \widehat{\dom\F}$, the set $\wh W_{\wh z}$ of points $\wh z'\in \widehat{\dom\F}$ that can be linked from $\wh z$ by a positively transverse arc is an open subset of $\widehat{\dom\F}$ whose boundary is a family $(\wh\phi_\alpha)_\alpha$ of leaves such that $\wh z\in \overline{L(\wh\phi_\alpha)}$ for any $\alpha$.

Let us denote $\check\phi$ the leaves of the topological plane $\wh V_0$; such leaves are connected components of the intersection of some leaf $\wh\phi$ of $\wh \F$ with $\wh V_0$. Hence, given $\wh z\in \widehat V_0$, the set $\check W_{\wh z}$ of points $\wh z'\in \widehat V_0$ that can be linked from $\wh z$ by a positively transverse arc included in $\wh V_0$ is an open subset of $\widehat{V}_0$ whose boundary is a family $(\check\phi_\alpha)_\alpha$ of leaves such that $\wh z\in \overline{L(\check\phi_\alpha)}$ for any $\alpha$. 
Hence, for any $n\in\N$, we have $\wh f(\wh z_n)\in \wh W_{\wh z_n}\setminus \check W_{\wh z_n}$. By the above characterization of $\check W_{\wh z}$, this implies that there exists a leaf $\check \phi_n$ of $\wh V_0$ such that $\wh z_n\in \overline{L(\check\phi_n)}$ and $\wh f(\wh z_n) \notin \overline{L(\check\phi_n)}$.

Because $\wh f(\wh z_n)\in \wh W_{\wh z}$ for any $n$, any leaf $\check \phi_n$ comes from a leaf $\wh\phi_n$ that meets $\wh{\dom\F}\setminus\wh V_0$. 
\medskip

Let $\wh\phi$ be any leaf of $\wh\F$ such that $\omega(\wh\phi)\subset \wh V_1$, and denote $\check\phi$ the connected component of $\wh\phi\cap \wh V_0$ containing a neighbourhood of $+\infty$ for $\wh\phi$. Let $\phi$ be the projection of $\wh\phi$ to $\Sigma$.

\begin{claim}\label{ClaimTechnical}
There is a neighbourhood $U$ of 
$\sing\F\cap V_1$ such that if $z\in U$, if $\wh z$ is a lift of $z$ belonging to $\wh V_1$ and if $\wh z\in \overline{L(\check\phi)}$, then $\wh f(\wh z) \in \overline{L(\check\phi)}$.
\end{claim}

\begin{proof}
Note that by the Poincaré-Bendixson Theorem (and as $\sing\F$ is totally disconnected), the set $\omega(\phi)$ is either a single singularity of $\sing\F$ together with leaves that are homoclinic to it, or a single closed leaf of $\F$, or the union of a set homeomorphic to a circle that is a union of leaves of $\F$ together with their $\alpha$ and $\omega$ limits that are singularities, with possibly some leaves that are homoclinic to these singularities.

Let $\wh x_0\in \check\phi$ such that $\check\phi_{\wh x_0}^+\subset \wh V_1$. Because $\wh\phi$ is a Brouwer line, the sets $\wh f(\check\phi_{\wh x_0}^+)$ and $\wh f^{-1}(\check\phi_{\wh x_0}^+)$ are disjoint from $\check\phi$. More precisely, let us prove that (in the plane $\wh V_0$)
\begin{equation}\label{EqInclusLeav}
\wh f(\check\phi_{\wh x_0}^+) \subset L(\check\phi)
\qquad \text{and}\qquad
\wh f^{-1}(\check\phi_{\wh x_0}^+) \subset R(\check\phi).
\end{equation}
We treat the first inclusion, the second being identical. If this first inclusion was false, then there would exist $\wt y_0\in \partial \widehat V_0\cap \widehat\phi$ such that 
\begin{equation}\label{EqInclusLeav3}
\wh f(\check\phi_{\wh x_0}^+) \subset \overline{L(\check\phi_{\wh y_0}^-)} \subset \overline{R(\check\phi)}.
\end{equation}
We have two cases. 
\begin{itemize}
\item Either $\omega(\phi) = \omega(f(\phi))$ is reduced to a single point $y$ (which is a singularity of $\F$) and then \eqref{EqInclusLeav3} implies that $\alpha(\phi)$ is reduced to $\{y\}$ and it is easy to get a contradiction (one of the projections of $L(\wh\phi)$ or $R(\wh\phi)$ on $\Sigma$ is a topological disk mapped into itself by $f$ or $f^{-1}$). 
\item 
Or $\omega(\phi)$ has at least two points and $\phi$ spirals around it in a neighbourhood of $+\infty$; then \eqref{EqInclusLeav3} implies that $\alpha(\phi) = \omega(\phi)$, which is impossible for reasons of orientation of the foliation around this limit set.
\end{itemize}

Let $U_0$ be a neighbourhood of $\sing\F\cap V_1$ included in $V_1$ (denote $\wh U_0$ the lift of $U_0$ included in $\wh V_1$) such that 
\begin{equation}\label{EqInclusLeav2}
\check\phi_{\wh x_0}^- \subset \wh V_0\setminus \wh U_0
\qquad \text{and}\qquad
\wh f^{-1}(\check\phi_{\wh x_0}^-) \subset \wh V_0\setminus \wh U_0.
\end{equation}
Let $z\in V_1$ such that the trajectory $(I^t(z))_{t\in[0,1]}$ of $z$ under the isotopy $I$ stays in $U_0\cap f^{-1}(U_0)$. By (uniform) continuity of $I$, the set of such $z$ is a neighbourhood $U$ of $\sing\F\cap V_1$.
Suppose by contradiction that some lift $\wh z$ of $z$ belonging to $\wh V_1$ satisfies $\wh z\in \overline{L(\check\phi)}$ and $\wh f(\wh z)\notin \overline{L(\check\phi)}$.
In this case, both $\wh z$ and $\wh f(\wh z)$ belong to the set $\wh U$ of points of $\wh V_1$ projecting in $U$; moreover these points are separated by $\check\phi$, so $\check\phi$ meets $\wh U$. As $(I^t(\wh z))_{t\in[0,1]}$ links $\wh z$ to $\wh f(\wh z)$, there exists $\wh \gamma$ a subpath of $(I^t(\wh z))_{t\in[0,1]}$ (hence included in $\wh U$) linking $\wh z$ to $\wh z'\in \check\phi$ and not meeting $\check\phi$ in its interior. It is included in $\overline{L(\check\phi)}$ and so (by \eqref{EqInclusLeav}) it is disjoint from $\wh f^{-1}(\check\phi_{\wh x_0}^+)$. It is also included in $\wh U$ and so (by \eqref{EqInclusLeav2}) it is disjoint from $\wh f^{-1}(\check\phi_{\wh x_0}^-)$. Hence, $\wh f(\wh \gamma)$ is disjoint from $\check\phi$; moreover (because of \eqref{EqInclusLeav}) it contains $\wh f(\wh z')\in \overline{L(\check\phi)}$. We deduce that $\wh f(\wh \gamma)$ is included in $\overline{L(\check\phi)}$ and this shows that $\wh f(\wh z)\in L(\check\phi)$.
\end{proof}

Because the leaves $\wh \phi_n$ separate the points $\wh z_n$ and $\wh f(\wh z_n)$, and because these points belong to $V_1$ eventually, the leaves $\phi_n$ meet $V_1$ eventually. Moreover, recall that $\phi_n$ meets $\dom\F\setminus V_0$. Hence, up to taking a subsequence, 
\begin{enumerate}
\item \label{case1technic}either all the leaves $\check\phi_n$ are incoming $\wh V_0$ with limit in $V_1$, meaning that there exists $\wh y_n^+\in \partial \wh V_0$ such that $\check\phi_n = (\wh \phi_n)_{\wh y_n^+}^+$ with $\omega(\phi_n)\subset V_1$;
\item \label{case2technic}or all the leaves $\check\phi_n$ are outgoing $\wh V_0$ with limit in $V_1$, meaning that there exists $\wh y_n^-\in \partial \wh V_0$ such that $\check\phi_n = (\wh \phi_n)_{\wh y_n^-}^-$ with $\alpha(\phi_n)\subset V_1$;
\item \label{case3technic}or for any $n$, there is a connected component of $\phi_n\cap V_1$ joining $y_n^+\in \partial V_1\cap \phi_n$ to $y_n^-\in \partial V_1\cap \phi_n$: we have $(\wh \phi_n)_{\wh y_n^-}^- \cap (\wh \phi_n)_{\wh y_n^+}^+\neq\varnothing$.
\end{enumerate}

Let us begin with case \ref{case1technic}. By taking a subsequence if necessary, we can suppose that the sets $\check\phi_n \cup \{y_n^-\}\cup \omega(\phi_n)$ converge for Hausdorff topology to some set $K\subset \overline{V_0}$. This set contains one incoming leaf $\phi'$.
By the Poincaré-Bendixson Theorem, and because the singularity set is totally disconnected, the limit $\omega(\phi')$ is either a single singularity, or a limit cycle (implying in prticular that any leaf coming close enough to this set stays forever close to this set). 
In both of these cases the set $\omega(\phi')$ is included in $V_1$.


Let $\gamma$ be an arc included in $V_0$, transverse to $\F$ and intersecting $\phi'$ in its interior and once. By changing the lifts of the $\phi_n$ if necessary, we can suppose that this arc lifts to an arc $\wh\gamma$ of $\wh{\dom\F}$ intersecting a lift $\wh\phi'$ of $\phi'$ in a point $\wh y_0$, such that $\wh y_0$ is the limit of points of $\wh\phi_n$. 
Considering a subsequence if necessary, we can suppose that either the sequence $(L(\wh\phi_n))_n$ is increasing and that $L(\wh\phi_n) \subset L(\wh\phi')$, or that the sequence $(L(\wh\phi_n))_n$ is decreasing and that $L(\wh\phi') \subset L(\wh\phi_n)$. Let us treat the first case, the second one being similar.

\begin{figure}
\begin{center}

\tikzset{every picture/.style={line width=0.75pt}} 

\begin{tikzpicture}[x=0.75pt,y=0.75pt,yscale=-.9,xscale=.9]


\draw  [draw opacity=0][fill={rgb, 255:red, 155; green, 155; blue, 155 }  ,fill opacity=0.1 ] (240.88,141.12) .. controls (240.4,86.3) and (266.62,81.88) .. (359.62,81.88) .. controls (452.62,81.88) and (437.99,91.6) .. (438.56,151.21) .. controls (439.13,210.82) and (435.4,211.3) .. (435,223.3) .. controls (419,225.5) and (380.8,228.9) .. (345.6,226.9) .. controls (311.8,223.5) and (264,229.3) .. (245.2,229.7) .. controls (246.6,218.9) and (240.88,175.19) .. (240.88,141.12) -- cycle ;
\draw [draw opacity=0][fill={rgb, 255:red, 95; green, 176; blue, 0 }  ,fill opacity=0.1 ]   (243.03,117.9) .. controls (249.87,88.07) and (265.88,82.33) .. (359.62,81.88) .. controls (454.54,81.88) and (436.77,98.33) .. (438.56,140.32) .. controls (395.01,144.32) and (392,184.67) .. (392.01,196.75) .. controls (378.05,198.01) and (337.19,193.97) .. (307.95,199.99) .. controls (303.16,164.44) and (268.87,112.93) .. (243.03,117.9) -- cycle ;
\draw [color={rgb, 255:red, 245; green, 166; blue, 35 }  ,draw opacity=1 ]   (346.33,233.58) .. controls (346.96,183.5) and (322.48,119.69) .. (336.58,82.26) ;
\draw [shift={(336.53,153.98)}, rotate = 80.71] [fill={rgb, 255:red, 245; green, 166; blue, 35 }  ,fill opacity=1 ][line width=0.08]  [draw opacity=0] (6.25,-3) -- (0,0) -- (6.25,3) -- cycle    ;
\draw [color={rgb, 255:red, 245; green, 166; blue, 35 }  ,draw opacity=1 ]   (392.83,233.25) .. controls (384.67,160.42) and (406.56,143.82) .. (438.56,140.32) ;
\draw [shift={(395.81,171.14)}, rotate = 106.5] [fill={rgb, 255:red, 245; green, 166; blue, 35 }  ,fill opacity=1 ][line width=0.08]  [draw opacity=0] (6.25,-3) -- (0,0) -- (6.25,3) -- cycle    ;
\draw [color={rgb, 255:red, 128; green, 128; blue, 128 }  ,draw opacity=1 ]   (245.2,229.7) .. controls (272.8,226.9) and (318,224.7) .. (345.6,226.9) .. controls (373.2,229.1) and (407.6,226.5) .. (435,223.3) ;
\draw [color={rgb, 255:red, 245; green, 166; blue, 35 }  ,draw opacity=1 ]   (243.03,117.9) .. controls (267.63,112.03) and (313,163.75) .. (309.67,235.08) ;
\draw [shift={(293.99,159.85)}, rotate = 64.22] [fill={rgb, 255:red, 245; green, 166; blue, 35 }  ,fill opacity=1 ][line width=0.08]  [draw opacity=0] (6.25,-3) -- (0,0) -- (6.25,3) -- cycle    ;
\draw [color={rgb, 255:red, 95; green, 176; blue, 0 }  ,draw opacity=1 ]   (290.61,204.75) .. controls (319.27,191.7) and (388.63,199.07) .. (402.33,196.12) ;

\draw (314.77,135.76) node [anchor=north] [inner sep=0.75pt]    {$\hat{z}_{n}$};
\draw (352.8,124.5) node  [font=\tiny]  {$\times $};
\draw (353.07,124.84) node [anchor=west] [inner sep=0.75pt]    {$\hat{f}(\hat{z}_{n})$};
\draw (313.79,132.33) node  [font=\tiny]  {$\times $};
\draw (404.33,196.12) node [anchor=west] [inner sep=0.75pt]  [color={rgb, 255:red, 88; green, 162; blue, 1 }  ,opacity=1 ]  {$\hat{\gamma }_{+}$};
\draw (295.31,104.15) node [anchor=west] [inner sep=0.75pt]  [color={rgb, 255:red, 88; green, 162; blue, 1 }  ,opacity=1 ]  {$\wh D$};
\draw (278.7,129.34) node [anchor=north east] [inner sep=0.75pt]  [color={rgb, 255:red, 210; green, 142; blue, 30 }  ,opacity=1 ]  {$\check{\phi }_{1}$};
\draw (413.49,144.99) node [anchor=north west][inner sep=0.75pt]  [color={rgb, 255:red, 210; green, 142; blue, 30 }  ,opacity=1 ]  {$\check{\phi } '$};
\draw (262.34,221.5) node [anchor=south] [inner sep=0.75pt]    {$\hat{V}_{0}$};
\draw (343.39,167.7) node [anchor=west] [inner sep=0.75pt]  [color={rgb, 255:red, 210; green, 142; blue, 30 }  ,opacity=1 ]  {$\check{\phi }_{n}$};

\end{tikzpicture}

\caption{The domain $\wh D$ of case \ref{case1technic}.\ in the proof of Proposition~\ref{LemLocalTransverse}.\label{FigFirstD}}
\end{center}
\end{figure}
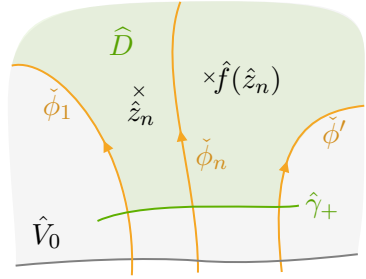

Claim~\ref{ClaimTechnical} gives us a neighbourhood $U\subset V_1$ of $\sing\F\cap V_0$ such that if $z\in U$, then any lift $\wh z$ of $z$ belonging to $\overline{L(\check\phi_1)}$ is such that $\wh f(\wh z) \in \overline{L(\check\phi_1)}$, and any lift $\wh z$ of $z$ belonging to $\overline{L(\check\phi')}$ is such that $\wh f(\wh z) \in \overline{L(\check\phi')}$. Because $z_n$ tend to $z_\infty\in \sing\F\cap V_0$, the points $z_n$ belong to $U$ eventually: by taking a subsequence we can suppose that they all belong to $U$.
Recall that we have supposed by contradiction that $\wh z_n\in \overline{L(\check\phi_n)}$ and $\wh f(\wh z_n) \notin \overline{L(\check\phi_n)}$.
\begin{itemize}
\item If $\wh z_n\notin \overline{L(\check\phi_1)}$ and $\wh f(\wh z_n) \in \overline{L(\check\phi')}$, then both $\wh z_n$ and $\wh f(\wh z_n)$ belong to $R(\check\phi_1)\cap L(\check\phi')\subset \wh V_0$; more precisely for $n$ large enough they both belong to the unbounded connected component $\wh D$ (see Figure~\ref{FigFirstD}) of $\big(R(\check\phi_1)\cap L(\check\phi')\big)\setminus\wh \gamma$. The boundary of this set is made of pieces of two leaves, together with the transverse arc $\wh\gamma$. The orientation of these pieces prevents $\wh \phi_n\setminus \check\phi_n$ from intersecting $\wh D$, and contradicts the existence of an $\wh\F$-transverse path joining $\wh z$ to  $\wh f(\wh z)$.
\item If $\wh z_n\in \overline{L(\check\phi_1)}$, then by Claim~\ref{ClaimTechnical} we have $\wh f(\wh z_n)\in \overline{L(\check\phi_1)}$.
On the other hand, $\wh f(\wh z_n) \in R(\check\phi_n)\subset R(\check\phi_1)$. This is a contradiction.
\item If $\wh f(\wh z_n) \notin \overline{L(\check\phi')}$, then by Claim~\ref{ClaimTechnical} we have $\wh z_n\notin \overline{L(\check\phi')}$.
On the other hand, $\wh z_n\in\overline{L(\check\phi_n)}\subset L(\check\phi_1)$. This is also a contradiction.
\end{itemize}


The case where all the leaves $\check\phi_n$ are outgoing (case \ref{case2technic}) is identical to the incoming case.
\bigskip

It remains to treat case \ref{case3technic}, \emph{i.e.}\ the case where for any $n\in\N$, we have $\check\phi_n := (\wh \phi_n)_{\wh y_n^-}^- \cap (\wh \phi_n)_{\wh y_n^+}^+$ with $\wh y_n^+, \wh y_n^-\in \partial \wh V_1$ (we now denote $\check\phi$ the intersections of a leaf $\wh\phi$ with $\wh V_1$). By taking a subsequence if necessary, we can suppose that the sets $\check\phi_n \cup \{y_n^-, y_n^+\}$ converge for Hausdorff topology to some set $K\subset \overline{V_0}$. 
This set contains one incomming leaf $\phi^+$ and an outgoing leaf $\phi^-$. 
By the Poincaré-Bendixson Theorem, and because the singularity set is totally disconnected, the limit sets $\omega(\phi^+)$ and $\omega(\phi^-)$ are either a single singularity, or a limit cycle (implying in prticular that any leaf coming close enough to this set stays forever close to this set in the future or the past). The fact that the leaves $\phi_n$ go in and out of $V_0$ rules out the second possibility. Hence, the sets $\omega(\phi^+)$ and $\alpha(\phi^-)$ are both reduced to a single singularity; in particular these sets are included in $V_2$. 

Let $\gamma^+$ (resp. $\gamma^-$) be a simple arc included in $V_1$, transverse to $\F$ and intersecting $\phi^+$ (resp. $\phi^-$) in its interior and once. Because the leaves $\overline{\pr_\Sigma(\check\phi_n)}$ meet $\partial V_1$, Poincaré-Bendixson theory implies that for any $n\in\N$, the pieces of leaves $\pr_\Sigma(\check\phi_n)$ intersect both transversals $\gamma^+$ and $\gamma^-$ once. Let $\{x_n^+\} = \pr_\Sigma(\check\phi_n)\cap \gamma^+$ and $\{x_n^-\} = \pr_\Sigma(\check\phi_n)\cap \gamma^-$. As before, up to taking a subsequence, we can suppose that either both sequences $x_n^+$ and $x_n^-$ are increasing (for some order on $\gamma^+$ and $\gamma^-$) and tend to respectively $x^+$ and $x^-$, or both sequences $x_n^+$ and $x_n^-$ are decreasing and tend to respectively $x^+$ and $x^-$. We treat only the second case, the first one being similar.

\begin{figure}
\begin{center}

\tikzset{every picture/.style={line width=0.75pt}} 

\begin{tikzpicture}[x=0.75pt,y=0.75pt,yscale=-1,xscale=1]

\draw  [draw opacity=0][fill={rgb, 255:red, 155; green, 155; blue, 155 }  ,fill opacity=0.1 ] (336.09,78.05) .. controls (336.09,20.75) and (369.36,1.17) .. (462.36,1.17) .. controls (555.36,1.17) and (560.72,25.64) .. (561.29,85.25) .. controls (561.86,144.86) and (512.61,169.12) .. (440.09,170.85) .. controls (367.56,172.59) and (336.09,135.34) .. (336.09,78.05) -- cycle ;
\draw [color={rgb, 255:red, 245; green, 166; blue, 35 }  ,draw opacity=1 ]   (86.11,140.08) .. controls (86.44,157.08) and (88.78,170.42) .. (86.44,184.08) ;
\draw [shift={(86.95,157.43)}, rotate = 86.85] [fill={rgb, 255:red, 245; green, 166; blue, 35 }  ,fill opacity=1 ][line width=0.08]  [draw opacity=0] (6.25,-3) -- (0,0) -- (6.25,3) -- cycle    ;
\draw  [fill={rgb, 255:red, 155; green, 155; blue, 155 }  ,fill opacity=0.1 ] (72.97,97.17) .. controls (105.63,68.51) and (158.3,72.51) .. (180.3,101.84) .. controls (202.3,131.17) and (201.63,155.84) .. (177.63,193.84) .. controls (153.63,231.84) and (73.63,227.84) .. (53.63,197.84) .. controls (33.63,167.84) and (40.3,125.84) .. (72.97,97.17) -- cycle ;
\draw [color={rgb, 255:red, 245; green, 166; blue, 35 }  ,draw opacity=1 ]   (80.78,68.75) .. controls (111.78,87.75) and (116.11,237.42) .. (147.44,246.08) ;
\draw [shift={(112.39,153.93)}, rotate = 77.75] [fill={rgb, 255:red, 245; green, 166; blue, 35 }  ,fill opacity=1 ][line width=0.08]  [draw opacity=0] (6.25,-3) -- (0,0) -- (6.25,3) -- cycle    ;
\draw [color={rgb, 255:red, 245; green, 166; blue, 35 }  ,draw opacity=1 ]   (112.59,62.85) .. controls (154.59,116.05) and (189.44,212.08) .. (182.11,239.42) ;
\draw [shift={(158.68,143.24)}, rotate = 66.99] [fill={rgb, 255:red, 245; green, 166; blue, 35 }  ,fill opacity=1 ][line width=0.08]  [draw opacity=0] (6.25,-3) -- (0,0) -- (6.25,3) -- cycle    ;
\draw [color={rgb, 255:red, 245; green, 166; blue, 35 }  ,draw opacity=1 ]   (49.78,72.42) .. controls (80.78,91.42) and (84.44,114.08) .. (86.11,140.08) ;
\draw [shift={(75.48,96.96)}, rotate = 61.84] [fill={rgb, 255:red, 245; green, 166; blue, 35 }  ,fill opacity=1 ][line width=0.08]  [draw opacity=0] (6.25,-3) -- (0,0) -- (6.25,3) -- cycle    ;
\draw [color={rgb, 255:red, 245; green, 166; blue, 35 }  ,draw opacity=1 ]   (86.44,184.08) .. controls (84.11,199.08) and (98.78,212.75) .. (99.11,232.75) ;
\draw [shift={(90.77,204.22)}, rotate = 66.44] [fill={rgb, 255:red, 245; green, 166; blue, 35 }  ,fill opacity=1 ][line width=0.08]  [draw opacity=0] (6.25,-3) -- (0,0) -- (6.25,3) -- cycle    ;
\draw  [draw opacity=0][fill={rgb, 255:red, 0; green, 91; blue, 198 }  ,fill opacity=1 ] (83.43,140.08) .. controls (83.43,138.6) and (84.63,137.4) .. (86.11,137.4) .. controls (87.59,137.4) and (88.79,138.6) .. (88.79,140.08) .. controls (88.79,141.56) and (87.59,142.76) .. (86.11,142.76) .. controls (84.63,142.76) and (83.43,141.56) .. (83.43,140.08) -- cycle ;
\draw  [draw opacity=0][fill={rgb, 255:red, 0; green, 91; blue, 198 }  ,fill opacity=1 ] (83.76,184.08) .. controls (83.76,182.6) and (84.96,181.4) .. (86.44,181.4) .. controls (87.92,181.4) and (89.12,182.6) .. (89.12,184.08) .. controls (89.12,185.56) and (87.92,186.76) .. (86.44,186.76) .. controls (84.96,186.76) and (83.76,185.56) .. (83.76,184.08) -- cycle ;
\draw [color={rgb, 255:red, 95; green, 176; blue, 0 }  ,draw opacity=1 ]   (83.28,202.89) .. controls (123.28,172.89) and (142.59,204.45) .. (182.59,174.45) ;
\draw [color={rgb, 255:red, 95; green, 176; blue, 0 }  ,draw opacity=1 ]   (69.61,117.22) .. controls (109.61,87.22) and (105.79,117.25) .. (145.79,87.25) ;
\draw    (213.24,92.64) .. controls (242.71,57.82) and (261.99,47.62) .. (311.62,42.22) ;
\draw [shift={(313.91,41.98)}, rotate = 174.07] [fill={rgb, 255:red, 0; green, 0; blue, 0 }  ][line width=0.08]  [draw opacity=0] (10.72,-5.15) -- (0,0) -- (10.72,5.15) -- (7.12,0) -- cycle    ;
\draw    (218.58,230.64) .. controls (264.31,260.7) and (274.19,258.74) .. (319.12,271.82) ;
\draw [shift={(321.91,272.64)}, rotate = 196.48] [fill={rgb, 255:red, 0; green, 0; blue, 0 }  ][line width=0.08]  [draw opacity=0] (10.72,-5.15) -- (0,0) -- (10.72,5.15) -- (7.12,0) -- cycle    ;
\draw [draw opacity=0][fill={rgb, 255:red, 95; green, 176; blue, 0 }  ,fill opacity=0.1 ]   (347.84,64.78) .. controls (389.64,55.84) and (405.32,48.26) .. (406.69,29.74) .. controls (416.26,30.29) and (462.49,24.58) .. (480.7,27.42) .. controls (485.7,43.57) and (487.99,110) .. (488.01,133.91) .. controls (474.05,135.17) and (433.19,131.13) .. (403.95,137.15) .. controls (399.39,109.37) and (374.89,84.25) .. (346.23,89.93) ;
\draw [color={rgb, 255:red, 245; green, 166; blue, 35 }  ,draw opacity=1 ]   (346.23,89.93) .. controls (370.84,84.06) and (407.76,107.88) .. (404.6,152.39) ;
\draw [shift={(387.24,103.59)}, rotate = 47.2] [fill={rgb, 255:red, 245; green, 166; blue, 35 }  ,fill opacity=1 ][line width=0.08]  [draw opacity=0] (6.25,-3) -- (0,0) -- (6.25,3) -- cycle    ;
\draw [color={rgb, 255:red, 245; green, 166; blue, 35 }  ,draw opacity=1 ]   (347.84,64.78) .. controls (372.79,58.37) and (411.8,53.03) .. (406.35,23.83) ;
\draw [shift={(389.22,52.5)}, rotate = 156.41] [fill={rgb, 255:red, 245; green, 166; blue, 35 }  ,fill opacity=1 ][line width=0.08]  [draw opacity=0] (6.25,-3) -- (0,0) -- (6.25,3) -- cycle    ;
\draw [color={rgb, 255:red, 245; green, 166; blue, 35 }  ,draw opacity=1 ]   (488.42,153.28) .. controls (489.04,103.19) and (484.26,32.48) .. (479.49,17.75) ;
\draw [shift={(486.56,81.71)}, rotate = 86.84] [fill={rgb, 255:red, 245; green, 166; blue, 35 }  ,fill opacity=1 ][line width=0.08]  [draw opacity=0] (6.25,-3) -- (0,0) -- (6.25,3) -- cycle    ;
\draw [color={rgb, 255:red, 95; green, 176; blue, 0 }  ,draw opacity=1 ]   (393.46,30.17) .. controls (417.97,30.8) and (479.64,21.96) .. (490.86,29.96) ;
\draw [color={rgb, 255:red, 95; green, 176; blue, 0 }  ,draw opacity=1 ]   (386.61,141.91) .. controls (415.27,128.87) and (484.63,136.23) .. (498.33,133.29) ;
\draw [color={rgb, 255:red, 245; green, 166; blue, 35 }  ,draw opacity=1 ]   (441.59,154.44) .. controls (442.21,104.35) and (418.56,56.33) .. (432.66,18.91) ;
\draw [shift={(432.32,83.3)}, rotate = 79.61] [fill={rgb, 255:red, 245; green, 166; blue, 35 }  ,fill opacity=1 ][line width=0.08]  [draw opacity=0] (6.25,-3) -- (0,0) -- (6.25,3) -- cycle    ;
\draw  [draw opacity=0][fill={rgb, 255:red, 0; green, 0; blue, 0 }  ,fill opacity=1 ] (427.8,28.53) .. controls (427.8,27.36) and (428.75,26.42) .. (429.92,26.42) .. controls (431.09,26.42) and (432.04,27.36) .. (432.04,28.53) .. controls (432.04,29.71) and (431.09,30.65) .. (429.92,30.65) .. controls (428.75,30.65) and (427.8,29.71) .. (427.8,28.53) -- cycle ;
\draw  [draw opacity=0][fill={rgb, 255:red, 0; green, 0; blue, 0 }  ,fill opacity=1 ] (479.62,27.08) .. controls (479.62,25.91) and (480.57,24.96) .. (481.74,24.96) .. controls (482.91,24.96) and (483.86,25.91) .. (483.86,27.08) .. controls (483.86,28.25) and (482.91,29.2) .. (481.74,29.2) .. controls (480.57,29.2) and (479.62,28.25) .. (479.62,27.08) -- cycle ;
\draw  [draw opacity=0][fill={rgb, 255:red, 0; green, 0; blue, 0 }  ,fill opacity=1 ] (404.57,29.74) .. controls (404.57,28.57) and (405.52,27.62) .. (406.69,27.62) .. controls (407.86,27.62) and (408.81,28.57) .. (408.81,29.74) .. controls (408.81,30.91) and (407.86,31.86) .. (406.69,31.86) .. controls (405.52,31.86) and (404.57,30.91) .. (404.57,29.74) -- cycle ;
\draw  [draw opacity=0][fill={rgb, 255:red, 0; green, 0; blue, 0 }  ,fill opacity=1 ] (438.53,134.06) .. controls (438.53,132.89) and (439.48,131.94) .. (440.65,131.94) .. controls (441.82,131.94) and (442.77,132.89) .. (442.77,134.06) .. controls (442.77,135.23) and (441.82,136.18) .. (440.65,136.18) .. controls (439.48,136.18) and (438.53,135.23) .. (438.53,134.06) -- cycle ;
\draw  [draw opacity=0][fill={rgb, 255:red, 0; green, 0; blue, 0 }  ,fill opacity=1 ] (486.44,134) .. controls (486.44,132.83) and (487.39,131.88) .. (488.56,131.88) .. controls (489.73,131.88) and (490.68,132.83) .. (490.68,134) .. controls (490.68,135.17) and (489.73,136.12) .. (488.56,136.12) .. controls (487.39,136.12) and (486.44,135.17) .. (486.44,134) -- cycle ;
\draw  [draw opacity=0][fill={rgb, 255:red, 0; green, 0; blue, 0 }  ,fill opacity=1 ] (401.83,137.15) .. controls (401.83,135.98) and (402.78,135.03) .. (403.95,135.03) .. controls (405.12,135.03) and (406.07,135.98) .. (406.07,137.15) .. controls (406.07,138.32) and (405.12,139.27) .. (403.95,139.27) .. controls (402.78,139.27) and (401.83,138.32) .. (401.83,137.15) -- cycle ;
\draw  [draw opacity=0][fill={rgb, 255:red, 155; green, 155; blue, 155 }  ,fill opacity=0.1 ] (340.09,267.05) .. controls (340.09,209.75) and (373.36,190.17) .. (466.36,190.17) .. controls (559.36,190.17) and (564.72,214.64) .. (565.29,274.25) .. controls (565.86,333.86) and (516.61,358.12) .. (444.09,359.85) .. controls (371.56,361.59) and (340.09,324.34) .. (340.09,267.05) -- cycle ;
\draw [draw opacity=0][fill={rgb, 255:red, 95; green, 176; blue, 0 }  ,fill opacity=0.1 ]   (351.84,253.78) .. controls (393.64,244.84) and (409.32,237.26) .. (410.69,218.74) .. controls (420.26,219.29) and (466.49,213.58) .. (484.7,216.42) .. controls (486.17,241.06) and (481,254.34) .. (524.92,249.06) .. controls (527,267.34) and (525.98,274.8) .. (526.42,282.56) .. controls (484.42,284.31) and (492,310.83) .. (492.01,322.91) .. controls (478.05,324.17) and (437.19,320.13) .. (407.95,326.15) .. controls (403.39,298.37) and (378.89,273.25) .. (350.23,278.93) ;
\draw [color={rgb, 255:red, 245; green, 166; blue, 35 }  ,draw opacity=1 ]   (350.23,278.93) .. controls (374.84,273.06) and (411.76,296.88) .. (408.6,341.39) ;
\draw [shift={(391.24,292.59)}, rotate = 47.2] [fill={rgb, 255:red, 245; green, 166; blue, 35 }  ,fill opacity=1 ][line width=0.08]  [draw opacity=0] (6.25,-3) -- (0,0) -- (6.25,3) -- cycle    ;
\draw [color={rgb, 255:red, 245; green, 166; blue, 35 }  ,draw opacity=1 ]   (351.84,253.78) .. controls (376.79,247.37) and (415.8,242.03) .. (410.35,212.83) ;
\draw [shift={(393.22,241.5)}, rotate = 156.41] [fill={rgb, 255:red, 245; green, 166; blue, 35 }  ,fill opacity=1 ][line width=0.08]  [draw opacity=0] (6.25,-3) -- (0,0) -- (6.25,3) -- cycle    ;
\draw [color={rgb, 255:red, 245; green, 166; blue, 35 }  ,draw opacity=1 ]   (524.92,249.06) .. controls (470.42,257.31) and (488.17,223.81) .. (486.17,204.81) ;
\draw [shift={(486.85,239.96)}, rotate = 54.48] [fill={rgb, 255:red, 245; green, 166; blue, 35 }  ,fill opacity=1 ][line width=0.08]  [draw opacity=0] (6.25,-3) -- (0,0) -- (6.25,3) -- cycle    ;
\draw [color={rgb, 255:red, 95; green, 176; blue, 0 }  ,draw opacity=1 ]   (397.46,219.17) .. controls (421.97,219.8) and (483.64,210.96) .. (494.86,218.96) ;
\draw [color={rgb, 255:red, 95; green, 176; blue, 0 }  ,draw opacity=1 ]   (390.61,330.91) .. controls (419.27,317.87) and (488.63,325.23) .. (502.33,322.29) ;
\draw [color={rgb, 255:red, 245; green, 166; blue, 35 }  ,draw opacity=1 ]   (445.59,343.44) .. controls (446.21,293.35) and (422.56,245.33) .. (436.66,207.91) ;
\draw [shift={(436.32,272.3)}, rotate = 79.61] [fill={rgb, 255:red, 245; green, 166; blue, 35 }  ,fill opacity=1 ][line width=0.08]  [draw opacity=0] (6.25,-3) -- (0,0) -- (6.25,3) -- cycle    ;
\draw  [draw opacity=0][fill={rgb, 255:red, 0; green, 0; blue, 0 }  ,fill opacity=1 ] (431.8,217.53) .. controls (431.8,216.36) and (432.75,215.42) .. (433.92,215.42) .. controls (435.09,215.42) and (436.04,216.36) .. (436.04,217.53) .. controls (436.04,218.71) and (435.09,219.65) .. (433.92,219.65) .. controls (432.75,219.65) and (431.8,218.71) .. (431.8,217.53) -- cycle ;
\draw  [draw opacity=0][fill={rgb, 255:red, 0; green, 0; blue, 0 }  ,fill opacity=1 ] (483.62,216.08) .. controls (483.62,214.91) and (484.57,213.96) .. (485.74,213.96) .. controls (486.91,213.96) and (487.86,214.91) .. (487.86,216.08) .. controls (487.86,217.25) and (486.91,218.2) .. (485.74,218.2) .. controls (484.57,218.2) and (483.62,217.25) .. (483.62,216.08) -- cycle ;
\draw  [draw opacity=0][fill={rgb, 255:red, 0; green, 0; blue, 0 }  ,fill opacity=1 ] (408.57,218.74) .. controls (408.57,217.57) and (409.52,216.62) .. (410.69,216.62) .. controls (411.86,216.62) and (412.81,217.57) .. (412.81,218.74) .. controls (412.81,219.91) and (411.86,220.86) .. (410.69,220.86) .. controls (409.52,220.86) and (408.57,219.91) .. (408.57,218.74) -- cycle ;
\draw  [draw opacity=0][fill={rgb, 255:red, 0; green, 0; blue, 0 }  ,fill opacity=1 ] (442.53,323.06) .. controls (442.53,321.89) and (443.48,320.94) .. (444.65,320.94) .. controls (445.82,320.94) and (446.77,321.89) .. (446.77,323.06) .. controls (446.77,324.23) and (445.82,325.18) .. (444.65,325.18) .. controls (443.48,325.18) and (442.53,324.23) .. (442.53,323.06) -- cycle ;
\draw  [draw opacity=0][fill={rgb, 255:red, 0; green, 0; blue, 0 }  ,fill opacity=1 ] (405.83,326.15) .. controls (405.83,324.98) and (406.78,324.03) .. (407.95,324.03) .. controls (409.12,324.03) and (410.07,324.98) .. (410.07,326.15) .. controls (410.07,327.32) and (409.12,328.27) .. (407.95,328.27) .. controls (406.78,328.27) and (405.83,327.32) .. (405.83,326.15) -- cycle ;
\draw [color={rgb, 255:red, 245; green, 166; blue, 35 }  ,draw opacity=1 ]   (493.47,338.33) .. controls (487.92,292.06) and (494.42,286.06) .. (526.42,282.56) ;
\draw [shift={(494.89,296.11)}, rotate = 109.28] [fill={rgb, 255:red, 245; green, 166; blue, 35 }  ,fill opacity=1 ][line width=0.08]  [draw opacity=0] (6.25,-3) -- (0,0) -- (6.25,3) -- cycle    ;
\draw  [draw opacity=0][fill={rgb, 255:red, 0; green, 0; blue, 0 }  ,fill opacity=1 ] (490.53,323) .. controls (490.53,321.83) and (491.48,320.88) .. (492.65,320.88) .. controls (493.82,320.88) and (494.77,321.83) .. (494.77,323) .. controls (494.77,324.17) and (493.82,325.12) .. (492.65,325.12) .. controls (491.48,325.12) and (490.53,324.17) .. (490.53,323) -- cycle ;
\draw  [draw opacity=0][fill={rgb, 255:red, 189; green, 16; blue, 224 }  ,fill opacity=1 ] (124.93,125.58) .. controls (124.93,124.1) and (126.13,122.9) .. (127.61,122.9) .. controls (129.09,122.9) and (130.29,124.1) .. (130.29,125.58) .. controls (130.29,127.06) and (129.09,128.26) .. (127.61,128.26) .. controls (126.13,128.26) and (124.93,127.06) .. (124.93,125.58) -- cycle ;

\draw (103.22,172.02) node  [font=\tiny]  {$\times $};
\draw (102.58,175) node [anchor=north] [inner sep=0.75pt]    {$z_{n}$};
\draw (135.06,148.68) node  [font=\tiny]  {$\times $};
\draw (139.17,152.98) node [anchor=north] [inner sep=0.75pt]    {$f( z_{n})$};
\draw (146.02,191.78) node [anchor=north] [inner sep=0.75pt]  [color={rgb, 255:red, 88; green, 162; blue, 1 }  ,opacity=1 ]  {$\gamma _{+}$};
\draw (144.4,88.92) node [anchor=north west][inner sep=0.75pt]  [color={rgb, 255:red, 88; green, 162; blue, 1 }  ,opacity=1 ]  {$\gamma _{-}$};
\draw (141,240.23) node [anchor=south west] [inner sep=0.75pt]  [color={rgb, 255:red, 210; green, 142; blue, 30 }  ,opacity=1 ]  {$\phi _{n}$};
\draw (181.27,212.39) node [anchor=south west] [inner sep=0.75pt]  [color={rgb, 255:red, 210; green, 142; blue, 30 }  ,opacity=1 ]  {$\phi _{1}$};
\draw (96.93,221.12) node [anchor=north east] [inner sep=0.75pt]  [color={rgb, 255:red, 210; green, 142; blue, 30 }  ,opacity=1 ]  {$\phi ^{+}$};
\draw (62.18,84.12) node [anchor=north east] [inner sep=0.75pt]  [color={rgb, 255:red, 210; green, 142; blue, 30 }  ,opacity=1 ]  {$\phi ^{-}$};
\draw (264.67,52.4) node [anchor=south east] [inner sep=0.75pt]    {$\hat{\phi }_{1}^{+} =\hat{\phi }_{1}^{-}$};
\draw (272,257.4) node [anchor=north east] [inner sep=0.75pt]    {$\hat{\phi }_{1}^{+} \neq \hat{\phi }_{1}^{-}$};
\draw (409.57,73.92) node [anchor=north] [inner sep=0.75pt]    {$\hat{z}_{n}$};
\draw (452.58,45.37) node [anchor=north west][inner sep=0.75pt]  [font=\tiny]  {$\times $};
\draw (460.77,50.51) node [anchor=north] [inner sep=0.75pt]    {$\hat{f}(\hat{z}_{n})$};
\draw (403.57,66.8) node [anchor=north west][inner sep=0.75pt]  [font=\tiny]  {$\times $};
\draw (440.17,104.44) node [anchor=west] [inner sep=0.75pt]  [color={rgb, 255:red, 210; green, 142; blue, 30 }  ,opacity=1 ]  {$\hat{\phi }_{n}$};
\draw (382.55,98.05) node [anchor=north east] [inner sep=0.75pt]  [color={rgb, 255:red, 210; green, 142; blue, 30 }  ,opacity=1 ]  {$\hat{\phi }^{+}$};
\draw (384.28,50.64) node [anchor=south east] [inner sep=0.75pt]  [color={rgb, 255:red, 210; green, 142; blue, 30 }  ,opacity=1 ]  {$\hat{\phi }^{-}$};
\draw (492.18,80.62) node [anchor=west] [inner sep=0.75pt]  [color={rgb, 255:red, 210; green, 142; blue, 30 }  ,opacity=1 ]  {$\hat{\phi }_{1}^{+} =\hat{\phi }_{1}^{-}$};
\draw (503.57,137.06) node [anchor=west] [inner sep=0.75pt]  [color={rgb, 255:red, 88; green, 162; blue, 1 }  ,opacity=1 ]  {$\hat{\gamma }_{+}$};
\draw (492.86,33.36) node [anchor=north west][inner sep=0.75pt]  [color={rgb, 255:red, 88; green, 162; blue, 1 }  ,opacity=1 ]  {$\hat{\gamma }_{-}$};
\draw (373.51,75.72) node [anchor=west] [inner sep=0.75pt]  [color={rgb, 255:red, 88; green, 162; blue, 1 }  ,opacity=1 ]  {$\hat D$};
\draw (410.35,27.37) node [anchor=south east] [inner sep=0.75pt]    {$\hat{x}^{-}$};
\draw (434.8,25.48) node [anchor=south west] [inner sep=0.75pt]    {$\hat{x}_{n}^{-}$};
\draw (483.74,25.8) node [anchor=south west] [inner sep=0.75pt]    {$\hat{x}_{1}^{-}$};
\draw (440.5,135.28) node [anchor=north east] [inner sep=0.75pt]    {$\hat{x}_{n}^{+}$};
\draw (405.75,157) node [anchor=south east] [inner sep=0.75pt]    {$\hat{x}^{+}$};
\draw (487.99,133.88) node [anchor=north east] [inner sep=0.75pt]    {$\hat{x}_{1}^{+}$};
\draw (413.57,262.92) node [anchor=north] [inner sep=0.75pt]    {$\hat{z}_{n}$};
\draw (446.98,264.77) node [anchor=north west][inner sep=0.75pt]  [font=\tiny]  {$\times $};
\draw (456.67,265.8) node [anchor=west] [inner sep=0.75pt]    {$\hat{f}(\hat{z}_{n})$};
\draw (407.57,255.8) node [anchor=north west][inner sep=0.75pt]  [font=\tiny]  {$\times $};
\draw (445.37,301.84) node [anchor=west] [inner sep=0.75pt]  [color={rgb, 255:red, 210; green, 142; blue, 30 }  ,opacity=1 ]  {$\hat{\phi }_{n}$};
\draw (386.55,287.05) node [anchor=north east] [inner sep=0.75pt]  [color={rgb, 255:red, 210; green, 142; blue, 30 }  ,opacity=1 ]  {$\hat{\phi }^{+}$};
\draw (388.28,239.64) node [anchor=south east] [inner sep=0.75pt]  [color={rgb, 255:red, 210; green, 142; blue, 30 }  ,opacity=1 ]  {$\hat{\phi }^{-}$};
\draw (501,229) node [anchor=west] [inner sep=0.75pt]  [color={rgb, 255:red, 210; green, 142; blue, 30 }  ,opacity=1 ]  {$\hat{\phi }_{1}^{-}$};
\draw (507.57,326.06) node [anchor=west] [inner sep=0.75pt]  [color={rgb, 255:red, 88; green, 162; blue, 1 }  ,opacity=1 ]  {$\hat{\gamma }_{+}$};
\draw (465.2,217.11) node [anchor=north] [inner sep=0.75pt]  [color={rgb, 255:red, 88; green, 162; blue, 1 }  ,opacity=1 ]  {$\hat{\gamma }_{-}$};
\draw (377.51,264.72) node [anchor=west] [inner sep=0.75pt]  [color={rgb, 255:red, 88; green, 162; blue, 1 }  ,opacity=1 ]  {$\hat D$};
\draw (414.35,216.37) node [anchor=south east] [inner sep=0.75pt]    {$\hat{x}^{-}$};
\draw (438.8,214.48) node [anchor=south west] [inner sep=0.75pt]    {$\hat{x}_{n}^{-}$};
\draw (487.74,214.8) node [anchor=south west] [inner sep=0.75pt]    {$\hat{x}_{1}^{-}$};
\draw (444.5,324.28) node [anchor=north east] [inner sep=0.75pt]    {$\hat{x}_{n}^{+}$};
\draw (409.75,347) node [anchor=south east] [inner sep=0.75pt]    {$\hat{x}^{+}$};
\draw (491.99,322.88) node [anchor=north east] [inner sep=0.75pt]    {$\hat{x}_{1}^{+}$};
\draw (510,288) node [anchor=north west][inner sep=0.75pt]  [color={rgb, 255:red, 210; green, 142; blue, 30 }  ,opacity=1 ]  {$\hat{\phi }_{1}^{+}$};
\draw (201.33,142.52) node [anchor=west] [inner sep=0.75pt]    {$V_{1}$};

\end{tikzpicture}

\caption{The two cases of case \ref{case3technic}.\ of the proof of Proposition~\ref{LemLocalTransverse}, whether there is a singularity of $\F$ such as the violet one (in the disk delimited by the leaves $\phi_1$ and $\phi_n$, and the transverse segments $\gamma_+$ and $\gamma_-$): if there does not exist some then $\hat{\phi }_{1}^{+} =\hat{\phi }_{1}^{-}$, otherwise $\hat{\phi }_{1}^{+} \neq \hat{\phi }_{1}^{-}$. The two other blue points are singularities.\label{FigDescripD}}
\end{center}
\end{figure}
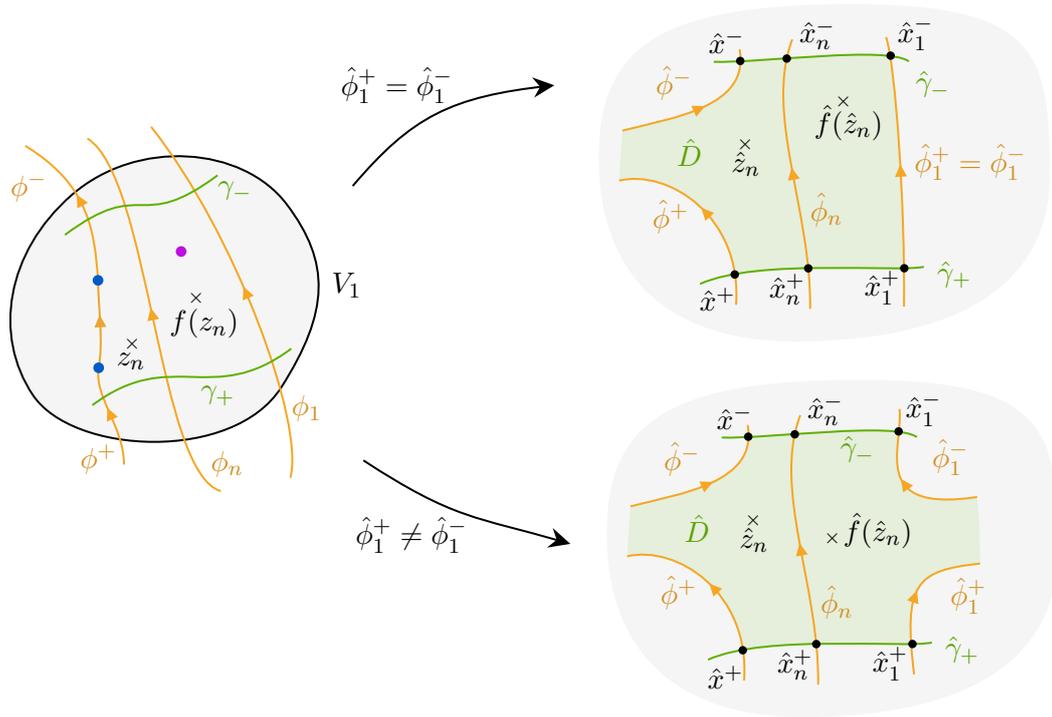

Fix $n\in\N$, $n\ge 2$, and let $\wh\gamma^+$ and $\wh\gamma^-$ be the lifts of the arcs $\gamma^+$ and $\gamma^-$ to $\wh V_1$ intersecting $\check\phi_{n}$. We denote $\wh x_1^+$ and $\wh x_1^-$ the lifts of $x_1^+$ and $x_1^-$ belonging to respectively $\wh\gamma^+$ and $\wh\gamma^-$, and $\wh x^+$ and $\wh x^-$ the lifts of $x^+$ and $x^-$ belonging to respectively $\wh\gamma^+$ and $\wh\gamma^-$. Let $\wh\phi_1^+$, $\wh\phi_1^-$, $\wh\phi^+$, $\wh\phi^-$ the lifts of respectively $\phi_1$, $\phi_1$, $\phi^+$ and $\phi^-$ meeting respectively $\wh x_1^+$, $\wh x_1^-$, $\wh x^+$, $\wh x^-$. Note that all these lifts do depend on $n$.

As $z_n$ tends to $z_\infty\in \sing\F$, for $n$ large enough the trajectory $(I^t(z_n))_{t\in[0,1]}$ does not meet $(\phi_1)_{x_1^-}^- \cap (\phi_1)_{x_1^+}^+$ nor $\gamma^+\cup\gamma^-$.

Recall that we have supposed by contradiction that $\wh z_n\in \overline{L(\check\phi_n)}$ and $\wh f(\wh z_n) \notin \overline{L(\check\phi_n)}$. We have three cases.
\begin{itemize}
\item Either $\wh f(\wh z_n) \in {R(\check\phi^+)}$. Recall that $\wh z_n \in \overline{L(\check\phi_n)}\subset {L(\check\phi^+)}$. One can apply Claim~\ref{ClaimTechnical} to get a contradiction for $n$ large enough.
\item Or $\wh f(\wh z_n) \in {R(\check\phi^-)}$. Recall that $\wh z_n \in \overline{L(\check\phi_n)}\subset {L(\check\phi^-)}$. Once again one can apply Claim~\ref{ClaimTechnical} to get a contradiction for $n$ large enough.
\item Or $\wh f(\wh z_n) \in \overline{L(\check\phi^+)}\cap \overline{L(\check\phi^-)}$. Let $\wh D$ be the domain containing $(\wh \phi_n)_{\wh x_n^+}^+ \cap (\wh \phi_n)_{\wh x_n^-}^-$ and bounded by (see Figure~\ref{FigDescripD})
\[\begin{cases}
(\wh\phi^+)_{\wh x^+}^+, (\wh\phi^-)_{\wh x^-}^-, (\wh \phi_1^+)_{\wh x_1^+}^+ \cap (\wh \phi_1^-)_{\wh x_1^-}^-, \wh\gamma^+|_{[\wh x_1^+, \wh x^+]} \text{ and }\wh\gamma^-|_{[\wh x_1^-, \wh x^-]}\qquad & \text{if } \wh\phi_1^+ = \wh\phi_1^- \\
(\wh\phi^+)_{\wh x^+}^+, (\wh\phi^-)_{\wh x^-}^-, (\wh \phi_1^+)_{\wh x_1^+}^+,  (\wh \phi_1^-)_{\wh x_1^-}^-, \wh\gamma^+|_{[\wh x_1^+, \wh x^+]} \text{ and }\wh\gamma^-|_{[\wh x_1^-, \wh x^-]}\qquad & \text{if } \wh\phi_1^+ \neq \wh\phi_1^- .
\end{cases}\]
Then both $\wh z_n$ and $\wh f(\wh z_n)$ belong to $\wh D$: because the trajectory $(I^t(z_n))_{t\in[0,1]}$ does not meet $(\phi_1)_{y_1^-}^- \cap (\phi_1)_{y_1^+}^+$ nor $\gamma^+\cup\gamma^-$, we have $\wh z_n, \wh f(\wh z_n)\in {R(\check\phi_1^+)}\cap {R(\check\phi_1^-)}$. On the other hand, the half leaves $(\wh\phi_n)_{\wh x_n^+}^-$ and $(\wh\phi_n)_{\wh x_n^-}^+$ cannot meet the domain $\wh D$ (for instance, $(\wh\phi_n)_{\wh x_n^+}^-$ is included in the domain bounded by\footnote{Note that we use the fact that, as transverse paths, the arcs $\wh\gamma^-$ and $\wh\gamma^-$ can meet leaves at most once.} $(\wh\phi^+)_{\wh x^+}^+$, $\wh\gamma^+|_{[\wh x_1^+, \wh x^+]}$ and $(\wh \phi_1^+)_{\wh x_1^+}^-$). This contradicts the existence of an $\wh\F$-transverse trajectory going from $\wh z_n$ to $\wh f(\wh z_n)$.
\end{itemize}
\end{proof}

\bibliographystyle{alpha}
\bibliography{biblio} 

\begin{thebibliography}{AZdPJ21}

\bibitem[ABP23]{MR4578317}
Juan Alonso, Joaqu\'{\i}n Brum, and Alejandro Passeggi.
\newblock On the structure of rotation sets in hyperbolic surfaces.
\newblock {\em J. Lond. Math. Soc. (2)}, 107(4):1173--1241, 2023.

\bibitem[AG74]{transitiveflows}
Samuil Aranson and Vyacheslav~Z. Grines.
\newblock On some invariants of dynamical systems on two-dimensional manifolds
  (necessary and sufficient conditions for the topological equivalence of
  transitive dynamical systems).
\newblock {\em Math. USSR, Sb.}, 19:365--393, 1974.

\bibitem[AG78]{aransonminimal}
Samuil Aranson and Vyacheslav~Z. Grines.
\newblock On the representation of minimal sets of currents on twodimensional
  manifolds by geodesics.
\newblock {\em Math. USSR, Izv.}, 12:103--124, 1978.

\bibitem[AGZ95]{anosovweil1}
Samuil Aranson, Vyacheslav~Z. Grines, and Evgeny~V. Zhuzhoma.
\newblock On the geometry and topology of flows and foliations on surfaces and
  the {Anosov} problem.
\newblock {\em Sb. Math.}, 186(8):1107--1146, 1995.

\bibitem[AGZ01]{anosovweil2}
Samuil Aranson, Vyacheslav~Z. Grines, and Evgeny~V. Zhuzhoma.
\newblock On {Anosov}-{Weil} problem.
\newblock {\em Topology}, 40(3):475--502, 2001.

\bibitem[AZdPJ21]{MR4190050}
Salvador Addas-Zanata and Bruno de~Paula~Jacoia.
\newblock A condition that implies full homotopical complexity of orbits for
  surface homeomorphisms.
\newblock {\em Ergodic Theory Dynam. Systems}, 41(1):1--47, 2021.

\bibitem[BCLR07]{BCLR}
Fran\c{c}ois B\'{e}guin, Sylvain Crovisier, and Fr\'{e}d\'{e}ric Le~Roux.
\newblock Construction of curious minimal uniquely ergodic homeomorphisms on
  manifolds: the {D}enjoy-{R}ees technique.
\newblock {\em Ann. Sci. \'{E}cole Norm. Sup. (4)}, 40(2):251--308, 2007.

\bibitem[BCLR20]{bguin2016fixed}
Fran\c{c}ois B\'{e}guin, Sylvain Crovisier, and Fr\'{e}d\'{e}ric Le~Roux.
\newblock Fixed point sets of isotopies on surfaces.
\newblock {\em J. Eur. Math. Soc. (JEMS)}, 22(6):1971--2046, 2020.

\bibitem[BDM06]{brodskiy}
N.~Brodskiy, J.~Dydak, and Atish Mitra.
\newblock Svarc-milnor lemma: a proof by definition.
\newblock {\em Topology Proceedings}, 04 2006.

\bibitem[BH09]{bridson}
Martin Bridson and Andr{\'e} Haefliger.
\newblock {\em Metric Spaces of Non-Positive Curvature}, volume 319.
\newblock 01 2009.

\bibitem[BHW24]{bowden2024boundary}
Jonathan Bowden, Sebastian Hensel, and Richard Webb.
\newblock Towards the boundary of the fine curve graph, 2024.

\bibitem[Boy94]{boyland2}
Philip Boyland.
\newblock Topological methods in surface dynamics.
\newblock {\em Topology Appl.}, 58(3):223--298, 1994.

\bibitem[Boy99]{zbMATH01408389}
Philip Boyland.
\newblock Isotopy stability of dynamics on surfaces.
\newblock In {\em Geometry and topology in dynamics.}, pages 17--45.
  Providence, RI: American Mathematical Society, 1999.

\bibitem[Boy00]{boyland}
Philip Boyland.
\newblock New dynamical invariants on hyperbolic manifolds.
\newblock {\em Isr. J. Math.}, 119:253--289, 2000.

\bibitem[Boy09]{zbMATH05634807}
Philip Boyland.
\newblock Transitivity of surface dynamics lifted to {Abelian} covers.
\newblock {\em Ergodic Theory Dyn. Syst.}, 29(5):1417--1449, 2009.

\bibitem[Cal07]{calegari}
Danny Calegari.
\newblock {\em Foliations and the geometry of 3-manifolds}.
\newblock Clarendon press, 2007.

\bibitem[CB88]{casson}
Andrew~J. Casson and Steven~A. Bleiler.
\newblock {\em Automorphisms of surfaces after {N}ielsen and {T}hurston},
  volume~9 of {\em London Mathematical Society Student Texts}.
\newblock Cambridge University Press, Cambridge, 1988.

\bibitem[D{\'a}v18]{zbMATH06914177}
Pablo D{\'a}valos.
\newblock On annular maps of the torus and sublinear diffusion.
\newblock {\em J. Inst. Math. Jussieu}, 17(4):913--978, 2018.

\bibitem[ES22]{erlandsson}
Viveka Erlandsson and Juan Souto.
\newblock {\em Mirzakhani's curve counting and geodesic currents}, volume 345
  of {\em Prog. Math.}
\newblock Cham: Birkh{\"a}user, 2022.

\bibitem[FH12]{frankshandel}
John Franks and Michael Handel.
\newblock Entropy zero area preserving diffeomorphisms of {{\(S^{2}\)}}.
\newblock {\em Geom. Topol.}, 16(4):2187--2284, 2012.

\bibitem[FM12]{farb}
Benson Farb and Dan Margalit.
\newblock {\em A Primer on Mapping Class Groups (PMS-49)}.
\newblock Princeton University Press, 2012.

\bibitem[Fra89]{MR0958891}
John Franks.
\newblock Realizing rotation vectors for torus homeomorphisms.
\newblock {\em Trans. Amer. Math. Soc.}, 311(1):107--115, 1989.

\bibitem[Fra96]{zbMATH00914982}
John Franks.
\newblock Rotation vectors and fixed points of area preserving surface
  diffeomorphisms.
\newblock {\em Trans. Am. Math. Soc.}, 348(7):2637--2662, 1996.

\bibitem[GKT14]{zbMATH06304088}
Nancy Guelman, Andres Koropecki, and Fabio~Armando Tal.
\newblock A characterization of annularity for area-preserving toral
  homeomorphisms.
\newblock {\em Math. Z.}, 276(3-4):673--689, 2014.

\bibitem[GLCP23]{guiheneuf2023area}
Pierre-Antoine Guihéneuf, Patrice Le~Calvez, and Alejandro Passeggi.
\newblock Area preserving homeomorphisms of surfaces with rational rotational
  direction, 2023.

\bibitem[GM22]{pa}
Pierre-Antoine Guihéneuf and Emmanuel Militon.
\newblock Homotopic rotation sets for higher genus surfaces.
\newblock {\em to appear in Memoirs of the AMS}, 2022.

\bibitem[GM23]{guiheneuf2023hyperbolic}
Pierre-Antoine Guihéneuf and Emmanuel Militon.
\newblock Hyperbolic isometries of the fine curve graph of higher genus
  surfaces, 2023.

\bibitem[Gui21]{guiheneufforcing}
Pierre-Antoine Guih{\'e}neuf.
\newblock Forcing theory for surface homeomorphisms (after {Le} {Calvez} and
  {Tal}).
\newblock In {\em S\'eminaire Bourbaki. Volume 2019/2021. Expos\'es
  1166--1180.}, pages 159--181, ex. SMF, 2021.

\bibitem[Ham66]{hamstrom}
Mary-Elizabeth Hamstrom.
\newblock Homotopy groups of the space of homeomorphisms on a 2-manifold.
\newblock {\em Ill. J. Math.}, 10:563--573, 1966.

\bibitem[Han85]{zbMATH03921585}
Michael Handel.
\newblock Global shadowing of pseudo-{Anosov} homeomorphisms.
\newblock {\em Ergodic Theory Dyn. Syst.}, 5:373--377, 1985.

\bibitem[Han86]{handel}
Michael Handel.
\newblock Zero entropy surface diffeomorphisms.
\newblock {\em preprint}, 1986.

\bibitem[Han99]{handel1}
Michael Handel.
\newblock A fixed-point theorem for planar homeomorphisms.
\newblock {\em Topology}, 38(2):235--264, 1999.

\bibitem[Hay95]{MR1334719}
Eijirou Hayakawa.
\newblock A sufficient condition for the existence of periodic points of
  homeomorphisms on surfaces.
\newblock {\em Tokyo J. Math.}, 18(1):213--219, 1995.

\bibitem[IIY03]{imayoshi}
Yoichi Imayoshi, Manabu Ito, and Hiroshi Yamamoto.
\newblock On the {Nielsen}-{Thurston}-{Bers} type of some self-maps of
  {Riemann} surfaces with two specified points.
\newblock {\em Osaka J. Math.}, 40(3):659--685, 2003.

\bibitem[Kat80]{MR573822}
A.~Katok.
\newblock Lyapunov exponents, entropy and periodic orbits for diffeomorphisms.
\newblock {\em Inst. Hautes \'{E}tudes Sci. Publ. Math.}, (51):137--173, 1980.

\bibitem[Kra81]{kra}
Irwin Kra.
\newblock On the {Nielsen}-{Thurston}-{Bers} type of some self-maps of
  {Riemann} surfaces.
\newblock {\em Acta Math.}, 146:231--270, 1981.

\bibitem[Kwa92]{zbMATH00120193}
Jaros{\l}aw Kwapisz.
\newblock Every convex polygon with rational vertices is a rotation set.
\newblock {\em Ergodic Theory Dyn. Syst.}, 12(2):333--339, 1992.

\bibitem[LC05]{lecalvezfoliations}
Patrice Le~Calvez.
\newblock An equivariant foliated version of {Brouwer}'s translation theorem.
\newblock {\em Publ. Math., Inst. Hautes {\'E}tud. Sci.}, 102:1--98, 2005.

\bibitem[LC06]{fixed2}
Patrice Le~Calvez.
\newblock A new proof of {Handel}'s fixed-point theorem.
\newblock {\em Geom. Topol.}, 10:2299--2349, 2006.

\bibitem[LC08]{zbMATH05518893}
Patrice Le~Calvez.
\newblock Why do the periodic points of homeomorphisms of the {Euclidean} plane
  rotate around certain fixed points?
\newblock {\em Ann. Sci. {\'E}c. Norm. Sup{\'e}r. (4)}, 41(1):141--176, 2008.

\bibitem[LC21]{fixed1}
Patrice Le~Calvez.
\newblock Handel's fixed point theorem: a {Morse} theoretical point of view.
\newblock {\em Acta Math. Sci., Ser. B, Engl. Ed.}, 41(6):2149--2172, 2021.

\bibitem[LCT18]{lecalveztalforcing}
Patrice Le~Calvez and Fabio Tal.
\newblock Forcing theory for transverse trajectories of surface homeomorphisms.
\newblock {\em Invent. Math.}, 212(2):619--729, 2018.

\bibitem[LCT22]{lct2}
Patrice Le~Calvez and Fabio Tal.
\newblock Topological horseshoes for surface homeomorphisms.
\newblock {\em Duke Math. J.}, 171(12):2519--2626, 2022.

\bibitem[Lel23]{lellouch}
Gabriel Lellouch.
\newblock On rotation sets for surface homeomorphisms in genus {$\geq2$}.
\newblock {\em M\'{e}m. Soc. Math. Fr. (N.S.)}, (178):iv+121, 2023.

\bibitem[Les11]{lessa}
Pablo Lessa.
\newblock Rotation vectors for homeomorphisms of non-positively curved
  manifolds.
\newblock {\em Nonlinearity}, 24(11):3237--3266, 2011.

\bibitem[LM91]{llibremackay}
Jaume Llibre and Robert~S. Mackay.
\newblock Rotation vectors and entropy for homeomorphisms of the torus isotopic
  to the identity.
\newblock Ergodic {Theory} {Dyn}. {Syst}., 1991.

\bibitem[LR17]{leroux}
Fr{\'e}d{\'e}ric Le~Roux.
\newblock An index for {Brouwer} homeomorphisms and homotopy {Brouwer} theory.
\newblock {\em Ergodic Theory Dyn. Syst.}, 37(2):572--605, 2017.

\bibitem[LT22]{liu2022noncontractible}
Xiao-Chuan Liu and Fabio~Armando Tal.
\newblock On non-contractible periodic orbits and bounded deviations, 2022.

\bibitem[Mat97]{matsumoto}
Shigenori Matsumoto.
\newblock Rotation sets of surface homeomorphisms.
\newblock {\em Bol. Soc. Bras. Mat., Nova S{\'e}r.}, 28(1):89--101, 1997.

\bibitem[McM13]{zbMATH06272204}
Curtis~T. McMullen.
\newblock Diophantine and ergodic foliations on surfaces.
\newblock {\em J. Topol.}, 6(2):349--360, 2013.

\bibitem[MZ90]{zbMATH04084609}
Micha{\l} Misiurewicz and Krystyna Ziemian.
\newblock Rotation sets for maps of tori.
\newblock {\em J. Lond. Math. Soc., II. Ser.}, 40(3):490--506, 1990.

\bibitem[MZ91]{zbMATH00009916}
Micha{\l} Misiurewicz and Krystyna Ziemian.
\newblock Rotation sets and ergodic measures for torus homeomorphisms.
\newblock {\em Fundam. Math.}, 137(1):45--52, 1991.

\bibitem[Nik01]{nikolaev}
Igor Nikolaev.
\newblock {\em Foliations on surfaces}, volume~41 of {\em Ergeb. Math.
  Grenzgeb., 3. Folge}.
\newblock Berlin: Springer, 2001.

\bibitem[Oxt53]{oxtoby}
John Oxtoby.
\newblock Stepanoff flows on the torus.
\newblock {\em Proceedings of the American Mathematical Society},
  4(6):982--987, 1953.

\bibitem[Pol92]{pollicott}
Mark Pollicott.
\newblock Rotation sets for homeomorphisms and homology.
\newblock {\em Trans. Am. Math. Soc.}, 331(2):881--894, 1992.

\bibitem[PPS18]{zbMATH06864334}
Alejandro Passeggi, Rafael Potrie, and Mart{\'{\i}}n Sambarino.
\newblock Rotation intervals and entropy on attracting annular continua.
\newblock {\em Geom. Topol.}, 22(4):2145--2186, 2018.

\bibitem[Sch57]{MR88720}
Sol Schwartzman.
\newblock Asymptotic cycles.
\newblock {\em Ann. of Math. (2)}, 66:270--284, 1957.

\bibitem[See64]{seeley}
Robert Seeley.
\newblock Extension of functions defined in a half space.
\newblock {\em Proceedings of the American Mathematical Society}, 15:625--626,
  1964.

\bibitem[SST22]{zbMATH07488214}
Guilherme Silva~Salom{\~a}o and Fabio~Armando Tal.
\newblock Non-existence of sublinear diffusion for a class of torus
  homeomorphisms.
\newblock {\em Ergodic Theory Dyn. Syst.}, 42(4):1517--1547, 2022.

\bibitem[Sta99]{stark}
Cristopher Stark.
\newblock Blow up and fixed points.
\newblock {\em Contemporary Mathematics}, 246:239--252, 1999.

\bibitem[Ste36]{stepanov}
Vyacheslaw Stepanov.
\newblock Sur une extension du théorem ergodique.
\newblock {\em Compositio Mathematica}, 3:239--253, 1936.

\bibitem[Thu88]{thurston}
William~P. Thurston.
\newblock On the geometry and dynamics of diffeomorphisms of surfaces.
\newblock {\em Bull. Am. Math. Soc., New Ser.}, 19(2):417--431, 1988.

\bibitem[Tra79]{Travaux}
{\em Travaux de {T}hurston sur les surfaces}, volume~66 of {\em
  Ast\'{e}risque}.
\newblock Soci\'{e}t\'{e} Math\'{e}matique de France, Paris, 1979.
\newblock S\'{e}minaire Orsay, With an English summary.

\bibitem[Whi34]{whitney}
Hassler Whitney.
\newblock Analytic extensions of differentiable functions defined in closed
  sets.
\newblock {\em Transactions of the American Mathematical Society},
  36(1):63--89, 1934.

\bibitem[Xav12]{juliana2}
Juliana Xavier.
\newblock Cycles of links and fixed points for orientation preserving
  homeomorphisms of the open unit disk.
\newblock {\em Fundam. Math.}, 219(1):59--96, 2012.

\bibitem[Xav13]{juliana1}
Juliana Xavier.
\newblock Handel's fixed point theorem revisited.
\newblock {\em Ergodic Theory Dyn. Syst.}, 33(5):1584--1610, 2013.

\bibitem[You77]{youngflows}
Lai-Sang Young.
\newblock Entropy of continuous flows on compact 2-manifolds.
\newblock {\em Topology}, 16:469--471, 1977.

\end{thebibliography}
\end{document}